\def\fileversion{v2.25a} \def\filedate{2009/10/10}
\documentclass[edeposit,fullpage]{uiucthesis2009}
\usepackage{amssymb}
\usepackage{mathrsfs}
\usepackage{amsmath}
\usepackage{amsthm}
\usepackage{url}
\usepackage{stackrel} 
\usepackage{courier}
\usepackage[all,cmtip]{xy} 
\usepackage{multicol} 
\usepackage{verbatim} 
\usepackage{mathtools}  
\usepackage{hyperref}
\usepackage{multicol}  
\usepackage{enumerate} 
\usepackage{bookmark} 
\usepackage{ytableau}
\usepackage{microtype}
 
\catcode`~=11 
\newcommand{\urltilde}{\kern -.15em\lower .7ex\hbox{~}\kern .04em}  
\catcode`~=13 

\newcommand{\diag}{\mathop{\mathrm{diag}}\nolimits}

\newcommand{\lie}{\mathop{\mathrm{Lie}}\nolimits}
\newcommand{\bipartite}{\mathop{\mathrm{bipartite}}\nolimits}

\newcommand{\Jordan}{\mathop{\mathrm{Jordan}}\nolimits}
\newcommand{\Proj}{\mathop{\mathrm{Proj}}\nolimits}

\newcommand{\incl}{\mathop{\mathrm{incl}}\nolimits}
\newcommand{\pr}{\mathop{\mathrm{pr}}\nolimits}

\newcommand{\Sympl}{\mathop{\mathrm{Sympl}}\nolimits}
\newcommand{\Ad}{\mathop{\mathrm{Ad}}\nolimits}
\newcommand{\ad}{\mathop{\mathrm{ad}}\nolimits}
 
\newcommand{\codim}{\mathop{\mathrm{codim}}\nolimits}  
  
\newcommand{\im}{\mathop{\mathrm{Im}}\nolimits} 
\newcommand{\ev}{\mathop{\mathrm{ev}}\nolimits} 
\newcommand{\gr}{\mathop{\mathrm{gr}}\nolimits} 
\newcommand{\In}{\mathop{\mathrm{In}}\nolimits} 
\newcommand{\nil}{\mathop{\mathrm{nil}}\nolimits} 
\newcommand{\tors}{\mathop{\mathrm{tors}}\nolimits} 
\newcommand{\I}{\mathop{\mathrm{I}_n}\nolimits} 
\newcommand{\Inoindex}{\mathop{\mathrm{I}}\nolimits}

\newcommand{\coker}{\mathop{\mathrm{coker}}\nolimits}
\newcommand{\rev}{\mathop{\mathrm{rev}}\nolimits}
\newcommand{\lex}{\mathop{\mathrm{lex}}\nolimits}
\newcommand{\spec}{\mathop{\mathrm{Spec}}\nolimits}  
\newcommand{\wt}{\mathop{\mathrm{wt}\,}\nolimits}
\newcommand{\tr}{\mathop{\mathrm{tr}}\nolimits}

\newcommand{\rk}{\mathop{\mathrm{rk}}\nolimits} 
\newcommand{\Tr}{\mathop{\mathrm{tr}}\nolimits}

\newcommand{\Rees}{\mathop{\mathbf{Rees}}\nolimits} 
\newcommand{\Filt}{\mathop{\mathbf{Filt}}\nolimits}
\newcommand{\Id}{\mathop{\mathbf{Id}}\nolimits}
\newcommand{\id}{\mathop{\mathrm{Id}}\nolimits} 
\newcommand{\vect}{\mathop{\mathrm{Vect}}\nolimits} 
\newcommand{\Mod}{\mathop{\mathrm{Mod}}\nolimits} 
\newcommand{\Rep}{\mathop{\mathrm{Rep}}\nolimits} 
\newcommand{\Hom}{\mathop{\mathrm{Hom}}\nolimits} 
\newcommand{\Hilb}{\mathop{\mathrm{Hilb}}\nolimits} 
\newcommand{\Ext}{\mathop{\mathrm{Ext}}\nolimits} 
 
\newcommand{\End}{\mathop{\mathrm{End}}\nolimits}

\makeatletter
\newif\if@borderstar
   \def\bordermatrix{\@ifnextchar*{%
       \@borderstartrue\@bordermatrix@i}{\@borderstarfalse\@bordermatrix@i*}%
   }
   \def\@bordermatrix@i*{\@ifnextchar[{\@bordermatrix@ii}{\@bordermatrix@ii[()]}}
   \def\@bordermatrix@ii[#1]#2{%
   \begingroup
     \m@th\@tempdima8.75\p@\setbox\z@\vbox{%
       \def\cr{\crcr\noalign{\kern 2\p@\global\let\cr\endline }}%
       \ialign {$##$\hfil\kern 2\p@\kern\@tempdima & \thinspace %
       \hfil $##$\hfil && \quad\hfil $##$\hfil\crcr\omit\strut %
       \hfil\crcr\noalign{\kern -\baselineskip}#2\crcr\omit %
       \strut\cr}}%
     \setbox\tw@\vbox{\unvcopy\z@\global\setbox\@ne\lastbox}%
     \setbox\tw@\hbox{\unhbox\@ne\unskip\global\setbox\@ne\lastbox}%
     \setbox\tw@\hbox{%
       $\kern\wd\@ne\kern -\@tempdima\left\@firstoftwo#1%
         \if@borderstar\kern2pt\else\kern -\wd\@ne\fi%
       \global\setbox\@ne\vbox{\box\@ne\if@borderstar\else\kern 2\p@\fi}%
       \vcenter{\if@borderstar\else\kern -\ht\@ne\fi%
         \unvbox\z@\kern-\if@borderstar2\fi\baselineskip}%
         \if@borderstar\kern-2\@tempdima\kern2\p@\else\,\fi\right\@secondoftwo#1 $%
     }\null \;\vbox{\kern\ht\@ne\box\tw@}%
   \endgroup
   }
\makeatother

\DeclarePairedDelimiter{\floor}{\lfloor}{\rfloor}

\newtheorem{theorem}{Theorem}[section] 
\newtheorem{proposition}[theorem]{Proposition}
\newtheorem{lemma}[theorem]{Lemma}

\newtheorem{example}[theorem]{Example}
\newtheorem{definition}[theorem]{Definition}
\newtheorem{remark}[theorem]{Remark} 

\newtheorem{corollary}[theorem]{Corollary}

\newtheorem{notation}[theorem]{Notation}
\newtheorem{weight}[theorem]{Remark (Weight Function)}

\newtheorem{motivation}[theorem]{Motivation}
\newtheorem{construction}[theorem]{Construction}
\newtheorem{trace}[theorem]{Trace}

\newtheorem{strategy}[theorem]{Strategy}
\newtheorem{setup}[theorem]{Set-up}

\newtheorem{monomial}[theorem]{Weighted Monomial Order}

\newtheorem{coordinatering}[theorem]{Systematic Procedure (off-diag coords of $\mathbb{C}[\mu^{-1}(0)^{rss}]$)}

\begin{document}

\title{On semi-invariants of filtered representations of quivers and the cotangent bundle of the enhanced Grothendieck-Springer resolution
        }
\author{Mee Seong Im}
\department{Mathematics}
\phdthesis
\advisor{Thomas Nevins}
\degreeyear{2014}
\committee{	Professor Rinat Kedem, Chair\\
						Professor Thomas Nevins, Director of Research\\
						Professor Maarten Bergvelt\\
						Professor Henry Schenck}
\maketitle 
 
\frontmatter 
 
\begin{abstract}   
We introduce the notion of filtered representations of quivers, 
which is related to usual quiver representations, but is a systematic generalization of conjugacy classes of  
$n\times n$ matrices to (block) upper triangular matrices up to conjugation by invertible (block) upper triangular matrices. 
With this notion in mind, we describe the ring of invariant polynomials for interesting families of quivers, namely, finite $ADE$-Dynkin quivers and affine type $\widetilde{A}$-Dynkin quivers. We then study their relation to an important and fundamental object in representation theory called the Grothendieck-Springer resolution, and we conclude by stating several conjectures, suggesting further research.  
\end{abstract}

\begin{dedication}
 In memoriam of    \\ 
 my Grandfather Heejae Im  (1917 - 1977),  \\ 
 my Grandmother Saun Im  (1916 - 2006),  \\ 
 my Grandfather Young Woon Lee (1907 - 2007),  \\ 
 my Aunt Deok Soon (Lee) Bae (1952 - 1992),  \\
 my Aunt Sam Soon (Lee) Yoon (1941 - 2010),  \\  
  and \\ 
 my Cousin Hye Joo Lee (1974 - 2007).  \\ 
\end{dedication}

\chapter*{Acknowledgments}  
 The author would like to thank her Committee Members Maarten Bergvelt, Rinat Kedem, Tom Nevins, and Hal Schenck for serving on her doctoral dissertation committee and for carefully reading the results in this thesis.   
 She would also like to thank Tom Nevins for proposing a series of interesting open problems in representation theory and algebraic geometry, proving a pathway into the immensely fascinating and delightful world of representation theory.  
 The author would like to acknowledge Maarten Bergvelt, Rinat Kedem, Tom Nevins, and Hal Schenck for positive encouragement and guidance during her studies at the University of Illinois. 

 The author is grateful to Michael DiPasquale, Arindam Roy, Jimmy Jianyun Shan, 
 and Jinhyung To (and others) for everlasting friendship over the years and 
 she would like to thank the Department of Mathematics and the Department of Physics at the University of Georgia, the Department of Mathematics at the University of Birmingham (United Kingdom), and the Department of Mathematics at the University of Illinois at Urbana-Champaign for vigorously and enthusiastically training the author in mathematics and in physics and providing numerous research opportunities in classical and computational algebraic geometry, thermodynamics and semiconductors, logic and nonstandard analysis, and representation theory. 
She is also eternally indebted to Mrs. Aileen Hutchins Lovern (mathematics) and Madame Diane Noonan (professeur de Fran\c{c}ais) for the distinctive and excellent teaching at Parkview High School through classes, mathematics club, club Fran\c{c}ais, Mu Alpha Theta (mathematics honors society), Soci\'et\'e Honoraire de Fran\c{c}ais, evening and weekend outings for mathematics competitions, et pour offrir des possibilit\'{e}s de voyager \`a France \`{a} des fins \'{e}ducatives. 
 Finally, she would like to thank her family and her Grandmother Geum Youn Lee (1917 - present) for their unwaivering patience and constant support.  
  
The author was supported by NSA grant H98230-12-1-0216, by Campus Research Board, and by NSF grant DMS 08-38434.  

  
\tableofcontents  



%

\mainmatter 

\chapter{Introduction}\label{chapter:introduction}

This dissertation studies a refinement of the notion of a representation of a quiver by attaching filtrations of vector spaces to the space of quiver representations and restricting to the subspace of representations that preserves the filtration.  
It is related to very interesting developments in algebraic geometry and representation theory, in particular, to Khovanov-Lauda and Rouquier algebras (KLR-algebras) which are graded algebras whose representation theory is related to categorification of quantum groups (\cite{MR2525917}, \cite{MR2763732}, \cite{Rouquier-2-Kac-Moody-algebras}, \cite{Brundan-quiver-Hecke-algebras}), graded cyclotomic $q$-Schur algebras as a quotient of a convolution algebra in the study of quiver varieties via a subset of quiver representations with a notion of a flag (\cite{Hu-Mathas-quiver-schur-algebras-I} and \cite{Stroppel-Webster-quiver-schur-algebras-q-fock-space}), 
  generalized Grothendieck-Springer resolutions whose fibers are quiver flag varieties (\cite{MR2838836}), etc.    
  In the case of KLR-algebras, KLR-algebras of type $A$ are isomorphic to cyclotomic Hecke algebras (\cite{MR2551762}, \cite{Rouquier-2-Kac-Moody-algebras}), which in turn are isomorphic to cyclotomic quiver Schur algebras (\cite{Stroppel-Webster-quiver-schur-algebras-q-fock-space}, Theorem 6.3).  
   However, these developments will not be treated here. Instead, we focus on the most basic algebraic problem about such spaces, namely to describe their function theory. 
   This amounts to a generalization of classical problems of invariant theory to this new context.    
 
We begin with some definitions.  
Let $Q=(Q_0,Q_1)$ be a quiver with vertices $Q_0$, arrows $Q_1$, and head $Q_1\stackrel{h}{\rightarrow}Q_0$ and tail 
$Q_1\stackrel{t}{\rightarrow}Q_0$ maps.  
A  {\em dimension vector} $\beta$ for a quiver $Q$ is an element of $\mathbb{Z}_{\geq 0}^{Q_0}$; in other words, a dimension vector consists of a choice of non-negative integer $\beta_i$ for each $i\in Q_0$.     
Let
\begin{displaymath}
Rep(Q,\beta) :=
\bigoplus_{a \in Q_1}\Hom_{\mathbb{C}}(\mathbb{C}^{\beta_{t(a)}},\mathbb{C}^{\beta_{h(a)}}). 
 \end{displaymath}
 An element of $Rep(Q,\beta)$ is determined by a choice of linear map $\mathbb{C}^{\beta_{t(a)}}\xrightarrow{A_a} \mathbb{C}^{\beta_{h(a)}}$ for each $a\in Q_1$. 
%
%
There is a natural action of the group
$\displaystyle{\mathbb{G}_{\beta} := \prod_{i\in Q_0} GL_{\beta_i}(\mathbb{C})}$ on $Rep(Q,\beta)$ where each factor $GL_{\beta_i}(\mathbb{C})$ acts via change of basis in $\mathbb{C}^{\beta_i}$.

 We say a quiver is an {\em (affine) Dynkin quiver} if the underlying graph has the structure of an (affine) Dynkin graph.  
We say a quiver is a {\em $k$-Kronecker quiver} if the quiver has two vertices and exactly $k$ arrows (of any direction) between them, and we say a quiver is an {\em $m$-Jordan quiver} if the quiver has exactly one vertex and $m$ arrows at that vertex such that $ta=ha$ for each arrow $a$. 
   
\begin{example}\label{example:1-Jordan-1-Kronecker-quivers}
The quiver 
$\xymatrix@-1pc{\bullet \ar@(ru,rd)}\hspace{6mm}$ is called a {\em $1$-Jordan quiver} while 
$\xymatrix@-1pc{\bullet \ar[rr] & & \bullet}$ is called a {\em $1$-Kronecker quiver} or an {\em $A_2$-Dynkin quiver}.   These are examples of more general classes of quivers described in Section~\ref{subsection:introduction-to-quivers}. 

Let $Q$ denote the $1$-Jordan quiver and  
choose the dimension vector $\beta=n$.  
Then $Rep(Q,\beta)=\mathfrak{gl}_n$, the space of $n\times n$ matrices, and  
$\mathbb{G}_{\beta} = GL_n(\mathbb{C})$ acts by conjugation on $\mathfrak{gl}_n$. 

Let $Q'$ denote the $1$-Kronecker quiver and let $\beta'=(m,n)$. Then $Rep(Q',\beta') = M_{n\times m}$ is the space of 
$n\times m$ matrices, and $\mathbb{G}_{\beta} = GL_m(\mathbb{C})\times GL_n(\mathbb{C})$ acts on $M_{n\times m}$ via 
$(g,h).A=hAg^{-1}$.   
\end{example}

   The study of the space $Rep(Q,\beta)$ with its $\mathbb{G}_{\beta}$-action represents a far-reaching generalization of the classical study of similarity classes of matrices, i.e., the study of linear operators up to change of basis.

\vspace{.5em}

Now, we introduce a new, refined analogue of quiver representations.  Given a quiver $Q$ and a dimension vector $\beta = (\beta_i)_{i\in Q_0}$, choose a sequence $\gamma^1, \gamma^2, \dots, \gamma^N = \beta$ of dimension vectors so that for each $k$ and each vertex $i\in Q_0$,
$\gamma^k_i\leq \gamma^{k+1}_i$.  
For each $i\in Q_0$, one gets a filtration of $\mathbb{C}^{\beta_i}$ by taking
\begin{displaymath}
\mathbb{C}^{\gamma^1_i}\subseteq \mathbb{C}^{\gamma^2_i}\subseteq\dots \subseteq \mathbb{C}^{\beta_i}
\end{displaymath}
to be the standard sequence in which each $\mathbb{C}^k$ is spanned by the first $k$ standard basis vectors.  For example, if
$\beta_i = n$ for every $i$, we could take $\gamma_i^k = k$ for every $i$ and $k$: then we are choosing the standard filtration by coordinate subspaces at each vertex of the quiver $Q$.

Given a sequence 
$\gamma^1, \gamma^2, \dots, \gamma^N = \beta$ of dimension vectors as above,
let $F^{\bullet}Rep(Q,\beta)$ denote the subspace of $Rep(Q,\beta)$ consisting of those representations $(A_a)_{a\in Q_1}$ whose linear maps $A_a$ preserve a fixed sequence of vector spaces at every level: that is, $A_a(\mathbb{C}^{\gamma^k_{t(a)}}) \subseteq \mathbb{C}^{\gamma^k_{h(a)}}$ for every $k$ and $a$.  
Let $P_i\subseteq GL_{\beta_i}(\mathbb{C})$ be the subgroup of linear automorphisms preserving the filtration of vector spaces at vertex $i$; this is a parabolic subgroup. 
Then the product $\mathbb{P}_{\beta}:= \displaystyle{\prod_{i\in Q_0}P_i}$ of parabolic groups acts on 
$F^{\bullet}Rep(Q,\beta)$ as a change-of-basis.

\begin{example}\label{example:1-Jordan-1-Kronecker-filtered-quiver-representations}
Consider the $1$-Jordan quiver. 
Equip $\mathbb{C}^n$ with the complete standard flag of vector spaces,
\begin{displaymath}
\{0\} \subset \mathbb{C}^1\subset \mathbb{C}^2\subset\dots\subset \mathbb{C}^n.
\end{displaymath}
 Then 
$F^{\bullet}Rep(Q,\beta)=\mathfrak{b}$, the vector space of upper triangular matrices, and the group $\mathbb{P}_{\beta}=B$ of invertible upper triangular matrices acts on $\mathfrak{b}$ via conjugation. 

Now consider the $1$-Kronecker quiver and let $m=n$ (so $\beta'=(n,n)$). Let $F^{\bullet}$ be the complete standard filtration of vector spaces at each vertex. Then  
$F^{\bullet}Rep(Q',\beta')=\mathfrak{b}$, and we have  
$\mathbb{P}_{\beta'}=B\times B$ acting on $\mathfrak{b}$ via the left-right action: $(b,d).A=dAb^{-1}$. 
\end{example}  

  The study of the space $F^\bullet Rep(Q,\beta)$ with its $\mathbb{P}_{\beta}$-action thus represents a generalization of the study of upper triangular linear operators up to upper-triangular change of basis. 
 
Next, we introduce some definitions from invariant theory.   
Let $\chi:G\rightarrow \mathbb{C}^*$ be an algebraic group homomorphism from a group $G$ to the multiplicative group $\mathbb{C}^*$ of $\mathbb{C}$; the map $\chi$ is called a {\em character} of $G$.   
\begin{definition}\label{definition:invariant-semi-invariant-polynomials-introduction}  
Let $X$ be a $G$-space and let $f\in \mathbb{C}[X]$.  
If $f(g \cdot x)=f(x)$ for all $g\in G$ and for all $x\in X$, then $f$ is called an 
{\em invariant polynomial}. 
If $f(g\cdot x)=\chi(g)f(x)$ for all $g\in G$ and $x\in X$, where $\chi$ is a character of $G$, then 
$f$ is called a {\em $\chi$-semi-invariant polynomial}. 
\end{definition} 

\begin{definition}\label{definition:invariant-semi-invariant-polynomials-coordinate-rings-introduction}
We define 
\[ \mathbb{C}[X]^G :=\{ f\in \mathbb{C}[X]: f(g\cdot x)=f(x) \mbox{ for all }g\in G \mbox{ and }x\in X\} \hspace{1em} \text{and}
\]  
\[ \mathbb{C}[X]^{G,\chi}:=\{f\in \mathbb{C}[X]:f(g\cdot x)=\chi(g)f(x)\mbox{ for all }g\in G \mbox{ and }x\in X\}. 
\]  
\end{definition}  

The problem of  characterizing the spaces $\mathbb{C}[Rep(Q,\beta)]^{\mathbb{G}_{\beta},\chi}$ for a quiver $Q$ and all algebraic characters $\chi$ of $\mathbb{G}_{\beta}$ is the problem of describing the {\em semi-invariants of quivers}. 

Let $\det^i$ be the $i^{th}$ power of the determinant function, as a polynomial function of a matrix. 
 
\begin{example}\label{example:1Jordan-1Kronecker-classical-setting-invariants} 
If $Q$ is the $1$-Jordan quiver and $\beta=n$, then 
$\mathbb{C}[\mathfrak{gl}_n]^{GL_n(\mathbb{C})}=\mathbb{C}[\tr(A),\ldots, \det(A)]$, the ring of symmetric functions in the eigenvalues of an $n\times n$ matrix, and 
$\displaystyle{\bigoplus_{i\in \mathbb{Z}}} \mathbb{C}[\mathfrak{gl}_n]^{GL_n(\mathbb{C}),\det^i} \cong \mathbb{C}[\mathfrak{gl}_n]^{GL_n(\mathbb{C})}$ since there are no $\det^i$-semi-invariants ($|i|>0$) for the $GL_n(\mathbb{C})$-adjoint action on $\mathfrak{gl}_n$.  
If 
$Q'$ is the $1$-Kronecker quiver and $\beta'=(n,n)$, then $GL_n(\mathbb{C})\times GL_n(\mathbb{C})$-invariants are 
$\mathbb{C}[M_{n\times n}]^{GL_n(\mathbb{C})\times GL_n(\mathbb{C})}=\mathbb{C}$, while semi-invariants are 
$\displaystyle{\bigoplus_{(i,j)\in\mathbb{Z}^2}} \mathbb{C}[M_{n\times n}]^{GL_n(\mathbb{C}) \times GL_n(\mathbb{C}),\det^i\det^j}\cong \mathbb{C}[\det(A)]$. 
\end{example} 
 \noindent
The semi-invariants of quivers were completely described in work of Derksen-Weyman, Schofield-Van den Bergh, and Domokos-Zubkov (cf. \cite{MR1958904}, \cite{MR1908144}, \cite{MR1825166}), following work of many authors over a number of years.  

The principal problem addressed in this dissertation is the characterization of $\mathbb{C}[F^\bullet Rep(Q,\beta)]^{\mathbb{P}_{\beta},\chi}$, i.e., of 
{\em semi-invariants of filtered quivers}.\footnote{Note that this terminology is slightly abusive, since it is the representations, not the quivers themselves, that are filtered.}  The techniques of Derksen-Weyman, Schofield-Van den Bergh, and Domokos-Zubkov are based on classical invariant theory of semisimple or reductive groups and do not apply in the filtered situation.  As a result, the problem seems to be much harder, and we do not achieve a complete solution here.  However, we achieve results in many cases of interest in 
representation theory and algebraic geometry.

As a sample of the results presented here, consider $ADE$-Dynkin quivers and affine type $\widetilde{A}_r$-Dynkin quivers.   
If $Q$ if an $ADE$-Dynkin quiver, we will write $r$ to mean $Q$ has $r$ vertices and $r-1$ arrows.

Let $\mathbb{P}_{\beta} = \displaystyle{\prod_{i\in Q_0}} P_i$ be a product of parabolic groups as defined above. 
 We will write the ring of $\mathbb{P}_{\beta}$-semi-invariant polynomials as 
 $\displaystyle{\bigoplus_{\chi}} \mathbb{C}[F^{\bullet}Rep(Q,\beta)]^{\mathbb{P}_{\beta},\chi}$.   
Let $\mathfrak{t}_n$ be the space of complex diagonal matrices in $\mathfrak{gl}_n$.    
 
\begin{theorem}\label{theorem:filtered-ADE-Dynkin-quiver-introduction} 
\mbox{}
\begin{enumerate}
\item\label{item:ADE-Dynkin-quiver-filtered-si-introduction} If $Q$ is an ADE-Dynkin quiver and  
$\beta=(n,\ldots, n)\in \mathbb{Z}_{\geq 0}^{Q_0}$, 
then  
\begin{displaymath}
\displaystyle{\bigoplus_{\chi}} \mathbb{C}[F^{\bullet}Rep(Q,\beta)]^{\mathbb{P}_{\beta},\chi} \cong \mathbb{C}[\mathfrak{t}_n^{\oplus r-1}].
\end{displaymath} 
\item\label{item:affine-Ar-Dynkin-quiver-filtered-si-introduction} If $Q$ is an affine type $\widetilde{A}_r$-Dynkin quiver and $\beta=(n,\ldots, n)\in \mathbb{Z}_{\geq 0}^{Q_0}$, then 
\begin{displaymath}
\displaystyle{\bigoplus_{\chi}} \mathbb{C}[F^{\bullet}Rep(Q,\beta)]^{\mathbb{P}_{\beta},\chi} \cong \mathbb{C}[\mathfrak{t}_n^{\oplus r+1}].
\end{displaymath}
\end{enumerate}
\end{theorem}

More general versions of this theorem are given in Section~\ref{section:framed-filtered-quiver-var-affine-Dynkin-type} (Theorem~\ref{theorem:invariants-of-framed-affine-Ar-quiver})  
  and 
Section~\ref{section:semi-invariants-quivers-at-most-two-pathways} (Theorem~\ref{theorem:two-paths-max-quiver-semi-invariants} and Theorem~\ref{theorem:two-paths-max-quiver-semi-invariants-framed}).

In Chapter~\ref{chapter:background}, we review some background essential to the results in this thesis.   
In Chapter~\ref{chapter:filtered-quiver-varieties}, we relate filtered quiver varieties to quiver Grassmannians and quiver flag varieties through objects called universal quiver Grassmannian (cf. Section~\ref{subsection:quiver-grassmannians}) 
and universal quiver flag (cf. Section~\ref{subsection:quiver-flag-variety}).   
In Chapter~\ref{chapter:categories-filt-graded-quiver-varieties}, we explain Wolf's construction of reflection functors for quiver flag varieties (\cite{wolf2009geometric}, Theorem 5.16) and in Chapter~\ref{chapter-semi-invariants-filtered-quiver-vars}, we study the ring of semi-invariant polynomials for various families of filtered quiver varieties. 

In Chapter~\ref{chapter:rss-locus}, we study the Hamiltonian reduction of the cotangent bundle of the enhanced Grothendieck-Springer resolution restricted to the regular semisimple locus, 
and in Chapter~\ref{chapter:future-direction}, we conclude this thesis with a precise description of open problems and statements of conjectures, further motivating and suggesting research in the direction of filtered quiver varieties.


\chapter{Background}\label{chapter:background}

\section{Representations of quivers}\label{section:reps-of-quivers-intro} 

We refer to lecture notes by   
Michel Brion \cite{Brion-rep-of-quivers},   
William Crawley-Boevey \cite{Crawley-Boevey-rep-quivers},  
Victor Ginzburg \cite{Ginzburg-Nakajima-quivers},  
and  
Alastair King \cite{MR1315461}  
for an in-depth discussion of algebraic and geometric aspects of quiver varieties. 
Here, we mention some basic properties of quivers.

\subsection{Introduction to quivers}\label{subsection:introduction-to-quivers}  

We begin with a definition of a quiver. 
 
\begin{definition}
A {\em quiver} 
$Q=(Q_0, Q_1)$ is a directed graph: that is, it has a set $Q_0=\{1,2,\ldots, p\}$ of vertices 
and a set $Q_1=\{ a_1,a_2,\ldots, a_q\}$ of arrows, 
which come equipped with two functions: for each arrow  
$\stackrel{i}{\bullet}
            \stackrel{a}{\longrightarrow} \stackrel{j}{\bullet}$, 
            $t:Q_1\rightarrow Q_0$ maps $t(a)=ta=i$ and 
            $h:Q_1\rightarrow Q_0$ maps $h(a)=ha=j$.  We will call $t(a)$ the 
            {\em tail} of arrow $a$ and $h(a)$ the {\em head} of arrow $a$. 
\end{definition}

%
%

\begin{definition} 
 We say a quiver $Q=(Q_0,Q_1)$ is {\em nontrivial} if $|Q_0|\geq 1$,  
 {\em finite} if $|Q_0|<\infty$ and $|Q_1|< \infty$, 
 and {\em connected} if the underlying graph is connected. 
 \end{definition}

In this thesis, we will restrict all discussions to nontrivial, finite, and connected quivers.

\begin{definition}
We say a vertex $i\in Q_0$ is a  
{\em sink}  
 if it is not the head of some arrow of the quiver and 
the vertex is a 
{\em source}  
if it is not the tail of some arrow of the quiver.  
\end{definition} 

\begin{definition}
We say a vertex 
$i\in Q_0$ is {\em $+$-admissible} 
if $i$ is a sink and it is {\em $-$-admissible} if $i$ is a source, i.e.,  
\[ 
\xymatrix@-1pc{
\ar[rrr] & & & \stackrel{+}{\bullet} & & & \ar[lll] \\ 
 & & &   & & &  \\ 
 & \ar[rruu] & & \cdots & & \ar[lluu] & \\ 
}, \hspace{4mm}
\xymatrix@-1pc{
& & &\ar[lll] \ar[lldd] \stackrel{-.}{\bullet} \ar[rrdd] \ar[rrr] & & & \\ 
& & &         & & &   \\  
& & &  \cdots & & &   \\  
} 
\]  
\end{definition}

\begin{definition}\label{definition:finite-ADE-Dynkin-quiver-introduction} 
We say $Q$ is an {\em $ADE$-Dynkin} quiver if the underlying graph of $Q$ is of type $A_r$, $D_r,$ $E_6$, $E_7$, or $E_8$:
\[
  \begin{tabular}{c|c}
  \multicolumn{1}{c}{\bfseries Type} & \multicolumn{1}{c}{\bfseries Graph}   \\ 
  \hline 
  $A_r$ &
  		$\xymatrix{
  		 \stackrel{1}{\bullet}\ar@{-}[r]^{ }    
  & \stackrel{2}{\bullet} \ar@{-}[r]^{ }   
  & \cdots \ar@{-}[r]^{ }  
  & \stackrel{r-1}{\bullet} \ar@{-}[r]^{ }  
  & \stackrel{r}{\bullet}
  		}$    \\ 
  $D_r$ &
    $\xymatrix{  
   & & & & \stackrel{r-1}{\bullet} \\  
  \stackrel{1}{\bullet}	\ar@{-}[r]^{ } &  
	\stackrel{2}{\bullet}	\ar@{-}[r]^{ } &  
	\ldots 								\ar@{-}[r]^{ } &  
\stackrel{r-2}{\bullet} \ar@{-}[ur]^{ } \ar@{-}[dr]^{ } & \\ 
  & & & & \stackrel{r}{\bullet} \\   
}$   \\ 
  $E_6$ & 
$ 
\xymatrix@-1pc{
& & \stackrel{4}{\bullet} & &  \\ 
\stackrel{1}{\bullet}\ar@{-}[r]^{ }  & 
\stackrel{2}{\bullet}\ar@{-}[r]^{ }  & 
\stackrel{3}{\bullet}\ar@{-}[r]^{ } \ar@{-}[u]^{ } & 
\stackrel{5}{\bullet}\ar@{-}[r]^{ }  & 
\stackrel{6}{\bullet} \\ 
}$     \\ 
  $E_7$ & 
  $  \xymatrix@-1pc{
 & & \stackrel{4}{\bullet} & & & \\ 
\stackrel{1}{\bullet}\ar@{-}[r]^{ }  & 
\stackrel{2}{\bullet}\ar@{-}[r]^{ }  & 
\stackrel{3}{\bullet}\ar@{-}[r]^{ } \ar@{-}[u]^{ } & 
\stackrel{5}{\bullet}\ar@{-}[r]^{ }  & 
\stackrel{6}{\bullet}\ar@{-}[r]^{ }  & 
\stackrel{7}{\bullet} \\ 
}$   \\ 
  $E_8$ & $\xymatrix@-1pc{
& & \stackrel{4}{\bullet} & & & &  \\ 
\stackrel{1}{\bullet}\ar@{-}[r]^{ }  & 
\stackrel{2}{\bullet}\ar@{-}[r]^{ }  & 
\stackrel{3}{\bullet}\ar@{-}[r]^{ } \ar@{-}[u]^{ } & 
\stackrel{5}{\bullet}\ar@{-}[r]^{ }  & 
\stackrel{6}{\bullet}\ar@{-}[r]^{ }  & 
\stackrel{7}{\bullet}\ar@{-}[r]^{ }  & 
\stackrel{8}{\bullet}. 
 \\ 
} $    \\ 
  \end{tabular}
\]  
\end{definition}

If $Q$ if an $ADE$-Dynkin quiver, we will assume $Q$ has $r$ vertices and $r-1$ arrows.

\begin{definition}\label{definition:affine-Ar-dynkin-quiver-intro}
We say $Q$ is of {\em affine (type) $\widetilde{A}_r$-Dynkin} if the underlying graph is the affine $\widetilde{A}_r$-Dynkin graph: 
\[ 
\xymatrix@-1pc{ 
 & &\stackrel{r+1}{\bullet} \ar@{-}[lld]_{} &\ldots  &\stackrel{5}{\bullet}  & &  \\   
 \stackrel{1}{\bullet} \ar@{-}[rdr]^{} & & & & & & \stackrel{4}{\bullet}. \ar@{-}[llu]_{} \\  
 & & \stackrel{2}{\bullet} \ar@{-}[rr]^{} & &  \stackrel{3}{\bullet}  \ar@{-}[rru] & & \\   
} 
\]
\end{definition} 

Note that $\widetilde{A}_0$ has one vertex and one arrow whose head equals its tail, and $\widetilde{A}_1$ has two vertices joined by two edges. 
Next, we will give two important classical constructions using quivers. 

\begin{definition}\label{definition:framing-to-construct-semi-invariants}   
Let $Q=(Q_0,Q_1)$ be a quiver. 
The quiver $Q^{\dagger}=(Q_0^{\dagger},Q_1^{\dagger})$ 
is called a 
{\em framed quiver} 
if $Q_0^{\dagger} = Q_0 \coprod Q_0^{\natural}$ where for each vertex $i\in Q_0$, there is a copy $i^{\natural}$ of $i$ in $Q_0^{\natural}$ such that $|Q_0^{\natural}|=|Q_0|$, 
		and    
${Q}_1^{\dagger}=Q_1\coprod Q_1^{\natural}$, where $Q_1^{\natural}$ contains arrows of the form  
$\stackrel{i^{\natural}}{\circ} \stackrel{\iota_i }{\longrightarrow} \stackrel{i}{\bullet}$ 
for each pair of vertices $i$ and $i^{\natural}$, $i\in Q_0$.  
\end{definition}

Other notions of framings are given in \cite{MR1302318}, \cite{MR1604167} and \cite{MR1834739}. 
Nakajima defines his framing as in Definition~\ref{definition:framing-to-construct-semi-invariants}, except the arrows in $Q_1^{\natural}$ are in opposite orientation (\cite{Ginzburg-Nakajima-quivers}, Section 3), 
while Crawley-Boevey adjoins exactly one framed vertex, calls it $\infty$,  
and adds $\beta_i$ additional arrows from vertex $i\in Q_0$ to $\infty$, for each vertex $i\in Q_0$ (\cite{MR1834739}, page 261).   

\begin{definition}\label{definition:double-quiver-classical-setting}
Let $Q=(Q_0, Q_1)$ be a quiver. 
The quiver 
$\overline{Q}=(\overline{Q}_0,\overline{Q}_1)$ 
is called a 
{\em double quiver} 
if  
$\overline{Q}_0 = Q_0$ and 
$\overline{Q}_1 = Q_1 \coprod Q_1^{op}$, where for each arrow $a\in Q_1$, there is an additional arrow $a^{op}\in Q_1^{op}$ 
such that 
$ta^{op}=ha$ and 
$ha^{op}=ta$, with $|Q_1^{op}|=|Q_1|$. 
\end{definition}

\subsection{Paths of a quiver}

\begin{definition}
A nontrivial {\em path} in $Q$ is a sequence $p=a_k\cdots a_2 a_1$ ($k\geq 1$) of arrows which satisfies 
$t(a_{i+1})=h(a_i)$  
for all $1\leq i\leq k-1$. 
\end{definition} 

The path $p= a_k\cdots a_2 a_1$ begins at the tail of $a_1$ and ends at the head of $a_k$ 
and we will write $h(p)=h(a_k)$ and $t(p)=t(a_1)$.

\begin{definition}
The  
{\em length} $l(p)$ of a path $p$ is the number of arrows in the path. 
\end{definition}

If  $p= a_k\cdots a_2 a_1$ is a nontrivial path, then $l(p)= k$. 

\begin{definition}
To each vertex 
$i\in Q_0$, we associate a path $e_i$ called the {\em trivial (empty) path} 
  whose head and tail are at $i$. 
\end{definition}  
 
We say the length of an empty path is $0$.

\begin{definition}
 If the tail of a nontrivial path   
 equals the head of the path,  
then the path is said to be a {\em cycle}, 
and we say a quiver is {\em acyclic} if it has no cycles.  
If the nontrivial path is actually a single arrow whose tail equals its head, then the arrow is said to be a {\em loop}.  
\end{definition}

\subsection{The path algebra of a quiver}
 
\begin{definition} 
The {\em path algebra} $\mathbb{C}Q$ of $Q$ 
is the $\mathbb{C}$-algebra with basis the paths in 
$Q$, with the product of two nontrivial paths $p$ and $q$ given by 
\[ 
p\circ q = 
\begin{cases}  
pq, \mbox{ concatenation of sequences } &\mbox{ if } tp=hq, \\  
0 &\mbox{ if } tp\not=hq, \\  
\end{cases}  
\]  
and trivial paths having the following properties: 
\[  
e_i e_j :=  
\begin{cases} 
e_i &\mbox{ if } i    = j \\ 
0   &\mbox{ if } i\not= j \\
\end{cases}, \hspace{4mm}
p e_i := 
\begin{cases}  
p   &\mbox{ if } tp    = i \\ 
0   &\mbox{ if } tp \not= i \\
\end{cases}, \hspace{4mm}
e_i p := 
\begin{cases}  
p   &\mbox{ if } hp    = i \\ 
0   &\mbox{ if } hp \not= i \\
\end{cases},  \mbox{ and } \sum_{i\in Q_0}e_i=1. 
\]       
%
 The $e_i$'s are called {\em (mutually) orthogonal idempotents} and 
 $\displaystyle{\sum_{i\in Q_0}e_i=1}$  
 is the identity element in $\mathbb{C}Q$.   
\end{definition} 
 
 Since the composition of paths is associative, $\mathbb{C}Q$ is an associative algebra. 
 We may identify $Q_0$ with the trivial paths   
 while we many identify $Q_1$ with the set of all paths of length $1$. 
 If we identify $Q_n$ with the set of all paths of length $n$, then $\mathbb{C}Q_n$ is a subspace of $\mathbb{C}Q$ spanned by $Q_n$, i.e., 
$\displaystyle{\mathbb{C}Q =\bigoplus_{n\geq 0}\mathbb{C}Q_n}$, 
 and 
 since $l(pq)=l(p)+l(q)$ whenever the composition of $p$ and $q$ is defined, 
 we have the inclusion 
 $(\mathbb{C}Q_k)(\mathbb{C}Q_l)\subseteq \mathbb{C}Q_{k+l}$. 
 Thus, the path algebra is a graded algebra.

 Note that $\mathbb{C}Q$ is a finite-dimensional vector space if and only if $Q$ contains no cycles.

\begin{example}\label{example:A3-quiver-background-path-algebra} 
Let $Q$ be an $A_3$-Dynkin quiver: 
\[ \xymatrix{
\stackrel{1}{\bullet} \ar[rr]^{a_1} & & \stackrel{2}{\bullet} \ar[rr]^{a_2} & & \stackrel{3}{\bullet}.  
}  
\]  
The path algebra of $Q$ is  $\mathbb{C}Q= \mathbb{C}e_1 \oplus \mathbb{C}e_2 \oplus \mathbb{C}e_3 \oplus \mathbb{C}a_1 \oplus \mathbb{C}a_2 \oplus \mathbb{C}a_2a_1$,   
which is $6$ dimensional, and 
$\mathbb{C}Q$ is isomorphic to the space of $3\times 3$ upper triangular matrices via the following map: 
\[ \begin{aligned} 
e_1 &\mapsto 
\begin{pmatrix}      
0 & 0 & 0 \\
0 & 0 & 0 \\
0 & 0 & 1      
\end{pmatrix}, 
\hspace{4mm} 
e_2 \mapsto 
\begin{pmatrix} 
0 & 0 & 0 \\ 
0 & 1 & 0 \\
0 & 0 & 0  
 \end{pmatrix}, \hspace{4mm} 
e_3 \mapsto 
\begin{pmatrix}    
1 & 0 & 0 \\ 
0 & 0 & 0 \\
0 & 0 & 0   
\end{pmatrix}, 										 \\ 
a_1    &\mapsto 
\begin{pmatrix}    
0 & 0 & 0 \\ 
0 & 0 & 1 \\
0 & 0 & 0   
 \end{pmatrix}, \hspace{4mm} 
a_2     \mapsto 
\begin{pmatrix}   
0 & 1 & 0 \\  
0 & 0 & 0 \\ 
0 & 0 & 0   
 \end{pmatrix}, \hspace{4mm}  
a_2a_1 \mapsto  
\begin{pmatrix} 
0 & 0 & 1 \\ 
0 & 0 & 0 \\ 
0 & 0 & 0   
 \end{pmatrix}. 											\\  
\end{aligned}  
\] 
One extends the above map by linearity to have 
\[  \mathbb{C}Q 
\cong  
\begin{pmatrix}              
\mathbb{C} 		  & \mathbb{C}  & \mathbb{C} \\ 
0 							& \mathbb{C}  & \mathbb{C} \\
0 							& 0 				  & \mathbb{C}  
\end{pmatrix}.  
\]   
\end{example}

\begin{example} 
Let $Q$ be a $3$-Kronecker {\em equioriented} (all arrows are pointing in the same direction) quiver: 
$
\xymatrix@-1pc{
\stackrel{1}{\bullet} \ar@/^/[rrr]^{a_1} \ar[rrr]|{a_2} \ar@/_/[rrr]_{a_3}& & & \stackrel{2}{\bullet} \\
}. 
$ 
Then the path algebra $\mathbb{C}Q$ $=$ 
$\mathbb{C}e_1\oplus \mathbb{C}e_2\oplus \mathbb{C}a_1\oplus \mathbb{C}a_2 \oplus \mathbb{C}a_3$ is isomorphic to 
\[ 
\left\{ 
\begin{pmatrix}  
v_1 & w   \\ 
0   & v_2    
\end{pmatrix} : v_1, v_2 \in \mathbb{C}, w\in \mathbb{C}^3 
\right\} 
\]  
with product defined as 
\[ 
\begin{pmatrix}  
v_1 & w \\ 
0 & v_2 \\ 
\end{pmatrix}  
\begin{pmatrix} 
v_1' & w' \\ 
0 & v_2' \\ 
\end{pmatrix} 
= 
\begin{pmatrix}   
v_1 v_1' & v_1 w'+ v_2' w \\   
0        & v_2 v_2'   
\end{pmatrix}.   
\] 
The isomorphism is given via the map 
\[ 
e_1\mapsto 
\begin{pmatrix} 
 0& (0,0,0)\\ 
 0& 1       
\end{pmatrix}
, \hspace{4mm}
e_2\mapsto 
\begin{pmatrix} 
 1& (0,0,0)\\ 
 0& 0       
\end{pmatrix}, 
\]
\[ 
a_1\mapsto \begin{pmatrix} 
 0& (1,0,0)\\ 
 0& 0        
\end{pmatrix}, \hspace{4mm}
a_2\mapsto \begin{pmatrix} 
 0& (0,1,0)\\  
 0& 0       
\end{pmatrix}, \hspace{4mm}
a_3\mapsto \begin{pmatrix} 
 0& (0,0,1)\\  
 0& 0        
\end{pmatrix}.   \\
\] 
\end{example}

\begin{example} 
Let $Q$ be the $1$-Jordan quiver $\xymatrix@-1pc{\stackrel{1}{\bullet}\ar@(ru,rd)^a}$. Then paths of $Q$ consist of $e_1, a,a^2,\ldots$. Thus there exists an isomorphism  $\mathbb{C}Q\rightarrow \mathbb{C}[x]$, sending $e_1\mapsto 1$ and $a\mapsto x$. 
\end{example} 

\begin{example} 
Let $Q$ be the $4$-Jordan quiver: 
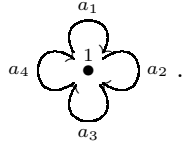
\[ 
\xymatrix@-1pc{
\stackrel{1}{\bullet} \ar@(ul,ur)^{a_1} \ar@(ur,dr)^{a_2} \ar@(dr,dl)^{a_3} \ar@(dl,ul)^{a_4} \\ 
}.
\] 
Then the map $\mathbb{C}Q\stackrel{\cong}{\longrightarrow}
\mathbb{C}\langle x_1,x_2,x_3,x_4\rangle$ given by 
$e_1\mapsto 1, a_i\mapsto x_i$ for $1\leq i\leq 4$ implies 
$\mathbb{C}Q$ is a free algebra on $4$ noncommuting variables.  
\end{example}

Definition~\ref{definition:quiver-with-relations-Brion} is given by Michel Brion. 

\begin{definition}\label{definition:quiver-with-relations-Brion} 
A 
{\em relation} 
of a quiver $Q$ is a subspace of $\mathbb{C}Q$ spanned by linear combinations of paths having a common source and a common target, and of length at least $2$. 
%
A 
{\em quiver with relations} is a pair $(Q,I)$, where $Q$ is a quiver and $I$ is a two-sided ideal of $\mathbb{C}Q$ generated by relations. 
The  
{\em quotient algebra}  
$\mathbb{C}Q/I$ is the path algebra of $(Q,I)$. 
\end{definition} 

\begin{example}
Let $Q$ be an $A_3$-Dynkin quiver given in Example~\ref{example:A3-quiver-background-path-algebra} 
and let  
$I = \langle a_2a_1 \rangle$, an ideal in $\mathbb{C}Q$. Then 
$\mathbb{C}Q/\langle a_2a_1 \rangle 
\cong  
\mathfrak{b}_3/\mathfrak{u}_3^+$,  
 where $\mathfrak{b}_3$ is the space of $3\times 3$ complex upper triangular matrices and 
\[ 
\mathfrak{u}_3^+ := \left\{ 
\begin{pmatrix} 
0 & 0 & u_{13} \\ 
0 & 0 & 0 \\ 
0 & 0 & 0 \\ 
\end{pmatrix}: 
u_{13}\in \mathbb{C} \right\}.  
\]  
\end{example}

Next, we introduce the following concepts since we restrict to quivers with specific properties in Section~\ref{section:semi-invariants-quivers-at-most-two-pathways}.

\begin{definition}\label{definition:reduced-path}
Let $Q=(Q_0,Q_1)$ be a quiver. 
Let $p=a_k\cdots a_2 a_1$ be a path where $a_i\in Q_1$ are arrows. 
If $p$ is a cycle, then 
we define $p^m$ to be the path composed with itself $m$ times, i.e., 
\[ 
p^m := p\circ p \circ \cdots \circ p = \underbrace{(a_k\cdots a_2 a_1)\cdots (a_k\cdots a_2 a_1)}_{m} = (a_k\cdots a_2 a_1)^m. 
\]  
A path $p$  
is {\em reduced} if $[p]\not=0$ in   
$\mathbb{C}Q/\langle q^2: q\in \mathbb{C}Q, l(q)\geq 1 \rangle$. 
%
\end{definition}


\begin{definition}\label{definition:pathway-of-quiver} 
 A {\em pathway} from vertex $i$ to vertex $j$ is 
 a reduced path 
 from $i$ to $j$.
 We define 
{\em pathways} of a quiver $Q$ to be the set of  
 all pathways from vertex $i$ to vertex $j$, where $i,j\in Q_0$.  
\end{definition}  

Note that pathways (of a quiver $Q$)  
include trivial paths and they 
form a finite set since $Q$ is a finite quiver.  
We will now give an example of Definition~\ref{definition:pathway-of-quiver}.  

\begin{example} 
Consider the $2$-Jordan quiver: \[ 
\xymatrix@-1pc{  
\stackrel{1.}{\bullet} \ar@(ru,rd)^{a_1} \ar@(ld,lu)^{a_2} 
}
\] 
Then $a_2^2 a_1$ is a path but not a pathway since it is not reduced. 
However, the path $a_2 a_1$ is a pathway. 
\end{example}

\subsection{Spaces of quiver representations}\label{subsection:quiver-representations}

\begin{definition} 
A   
{\em representation} 
 $W$ of a quiver $Q$ 
         assigns a vector space $W(i) = W_i$ 
         to each vertex $i\in Q_0$ and a linear map $W(a):W(ta)\rightarrow W(ha)$ 
         to each arrow $a\in Q_1$. 
\end{definition} 

\begin{definition} 
A representation $W=(W(i)_{i\in Q_0}, W(a)_{a\in Q_1})$ of $Q$ is  
{\em finite dimensional} 
if each vector space $W(i)$ is finite dimensional over $\mathbb{C}$.  
\end{definition}

\begin{definition}
A   
{\em subrepresentation}  
of a representation $W$ of $Q$   
is a subspace $V\subseteq W$ which is invariant under all operators: 
$W(a)(V(ta))\subseteq V(ha)$, where $a\in Q_1$. 
\end{definition}

\begin{definition}
A nonzero representation $W$ of $Q$ is {\em irreducible}, or {\em simple}, 
if the only subrepresentations of $W$ are $0$ and itself. 
\end{definition} 

Let $V,W$ be representations of $Q$. 
Then $V\oplus W$ has the structure of a representation of $Q$ given by 
$V(a)\oplus W(a):V(ta)\oplus W(ta)\rightarrow V(ha)\oplus W(ha)$ for all $a\in Q_1$. 

\begin{definition}
A nonzero representation $W$ of $Q$ is {\em indecomposable} if it is not isomorphic to a direct sum of two nonzero representations. 
\end{definition}

A map $W\rightarrow V$ of finite-dimensional representations $W$ and 
$V=(V(i)_{i\in Q_0}, V(a)_{a\in Q_1})$ is a family of $\mathbb{C}$-linear maps 
$\phi_i:W(i)\rightarrow V(i)$ for $i\in Q_0$ such that the diagram 
\[   
\xymatrix@-1pc{   
W(ta) \ar[dd]_{\phi_{ta}} \ar[rrr]^{W(a)} & & & W(ha) \ar[dd]^{\phi_{ha}}\\ 
& & & \\   
V(ta) \ar[rrr]^{V(a)} & & & V(ha) \\   
}   
\]   
commutes for all $a\in Q_1$. 
We will write $\phi=(\phi_i)_{i\in Q_0}$ to mean the family of linear maps for $W\rightarrow V$. 
The composition of maps of representations 
\[ 
W \stackrel{\phi}{\longrightarrow}V \stackrel{\psi}{\longrightarrow}U 
\] 
is given by $(\psi\circ\phi)_{i\in Q_0} := (\psi_i\circ \phi_i)_{i\in Q_0}$.  
Thus the composition of maps is associative and we will write   
$\id_W := (\id_{W(i)})_{i\in Q_0}$ to be the identity map of representations.

The definition of a dimension vector $\beta\in \mathbb{Z}_{\geq 0}^{Q_0}$ is given in Chapter~\ref{chapter:introduction}.

\begin{definition}\label{definition:epsilon-i-simple-dimension-vector}
We define $\epsilon_i \in \mathbb{Z}_{\geq 0}^{Q_0}$ to be the dimension vector with $1$ in the coordinate corresponding to $i\in Q_0$ and $0$ elsewhere.  
\end{definition}

Definition~\ref{definition:Euler-inner-product} and  Definition~\ref{definition:symmetrization-of-Ringel-form} 
are important in Section~\ref{subsection:derksen-weyman-technique} and Chapter~\ref{chapter:categories-filt-graded-quiver-varieties}. 

\begin{definition}\label{definition:Euler-inner-product}  
We define 
$\displaystyle{\langle \alpha,\beta\rangle}$ $:=$ 
$\displaystyle{\sum_{i\in Q_0} \alpha_i \beta_i - \sum_{a\in Q_1} \alpha_{ta}\beta_{ha}}$
to be the {\em Euler inner product}  
(or {\em Euler bilinear form} or {\em Ringel form})  
on dimension vectors $\alpha$ and $\beta$.    
\end{definition} 

\begin{definition}\label{definition:symmetrization-of-Ringel-form}
We define $(\alpha,\beta)_Q := \langle \alpha,\beta \rangle + \langle \beta,\alpha \rangle$ to be the symmetrization of the Ringel form on dimension vectors $\alpha$ and $\beta$. 
\end{definition}

\begin{definition}\label{definition:simple-representation-one-dim-vs-one-vertex}
We will define $S_i$ to be the representation of a quiver $Q$ of dimension vector $\epsilon_i\in \mathbb{Z}_{\geq 0}^{Q_0}$. 
\end{definition}

That is, $S_i$ is the representation with a $1$-dimensional vector space at vertex $i\in Q_0$ and $0$ elsewhere. 
 
Now, let $W$ be a representation of $Q$ of dimension vector $\beta\in \mathbb{Z}_{\geq 0}^{Q_0}$.  
Upon   
fixing a basis for each finite-dimensional vector space $W(i)$,   
each $W(i)$ is identified with $\mathbb{C}^{\beta_{i}}$ and each linear map 
$\xymatrix{ \stackrel{W(ta)}{\bullet} \ar[r]^{W(a)} & \stackrel{W(ha)}{\bullet}}$ 
may be identified with a  
$\beta_{ha}\times \beta_{ta}$ matrix.   
We will thus define 
\[ 
Rep(Q,\beta) := 
\displaystyle{\bigoplus_{a\in Q_1 }\Hom_{\mathbb{C}}(\mathbb{C}^{\beta_{ta}}, \mathbb{C}^{\beta_{ha} })} 
\]
throughout this thesis. 
Points in $Rep(Q,\beta)$ parameterize finite-dimensional representations of $Q$ of dimension vector $\beta$, 
and there is a natural 
$\displaystyle{\mathbb{G}_{\beta}=\prod_{i\in Q_0} GL_{\beta_i}(\mathbb{C})}$-action on $Rep(Q,\beta)$ 
as a change-of-basis; that is, given 
$(g_{\beta_i})_{i\in Q_0}\in \mathbb{G}_{\beta}$ and $W\in Rep(Q,\beta)$, 
\[ (g_{\beta_i})_{i\in Q_0}.(W(a))_{a\in Q_1} = (g_{\beta_{ha}}W(a)g_{\beta_{ta}}^{-1} )_{a\in Q_1}. 
\]  
Whenever the composition $pq$ of paths is defined,  
we set $W(pq):=W(p)W(q)$, i.e., the representation of a composition of paths is the product of representations of the paths. 
  
%
%
%
%
When we write ``general matrix" or ``general representation" of a quiver, we will mean a matrix or representation whose entries are indeterminates.

Recall Definition~\ref{definition:framing-to-construct-semi-invariants}.   
We will write ${\beta}^{\dagger} :=(\beta,\beta^{\natural})\in (\mathbb{Z}_{\geq 0}^{Q_0})^2$, where 
$\beta \in \mathbb{Z}_{\geq 0}^{Q_0}$ is the dimension vector for the original set of vertices $Q_0$
and $\beta^{\natural} \in \mathbb{Z}_{\geq 0}^{Q_0}$ is the dimension vector for the framed vertices $Q_0^{\natural}$; 
we will write 
				$W(\iota_i)\in \Hom(\mathbb{C}^{\beta_i^{\natural}}, \mathbb{C}^{\beta_i})$.  

\begin{definition}
We define 
\[ 
Rep(Q^{\dagger},\beta^{\dagger}) := \bigoplus_{a\in Q_1}\Hom_{\mathbb{C}}(\mathbb{C}^{\beta_{ta}}, \mathbb{C}^{\beta_{ha}}) 
\oplus 
\bigoplus_{i\in Q_0}\Hom_{\mathbb{C}}(\mathbb{C}^{\beta_i^{\natural}},\mathbb{C}^{\beta_i}) 
\] 
as the {\em framed quiver variety} 
of $Q^{\dagger}$ of dimension vector $\beta^{\dagger}$. 
\end{definition}

Classically, a basis for vector spaces at framed vertices is fixed; 
thus, 
there is no group action on vector spaces at framed vertices.  
Thus, $\mathbb{G}_{\beta^{\dagger}} := \mathbb{G}_{\beta} =\displaystyle{\prod_{i\in Q_0}} GL_{\beta_i}(\mathbb{C})$ acts on 
$Rep(Q^{\dagger},\beta^{\dagger})$ via the following: for 
$(g_{\beta_i})_{i\in Q_0}\in \mathbb{G}_{\beta}$ and for 
$(W(a),W(\iota_i))_{a\in Q_1, i\in Q_0}\in Rep(Q^{\dagger},\beta^{\dagger})$, 
\[ 
(g_{\beta_i})_{i\in Q_0}.(W(a),W(\iota_i))_{a\in Q_1, i\in Q_0} 
= 
(g_{\beta_{ha}} W(a) g_{\beta_{ta}}^{-1}, g_{\beta_i} W(\iota_i))_{a\in Q_1, i\in Q_0}.  
\]

Next, recall Definition~\ref{definition:double-quiver-classical-setting}. Then 
$\overline{\beta} = \beta\in \mathbb{Z}_{\geq 0}^{Q_0}$ since $\overline{Q}_0=Q_0$.  
\begin{definition}\label{definition:double-quiver-variety-classical-setting} 
We define $Rep(\overline{Q},\overline{\beta})$ $:=$  
$Rep(Q,\beta)\times Rep(Q^{op},\beta)$   
to be the {\em double quiver variety}, 
where 
$Q^{op}=(Q_0,Q_1^{op})$. 
\end{definition} 

The product $\mathbb{G}_{\overline{\beta}} := \mathbb{G}_{\beta}$ of general linear matrices acts on 
$Rep(\overline{Q},\overline{\beta})$ as a change-of-basis: given $(g_{\beta_i})_{i\in Q_0}\in \mathbb{G}_{\beta}$ and for $(W(a),W(a^{op}))_{a\in Q_1} \in Rep(\overline{Q},\overline{\beta})$, 
\[ 
(g_{\beta_i})_{i\in Q_0}. (W(a),W(a^{op}))_{a\in Q_1}  
= (g_{\beta_{ha}} W(a) g_{\beta_{ta}}^{-1}, g_{\beta_{ta}}W(a^{op})g_{\beta_{ha}}^{-1}  
)_{a\in Q_1}. 
\]

Now, given any two finite-dimensional vector spaces $V$ and $W$, there exists a canonical perfect pairing   
$\Hom(V,W) \times \Hom(W,V) \longrightarrow \mathbb{C}$, where  
$(f,g) \mapsto \Tr(f\circ g)$ $=$ $\Tr(g\circ f)$.   
This means the opposite representation space $Rep(Q^{op},\beta)$ is isomorphic to the dual $Rep(Q,\beta)^*$ of $Rep(Q,\beta)$. So 
$Rep(\overline{Q},\overline{\beta})$ $\cong$ $Rep(Q,\beta)\times Rep(Q,\beta)^*$ $\cong$ $T^*Rep(Q,\beta)$.  
The group   
$\mathbb{G}_{\beta}$ acts on    
$T^*Rep(Q,\beta)$ 
    by inducing the action of 
$\mathbb{G}_{\beta}$ on $Rep(Q,\beta)$ to $Rep(\overline{Q},\overline{\beta})$. 
The variety $T^*Rep(Q,\beta)$ is called the {\em cotangent bundle} of $Rep(Q,\beta)$. 
Associated to the $\mathbb{G}_{\beta}$-action is a moment map 
$\mu_{\beta}:T^*Rep(Q,\beta)\rightarrow \mathfrak{g}_{\beta}^* \stackrel{\Tr}{\cong} \mathfrak{g}_{\beta}$, where $\mathfrak{g}_{\beta}:=\displaystyle{\prod_{i\in Q_0}}\mathfrak{gl}_{\beta_i}$. 
Further discussion is given in Section~\ref{section:ham-reduction-symplectic-geometry}. 
  
\begin{example} 
Let $Q$ be a $2$-Dynkin quiver 
$\xymatrix@-1pc{
\stackrel{1}{\bullet}\ar[rr]^a & & \stackrel{2}{\bullet}
}$ 
and let $\beta=(1,2)$.  
Then $Rep(Q,\beta)\cong \mathbb{C}^2$.  
Let $\mathbb{G}_{\beta} = GL_1(\mathbb{C})\times GL_2(\mathbb{C})$ act on $Rep(Q,\beta)$ 
via $(g,h).W(a)$ $=$ $hW(a)g^{-1}$. 
Then the double quiver $\overline{Q}$ of $Q$ is 
$\xymatrix@-1pc{
\stackrel{1}{\bullet}\ar@/^/[rr]^a & & \ar@/^/[ll]^{a^{op}} \stackrel{2}{\bullet}
}$ 
and $T^*Rep(Q,\beta)\cong \mathbb{C}^2 \times (\mathbb{C}^2)^*$, where $(\mathbb{C}^2)^*$ is the dual vector space to $\mathbb{C}^2$. For $(g,h)\in \mathbb{G}_{\beta}$ 
and $(W(a),W(a^{op}))\in T^*Rep(Q,\beta)$, we have 
$(g,h).(W(a),W(a^{op}))$ $=$ $(hW(a)g^{-1}, gW(a^{op})h^{-1} )$. 
Let $i=W(a)$ and $j=W(a^{op})$. 
Then the moment map for the $\mathbb{G}_{\beta}$-action is  
$T^*Rep(Q,\beta)=\mathbb{C}^2 \times (\mathbb{C}^2)^*\stackrel{\mu_{\beta}}{\longrightarrow} \mathfrak{gl}_1^*\times \mathfrak{gl}_2^*$, where 
$(i,j)\mapsto (ji,ij)$. 
\end{example}

\subsection{The category of quiver representations}

Let $\Rep(Q)$ be the category of finite-dimensional representations of $Q$ and let $\mathbb{C}Q-\Mod$ be the category of finitely generated left $\mathbb{C}Q$-modules.

\begin{proposition} 
We have an equivalence of categories: $\Rep(Q)\simeq \mathbb{C}Q-\Mod$. 
\end{proposition}

\begin{proof} 
Let $W=(W(i)_{i\in Q_0}, W(a)_{a\in Q_1})$ 
be a finite-dimensional representation of $Q$, where 
$W(i)$ are finite-dimensional vector  spaces and $W(a)$ is a linear map from $W(ta)$ to $W(ha)$. 
Let $M:=
\displaystyle{\bigoplus_{i\in Q_0} W(i)}$, a $\mathbb{C}$-vector space. 
We define $\mathbb{C}Q$-module structure on $M$ as follows: for each $e_i,a\in Q_1$ and $i\in Q_0$, $e_i$ and $a$ act on $m\in M$ as follows: 
\[ 
e_i.m:= 
\begin{cases} 
m =m_i  &\mbox{ if } m\in W(i) \\
0 			&\mbox{ if } m\not\in W(i) \\
\end{cases} 
\hspace{2mm}\mbox{ and } \hspace{2mm} 
a.m := 
\begin{cases} 
W(a)(m_{ta}) &\mbox{ if } m\in W(ta) \\ 
0 					 &\mbox{ if } m\not\in W(ta), \\ 
\end{cases} 
\] 
and we extend the above actions linearly. 
This gives us a functor $F:\Rep(Q)\longrightarrow \mathbb{C}Q-\Mod$.  
Conversely, let's construct the functor $G:\mathbb{C}Q-\Mod\longrightarrow \Rep(Q)$.  
So let $M$ be a $\mathbb{C}Q$-module. 
We set $W(i):=Me_i$ for $i\in Q_0$ and for each $a\in Q_1$, define the linear maps
$W(a):W(ta)\longrightarrow W(ha)$ as $m\mapsto a.m$. 
We will now check that the maps of representations of $Q$ are precisely $\mathbb{C}Q$-module homomorphisms. 

So suppose $\{\phi_i:W(i)\rightarrow V(i), i\in Q_0 \}$ is a map of representations $W$ and $V$ of $Q$. Then we construct 
$\mathbb{C}Q$-module homomorphisms by 
\[ 
\begin{aligned} 
M &:=\bigoplus_{i\in Q_0}W(i)\stackrel{\psi}{\longrightarrow}N:=\bigoplus_{i\in Q_0}V(i), 
\mbox{ where } \\
\psi(e_i.m) &=
\begin{cases} 
\phi_i(m) &\mbox{ if } m\in W(i) \\  
 0 				&\mbox{ if } m\not\in W(i) \\  
\end{cases}  \hspace{2mm}\mbox{ and }\hspace{2mm}
\psi(a.m) = 
\begin{cases} 
\phi_{ha}(W(a)(m_{ta})) &\mbox{ if } m\in W(ta) \\ 
						0	 					&\mbox{ if } m\not\in W(ta). \\ 
\end{cases} \\
\end{aligned}
\] 
Since $\{\phi_i:i\in Q_0 \}$ is a collection of linear maps, 
we have 
\[
a.\psi(e_{ta}.m) = a.\psi(m_{ta}) = V(a)(\phi_{ta}(m_{ta})) \mbox{ and }
\psi(a.m_{ta})= \psi(W(a)(m_{ta}))
= \phi_{ha}(W(a)(m_{ta})), \mbox{ where } a\in Q_1.  
\] 
So $a.\psi(m_{ta}) = \psi(a.m_{ta})$ 
since $\{\phi_i\}$ is a map of representations. 
Thus $\psi$ is a $\mathbb{C}Q$-module homomorphism.

On the other hand, let $M\stackrel{\psi}{\longrightarrow}N$ be a $\mathbb{C}Q$-module homomorphism. 
Then we define $\{\phi_i: i\in Q_0 \}$ as follows: 
\[ 
\phi_i:W(i)\rightarrow V(i) \mbox{ where }\phi_i(e_i.m):=\psi(e_i.m) \mbox{ for all }i\in Q_0. 
\] 
If $m\in W(i)$, we will write $\phi_i(m_i)=\psi(m_i)$. 
Now, we will show that $\{\phi_i : i\in Q_0\}$ is a homomorphism of representations of $Q$. Since 
\[ 
\phi_{ha}(W(a)(m_{ta})) = \psi(W(a)(m_{ta})) = \psi(a.m_{ta}) \mbox{ and } 
V(a)(\phi_{ta}(m_{ta})) = V(a)(\psi(m_{ta}))=a.\psi(m_{ta}), 
\] 
we have $\phi_{ha}(W(a)(m_{ta}))= V(a)(\phi_{ta}(m_{ta}))$  
since $\psi$ is a $\mathbb{C}Q$-module homomorphism.  

Now consider the two functors $F$ and $G$ in 
$\xymatrix@-1pc{ Rep(Q) \ar@/^/[rr]^F & & \ar@/^/[ll]^G \mathbb{C}Q-\Mod \\
}$ 
defined above. 
Since it is clear that the two functors are mutually inverses, i.e., $F\circ G = \Id$ and $G\circ F = \Id$, we are done. 
\end{proof}

%



\subsection{Reflection functors}\label{subsection:reflection-functors-classical-setting} 
We will now define BGP reflection functors as introduced in \cite{MR0393065}.  
Let $\sigma_i(Q)$ denote the quiver obtained from $Q$ by inverting all arrows whose head or tail is $i$. 
For each vertex $i\in Q_0$, 
\[ 
\sigma_i: \mathbb{Z}Q_0 \longrightarrow \mathbb{Z}Q_0, \hspace{4mm} \beta \mapsto \beta-(\beta, \epsilon_i)_Q \epsilon_i, 
\] 
where $\epsilon_i\in \mathbb{Z}_{\geq 0}^{Q_0}$ is given in Definition~\ref{definition:epsilon-i-simple-dimension-vector} and 
$(\beta,\epsilon_i)_Q$ 
is the symmetrization of the Ringel form on $\beta$ and $\epsilon_i$ (Definition~\ref{definition:symmetrization-of-Ringel-form}).   
If $Q$ has no loops at $i\in Q_0$, then $\sigma_i^2 \beta = \beta$.

Let $i$ be a $+$-admissible vertex and $Q'=\sigma_i(Q)$. 
Let $W\in Rep(Q,\beta)$ and $W_i:=W(i)$ is the vector space at vertex $i$. 
We define the functor $S_i^+: Rep (Q,\beta)\rightarrow Rep(Q',\beta')$ as follows:  
let $W'=S_i^+ (W)$, 
where $W_j'=W_j$ if $j\not=i$ and $W_i'=\ker \phi_i^W$, where 
\[ 
\phi_i^W := \sum_{\nu=(j\rightarrow i)\in Q_1} \rho_{\nu}: \bigoplus_j W_j \rightarrow W_i, 
\] 
and for each inverted arrow $\nu=(i\rightarrow j)\in Q_1'$, we define 
$\rho_{\nu}': W_i'\rightarrow W_j=W_j'$ as the natural projection on the component 
$W_j\in \displaystyle{\bigoplus_j W_j}$.

If $i$ is a $-$-admissible vertex and $Q'=\sigma_i(Q)$, then we define the functor $S_i^-: Rep(Q,\beta)\rightarrow Rep(Q',\beta')$ as follows: 
let $W' = S_i^-(W)$, where $W_j'= W_j$ for $i\not=j$ and $W_i'=\coker \widetilde{\phi}_i^W$, where 
\[
\widetilde{\phi}_i^W := \sum_{\nu=(i\rightarrow j)\in Q_1}\rho_{\nu}: W_i \rightarrow \bigoplus_j W_j, 
\] 
and for each $\nu=(j\rightarrow i)\in Q_1'$, we define $\rho_{\nu}':W_j=W_j'\rightarrow W_i'$ 
by restriction of the projection 
$\displaystyle{\bigoplus_j W_j} \rightarrow \coker \widetilde{\phi}_i^W$ to $W_j$. 

Let $S_i$ be the representation defined in Definition~\ref{definition:simple-representation-one-dim-vs-one-vertex}. 
Theorems~\ref{theorem:reflection-functors-classical-setting} and \ref{theorem:reflection-functor-classical-setting-bijective-correspondence} 
are classical results used to relate representations of two quivers. 
\begin{theorem}\label{theorem:reflection-functors-classical-setting} 
Let $W$ be an indecomposable representation of $Q$. 
\begin{enumerate} 
\item Let $i\in Q_0$ be $+$-admissible. Then $S_i^+(W)=0$ if and only if $W=S_i$. 
If $W\not= S_i$, then $S_i^+(W)=W'$ is an indecomposable representation, 
\begin{equation}\label{equation:reflection-functor-classical-setting}
\dim W_i' = \sum_{j\rightarrow i} \dim W_j-\dim W_i, 
\end{equation}  
and $S_i^- S_i^+(W)\cong W$. 
\item Let $i\in Q_0$ be $-$-admissible. Then $S_i^-(W)=0$ if and only if $W=S_i$. 
If $W\not=S_i$, then 
$S_i^-(W)=W'$ is indecomposable, \eqref{equation:reflection-functor-classical-setting} holds, 
and $S_i^+S_i^-(W)\cong W$. 
\end{enumerate}
\end{theorem}

Let $Q$ be a quiver and let $Q'=\sigma_i(Q)$. 
Let $S_i'$ be the representation of $Q'$ with one dimensional vector space at vertex $i$ and zero dimensional vector space elsewhere. 
\begin{theorem}\label{theorem:reflection-functor-classical-setting-bijective-correspondence} 
Let $i\in Q_0$ be $+$-admissible and let $Q'=\sigma_i(Q)$. 
Then $S_i^+$ and $S_i^-$ establish a bijection between indecomposable representations of $Q$ (nonisomorphic to $S_i$) and indecomposable representations of $Q'$ (nonisomorphic to $S_i'$).
\end{theorem}


\section{Invariants of quiver varieties}\label{section:invariants-of-quiver-varieties}

The coordinate ring $\mathbb{C}[Rep(Q,\beta)]$ has two gradings. 
One way is called {\em $Q_1$-grading}, where the ring is graded by $\mathbb{Z}^{Q_1}$ 
since the quiver variety 
$Rep(Q,\beta)$ $=$ 
$\displaystyle{\bigoplus_{a\in Q_1}M_{\beta(ha)\times \beta(ta)}(\mathbb{C})}$
is a product of matrices.  
The second way is called {\em $Q_0$-grading}, where the ring is graded by $\mathbb{Z}^{Q_0}$. To explain further, 
there exists a natural action of 
$GL_{\beta}(\mathbb{C}) := \mathbb{G}_{\beta} =
\displaystyle{\prod_{i\in Q_0}GL_{\beta_i}(\mathbb{C})}$ on $Rep(Q,\beta)$ 
which induces an action on the ring $\mathbb{C}[Rep(Q,\beta)]$. 
So $\displaystyle{\prod_{i\in Q_0}\mathbb{C}^*}$ acts on $\mathbb{C}[Rep(Q,\beta)]$ via the characters of the group, where $\mathbb{C}^*$ is the set of all units in $\mathbb{C}$. Thus, 
we decompose the ring as a direct sum of weight spaces for the action of 
$\displaystyle{\prod_{i\in Q_0}\mathbb{C}^*}$. 
Let $\mathbb{C}[Rep(Q,\beta)]^{GL_{\beta}(\mathbb{C}),\bullet}$ $:=$ 
$\displaystyle{\bigoplus_{\chi} \mathbb{C}[Rep(Q,\beta)]^{GL_{\beta}(\mathbb{C}),\chi}}$, where $\chi$ is a character of $GL_{\beta}(\mathbb{C})$.  
Then polynomials $f\in \mathbb{C}[Rep(Q,\beta)]^{GL_{\beta}(\mathbb{C}),\bullet}$ 
are homogeneous with respect to $Q_0$-grading.


Recall from Definition~\ref{definition:invariant-semi-invariant-polynomials-introduction} 
that a polynomial $f\in \mathbb{C}[Rep(Q,\beta)]$ is an  
{\em invariant polynomial}  
if $g.f=f$ for all $g\in GL_{\beta}(\mathbb{C})$, and the polynomial $f$ is 
{\em $\chi$-semi-invariant} 
if 
$g.f=\chi(g)f$ for all $g\in GL_{\beta}(\mathbb{C})$, where $\chi:GL_{\beta}(\mathbb{C})\longrightarrow\mathbb{C}^*$ is a group homomorphism. 
We will refer (semi-)invariant polynomials as (semi-)invariants.  
Semi-invariants under the $GL_{\beta}(\mathbb{C})$-action are invariants for $SL_{\beta}(\mathbb{C})$ $:=$ 
$\displaystyle{\prod_{i\in Q_0}SL_{\beta_i}(\mathbb{C})}$-action and 
$SL_{\beta}(\mathbb{C})$-invariant polynomials that are homogeneous with respect to the $Q_0$-grading are also semi-invariant (for some $\chi$) for the $GL_{\beta}(\mathbb{C})$-action.  
Therefore, $\mathbb{C}[Rep(Q,\beta)]^{GL_{\beta}(\mathbb{C}),\bullet} \cong  \mathbb{C}[Rep(Q,\beta)]^{SL_{\beta}(\mathbb{C})}$. 
In the literature, one writes 
$SI(Rep(Q,\beta))$ to mean $\mathbb{C}[Rep(Q,\beta)]^{SL_{\beta}(\mathbb{C})}$. 
Although \cite{MR557581}, \cite{MR1113382} and \cite{MR1162487} are some of the earlier works in the study of invariants of quiver varieties, we will focus on techniques given in 
\cite{MR1758750},  
\cite{MR1825166}, and \cite{MR1908144}. 

The following techniques were constructed independently by a number of mathematicians. 
In 1999, Schofield-van den Bergh wrote down all semi-invariant polynomials for any quiver variety, in 2000, Derksen-Weyman gave a strategy to produce all semi-invariant polynomials for acyclic quivers by using representation theory techniques, 
and in 2001, Domokos-Zubkov gave a strategy on producing all semi-invariant polynomials for all quiver varieties using combinatorical techniques and a notion of a large matrix. 
We will summarize each technique in the following sections.

\subsection{Derksen-Weyman technique}\label{subsection:derksen-weyman-technique}
The following technique is found in \cite{MR1758750}. 
For a fixed quiver $Q=(Q_0,Q_1)$ and a dimension vector $\beta\in \mathbb{Z}_{\geq 0}^{Q_0}$, 
consider the vector space $Rep(Q,\beta)$ under 
$SL_{\beta}(\mathbb{C})$ $:=$ 
$\displaystyle{\prod_{i\in Q_0}SL_{\beta_i}(\mathbb{C})}$-action as a change-of-basis. 
Let $W$ be a general representation of $Rep(Q,\beta)$.   

Let $\alpha\in\mathbb{Z}_{\geq 0}^{Q_0}$ be a dimension vector. Recall the Ringel form on $\alpha$ and $\beta$ from Section~\ref{subsection:quiver-representations}.  
 It follows from \cite{MR658729} that  
$\langle \alpha,\beta\rangle = \dim \Hom(V,W)-\dim \Ext(V,W)$ where 
$V\in Rep(Q,\alpha)$ and $W\in Rep(Q,\beta)$.  
We choose $\alpha$ such that the Euler form $\langle \alpha,\beta\rangle$ equals zero and a general representation in 
$Rep(Q,\alpha)$ is indecomposable 
(we will follow Derksen and Weyman's technique with some examples). 
Let $V$ be a general representation of the vector space $Rep(Q,\alpha)$. 
Consider  
\[ 
\bigoplus_{i\in Q_0} \Hom(V(i),W(i))\stackrel{d_{W}^{V}}{\longrightarrow}
\bigoplus_{a\in Q_1} \Hom(V(ta),W(ha)), \mbox{ where }
(X_1,\ldots, X_{Q_0})\mapsto (W(a)X_{ta}-X_{ha}V(a))_{a\in Q_1},  
\] 
where $X_i\in \Hom(V(i), W(i))$ is (not necessarily a square) matrix with variables as coordinates, 
$V(a)$ is a general $\alpha(ha)\times \alpha(ta)$ matrix, and 
$W(a)$ is a general $\beta(ha)\times \beta(ta)$ matrix. 
 Write the map $d_{W}^{V}$ in terms of the basis $\{X_i: i\in Q_0 \}$.
 Since 
$\displaystyle{\sum_{i\in Q_0} \alpha_i \beta_i}$  
$=$ 
 $\displaystyle{\sum_{a\in Q_1} \alpha_{ta}\beta_{ha}}$, 
 $d_{W}^{V}$ is a square matrix so taking the determinant of $d_{W}^{V}$ makes sense. 
 Consider $c_{W}^{V} =\det(d_{W}^{V})\in \mathbb{C}[Rep(Q,\alpha)\times Rep(Q,\beta)]$; collect the polynomial coefficients of $\{V(a):a\in Q_1 \}$. 
\begin{theorem}[Derksen-Weyman]  
Let $Q$ be an acyclic quiver and let $\beta$ be a dimension vector. 
Let $SL_{\beta}(\mathbb{C})$ naturally act on $Rep(Q,\beta)$ and let $\alpha$ be a dimension vector such that $\langle \alpha,\beta\rangle=0$.  
Let $V\in Rep(Q,\alpha)$ be an indecomposable general representation.    
Then the polynomial coefficients of $\{V(a):a\in Q_1 \}$   
obtained from $c_{W}^{V}$ generate the ring of $SL_{\beta}(\mathbb{C})$-invariant polynomials. 
 \end{theorem}
 
\begin{example}\label{example:1-Kronecker-quiver-Derksen-Weyman-method}
 Let $Q$ be the $1$-Kronecker quiver 
 $\xymatrix@-1pc{\stackrel{1}{\bullet} \ar[rr]^a & &  \stackrel{2}{\bullet}}$ and let $\beta=(2,2)$. 
 We will write down the map $d_{W}^{V}$ from the diagram 
 \[ 
\xymatrix@-1pc{
 \stackrel{\mathbb{C}^2=W(1)}{\bullet} \ar[rr]^{W(a)} & & \stackrel{\mathbb{C}^2=W(2)}{\bullet} \\
 & & \\
 \stackrel{V(1)}{\bullet} \ar[uu]^{X_1} \ar[rr]^{V(a)} & & \stackrel{V(2)}{\bullet} \ar[uu]_{X_2.} \\   
 }
 \]  
Since we want $\langle \alpha,\beta\rangle=2\alpha_1+2\alpha_2-2\alpha_1$ to equal zero, 
this implies $\alpha_2=0$.  
So regardless of the choice of $\alpha_1$, $V(a)$ is the zero matrix. 
Let $\alpha_1=1$, $W(a)=A$, and $X_1=X$ (note that we may let $X_2=0$).  
Then  
\[ 
\xymatrix@-1pc{
\Hom(\mathbb{C}^1,\mathbb{C}^2)\ar[rrrr]^{d_{W}^{V}} & & & & \Hom(\mathbb{C}^1,\mathbb{C}^2), \hspace{4mm}
X\mapsto AX, 
}
\]   
and the map $[d_{W}^{V}]_X$ in terms of the basis $X$ is $A$. 
Therefore, we have 
$\mathbb{C}[\mathfrak{gl}_2]^{SL_2\times SL_2}=\mathbb{C}[\det(A)]$. 
\end{example}

\begin{example}\label{example:2-Kronecker-quiver-Derksen-Weyman-method}
 Let $Q$ be the $2$-Kronecker quiver 
 $\xymatrix@-1pc{\stackrel{1}{\bullet} \ar@/^/[rr]^{a} \ar@/_/[rr]_{b} & & 
 \stackrel{2}{\bullet}}$ 
 and let $\beta=(2,2)$. Since 
 $\langle \alpha,\beta\rangle$ $=$ $2\alpha_1+2\alpha_2-2\alpha_1-2\alpha_1 = 0$, 
 we must have 
 $\alpha_1=\alpha_2$. Let $\alpha_1=1$. 
 Since $V$ is an indecomposable representation in $Rep(Q,\alpha)$, let 
 $V(a)=1$ and $V(b)=\lambda$. 
 Let $W(a)=A, W(b)=B,X_1=X,X_2=Y$, where $a$ and $b$ are the two arrows of $Q$. 
 Let's write the coordinates of the matrices as follows: 
 $A=(a_{ij}), B=(b_{ij}), X=(x_i)$, and $Y=(y_j)$. 
 So we have 
 \[
 \xymatrix@-1pc{
 \stackrel{\mathbb{C}^2=W(1)}{\bullet} \ar@/^/[rr]^{W(a)} \ar@/_/[rr]_{W(b)}& & \stackrel{\mathbb{C}^2=W(2)}{\bullet} \\
 & & \\
 \stackrel{\mathbb{C}^1}{\bullet} \ar[uu]^{X_1} \ar@/^/[rr]^{1} \ar@/_/[rr]_{\lambda}& & \stackrel{\mathbb{C}^1}{\bullet} \ar[uu]_{X_2,} \\   
 }
 \]   
 where 
$\Hom(\mathbb{C}^1,\mathbb{C}^2)\oplus \Hom(\mathbb{C}^1,\mathbb{C}^2)
 \stackrel{d_{W}^{V}}{\longrightarrow}
 \Hom(\mathbb{C}^1,\mathbb{C}^2)\oplus \Hom(\mathbb{C}^1,\mathbb{C}^2)$ is the map  
 \[
 \begin{aligned} 
 (X,Y) &\mapsto (AX-Y,BX-\lambda Y) \\ 
 &= 
 \left(
  \left(
    \begin{tabular}{ll}
      $a_{11}$ & $a_{12}$ \\
      $a_{21}$ & $a_{22}$ \\
    \end{tabular}
  \right)  	
    \left(
    \begin{tabular}{l}
      $x_{1}$  \\
      $x_{2}$  \\
    \end{tabular}
  \right)   
 			- 
    \left(
    \begin{tabular}{l}
      $y_{1}$  \\
      $y_{2}$  \\
    \end{tabular}
  \right),   			
  \left(
    \begin{tabular}{ll}
      $b_{11}$ & $b_{12}$ \\
      $b_{21}$ & $b_{22}$ \\
    \end{tabular}
  \right) 			
     \left(
    \begin{tabular}{l}
      $x_{1}$  \\
      $x_{2}$  \\
    \end{tabular}
  \right)  			 
  			-  
    \left(
    \begin{tabular}{l}
      $\lambda y_{1}$  \\
      $\lambda y_{2}$  \\
    \end{tabular}
  \right)  			 
 						\right) \\ 
 &=
 \left( 
 \left( 
  \begin{tabular}{l}   
   $a_{11}x_{1}+a_{12}x_{2}-y_{1}$ \\
   $a_{21}x_{1}+a_{22}x_{2}-y_{2}$ \\ 
 \end{tabular}
 \right), 
 \left( 
 \begin{tabular}{l} 
   $b_{11}x_{1}+b_{12}x_{2}-\lambda y_{1}$\\ 
   $b_{21}x_{1}+b_{22}x_{2}-\lambda y_{2}$\\
 \end{tabular}
 \right) 
 \right). \\ 
 \end{aligned}
 \] 
So the matrix $[d_{W}^{V}]_{X,Y}$ in terms of basis $X$ and $Y$ is 
\[ 
\bordermatrix{
&x_1 &x_2 &y_1 &y_2 \cr 
& a_{11}& a_{12}& -1& 0\cr 
& a_{21}&a_{22} & 0& -1\cr 
& b_{11}&b_{12} &-\lambda & 0\cr 
& b_{21}&b_{22} & 0& -\lambda\cr 
}, 
\] and 
$\det(d_{W}^{V})=(a_{11}a_{22}-a_{12}a_{21})\lambda^2
-(a_{11}b_{22}-a_{12}b_{21}-a_{21}b_{12}+a_{22}b_{11})\lambda
+(b_{11}b_{22}-b_{12}b_{21})$. 
We conclude 
$\mathbb{C}[\mathfrak{gl}_2^{\oplus 2}]^{SL_2\times SL_2}$ $=$ $\mathbb{C}[\det(A), a_{11}b_{22}-a_{12}b_{21}-a_{21}b_{12}+a_{22}b_{11},\det(B)]$. 
\end{example}

\subsection{Domokos-Zubkov technique}\label{subsection:Domokos-Zubkov-technique}

The following strategy is proved in \cite{MR1825166}. 
Let $Q=(Q_0,Q_1)$ be an arbitary quiver (where cycles are allowed) and let $\beta$ be a dimension vector. 
Choose a set $v_1,\ldots, v_n, w_1,\ldots, w_m\in Q_0$ 
of vertices (possibly repeating) such that 
\begin{equation}\label{equation:domokos-zubkov-dimension-equality}
\sum_{i=1}^n \beta(v_i)= \sum_{j=1}^m \beta(w_j). 
\end{equation} 
Let $W\in Rep(Q,\beta)$ be a general representation and consider 
\[ M: \bigoplus_{i=1}^n W(v_i) \longrightarrow \bigoplus_{j=1}^m W(w_j),  
\] 
where $M=(m_{ij})$, with each $m_{ij}$ being a linear combination of a general representation of paths in $Q$ from $v_i$ to $w_j$, including the zero path which corresponds to the zero matrix and the identity matrix if $v_i=w_j$. 

\begin{theorem}[Domokos-Zubkov]  
Polynomial coefficients of the determinant of $M$ are in the algebra 
$\mathbb{C}[Rep(Q,\beta)]^{SL_{\beta}}$.  
Choose all possible combination of vertices satisfying 
\eqref{equation:domokos-zubkov-dimension-equality} 
and all possible combination of representations of paths to obtain polynomial generators of 
$\mathbb{C}[Rep(Q,\beta)]^{SL_{\beta}}$. 
\end{theorem}

\begin{example}\label{example:1-Jordan-Domokos-Zubkov-example}
Let $Q$ be the $1$-Jordan quiver $\xymatrix@-1pc{ \stackrel{1}{\bullet}\ar@(ru,rd)^{a}}$
and let $\beta=2$. 
Let $W(a)=(a_{ij})$ be a general representation of $Rep(Q,\beta)$. 
Let $n=m=1$ with $v_1=w_1=1$. 
Let $M:W(1)=\mathbb{C}^2 \longrightarrow  W(1)=\mathbb{C}^2$, 
where $M=(t W(a))$. 
Then $\det(M)=t^2\det(W(a))$. 
So $\det(W(a))$ is an invariant polynomial. 
Now let $n=m=2$ with $v_1=v_2=w_1=w_2=1.$
Then $M:\mathbb{C}^2\oplus \mathbb{C}^2\longrightarrow \mathbb{C}^2\oplus \mathbb{C}^2$, 
and let 
\[ 
M = \begin{pmatrix}
 s W(a)& t\Inoindex_2 \\ 
 u\Inoindex_2 & v\Inoindex_2 
\end{pmatrix} 
= \begin{pmatrix}
 s a_{11}&s a_{12} &t &0 \\  
 s a_{21}& s a_{22} &0 &t \\
 u& 0&v & 0\\  
 0& u&0 & v  
\end{pmatrix}, 
\] 
where $s,t,u,v$ are formal variables. 
Then 
\[ 
\det(M)=t^2 u^2 -(a_{11}+a_{22}) stuv +(a_{11}a_{22}-a_{12}a_{21}) s^2 v^2, 
\] 
and $\tr(W(a))=a_{11}+a_{22}$ is also an invariant polynomial. 
Thus, $\mathbb{C}[Rep(Q,\beta)]^{SL_{\beta}}=\mathbb{C}[\tr(W(a)),\det(W(a))]$, 
which coincides with the classical result that generators of the ring of invariant polynomials are precisely the coefficients of the characteristic polynomial of $W(a)$. 
\end{example}

\begin{example}\label{example:2-Jordan-Domokos-Zubkov-example}
Let $Q$ be the $2$-Jordan quiver $\xymatrix@-1pc{ \ar@(ld,lu)^a \stackrel{1}{\bullet}\ar@(ru,rd)^{b}}$
and let $\beta=2$. 
Let $n=m=2$ with $v_1=v_2=w_1=w_2=1.$
First,  
let $M:\mathbb{C}^2\oplus \mathbb{C}^2\longrightarrow \mathbb{C}^2\oplus \mathbb{C}^2$, 
where 
\[ 
M = \begin{pmatrix}
 s W(a)& t\Inoindex_2 \\ 
 u\Inoindex_2 & v\Inoindex_2 
\end{pmatrix}
= \begin{pmatrix} 
 s a_{11}&s a_{12} &t &0 \\ 
 s a_{21}& s a_{22} &0 &t \\
 u& 0&v & 0\\ 
 0& u&0 & v  
\end{pmatrix}. 
\] 
Then 
\[ 
\det(M)=t^2 u^2 -(a_{11}+a_{22}) stuv +(a_{11}a_{22}-a_{12}a_{21}) s^2 v^2.  
\] 
So $\Tr(W(a))$ and $\det(W(a))$ are invariant polynomials. 
Continuing to let $m=n=2$, 
we replace $W(a)$ with $W(b)$ to conclude that 
$\Tr(W(b))$ and $\det(W(b))$ are invariant polynomials. 
Next, consider 
$M:\mathbb{C}^2\oplus \mathbb{C}^2\longrightarrow \mathbb{C}^2\oplus \mathbb{C}^2$, 
where 
\[ 
M = \begin{pmatrix}  
 s W(a)			 & t W(b)\\ 
 u\Inoindex_2 & v\Inoindex_2 
\end{pmatrix}
= \begin{pmatrix} 
 s a_{11}& s a_{12} &t b_{11}&tb_{12} \\  
 s a_{21}& s a_{22} &t b_{21}&tb_{22} \\ 
 u& 0&v & 0\\ 
 0& u&0 & v  
\end{pmatrix}. 
\] 
Then $\det(M)=\det(W(b))t^2 u^2
-(a_{11}b_{22}-a_{12}b_{21}+a_{22}b_{11}-a_{21}b_{12}) stuv
+\det(W(a))s^2 v^2$.
We obtain another polynomial $a_{11}b_{22}-a_{12}b_{21}+a_{22}b_{11}-a_{21}b_{12}$, which is, in fact, an invariant polynomial. Thus, 
$\mathbb{C}[\mathfrak{gl}_2^{\oplus 2}]^{SL_2(\mathbb{C})}$ $=$ 
$\mathbb{C}[\det(W(a)), \Tr(W(a)), \det(W(b)), \Tr(W(b)), a_{11}b_{22}-a_{12}b_{21}+a_{22}b_{11}-a_{21}b_{12}]$. 
\end{example}

\subsection{Schofield-van den Bergh technique}
 
Schofield and van den Bergh's strategy 
(\cite{MR1908144}, Theorem 2.3) 
to produce semi-invariants for any quiver is to prove that (homogeneous multilinear) semi-invariants of $Rep(Q,\beta)$ are spanned by determinantal semi-invariants. 
Since this technique is similar to Domokos and Zubkov's technique, we will omit the details.

\chapter{Filtered quiver varieties}\label{chapter:filtered-quiver-varieties}

\section{Construction of filtered quiver varieties}

Parabolic group actions arise naturally in mathematics. For example, let $B$ be the standard Borel in $GL_n(\mathbb{C})$ (the set of complex invertible upper triangular matrices) and let $\mathfrak{b}=\lie(B)$ be the set of upper triangular matrices in $\mathfrak{gl}_n$, where $\mathfrak{gl}_n$ is the set of all $n\times n$ matrices. 
We are interested in studying the orbit structure of $B$ acting on $\mathfrak{b}$ by conjugation or, in another way to put it, we are interested in studying a variety that has a certain notion of filtrations associated to it. 
Let $\widetilde{\mathfrak{g}}$ be the Grothendieck-Springer resolution (cf. Section~\ref{section:GS-resolution-intro}) and let $G=GL_n(\mathbb{C})$.
Lemma~\ref{lemma:g-s-resolution-and-B-action-on-borel} gives the isomorphism 
$\mathfrak{b}/B\cong \widetilde{\mathfrak{g}}/G$ so the $B$-orbits on $\mathfrak{b}$ are of significant interest in and of itself.

Let us now discuss the $B$-adjoint action on $\mathfrak{b}$ further.  
Let $F^{\bullet}:\{ 0\}\subseteq \mathbb{C}^1\subseteq \mathbb{C}^2 \subseteq \ldots \subseteq \mathbb{C}^n$ be the complete standard filtration of vector spaces in $\mathbb{C}^n$.  
Then $\mathfrak{b}$ can be identified with the subspace of linear maps $\mathbb{C}^n\stackrel{f}{\rightarrow} \mathbb{C}^n$ such that $f(\mathbb{C}^k)\subseteq \mathbb{C}^k$ for all $k$. 
Since one must preserve the filtration of vector spaces while changing the basis, we have the $B$-action on the domain and the codomain.   
So points of $\mathfrak{b}/B$ correspond to equivalence classes of linear maps preserving the complete standard filtration of vector spaces, where two maps $f$ and $g$ are equivalent if there exists a change-of-basis that will take $f$ to $g$.

We now give the construction of filtered quiver varieties in the general setting. 
Let $Q=(Q_0,Q_1)$ be a quiver and let $\beta=(\beta_1,\ldots, \beta_{Q_0})\in \mathbb{Z}_{\geq 0}^{Q_0}$ be a dimension vector.  
Let 
$F^{\bullet}:0\subseteq \mathbb{C}^{\gamma^1}\subseteq \mathbb{C}^{\gamma^2}\subseteq \ldots \subseteq \mathbb{C}^{\beta}$ 
be a filtration of vector spaces such that we have the filtration 
$F_i^{\bullet}:0\subseteq \mathbb{C}^{\gamma^1_i}\subseteq \mathbb{C}^{\gamma^2_i}\subseteq 
\ldots \subseteq \mathbb{C}^{\beta_i}$ of vector spaces at vertex $i$. 
Let $Rep(Q,\beta)$ be the quiver variety in the classical sense (without the filtration of vector spaces imposed). 
Then $F^{\bullet}Rep(Q,\beta)$ is a subspace of $Rep(Q,\beta)$ whose linear maps preserve the filtration of vector spaces at every level.  
Let $P_i\subseteq GL_{\beta_i}(\mathbb{C})$ be the largest parabolic group preserving the filtration of vector spaces at vertex $i$. 
Then the product $\mathbb{P}_{\beta}:= \displaystyle{\prod_{i\in Q_0}P_i}$ of parabolic groups acts on 
$F^{\bullet}Rep(Q,\beta)$ as a change-of-basis.

Now, we will say a parabolic Lie algebra $\mathfrak{p}=\lie(P)$ is in {\em standard form} 
if a general parabolic matrix has indeterminates along its block diagonal and upper triangular portion of the matrix and zero below the diagonal blocks.   
We refer to \cite{MR2772068} for a discussion on quiver Grassmannians and quiver flag varieties, which are related to filtered quiver varieties but they are not the same. In the next three sections, we will discuss quiver Grassmannians, quiver flag varieties, and filtered quiver varieties. 
 
\subsection{Quiver Grassmannians}\label{subsection:quiver-grassmannians}
  
Let $\gamma=(\gamma_1,\ldots, \gamma_{Q_0})\in \mathbb{Z}_{\geq 0}^{Q_0}$ be a dimension vector. 
If each $\gamma_k \leq \beta_k$ for each $k\in Q_0$, then we define 
$Gr_{\gamma}(\beta) := \displaystyle{\prod_{k\in Q_0}Gr_{\gamma_k}(\beta_k)}$ to be the product of Grassmannians,  
where each $Gr_{\gamma_k}(\beta_k)$ parameterizes $\gamma_k$-dimensional subspaces in a $\beta_k$-dimensional vector space. So $Gr_{\gamma}(\beta)$ has dimension 
$\displaystyle{\sum_{k\in Q_0} \gamma_k(\beta_k-\gamma_k)}$,  
which is obtained by counting the number of indeterminates after row reducing a certain product of matrices, i.e., each matrix corresponding to $Gr_{\gamma_k}(\beta_k)$ is of size $\gamma_k\times \beta_k$ with the rows spanning the $\gamma_k$-dimensional subspace. 
Writing $\mathbb{G}_{\beta} :=
\displaystyle{\prod_{k\in Q_0}GL_{\beta_{k}}(\mathbb{C})}$, $Gr_{\gamma}(\beta)$ can also be thought of as $\mathbb{G}_{\beta}/\mathbb{P}_{\gamma}$, where $\mathbb{P}_{\gamma} := \displaystyle{\prod_{k\in Q_0}} P_{\gamma_k}\subseteq \mathbb{G}_{\beta}$ is a product of largest parabolic matrices with each $P_{\gamma_k}$ stabilizing a $\gamma_k$-dimensional subspace in a $\beta_k$-dimensional vector space. 
%
%
%
We now define 
\[ \begin{aligned}
 	Gr_{\gamma}^Q(\beta) := 
		\{ ((U(k))_{k\in Q_0}, (W(a))_{a\in Q_1} ) : 
				U(k)\cong \mathbb{C}^{\gamma_k}& \mbox{ is a subspace of } W(k)\cong \mathbb{C}^{\beta_k} \mbox{ and } \\
		   	&W(a)U(i)\subseteq U(j), \mbox{ where } a:i\rightarrow j 
		\}  																																								\\
		\end{aligned}
 \]  to be the 
 {\em universal quiver Grassmannian}  
  of $\gamma$-dimensional subrepresentations in $\beta$-dimensional representations of $Q$. 
Note that $Gr_{\gamma}^Q(\beta)$ is a subvariety of the pair $Gr_{\gamma}(\beta)\times Rep(Q,\beta)$, with 
 the product $\mathbb{G}_{\beta}$ of invertible matrices acting on $Gr_{\gamma}^Q(\beta)$ diagonally on each factor, i.e., for $(g_{\beta_k})_{k\in Q_0}\in \mathbb{G}_{\beta}$ and having chosen a basis for a sequence $(W(k))_{k\in Q_0}$  of vector spaces (and thus for $(U(k))_{k\in Q_0}$ as well), the $\mathbb{G}_{\beta}$-action is given as 
 \[ 
 (g_{\beta_k})_{k\in Q_0}. ( ( u_k )_{k\in Q_0}, (W(a))_{a\in Q_1} ) 
 = ((g_{\beta_k} u_k )_{k\in Q_0}, (g_{\beta_{ha}}W(a)  g_{\beta_{ta}}^{-1} )_{a\in Q_1} ),    
 \] where $u_k$ is a vector in $U(k)$ and $W(a)$ is a matrix in  
 $M_{\beta_{ha}\times \beta_{ta}}$. 
For notational simplicity (this will be apparent in Section~\ref{subsection:quiver-flag-variety} when we discuss quiver flag varieties), 
we may also write  
\[ 
Gr_{\gamma}^Q(\beta) = \{ (U,W) \in Gr_{\gamma}(\beta)\times Rep(Q,\beta): U \mbox{ is a subrepresentation of }W \}.  
\]

Now consider the projections $Gr_{\gamma}^Q(\beta)\stackrel{p_1}{\longrightarrow} Gr_{\gamma}(\beta)$ and 
$Gr_{\gamma}^Q(\beta)\stackrel{p_2}{\longrightarrow} Rep(Q,\beta)$ onto the two factors, where $p_1$ and $p_2$ are $\mathbb{G}_{\beta}$-equivariant maps.  
We will give explicit description of the fibers $p_1^{-1}(U)$ and $p_2^{-1}(W)$. 
For a fixed sequence $(U(k) )_{k\in Q_0}$ of subspaces in the product $Gr_{\gamma}(\beta)$ of Grassmannians,  
we can choose a sequence $( U(k)^{\perp} )_{k\in Q_0}$ of complementary subspaces such that $U(k) \oplus U(k)^{\perp}=W(k)\cong \mathbb{C}^{\beta_k}$ for each vertex $k\in Q_0$. 
Since the preimage of $(U(k))_{k\in Q_0}$ under the first projection map $p_1$ must satisfy $W(a)U(i)\subseteq U(j)$ for each arrow $a:i\rightarrow j$, 
we have 
\begin{equation}\label{equation:quiver-grassmannian-filtered-quiver-variety-one-step-filtration} 
		p_1^{-1}( U ) = 
		  U \times \left( \left( 
		\begin{array}{cc}
		\Hom_Q(U(i), U(j)) & \Hom_Q(U(i)^{\perp},U(j)) \\
		0 & \Hom_Q(U(i)^{\perp},U(j)^{\perp})   \\
		\end{array}
		\right) \right)_{a:i\rightarrow j}.   
\end{equation}  
Thus the fiber of $p_1$ over $U$ is a vector space of rank 
  \[ \sum_{a:i\rightarrow j, a\in Q_1} \beta_i\gamma_j + (\beta_i-\gamma_i)(\beta_j-\gamma_j)
  =\sum_{a:i\rightarrow j, a\in Q_1} \beta_i \beta_j +\gamma_i\gamma_j - \beta_j\gamma_i,  
  \]  
since $\Hom_Q( U(i),U(j) )$ is a vector space of dimension $\gamma_i\gamma_j$,  
$\Hom_Q( U(i)^{\perp},U(j))$ is a vector space of dimension $(\beta_i-\gamma_i)\gamma_j$, and 
$\Hom_Q( U(i)^{\perp},U(j)^{\perp})$ is a vector space of dimension $(\beta_i-\gamma_i)(\beta_j-\gamma_j)$. 
So $p_1$ is flat.

For $W\in Rep(Q,\beta)$, the preimage $p_2^{-1}(W)$ parameterizes all sequences $(U(k))_{k\in Q_0}$ of subspaces  
  of dimension vector $\gamma$ such that $W(a)U(i)\subseteq U(j)$ 
  for each arrow $a:i\rightarrow j$ in $Q$. Thus, we identify the fiber of 
    $p_2$ over a fixed representation $W$ as the  
    {\em quiver Grassmannian}  
    $Gr_{\gamma}(W) :=p_2^{-1}(W)$ of $W$ of dimension vector $\gamma$; it essentially parameterizes all $\gamma$-dimensional subrepresentations of $W$.  Thus, $p_2$ is projective.

\begin{example}\label{example:quiver-grassmannian-A1}
Consider the $A_1$-Dynkin quiver, and consider 
$\beta= n$.  
Then $Rep(Q,\beta)\cong \mathbb{C}^n$, an $n$-dimensional complex vector space. 
Now let the dimension vector $\gamma$ be $m \leq n$.  
Then the universal quiver Grassmannian is   
$Gr_{m}^{Q}(n) =Gr_m(n)\times \mathbb{C}^n$  
with the quiver Grassmannian $Gr_m(\mathbb{C}^n) = Gr_m(n)$ coinciding with the classical notion of a Grassmannian: the space of $m$-dimensional subspaces in an $n$-dimensional complex vector space.   
\end{example}

\begin{example}\label{example:quiver-grassmannian-A2} 
Consider the $A_2$-Dynkin quiver 
$\xymatrix@-1pc{ \stackrel{1}{\bullet} \ar[rr]^a & & \stackrel{2}{\bullet}}$ and let $\beta =(n_1,n_2)$. 
Consider $\gamma = (m_1, m_2)$, with $m_{k}\leq n_k$ for $k=1,2$. Then 
for $W\in M_{n_2\times n_1}\cong Rep(Q,\beta)$,  
the quiver Grassmannian $Gr_{(m_1,m_2)}(W)$ is a closed subvariety of $Gr_{m_1}(\mathbb{C}^{n_1})\times Gr_{m_2}(\mathbb{C}^{n_2})$ which satisfies 
$W(U(1))\subseteq U(2)$, where $\dim (U(i))=m_i$ for $i=1,2$.    
\end{example}

\subsection{Quiver flag varieties}\label{subsection:quiver-flag-variety}

In this section, we generalize Section~\ref{subsection:quiver-grassmannians}   
by discussing the notion of quiver flag varieties. 
These varieties also appear in the literature as quiver flag manifolds. 
Similar as before, 
let $\gamma^{1}, \gamma^2, \ldots, \gamma^{N}=\beta \in \mathbb{Z}_{\geq 0}^{Q_0}$
be dimension vectors with coordinates 
$\gamma^{k}=(\gamma_1^{k},\ldots, \gamma_{Q_0}^{k})$ 
satisfying $\gamma_i^k \leq \gamma_i^{k+1}$  
for each $i\in Q_0$ and for each $1\leq k< N$.  
Define $Fl_{\gamma^{\bullet}}(\beta) := 
\displaystyle{\prod_{i\in Q_0} Fl_{\gamma_i^{\bullet}}(\beta_i)}$  
to be the product of flag varieties,  
where each $Fl_{\gamma_i^{\bullet}}(\beta_i)$
is the usual flag variety parametrizing flags of subspaces 
\[ 0 \subseteq U_i^{1}\subseteq U_i^{2} \subseteq \ldots \subseteq U_i^{N} = W(i) \cong \mathbb{C}^{\beta_i}
\] 
   with $\dim U_i^{k} =\gamma_i^k$ for all $i\in Q_0$ and for all $1\leq k< N$. 
We define the {\em universal quiver flag} to be 
\begin{equation}\label{equation:universal-quiver-flag}
  \begin{aligned}
  Fl_{\gamma^{\bullet}}^Q(\beta) := 
  \{ &(U^{1},\ldots, U^{N}, W) \in Fl_{\gamma^{\bullet}}(\beta)\times Rep(Q,\beta): 
    0 \subseteq U^{1} \subseteq U^{2} \subseteq \ldots \subseteq U^{N} = W \mbox{ is } \\
    &\mbox{ a chain of subrepresentations of }W   
    \mbox{ and } W(a)U_i^{k}\subseteq U_j^{k}\:\: \forall \: a:i\rightarrow j \mbox{ and }  
    \forall \: 1\leq k \leq N
            \}.   \\ 
    \end{aligned}
 \end{equation}

Consider the two $\mathbb{G}_{\beta}$-equivariant projections $Fl_{\gamma^{\bullet}}^Q(\beta)\stackrel{p_1}{\longrightarrow} Fl_{\gamma^{\bullet}}(\beta)$
and 
$Fl_{\gamma^{\bullet}}^Q(\beta)\stackrel{p_2}{\longrightarrow} Rep(Q,\beta)$. 
First, let us view $U_i^{k}$ as subspaces in $W(i)$ for each $i\in Q_0$. 
The fiber of $p_1$ over the tuple 
$(U^{1},\ldots, U^{N})$ is 
again a homogeneous vector bundle isomorphic to 
\begin{equation}\label{eq:quiver-flag-vector-bundle} 
\bigoplus_{a\in Q_1}
\left( 
\bigoplus_{k=1}^{N} \Hom\left(U_{ta}^{k}/U_{ta}^{k-1}, U_{ha}^{k} \right) 
\right) 
\end{equation}
where $U_{ta}^{k}/U_{ta}^{k-1}:=\{v\in U_{ta}^{k}:v\perp u \:\:\forall u\in U_{ta}^{k-1} \}$    
and $U_{ta}^{0}:=\{ 0\}$. So $p_1$ is flat. 
On the other hand, 
$p_2^{-1}(W)$ parameterizes all flags 
$0 \subseteq U^{1} \subseteq U^{2} \subseteq \ldots \subseteq U^{N} = W$ of subspaces with prescribed dimension vectors $\gamma^{k}$, with each $U^{k}$ being a subrepresentation of $U^{k+1}$. 
So $p_2$ is projective.  
The fiber $p_2^{-1}(W) \cong Fl_{\gamma^{\bullet}}(W)$ of $p_2$ over $W$ is called the  
{\em quiver flag variety}.

\begin{example}\label{example:quiver-flag-A1}
Consider the $A_1$-Dynkin quiver and consider $\beta=n$. 
Let $\gamma^{i}=i$ for $1\leq i\leq n$.  
Then the quiver flag variety $Fl_{\gamma^{\bullet}}(\mathbb{C}^n)$ 
is isomorphic to the complete flag variety for $\mathbb{C}^n$. 
\end{example}  
 
\begin{example}\label{example:quiver-flag-generalizing-grassmannian}
Let $Q$ be any quiver and consider $N=2$ to obtain the quiver Grassmannian from Section~\ref{subsection:quiver-grassmannians}. 
\end{example}

It is worth mentioning that when rewriting the direct sum of vector bundles in (\ref{eq:quiver-flag-vector-bundle}) as a product of matrices (one matrix representation for each arrow $a\in Q_1$), each matrix has the form of a parabolic matrix (cf. see \eqref{equation:quiver-grassmannian-filtered-quiver-variety-one-step-filtration} when $N=2$). 
In the next section, we will further investigate the fibers of $p_1$.

\subsection{Filtered quiver varieties}\label{subsection:filtered-quiver-varieties}

{\em Filtered quiver varieties} are precisely the fibers of $p_1$ over a fixed sequence of vector spaces in \eqref{equation:universal-quiver-flag}. 
A more direct construction is given as follows: 
let $Q$ be any quiver. 
Fix a set of dimension vectors $\gamma^{1},\gamma^2, \ldots, \gamma^{N}= \beta \in \mathbb{Z}_{\geq 0}^{Q_0}$ such that   
$\gamma_i^k\leq \gamma_i^{k+1}$ for each $i\in Q_0$ and for each $1\leq k < N$.   
Let $F^{\bullet}: 0 \subseteq U^{1} \subseteq U^{2} \subseteq \ldots \subseteq U^{N} = W$ be a flag of subspaces such that 
$\dim U_i^{k} = \gamma_i^k$ for each $i\in Q_0$ and for each $1\leq k\leq N$.


\begin{definition}\label{definition:filtered-quiver-representation}   
The {\em filtered quiver variety} is   
\[  
F^{\bullet}Rep(Q,\beta):= \{ W\in Rep(Q,\beta):
W(U_{ta}^{k})\subseteq U_{ha}^{k}\:\: \forall \: 1\leq k\leq N, \:\:\forall \:a\in Q_1
\}. 
\]   
\end{definition}  

If $P_i\subseteq GL_{\beta_i}(\mathbb{C})$ is the largest parabolic subgroup acting as a change-of-basis while preserving the filtration of vector spaces at vertex $i\in Q_0$, 
then $\mathbb{P}_{\beta}:= \displaystyle{\prod_{i\in Q_0}P_i}$ naturally acts on $F^{\bullet}Rep(Q,\beta)$ as a change-of-basis.

\begin{remark}
A filtered quiver variety is a representation space where the filtration is a structure on a representation, not on the quiver itself. 
\end{remark}

\begin{remark}\label{remark:filtration-on-framed-quiver}    
Recall a framed quiver in Definition~\ref{definition:framing-to-construct-semi-invariants}. 
We construct a 
{\em framed filtered quiver variety}, 
denoted as 
$F^{\bullet}Rep({Q}^{\dagger},\beta^{\dagger} )$, as follows: 
the filtration $F^{\bullet}$ on the framed variety 
$F^{\bullet}Rep({Q}^{\dagger},\beta^{\dagger} )$ 
modifies vector spaces assigned to vertices in $Q_0$,   
but not those vector spaces assigned to framed vertices 
$Q_0^{\natural}$.  
Writing the dimension vector for a framed quiver as 
$\beta^{\dagger}=(\beta,\beta^{\natural})$, 
$\mathbb{P}_{\beta}$  
naturally acts on  
$F^{\bullet}Rep({Q}^{\dagger},\beta^{\dagger})$. 
If the tail of the underlying arrow 
of a general representation is not a framed vertex, 
then a general representation must preserve the fixed filtration $F^{\bullet}$ of vector spaces at both the head and the tail of the underlying arrow; 
such representation is possibly block upper triangular and has indeterminates in those entries that are not zero. 
For all other linear maps whose tail is a framed vertex, i.e., 
these are of the form 
$W(\iota_i)\in\Hom(\mathbb{C}^{\beta_i^{\natural}}, \mathbb{C}^{\beta_i} )$,  
we define these general representations to have indeterminates in all coordinate entries 
since we do not impose a filtration of vector spaces on framed vertices.   
\end{remark}  
     
\begin{example} 
Consider the framed quiver $\xymatrix@-1pc{\stackrel{1^{\natural}}{\circ}\ar[rr]^{\iota_1} & & \stackrel{1}{\bullet} \ar@(ru,rd)^{a}}$, and consider $\beta=(2,2)$. 
Let $F_1^{\bullet}:0\subseteq \mathbb{C}^1\subseteq \mathbb{C}^2$ be the complete standard filtration of vector spaces at vertex $1$. 
Then $F^{\bullet}Rep(Q^{\dagger},\beta^{\dagger})\cong \mathfrak{b}_2\oplus \Hom(\mathbb{C}^2,\mathbb{C}^2)$; that is, if $W$ is a general representation in $F^{\bullet}Rep(Q^{\dagger},\beta^{\dagger})$, then 
\[ 
W(\iota_1) = \begin{pmatrix}
x_{11} & x_{12} \\ 
x_{21} & x_{22}  
\end{pmatrix} 
\mbox{ while }
W(a) = \begin{pmatrix}
a_{11} & a_{12} \\ 
0 & a_{22}  
\end{pmatrix}.  
\] 
\end{example}

\begin{remark}\label{remark:double-quiver-filtered-setting}
Using Definition~\ref{definition:double-quiver-classical-setting}, we define the 
{\em filtered double quiver variety} 
as 
$F^{\bullet}Rep(\overline{Q},\overline{\beta}) := F^{\bullet}Rep(Q,\beta)\times F^{\bullet}Rep(Q^{op},\beta)$, where $Q^{op}=(Q_0,Q_1^{op})$.  
The group $\mathbb{P}_{\overline{\beta}}=\mathbb{P}_{\beta}$ acts on $F^{\bullet}Rep(\overline{Q},\overline{\beta})$ by acting on  
$F^{\bullet}Rep(Q,\beta)$ and $F^{\bullet}Rep(Q^{op},\beta)$ as a change-of-basis. 
\end{remark}   
 
In the classical setting, a double quiver variety is the same as the cotangent bundle of a quiver variety (cf. the paragraph following Definition~\ref{definition:double-quiver-variety-classical-setting}).  
However, in our setting,  
filtered double quiver varieties are not the same as the cotangent bundle of filtered quiver varieties since $F^{\bullet}Rep(Q^{op},\beta)$ may consist of representations that are block upper triangular. 
We will give a plausible definition of the cotangent bundle of a filtered quiver variety which is used throughout this thesis 
(cf. Definition~\ref{definition:cotangent-bundle-of-filtered-quiver-variety}).

\begin{example}\label{example:quiver-with-oriented-path}  
Consider the double $1$-Kronecker quiver 
$\xymatrix{\stackrel{1}{\bullet} \ar@/^/[rr]| {a} 	&& \stackrel{2}{\bullet} \ar@/^/[ll]|{a^{op}}}$ 
with the following filtration: 
$\gamma^1=(3,1)$, $\gamma^2 = (4,3)$, and $\gamma^3 =\beta=(5,4)$, 
where the vector space representation of $\gamma_i^k$ is the span of the first $\gamma_i^k$ standard basis vectors.   
If $W$ is a general representation in $F^{\bullet}Rep(\overline{Q},\overline{\beta})$, then 
\[ W(a) = \left( \begin{array}{ccccc} 
a_{11}&a_{12}&a_{13}& a_{14}&a_{15} \\ 
0&0&0& a_{24}&a_{25} \\
0&0&0& a_{34}&a_{35} \\
0&0&0& 0&a_{45} \\
\end{array} 
\right) \mbox{ while }
W(a^{op}) = \left( \begin{array}{cccc} 
c_{11}&c_{12}&c_{13}&c_{14} \\ 
c_{21}&c_{22}&c_{23}&c_{24} \\ 
c_{31}&c_{32}&c_{33}&c_{34} \\ 
0&c_{42}&c_{43}&c_{44} \\ 
0&0&0&c_{54} \\ 
\end{array}
\right).  
\]   
Note that both $W(a)$ and $W(a^{op})$ are in block upper triangular form. 
\end{example}

Next, we want to understand the cotangent bundle of filtered quiver varieties. 
We first begin with a definition. 

\begin{definition}\label{definition:-parabolic-lie-algebra-intersection-of-all-kernels}
Let $\lie(P)=\mathfrak{p}\subseteq \mathfrak{gl}_n$ be the (standard) parabolic Lie algebra. Define
 $\mathfrak{n}_{\mathfrak{p}}\subseteq \mathfrak{p}$ as the 
 Lie algebra of the 
 intersection of the kernels of all homomorphisms from $P$ to the multiplicative group $\mathbb{C}^*$. 
\end{definition} 

Next, we give a lemma to reinterpret 
$T^*(\mathfrak{b})=\mathfrak{b}\times \mathfrak{b}^*$  
as a filtered quiver variety, where $\mathfrak{b}^*=\Hom_{\mathbb{C}}(\mathfrak{b},\mathbb{C})$.

\begin{lemma}
Let $\mathfrak{b}$ be the standard Borel in $\mathfrak{gl}_n$ and let $\mathfrak{n}^+$ be the space of strictly upper triangular matrices in $\mathfrak{b}$. 
Then $\mathfrak{b}^*\cong \mathfrak{gl}_n/\mathfrak{n}^+$.
 Similarly, let $\lie(P)=\mathfrak{p}\subseteq \mathfrak{gl}_n$ be the (standard) parabolic Lie algebra and let $\mathfrak{n}_{\mathfrak{p}}\subseteq \mathfrak{p}$ be the 
 Lie algebra given in Definition~\ref{definition:-parabolic-lie-algebra-intersection-of-all-kernels}.   
 Then 
$\mathfrak{p}^*\cong \mathfrak{gl}_n/\mathfrak{n}_{\mathfrak{p}}$. 
\end{lemma} 

\begin{proof}
Without loss of generality, let $\mathfrak{b}$ be the Borel in standard form (block upper triangular) 
and let $\mathfrak{n}^+\subseteq \mathfrak{b}$ be the set of strictly upper triangular matrices. 
Consider 
\[ \mathfrak{gl}_n \stackrel{\Psi}{\longrightarrow}\mathfrak{b}^* =\Hom(\mathfrak{b}, \mathbb{C}), 
\hspace{4mm} \mbox{ where }
\Psi(s):\mathfrak{b}\rightarrow \mathbb{C} \mbox{ maps }\Psi(s)(r)=\Tr(rs). 
\]   
For $s\in \mathfrak{n}^+$, $\Tr(rs)=0$. 
So $\Psi$ factors through  
$\mathfrak{gl}_n/\mathfrak{n}^+$ and  
$\ker\Psi=\mathfrak{n}^+$. 
Next, we claim that $\Psi:\mathfrak{gl}_n/\mathfrak{n}^+\rightarrow \mathfrak{b}^*$  
is a bijection and $B$-equivariant.  
It is a bijection since if 
$\Psi(\overline{s})(r)=\Tr(\overline{s}r)=0$ for all $r\in \mathfrak{b}$, then $s\in \mathfrak{n}^+$ 
(so $\overline{s}=0+\mathfrak{n}^+$); so $\Psi$ is injective.
For surjectivity, 
for $\Tr(r - )=t\in \mathfrak{b}^*$, it is clear that there exists $\overline{s}\in \mathfrak{gl}_n/\mathfrak{n}^+$ such that $\Tr(r\overline{s})=t$. 
Finally, $\Psi$ is $B$-equivariant since  
$\Psi(b\overline{s}b^{-1})(brb^{-1})=\Tr(r\overline{s})=\Psi(\overline{s})(r)$.  
Thus we have  
$\mathfrak{gl}_n/\mathfrak{n}^+ \cong \mathfrak{b}^*$.

Since we have a similar argument for  $\mathfrak{p}\subseteq \mathfrak{gl}_n$ 
and 
$\mathfrak{n}_{\mathfrak{p}}$ of $\mathfrak{p}$, we are done. 
\end{proof}

Note if $\mathfrak{b}$ is the space of upper triangular matrices, 
then we can essentially think of $\mathfrak{b}^*$ as the space of lower triangular matrices.  
Similarly, if $\mathfrak{p}$ is the space of block upper triangular matrices, then we can essentially think of 
$\mathfrak{p}^*$ as the space of block lower triangular matrices.

\begin{definition}  
Let $Q$ be a quiver and consider $\beta=(n,\ldots, n)\in \mathbb{Z}_{\geq 0}^{Q_0}$, a dimension vector. 
Let $Q^{op}$ be the opposite quiver and let $F_i^{\bullet}$ be the filtration 
 $0 \subseteq \mathbb{C}^{\gamma_1}\subseteq \mathbb{C}^{\gamma_2}\subseteq \ldots \subseteq \mathbb{C}^{\gamma_{N-1}} \subseteq \mathbb{C}^{n}$  
 of vector spaces at each vertex $i\in Q^{op}$, where $\mathbb{C}^{\gamma_k}$ 
 is the vector space spanned by the first $\gamma_k$ standard basis vectors.    
Since $F^{\bullet}Rep(Q^{op},\beta)\cong \mathfrak{p}^{Q_1}$, we define  
$\nil(F^{\bullet}Rep(Q^{op},\beta)) := \mathfrak{n}_{\mathfrak{p}}^{Q_1}$, 
where $\mathfrak{n}_{\mathfrak{p}}$  
is defined in Definition~\ref{definition:-parabolic-lie-algebra-intersection-of-all-kernels}. 
\end{definition}

\begin{definition}\label{definition:cotangent-bundle-of-filtered-quiver-variety}  
Let 
$F_i^{\bullet}: 0 \subseteq \mathbb{C}^{\gamma_1}\subseteq \mathbb{C}^{\gamma_2}\subseteq \ldots \subseteq \mathbb{C}^{\gamma_{N-1}} \subseteq \mathbb{C}^{n}$ 
be a filtration at each $i\in Q_0$, where $\mathbb{C}^{\gamma_k}$ is the space spanned by the first $\gamma_k$ standard basis vectors.  
We define 
$T^*F^{\bullet}Rep(Q,\beta)$ $:=$ 
$F^{\bullet}Rep(Q,\beta)\times F^{\bullet}Rep(Q,\beta)^*$, where 
\[ F^{\bullet}Rep(Q,\beta)^* := Rep(Q^{op},\beta)/\nil(F^{\bullet}Rep(Q^{op},\beta));   
\]   
$T^*F^{\bullet}Rep(Q,\beta)$ 
is the 
{\em cotangent bundle of a filtered quiver variety}.  
\end{definition}

The $\mathbb{P}_{\beta}$-action on $F^{\bullet}Rep(Q,\beta)$ naturally extends to $T^*F^{\bullet}Rep(Q,\beta)$. 

\begin{example} 
Let $Q$ be the $1$-Jordan quiver $\xymatrix@-1pc{\stackrel{1}{\bullet} \ar@(ru,rd)}\hspace{4mm}$ 
and consider $\beta=2$. Impose the complete standard filtration of vector spaces at vertex $1$. 
Then $F^{\bullet}Rep(Q,\beta)= \mathfrak{b}_2 \cong F^{\bullet}Rep(Q^{op},\beta)$ 
and $\nil(F^{\bullet}Rep(Q^{op},\beta))=\mathfrak{n}^+$. 
So $F^{\bullet}Rep(Q,\beta)^*=\mathfrak{gl}_2/\mathfrak{n}^+\cong \mathfrak{b}_2^*$ 
and 
$T^*F^{\bullet}Rep(Q,\beta)=\mathfrak{b}_2\times \mathfrak{b}_2^*$. 
\end{example}

\begin{example} 
In the degenerate case when $N=1$, i.e., the case when the filtration consists of $0$-dimensional vector spaces and the entire space $\mathbb{C}^n$, we have $F^{\bullet}Rep(Q,\beta)^*\cong Rep(Q^{op},\beta)$.
Thus, the cotangent bundle  
$T^*F^{\bullet}Rep(Q,\beta)$ coincides with the classical setting $T^*Rep(Q,\beta)\cong \mathfrak{gl}_n^{Q_1}\times \mathfrak{gl}_n^{Q_1}$. 
\end{example}

Next, recall Definition~\ref{definition:framing-to-construct-semi-invariants} and Remark~\ref{remark:filtration-on-framed-quiver}. 

\begin{definition} 
Let $Q^{\dagger}$ be a framed quiver, and consider the assumptions in Definition~\ref{definition:cotangent-bundle-of-filtered-quiver-variety}. 
We define the 
{\em cotangent bundle of a framed filtered quiver variety} 
as  
\[ T^*F^{\bullet}Rep(Q^{\dagger},\beta^{\dagger})
= F^{\bullet}Rep(Q^{\dagger},\beta^{\dagger})
\times 
F^{\bullet}Rep(Q^{\dagger},\beta^{\dagger})^*, 
\] 
where 
\[
\begin{aligned} 
F^{\bullet}Rep(Q^{\dagger},\beta^{\dagger})^*
&= (Rep(Q^{op},\beta)/\nil(F^{\bullet}Rep(Q^{op},\beta))) 
\times 
\Hom(\mathbb{C}^{\beta^{\natural}},\mathbb{C}^{\beta})^* \\ 
&\cong  
(Rep(Q^{op},\beta)/\nil(F^{\bullet}Rep(Q^{op},\beta))) 
\times 
\Hom(\mathbb{C}^{\beta}, \mathbb{C}^{\beta^{\natural}}),  \\ 
\end{aligned} 
\] 
with $\Hom(\mathbb{C}^{\beta^{\natural}},\mathbb{C}^{\beta})^* \cong 
\Hom(\mathbb{C}^{\beta}, \mathbb{C}^{\beta^{\natural}}) 
:= 
\displaystyle{\prod_{i\in Q_0} 
\Hom(\mathbb{C}^{\beta_i}, \mathbb{C}^{\beta_i^{\natural}})}$. 
\end{definition}

Similar as before, $\mathbb{P}_{\beta}$ acts on 
$T^*F^{\bullet}Rep(Q^{\dagger},\beta^{\dagger})$ in the obvious way.

\chapter{Categories of filtered and graded quiver varieties}\label{chapter:categories-filt-graded-quiver-varieties}

In this chapter, we will discuss categories of associated graded, graded, and filtered quiver varieties. 
We begin with some basic facts. 

\begin{definition}\label{definition:filtration-of-a-finitely-generated-module-vs-algebra}
Let $[a,b]$ be a set of integers between $a\in \mathbb{Z}$ and $b\in \mathbb{Z}$. 
We define a {\em filtration}  
of a finitely generated module $F$ to be 
\[ 
F_a \subseteq F_{a+1}\subseteq \ldots \subseteq F_b 
\] 
such that 
$F_a =\{ 0\}$ and $F_b=F$. 
\end{definition} 

Equivalently, a filtration of $F$ is $F_k\subseteq \displaystyle{\bigcup_k F_k}$ such that 
$F_a =0$ for $a <\!\!< 0$ and $F_b=\displaystyle{\bigcup_k F_k}$ for $b>\!\!>0$. 
Throughout this thesis, we assume a filtration (of a module or a vector space or a ring) is separated and exhaustive.

\begin{definition}\label{definition:induced-filtration}  
Let $X_j$ be a {\em filtered vector space}, 
i.e., $\{ 0\}= X_{j,0} \subseteq X_{j,1} \subseteq \ldots \subseteq X_{j,k}=X_j$.  
If $Y_j \subseteq X_j$ is a subspace, then the {\em induced filtration of $Y_j$} is 
$Y_{j,l} = Y_j\cap X_{j,l}$.  
If $X_j \stackrel{\varphi}{\twoheadrightarrow} Y_j$ is a quotient map, 
then the {\em induced filtration of $Y_j$}  
is defined as 
$Y_{j,l} = \varphi(X_{j,l})$. 
\end{definition}  

\begin{definition}\label{definition:filtered-linear-map}
Let $X_j$ and $Y_j$ be filtered vector spaces and let $X_j\stackrel{\varphi}{\longrightarrow}Y_j$ be a linear map.
We say $\varphi$ is {\em filtered} if $\varphi(X_{j,l})\subseteq Y_{j,l}$ for each $1\leq l\leq k$. 
\end{definition}

\begin{definition}\label{definition:filtration-on-direct-sum}  
For each $j\in J$, let $X_j$ be a filtered vector space.  
We define the 
{\em induced filtration $F_l(\displaystyle{\bigoplus_j X_j})$ on  
the direct sum} 
of $X_j$ to be $F_l(\displaystyle{\bigoplus_j X_j}):= 
\displaystyle{\bigoplus_j X_{j,l}}$, where $1\leq l\leq k$. 
\end{definition} 

\begin{definition}\label{definition:filtered-hom}
Let $X$ and $Y$ be filtered vector spaces. We define $F^{\bullet}\Hom(X,Y)\subseteq \Hom(X,Y)$  
to be the subspace of all linear maps  
$\varphi:X\rightarrow Y$ such that  
$\varphi(X_l)\subseteq Y_l$ for each $0\leq l\leq k$. 
\end{definition}

The following are some properties of filtered maps.  
\begin{lemma}\label{lemma:filtered-maps-basic-properties} 
	  \mbox{}
\begin{enumerate} 
\item\label{item:projection} Each projection 
$\displaystyle{\bigoplus_j X_j} \stackrel{\pi_j}{\longrightarrow} X_j$ is a filtered map. 
\item\label{item:inclusion} Each inclusion $X_j\hookrightarrow \displaystyle{\bigoplus_j X_j}$ is a filtered map. 
\item\label{item:induced-filtered-incl} If $Y_j\subseteq X_j$ is a subspace of a filtered vector space $X_j$, 
	then giving $Y_j$ the induced filtration,   
	the inclusion $Y_j\hookrightarrow X_j$ is a filtered map.  
\item\label{item:filtered-quotient-map} If $X_j\twoheadrightarrow Y_j$ is a quotient map where $X_j$ is a filtered vector space, 
then giving $Y_j$ the induced filtration, $X_j\twoheadrightarrow Y_j$ is a filtered map.  
\item\label{item:composition} If $X_{a}\stackrel{\varphi}{\longrightarrow} X_{b}$   
 and   
 $X_{b}\stackrel{\psi}{\longrightarrow} X_{c}$ are filtered maps, then so is $\psi\circ \varphi$.  
\end{enumerate} 
\end{lemma} 

\begin{proof} 
We will begin with \eqref{item:projection}. 
Since  $\pi_j (F_l( \displaystyle{\bigoplus_j X_j} ) ) 
= 
 \pi_j ( \displaystyle{\bigoplus_j X_{j,l}} ) = X_{j,l}$, we are done. 
For \eqref{item:inclusion}, we have $X_{j,l}\hookrightarrow X_{j,l}\oplus 0 = F_l(X_j\oplus 0)$. 
Let us now prove \eqref{item:induced-filtered-incl}.  
Let $Y_j\stackrel{i}{\hookrightarrow} X_j$ be an inclusion. 
Then since 
    $i(Y_{j,l})=i(Y_j\cap X_{j,l})\cong Y_j\cap X_{j,l} \subseteq X_{j,l}$ 
    for each $0\leq l\leq k$, the inclusion $i$ is a filtered map. 
      For \eqref{item:filtered-quotient-map}, let $X_j \stackrel{\varphi}{\twoheadrightarrow} Y_j$ be a quotient map. 
      Then $\varphi(X_{j,l})=Y_{j,l}$ since $Y_j$ has the induced filtration. Thus $\varphi$ is a filtered map.
      Finally for \eqref{item:composition},  
      since $\psi\circ\varphi(X_{a,l})=\psi(\varphi(X_{a,l}))\subseteq \psi(X_{b,l})\subseteq X_{c,l}$, 
$\psi\circ\varphi$ is a filtered map. 
\end{proof}

\begin{definition}\label{definition:forgetful-functor}
Let $F^{\bullet}Rep(Q,\beta)\stackrel{forget}{\longrightarrow} Rep(Q,\beta)$ 
be the forgetful functor which sends a filtered representation to the representation 
in $Rep(Q,\beta)$,  
where the data consisting of the fixed filtration of vector spaces at each vertex is ignored. 
\end{definition} 
 
We refer the reader to Section~\ref{subsection:reflection-functors-classical-setting}  
for definitions of reflection functors in the classical setting. 

\begin{proposition}\label{proposition:filtered-representation-reflection-functors}
 If $X$ is a filtered representation of $Q$, then 
$S_i^+ X$ and $S_i^- X$ with induced filtrations are filtered representations of $F^{\bullet}Rep(Q',\beta')$, 
where $i\in Q_0$ is a sink or a source, respectively. 
Furthermore, let $F^{\bullet}Rep(Q,\beta)\longrightarrow Rep(Q,\beta)$ be the forgetful functor. 
Then the maps in the diagrams 
\[ 
\xymatrix@-1pc{ 
F^{\bullet} Rep(Q,\beta) \ar[rr]^{S_i^+} \ar[dd]_{forget} & & F^{\bullet}Rep(\sigma_i(Q),\sigma_i(\beta)) \ar[dd]^{forget} & & 
F^{\bullet} Rep(Q,\beta) \ar[rr]^{S_i^-} \ar[dd]_{forget} & & F^{\bullet}Rep(\sigma_i(Q),\sigma_i(\beta)) \ar[dd]^{forget} \\ 
& \circlearrowleft & & & &\circlearrowleft & \\ 
Rep(Q,\beta) \ar[rr]^{S_i^+} & & Rep(\sigma_i(Q),\sigma_i(\beta)) & & 
Rep(Q,\beta) \ar[rr]^{S_i^-} & & Rep(\sigma_i(Q),\sigma_i(\beta))\\
}
\]  
commute, where $i\in Q_0$ is a sink or a source, respectively. 
\end{proposition} 

\begin{proof}  
Let $Q$ be a quiver and let $i\in Q_0$ be a sink of $Q$.  By the definition of a reflection functor, 
$X_i'=\ker \phi_i^X$, 
where $\phi_i^X 
= 
\displaystyle{\sum_{ (j\stackrel{\nu}{\rightarrow} i)\in Q_1} \rho_{\nu}}: 
\displaystyle{\bigoplus_j X_j} \rightarrow X_i$
 and for each arrow $i\stackrel{\nu}{\rightarrow} j \in Q_1'$, 
 the linear map $\rho_{\nu}':X_i'\rightarrow X_j=X_j'$ 
 is the natural projection on the component $X_j\in \displaystyle{\bigoplus_j X_j}$. 
 Thus $X_i'\hookrightarrow \displaystyle{\bigoplus_j X_j}$ is the inclusion followed by a projection onto the component $X_j$. 
 Since each $X_j$ is a filtered vector space: $\{0\} =X_{j,0}\subseteq X_{j,1}\subseteq \ldots \subseteq X_{j,k}=X_j$,
 we have 
 $(X_i')_l := X_i' \cap (\displaystyle{\bigoplus_j X_{j,l}}) \hookrightarrow \displaystyle{\bigoplus_j X_{j,l}}\twoheadrightarrow X_{j,l}$
 for each $0\leq l\leq k$. 
 Since $(X_i')_{l-1}\subseteq (X_i')_l$ for each $1\leq l\leq k$,  
 the representation 
 $X'$ with the induced filtration is a filtered representation of $F^{\bullet}Rep(Q',\beta')$.  
     Now for a sink $i\in Q_0$, consider a filtered representation $X$ of $Q$ with dimension vector $\beta$. 
     Then $S_i^+X$ is a filtered representation of $\sigma_i(Q)$ with dimension vector $\sigma_i(\beta)$. 
     Applying the forgetful functor, $S_i^+X$ is now a representation in $Rep(\sigma_iQ,\sigma_i\beta)$, 
     which coincides with applying the forgetful functor to $X$ followed by $S_i^+$.

 Now, suppose $i\in Q_0$ is a source of $Q$. By the definition of a reflection functor,   
 $X_i'=\coker \widetilde{\phi}_i^X$ with  
 $\rho_{\nu}':X_j\rightarrow X_i'$ is given by  
 $X_j\cong \displaystyle{\bigoplus_j X_j|_{X_j\oplus 0}}
 \hookrightarrow 
 \displaystyle{\bigoplus_j X_j} 
 \twoheadrightarrow 
 \displaystyle{\bigoplus_j X_j/\im \widetilde{\phi}_i^X}$. 
 Since $X_j$ is a filtered vector space, the induced filtration on $X_i'$ is 
 $\rho_{\nu}'(X_{j,l})= ((\displaystyle{\bigoplus_j X_j}) /\im \widetilde{\phi}_i^X)\cap (\displaystyle{\bigoplus_j X_{j,l}}) =: (X_{i}')_l$   
 for $0\leq l\leq k$.    
 Since $(X_{i}')_{l-1} \subseteq (X_{i}')_l$ for each $1\leq l\leq k$,  
 the representation  
 $X'$ with the induced filtration is a well-defined filtered representation of $F^{\bullet}Rep(Q',\beta')$. 
 Now for a source $i\in Q_0$, consider a filtered representation $X$ of $Q$ with dimension vector $\beta$. 
 Then $S_i^-X$ is a filtered representation of $\sigma_i(Q)$ with dimension vector $\sigma_i(\beta)$. 
 Applying the forgetful functor, $S_i^-X$ is now a representation in $Rep(\sigma_iQ,\sigma_i\beta)$, which coincides with applying the forgetful functor to $X$ followed by $S_i^-$. 
\end{proof}

We refer the reader to \cite{MR1741551} for the construction of multifiltered modules over multifiltered associative rings. 
Next, we will discuss the notion of strict filtered maps in the context of filtered representations of a quiver.  

\begin{definition}\label{definition:strict-filtered-morphisms}
Let $M\stackrel{\varphi}{\longrightarrow}N$ be a filtered map of left $\mathbb{C}Q$-modules, where 
$\{0\} = M_0\subseteq M_1\subseteq \ldots \subseteq M_k=M$ and 
$\{0\} = N_0\subseteq N_1\subseteq \ldots \subseteq N_k=N$ are filtrations of $M$ and $N$, respectively. 
If $M_l$ is maximal with the property that  
$\varphi(M_l)=\varphi(M)\cap N_l$ for all $1\leq l\leq k$, then $\varphi$ is called  
{\em strict}.  
\end{definition} 

Note that $\varphi$   
is strict if the filtration on the image is induced by the filtration of the target module, i.e., we have $M_l=\varphi^{-1}(N_l)$ for all $1\leq l\leq k$.

\begin{proposition}\label{proposition:associated-graded-and-strict-exact-sequence} 
Let $\mathbb{C}Q$ be the path algebra of $Q$ and let 
\begin{equation}\label{equation:SES-strict-exact-proposition}
\xymatrix@-1pc{
M \ar[rrr]^{\varphi_M} & & & N \ar[rrr]^{\varphi_N} & & & S  \\ 
}
\end{equation}  
be a sequence of filtered maps exact at $N$. Then 
\begin{equation}\label{equation:SES-strict-exact-gr-proposition}  
\xymatrix@-1pc{ 
\gr M \ar[rrr]^{\overline{\varphi}_M} & & & \gr N \ar[rrr]^{\overline{\varphi}_N} & & & \gr S  \\  
} 
\end{equation}
 is exact at $\gr N$ if and only if $\varphi_M$ and $\varphi_N$ are strict.  
\end{proposition}

We sometimes say the sequence is {\em strict exact} if it is exact and both $\varphi_M$ and $\varphi_N$ are strict.  
We will now prove Proposition~\ref{proposition:associated-graded-and-strict-exact-sequence}. 

\begin{proof} 
Suppose 
\[ 
\xymatrix@-1pc{ 
\displaystyle{\bigoplus_{l=1}^{k} M_l/M_{l-1}} \ar[rrr]^{\overline{\varphi}_M} & & & \displaystyle{\bigoplus_{l=1}^k N_l/N_{l-1}} \ar[rrr]^{\overline{\varphi}_N} & & & 
\displaystyle{\bigoplus_{l=1}^k S_l/S_{l-1}} 
}
\] is exact. We want to show 
$\varphi_M(M_l)=\varphi_M(M)\cap N_l$ and 
$\varphi_N(N_l)=\varphi_N(N)\cap S_l$. 
Since $M_l\subseteq M$, $\varphi_M(M_l)\subseteq \varphi_M(M)$. 
Since $\varphi$ is a filtered map, $\varphi_M(M_l)\subseteq N_l$. 
Thus, $\varphi_M(M_l)\subseteq \varphi_M(M)\cap N_l$. 
Now, we will show the other inclusion. 
 Let $n\in \varphi_M(M)\cap N_l=\ker(N\stackrel{\varphi_N}{\longrightarrow} S)\cap N_l$. 
 Let's write 
 $[n]_l =n_l+N_{l-1}\in N_l/N_{l-1}$. 
 Since $\varphi_N(n)=0$, 
 $\overline{\varphi}_N([n]_l)=[0]_l$. 
 By \eqref{equation:SES-strict-exact-gr-proposition},   
 there exists $[m]_l=m_l+M_{l-1}\in M_l/M_{l-1}$ such that 
 $\overline{\varphi}_M([m]_l)=[n]_l$. 
 Lifting $\overline{\varphi}_M$ to $\varphi_M$, $\varphi_M(m)=n$, where $m\in M_l$. 
 Thus, we have $n\in \varphi_M(M_l)$.  
  
 Similarly, since $N_l\subseteq N$,  
 we have $\varphi_N(N_l)\subseteq \varphi_N(N)$. 
 Since $\varphi$ is a filtered map, 
 $\varphi_N(N_l)\subseteq S_l$.  
 Thus, $\varphi_N(N_l)\subseteq \varphi_N(N)\cap S_l$. 
 On the other hand, suppose $s\in \varphi_N(N)\cap S_l$. 
 This means $[s]_l=s_l+S_{l-1}\in S_l/S_{l-1}$ and there exists $n\in N$ such that $\varphi_N(n)=s$. 
 We want to show that $n\in N_l$. 
 Write $[n]_j=n_j+N_{j-1}\in N_j/N_{j-1}$. 
 Suppose $j>l$. 
 Then $\overline{\varphi}_N([n]_j)=[0]_j = [s]_j \in S_j/S_{j-1}$. 
 By \eqref{equation:SES-strict-exact-gr-proposition}, there exists 
 $[m]_j =m_j+M_{j-1}$ such that $\overline{\varphi}_M([m]_j)=[n]_j$. 
 Lifting $\overline{\varphi}_M$ and $\overline{\varphi}_N$ to $\varphi_M$ and $\varphi_N$, respectively, 
 $\varphi_N\circ\varphi_M(m)=\varphi_N(n)=0=s$. 
 So $s$ was zero. 
 If $s$ is not zero, then the same argument implies $s$ must be zero, which is a contradiction. Thus, $j\leq l$. 
 If $0\not= s\in S_l$ such that $[s]_l = [0]_l$, 
 then there exists smallest $l'$ such that $[s]_{l'}\not=[0]_{l'}$. 
 Rewrite the indices such that $[s]_l\not=[0]_l$ in $S_l/S_{l-1}$. 
 Now if $j<l$, then 
 $[0]_l =[n]_l = 0+N_{l-1}\in N_l/N_{l-1}$ implies $\overline{\varphi}_N([0]_l)=[s]_l\not=[0]_l$. 
 This is a contradiction and therefore, we conclude $j=l$ and $\varphi_N$ is strict. 
 
 Now suppose $\varphi_M$ and $\varphi_N$ are strict, i.e., 
 $\varphi_M(M_l)=\varphi_M(M)\cap N_l$ and $\varphi_N(N_l)=\varphi_N(N)\cap S_l$. 
 We want to show the sequence in \eqref{equation:SES-strict-exact-gr-proposition} is exact. 
Since  $M\stackrel{\varphi_M}{\longrightarrow} N\stackrel{\varphi_N}{\longrightarrow} S$ 
is an exact sequence of filtered maps, by definition we have $\varphi_M(M_l)=\ker(N_l\stackrel{\varphi_N}{\longrightarrow}S_l)$ for all $0\leq l\leq k$. 
This means given the complex 
$\gr M \stackrel{\overline{\varphi}_M}{\longrightarrow}  \gr N \stackrel{\overline{\varphi}_N}{\longrightarrow} \gr S$, $\im \overline{\varphi}_M\subseteq \ker \overline{\varphi}_N$. 
We also have for all $0\leq l\leq k$, 
\begin{equation}\label{eq:strict-SES-proposition}
\ker(N_l\stackrel{\varphi_N}{\longrightarrow} S_l)=\varphi_M(M)\cap N_l. 
\end{equation}
Suppose $[n]_l=n_l+N_{l-1}\in \ker \overline{\varphi}_N$. 
Lifting $[n]_l$ in $N_l/N_{l-1}$ to $n$ in $N_l$, 
$\varphi_N(n)=0$.   
By the right-hand side of \eqref{eq:strict-SES-proposition}, $n\in \varphi_M(M)$.
So there exists $m\in M$ such that $\varphi_M(m)=n$. 
Since $\varphi_M(M_l)=\varphi_M(M)\cap N_l$, 
$\varphi_M(m)=n$ implies $m$ is really in $M_l$. 
So 
$[m]_l=m_l+M_{l-1}\in M_l/M_{l-1}$ such that 
$\overline{\varphi}_M([m]_l)=[n]_l$. 
So $[n]_l\in \im\overline{\varphi}_M$ and thus, we obtain that the associated graded complex is exact. 
\end{proof} 

\begin{corollary}\label{corollary:associated-graded-and-strict-exact-sequence}
Let $M\stackrel{\varphi}{\longrightarrow}N$ be a filtered map of left $\mathbb{C}Q$-modules. Then 
\begin{enumerate}[ {(}1{)} ]  
\item\label{item:injective-strict} the morphism $\varphi$ is injective and strict if and only if $\gr M \stackrel{\overline{\varphi}}{\longrightarrow}\gr N$ is injective. 
\item\label{item:surjective-strict} the morphism $\varphi$ is surjective and strict if and only if $\gr M\stackrel{\overline{\varphi}}{\longrightarrow}\gr N$ is surjective. 
\end{enumerate}
\end{corollary}

Corollary~\ref{corollary:associated-graded-and-strict-exact-sequence} is a special case of Proposition~\ref{proposition:associated-graded-and-strict-exact-sequence} by 
taking $M=0$ (so that $\varphi_N$ is injective) or $S=0$ (so that $\varphi_M$ is surjective).    
It follows from Proposition~\ref{proposition:associated-graded-and-strict-exact-sequence} and  Corollary~\ref{corollary:associated-graded-and-strict-exact-sequence}  
that a map of filtered modules $0\rightarrow M\rightarrow N\rightarrow S\rightarrow 0$  
is a strict short exact sequence if and only if 
$0\rightarrow \gr M \rightarrow \gr N \rightarrow \gr S\rightarrow 0$ is a short exact sequence.

\begin{remark}\label{remark:relating-associated-graded-and-SES} 
Let $X_j$ and $Y_j$ be filtered vector spaces. 
If $\gr X_j \hookrightarrow \gr Y_j$ is injective, then   
$X_j\stackrel{\varphi}{\rightarrow} Y_j$ is also injective and 
the filtration of $X_j$ equals the filtration induced from $Y_j$. 
And similarly, if $\gr X_j \twoheadrightarrow \gr Y_j$ is surjective, 
then $X_j \stackrel{\psi}{\twoheadrightarrow} Y_j$  
is also surjective and the filtration of vector spaces at $X_j$ 
is determined by the filtration of vector spaces at $Y_j$. 
\end{remark}

Now let $X_j := X_{j,\bullet}$ be a filtered vector space. 
We extend the filtration of vector spaces such that it is torsion-free, i.e., 
if $X_{j,k}= X_j$ is the entire vector space, 
then let 
$X_{j,k+l} :=X_{j,k}$ for all $l\geq 0$. 

\begin{definition}\label{definition:rees-vector-space} 
The {\em Rees vector space} of $X_j$ is defined to be 
\[ 
\Rees(X_j) := \bigoplus_{l\in \mathbb{Z}} X_{j,l} t^l. 
\] 
\end{definition}

It is a graded module over the graded algebra $\mathbb{C}[t]$ as 
\[ \sum_l c_l t^l \cdot \sum_k d_k t^k 
 = \sum_m \sum_{l+k=m} c_l d_k t^m.  
\]

\begin{definition}\label{definition:graded-representation}  
Let $X$ be a filtered representation of a quiver $Q$.  
We define the   
{\em graded representation}  
$\Rees(X)$ of $X$  
to be the representation whose graded vector space at each vertex $i\in Q_0$ is   
$\Rees(X)_i :=\Rees(X_i)$, and for each arrow $i\stackrel{a}{\rightarrow} j$,    
the graded linear maps are  
$\Rees(X)_i \stackrel{\Rees(A_X)}{\longrightarrow}\Rees(X)_j$, 
where   
$\Rees(A_X)$ sends $v_{i,l}t^l\mapsto v_{j,l}t^l$ for every $l$. 
\end{definition}

\begin{definition}\label{definition:indecomposable-graded-rep} 
The Rees representation $\Rees(X)$ of $X$ is {\em decomposable} 
if 
$\Rees(X)\cong \Rees(Y)\oplus \Rees(Z)$, 
where $\Rees(Y)$ and $\Rees(Z)$ are nonzero graded $\mathbb{C}[t]$-representations of $Q$. 
If such decomposition does not exist, then we say 
$\Rees(X)$ is an {\em indecomposable graded representation}.  
\end{definition}

\begin{definition}\label{definition:category-of-graded-poly-rep}
A graded $\mathbb{C}[t]$-representation of $Q$ is a representation of $Q$ in graded $\mathbb{C}[t]$-modules. 
Let $GPRep(Q)$ be the {\em category of graded $\mathbb{C}[t]$-representations}. 
\end{definition}

\begin{lemma}\label{lemma:rees-functor}
$\Rees(-)$ defines a functor from  $F^{\bullet}Rep(Q)$ to graded $\mathbb{C}[t]$-representations: 
\[  
\xymatrix@-1pc{  
F^{\bullet}Rep(Q)\ar[rr] & & GPRep(Q) \\   
}  
\]   
by sending $X\mapsto \Rees(X)$.    
 \end{lemma}

\begin{definition}\label{definition:indecomposable-filtered-reps} 
We say a filtered representation $X$ is decomposable if $X\cong Y\oplus Z$, 
where $Y$ and $Z$ are nonzero filtered representations. 
If such decomposition does not exist, we say $X$ is an indecomposable filtered representation. 
\end{definition}

\begin{lemma}\label{lemma:total-space-of-graded-space} 
Let 
\[  \xymatrix@-1pc{ 
GPRep(Q) \ar[rr] & & F^{\bullet}Rep(Q)
}
\] be the functor defined as follows:  
writing 
$\Rees(X)_i$ $=$ $\displaystyle{\bigoplus_l R_{i,l}}$, 
if $t\cdot -:R_{i,l}\rightarrow R_{i,l+1}$ is injective for all $l$, 
then  
$\Rees(X)_i/\Rees(X)_i(t-1)$ is isomorphic to $\bigcup_l R_{i,l}$ 
via $R_{i,l}\stackrel{t\cdot -}{\hookrightarrow} R_{i,l+1}$. 
\end{lemma}   
 
That is, the total space at vertex $i$ is $\Rees(X)_i/\Rees(X)_i(t-1)$.   
Moreover, the vector space $\displaystyle{\bigcup_l R_{i,l}} = R_i$    
has the induced filtration from $\Rees(X)/\Rees(X)(t-1)$.

\begin{proof}   
Suppose $t\cdot -:R_{i,l}\rightarrow R_{i,l+1}$ is injective for each $l$.   
We will show the existence of an isomorphism   
\[   
\xymatrix@-1pc{   
\displaystyle{\bigcup_l R_{i,l}} \ar[rr]^{} & & \Rees(X)_i/\Rees(X)_i(t-1).   
}   
\]  
Consider the composite   
\[   
\xymatrix@-1pc{  
R_{i,l} \ar[ddrr]_{\pi_l} \ar@{^{(}->}[rr]^{\incl} & & \Rees(X)_i \ar@{->>}[dd]^{\pr_i} \\  
 				& & \\  
 				& & \Rees(X)_i/\Rees(X)_i(t-1). \\  
}
\] 
Then the kernel of $\pi_l$ is 
$R_{i,l}\cap \Rees(X)_i(t-1)$. 
Note that an element of $\Rees(X)_i(t-1)$ is of the form 
\[ 
v = (t-1)\left( \sum_{l=l_0}^{l_0+k} v_l \right), 
\]   
where  
$v_l\in R_{i,l}$, $v_{l_0}\not=0$, $v_{l_0+k}\not=0$.  
Expanding the product for $v$, we get   
\[   
v = -v_{l_0} + t v_{l_0}-v_{l_0+1}+\ldots + tv_{l_0+k}, 
\] 
where $-v_{l_0}\in R_{i,l_0}$, $t v_{l_0}-v_{l_0+1}\in R_{i,l_0+1}$, $\ldots$,  
$tv_{l_0+k}\in R_{i,l_0+k+1}$.  
Since $v_{l_0}\not=0$ by assumption and $tv_{l_0+k}\not=0$ since multiplication by $t$ is injective,  
any such $v$ has nonzero components in at least two graded module pieces.  
So $R_{i,l}\cap \Rees(X)_i(t-1)=\{ 0\}$ for every $l$ since $R_{i,l}$ is just one graded piece.  
Thus, $R_{i,l}\stackrel{\pi_l}{\hookrightarrow} \Rees(X)_i/\Rees(X)_i(t-1)$ is injective.   

Next, we will prove that the diagram 
\[ 
\xymatrix@-1pc{ 
R_{i,l} \ar@{^{(}->}[ddrr]_{\pi_l} \ar[rr]^{t\cdot -} & & R_{i,l+1} \ar@{^{(}->}[dd]^{\pi_{l+1}} \\ 
 & & \\  
 & & \Rees(X)_i/\Rees(X)_i(t-1) \\ 
} 
\] 
commutes, where $\pi_l:v\mapsto v$, $t\cdot -:v\mapsto tv$, and $\pi_{l+1}:tv\mapsto tv$.  
So if $v\in R_{i,l}$, then $tv-v = (t-1)v=0$ in $\Rees(X)_i/\Rees(X)_i(t-1)$. 
So $tv=v$ in $\Rees(X)_i/\Rees(X)_i(t-1)$. 

Finally, we need to show   
\[  
\bigcup_l \im \pi_l = \Rees(X)_i/\Rees(X)_i(t-1).  
\]   
Since $\Rees(X)_i/\Rees(X)_i(t-1)$ consists of images of elements of the form 
$\displaystyle{\sum_{l=l_0}^{l_0+k} v_l}$,  
where $v_l\in R_{i,l}$, 
the image of such element equals the image of 
\[ \sum_{l=l_0}^{l_0+k} v_l t^{l_0+k-l} 
  = v_{l_0}t^k + v_{l_0+1}t^{k-1} + v_{l_0+2}t^{k-2} +\ldots + v_{l_0+k} \in R_{i,l_0+k}. 
\] 
\end{proof}

Now, consider the map from a graded representation to a filtered representation: 
$X_{\bullet,\bullet}\mapsto (X/X(t-1))_{\bullet,\bullet}$, 
where 
$(X/X(t-1))_{i,l}=\im (X_{i,l}\rightarrow X_{i,\bullet}/X_{i,\bullet}(t-1))$. 
That is, let $X$ be a graded $\mathbb{C}[t]$-representation; so 
$X_i = \displaystyle{\bigoplus_l X_{i,l}}$ 
is a graded $\mathbb{C}[t]$-module. 
Define 
\[  \xymatrix@-1pc{ 
\Filt(X_{\bullet})_{i,l} = \im(X_{i,l}\ar[rr] & & X_i/X_i(t-1)).  
}
\] 

\begin{lemma}\label{lemma:filt-functor} 
The module $\Filt(X_{\bullet})$ defined above is a filtered representation of $Q$. 
\end{lemma}

Regardless of $t$ having torsion or being torsion-free, 
defining $\Filt(X_{\bullet})_{i,l}$ as in Lemma~\ref{lemma:filt-functor} defines a filtration of $X_i/X_i(t-1)$ for each $i$.
Note if $t$ has torsion,  
then multiplication by $t$ sends torsion submodules to the zeroth graded piece while sending the rest of the module to the next graded module.

\begin{proof} 
Since $\Filt(X_{\bullet})_{i,l}$ is naturally contained in $\Filt(X_{\bullet})_{i,l+1}$ for each $i$ and $l$ and 
for each arrow $i\rightarrow j$, 
each graded vector space $X_{i,l}$ at vertex $i$ maps to each graded piece 
$X_{j,l}$ at vertex $j$, 
$\Filt(X_{\bullet})$ is indeed a filtered representation of $Q$. 
\end{proof}

\begin{proposition}\label{proposition:rees-functor-and-filtered-functor} 
Consider the $\Rees$ and the $\Filt$ functors: 
\[ 
\xymatrix@-1pc{
F^{\bullet}Rep(Q) \ar@/^/[rr]^{\Rees} & & GPRep(Q) \ar@/^/[ll]^{\Filt}. 
}
\] 
Then $\Filt\circ \Rees\cong \Id$ on $F^{\bullet}Rep(Q)$  
while   
$\Rees \circ\Filt \cong \Id$   
on the subcategory of $t$-torsion-free graded modules in $GPRep(Q)$.   
\end{proposition} 

Proposition~\ref{proposition:rees-functor-and-filtered-functor} implies  
$\Rees \circ\Filt(M)=M/\tors(M)$, where  
$\tors(M)$ is a $t$-torsion submodule of $M$.   

\begin{proof}  
 Let $X$ be a filtered representation. Then 
 $\Filt \circ \Rees(X_{i,l}) = \Filt(X_{i,l}t^l) = X_{i,l}$
  for each $l$, 
 so 
 $\Filt (\Rees(X_{i,l}))$ is naturally contained inside $\Filt(\Rees(X_{i,l+1}))$. 
 It is then clear that for each $l$ and for each arrow $i\rightarrow j$, linear maps of the form 
 $\Filt \circ \Rees(X_{i,l})\rightarrow \Filt \circ \Rees(X_{j,l})$  
 are compatible with filtered linear maps $X_{i,l}\rightarrow X_{j,l}$.  
On the other hand, let  
$\displaystyle{\bigoplus_{l\in \mathbb{Z}}X_{i,l}t^l}$  
be $t$-torsion-free graded module.  
Then  
 $\Rees\circ \Filt(X_{i,l}t^l)$ $=$ $\Rees(X_{i,l})$ $=$ $X_{i,l}t^l$, so we are done. 
 
 Now, we will show that the functors preserve direct sums. 
Suppose $X$ is an indecomposable filtered representation in $F^{\bullet} Rep(Q)$. 
Consider $\Rees(X)$. If $\Rees(X)$ were decomposable, then 
$\Rees(X)$ $=$ $\Rees(Y)$ $\oplus$ $\Rees(Z)$, where 
$\Rees(Y)$ and $\Rees(Z)$ are nonzero graded $\mathbb{C}[t]$-representations. 
This means 
for each arrow $i\stackrel{a}{\rightarrow} j$, we have 
\[ 
\xymatrix@-1pc{ 
\Rees(X)_i \ar[dd]_{\Rees(A_X)} &   & \Rees(Y)_i \ar[dd]^{\Rees(A_Y)} & &  & \Rees(Z)_i \ar[dd]^{\Rees(A_Z)} \\ 
 						&   \cong &    &   & \bigoplus  &  & \\ 
 \Rees(X)_j &     &     \Rees(Y)_j & &  &  \Rees(Z)_j. \\ 
}
\] 
Applying the functor $GPRep(Q)\rightarrow F^{\bullet}Rep(Q)$, which maps  
$\Rees(X)\mapsto \Rees(X)/\Rees(X)(t-1)$, 
where 
$\Rees(X)_i/\Rees(X)_i(t-1)\cong \displaystyle{\bigcup_{l\in \mathbb{Z}} X_{i,l}}$,  
we have 
\[ 
\xymatrix@-1pc{ 
\displaystyle{\bigcup_l X_{i,l}} \ar[dd]_{ } &   & \displaystyle{\bigcup_l Y_{i,l}} \ar[dd]^{} & &  & \displaystyle{\bigcup_l Z_{i,l}} \ar[dd]^{} \\ 
 						&   \cong &    &   & \bigoplus  &  & \\ 
\displaystyle{\bigcup_l X_{j,l}} &     &     \displaystyle{\bigcup_l Y_{j,l}} & &  &  \displaystyle{\bigcup_l Z_{j,l}}, \\ 
}
\] 
where 
$\im (X_{i,l})\subseteq X_{j,l}$, 
$\im (Y_{i,l})\subseteq Y_{j,l}$, 
$\im (Z_{i,l})\subseteq Z_{j,l}$ for all $l$ 
since $X$ is a filtered representation. 
So $X$ is decomposable, which is a contradiction. 
Thus $\Rees(X)$ is indecomposable.

Now suppose $\Rees(X)$ is indecomposable. 
For a contradiction, suppose $X$ is a decomposable filtered representation. 
This means $X$ is isomorphic to $Y\oplus Z$, where $Y$ and $Z$ are nonzero filtered representations. 
Applying $\Rees(-)$, we have 
$\Rees(X)$ $\cong$ $\Rees(Y\oplus Z)$, where 
the graded vector space at vertex $i$ is $\Rees(Y\oplus Z)_i$ $\cong$ 
$\Rees(Y)_i\oplus \Rees(Z)_i$ and   
$\Rees(Y\oplus Z)_i$ $\rightarrow$ $\Rees(Y\oplus Z)_j$ is the map  
$\Rees(Y)_i\oplus \Rees(Z)_i$ $\rightarrow$ $\Rees(Y)_j \oplus \Rees(Z)_j$,   
which decomposes as   
$\Rees(Y)_i\rightarrow \Rees(Y)_j$ together with   
$\Rees(Z)_i\rightarrow \Rees(Z)_j$.    
So since   
$\Rees(Y\oplus Z) \cong \Rees(Y)\oplus \Rees(Z)$, $\Rees(X)$ must be decomposable, which is a contradiction.   
Thus, $X$ is an indecomposable filtered representation.   
\end{proof}

Since Proposition~\ref{proposition:rees-functor-and-filtered-functor} preserves direct sums, Corollary~\ref{corollary:relating-indecomp-filt-rep-indecomp-graded-rep} follows. 

\begin{corollary}\label{corollary:relating-indecomp-filt-rep-indecomp-graded-rep} 
The filtered representation $X$ is indecomposable in $F^{\bullet}Rep(Q)$ if and only if 
the graded $\mathbb{C}[t]$-representation 
$\Rees(X)$ is indecomposable 
in $GPRep(Q)$. 
\end{corollary}

\begin{lemma}\label{lemma:graded-rep-to-associated-graded-rep} 
Let 
\[ \xymatrix@-1pc{
GPRep(Q) \ar[rr] & & GRep(Q) 
}
\] be the functor defined as follows: writing 
$\Rees(X)_i 
= 
\displaystyle{\bigoplus_l R_{i,l}}$,   
if $t\cdot -:R_{i,l} \rightarrow R_{i,l+1}$  
is torsion-free for all $l$, then 
$\Rees(X)_i/\Rees(X)_i\cdot t$ is isomorphic to $\gr(X)_i$.
\end{lemma}

That is, the associated graded vector space at vertex $i$ is $\Rees(X)_i/\Rees(X)_i \cdot t$. 
Furthermore, $\gr(X)_i$ is a $\gr_{(t)}\mathbb{C}[t]$-module.

\begin{proof} 
Suppose $t\cdot -:R_{i,l}\rightarrow R_{i,l+1}$ is injective for all $l$. 
We will show that 
\[ \xymatrix@-1pc{
\gr (X)_i \ar[rr] & & \Rees(X)_i/\Rees(X)_i \cdot t
}
\] 
is an isomorphism. 
Consider the composite 
\[ 
\xymatrix@-1pc{ 
R_{i,l+1}  \ar[ddrr]_{\pi_{l+1}} \ar@{^{(}->}[rr]^{\incl} & & \Rees(X)_i \ar@{->>}[dd]^{\pr_i} \\ 
 & & \\ 
 & & \Rees(X)_i/\Rees(X)_i \cdot t. \\ 
}
\]  
Show that the kernel of $R_{i,l+1}\cap \Rees(X)_i \cdot t$ is $R_{i,l}\cdot t$. 
Consider $v = t v_l\in R_{i,l}\cdot t$, where $v_l\in R_{i,l}=X_{i,l}t^l$.   
Then  $t v_l\in X_{i,l}t^{l+1}\subseteq X_{i,l+1}t^{l+1}=R_{i,l+1}$.  
Since $v\in \Rees(X)_i\cdot t$ and since $t$ is injective, 
the kernel is indeed $R_{i,l}\cdot t$, 
and thus, 
$R_{i,l+1}/R_{i,l}\cdot t \cong X_{i,l+1}/X_{i,l}\hookrightarrow \Rees(X)_i/\Rees(X)_i\cdot t$ is injective. 
Since 
\[  
(t)/(t^2)\cdot X_{i,l+1}/X_{i,l} = X_{i,l+2}/X_{i,l+1}, 
\]   
we conclude that $\Rees(X)_i/\Rees(X)_i\cdot t\cong \gr(X)_i$. 
\end{proof}


\begin{definition}\label{definition:associated-graded-functor} 
Let $F^{\bullet} Rep(Q,\beta)\stackrel{\gr}{\longrightarrow} GRep(Q,\beta)$ be the 
{\em associated graded functor} which sends a filtered representation to its associated graded representation. 
\end{definition}

Recall that $S_i$ is the representation of $Q$ with $\mathbb{C}$ at vertex $i$ and $0$ in all other vertices (Definition~\ref{definition:simple-representation-one-dim-vs-one-vertex}), and also recall that if 
$i$ is a sink, then $S_i^+(S_i)$ is the the zero representation, and if $i$ is a source, 
then $S_i^-(S_i)=0$ (Theorem~\ref{theorem:reflection-functors-classical-setting}).

%
%

\begin{example}\label{example:counter-example-to-functorial-isom}
Consider the $A_3$-quiver with dimension vector $\beta=(2,2,2)$ and representation 
\[ 
X: \xymatrix@-1pc{  
\mathbb{C}^2 \ar@{-}[d] & & & \mathbb{C}^2 \ar@{-}[d] & &  & \mathbb{C}^2 \ar@{-}[d] \\
\stackrel{\mathbb{C}^1}{\bullet} \ar[rrr]^{\bordermatrix{ & & \cr & 0& 1\cr &0 &0 \cr } } & & & \stackrel{\mathbb{C}^1}{\bullet} & & & \ar[lll]_{\bordermatrix{ & & \cr & 0& 0\cr &0 &0 \cr }} \stackrel{\mathbb{C}^1}{\bullet}.  \\ 
}  
\]   
One calculates that $S_2^+ \gr X \not\cong \gr S_2^+ X$. 
\end{example}  

It would be interesting to find appropriate conditions such that 
 $S_i^+ \gr X \cong \gr S_i^+ X$ and $S_i^- \gr X \cong \gr S_i^- X$ for any $i\in Q_0$.   
More interestingly, it remains an open problem to construct a new filtration  
$\sigma_iF^{\bullet}$ such that 
$\widehat{D}:F^{\bullet}Rep(Q,\beta)\rightarrow (\sigma_iF^{\bullet})Rep(Q^{op},\beta')$ is an isomorphism of varieties,  
and if $i$ is a sink, construct a map 
$S_i^+:F^{\bullet}Rep(Q,\beta)\rightarrow (\sigma_iF^{\bullet})Rep(\sigma_iQ,\sigma_i\beta)$ such that 
$\widehat{D}\circ S_i^+=S_i^-\circ \widehat{D}$.

Now, we will discuss Stefan Wolf's construction of reflection functors $S_a^{+}$ and $S_a^{-}$ 
(\cite{wolf2009geometric}, Theorem 5.16) for quiver flag varieties $Fl_{\gamma^{\bullet}}(W)$,   
i.e., these are precisely the fibers of $p_2$ over $W\in Rep(Q,\beta)$ of the universal quiver flag given in \eqref{equation:universal-quiver-flag}.  
%
%
We refer to Section~\ref{subsection:reflection-functors-classical-setting} for reflection functors for quiver varieties in the classical setting. 
Let $Q$ be a quiver and let $\beta\in \mathbb{Z}_{\geq 0}^{Q_0}$ be a dimension vector. 
Let $W\in Rep(Q,\beta)$. This means $\dim W_i=\beta_i$ for each $i\in Q_0$. 
Let $a$ be a sink and consider 
$\phi_a^W: \displaystyle{\bigoplus_{j\stackrel{\alpha}{\rightarrow} a}} W_j \rightarrow W_a$, where $\alpha$ is an arrow from vertex $j$ to vertex $a$. 
Let 
$\sigma_a(Q)$ be the reflection of the arrows at vertex $a\in Q_0$, that is, 
$\sigma_a:\mathbb{Z}Q_0\rightarrow \mathbb{Z}Q_0$ is the map 
$\sigma_a(\beta)=\beta - (\beta,\epsilon_a)_Q \epsilon_a$, 
where 
$(\beta,\epsilon_a)_Q := \langle \beta,\epsilon_a \rangle + \langle \epsilon_a, \beta \rangle$  
is the symmetrization of the Ringel form (cf. Definition~\ref{definition:symmetrization-of-Ringel-form})
and $\epsilon_a$ is the dimension vector with $1$ in the component corresponding to vertex $a$ and zero elsewhere. 
Let 
$\gamma^{1},\ldots, \gamma^{N}\in \mathbb{Z}_{\geq 0}^{Q_0}$ such that 
$\gamma_i^k\leq \gamma_i^{k+1}$, 
$\dim U_i^{k} = \gamma_i^{k}$,  
$W(\alpha)U_{t\alpha}^{k}\subseteq U_{h\alpha}^{k}$, 
and 
$U^N=W$  
 for each $ k $, $i\in Q_0$, and $\alpha\in Q_1$. 
So 
$U^{\bullet}: U^{0}= \{0\} \subseteq U^{1}\subseteq U^{2}\subseteq \ldots \subseteq U^{N}= W$ is a quiver flag of $W$.  
We will construct $S_a^{+}U^{\bullet}$ as follows.  
For each $0\leq i\leq N$,  
we have a short exact sequence of vector spaces: 
\[  
\xymatrix@-1pc{  
0 \ar[rrr] & & & (S_a^+ U^{i-1})_a \ar[rrr] \ar[dd]_f  & & & \bigoplus_{j\rightarrow a} U_j^{i-1} \ar[rrr]^{\phi_a^{U^{i-1}}} \ar@{^{(}->}[dd]_g & & & \im \phi_a^{U^{i-1}} \ar[rrr] \ar@{^{(}->}[dd]^{h} & & & 0\\ 
& & &   & & &     & & &   & & & \\ 
0 \ar[rrr] & & & (S_a^+ U^{i})_a \ar[rrr]  & & & \bigoplus_{j\rightarrow a} U_j^{i}  \ar[rrr]^{\phi_a^{U^{i}}}  & & &   \im \phi_a^{U^{i}} \ar[rrr] & & & 0, \\ 
}  
\]  
where $g$ and $h$ are injective. This implies that $f$ must also be injective. 
So $S_a^+ U^{\bullet}$ defined in the commutative diagram is the new quiver flag of 
$S_a^+ W$.  
Thus given 
$\underline{\gamma}$ $=$ 
$(\gamma^{0}= \{0\} \leq \gamma^{1} \leq \gamma^{2}\leq \ldots \leq  \gamma^{N} = \beta)$, 
$\sigma_a \underline{\gamma}$ is defined in the following way: 
for $s := \beta_a -\rk \phi_a^W$ being the codimension of $\im \phi_a^W$ in the vector space $W_a$, 
we have 
\[  
\begin{aligned}
e_a &:= \dim(S_a^+W)_a  = \sum_{\alpha=(i\rightarrow a)}\beta_i - \rk \phi_a^W  \\ 
&= \beta_a - (\beta_a-0)-(\beta_a-\sum_{\alpha=(i\rightarrow a)} \beta_i) 
+ \beta_a -\rk \phi_a^W  \\ 
&= \beta_a -( \langle \beta,\epsilon_a \rangle + \langle \epsilon_a,\beta \rangle ) 
+ \beta_a -\rk \phi_a^W  \\ 
&= \beta_a -(\beta,\epsilon_a)_Q + s  
= \beta_a -s-(\beta -s\epsilon_a,\epsilon_a)_Q.  \\  
\end{aligned} 
\]    
Therefore, 
\[ 
\begin{aligned} 
\dim (S_a^+W)
&= \sigma_a(\beta-s\epsilon_a)
= \beta-s\epsilon_a -(\beta-s\epsilon_a,\epsilon_a)_Q \epsilon_a  
= \beta-s\epsilon_a -\langle \beta-s\epsilon_a,\epsilon_a \rangle \epsilon_a
-\langle \epsilon_a,\beta-s \epsilon_a \rangle \epsilon_a \\
&=\beta-s\epsilon_a -(\beta_a-s-\sum_{(c\rightarrow a)\in Q_1} \beta_c )\epsilon_a -(\beta_a-s )\epsilon_a  
=\beta-(\beta_a-\sum_{(c\rightarrow a)\in Q_1}\beta_c)\epsilon_a 
-\beta_a\epsilon_a +s \epsilon_a \\
&=\beta - \langle \beta,\epsilon_a \rangle \epsilon_a 
-\langle \epsilon_a,\beta \rangle \epsilon_a +s \epsilon_a  
=\beta-(\beta,\epsilon_a )_Q\epsilon_a + s \epsilon_a 
=\sigma_a(\beta)+s \epsilon_a 
\end{aligned} 
\]  
and 
$e_a^{j}:=\dim(S_a^+ U^{j})_a = 
\displaystyle{\sum_{\alpha=(i\rightarrow a)} \gamma_i^{j} -\rk \phi_a^{U^{j}}}$.

Since the dimension of $S_a^+ U^{i}$ depends on the rank of $\phi_a^{U^{i}}$, Wolf gives the next definition. 
Recall that $\dim \Hom(W,S_a)=\beta_a-\rk \phi_a^W$ for a representation $W\in Rep(Q,\beta)$,  
where 
\begin{center}
$S_a$ 
is the representation of $Q$ with one dimensional vector space at vertex $a$ and zero dimensional vector space elsewhere 
\end{center}
(cf. Definition~\ref{definition:simple-representation-one-dim-vs-one-vertex}).  
In general, we have 
$\dim \Hom(U^{i},S_a)=\dim U_a^{i} - \rk \phi_a^{U^{i}}$ for a representation 
$U^{i}\in Rep(Q,  \gamma^{i} )$.

\begin{definition} 
Let $a$ be a sink and $\beta\in \mathbb{Z}_{\geq 0}^{Q_0}$ be a dimension vector. 
Let $s\in \mathbb{Z}$. Define 
\[ 
Rep(Q,\beta)\langle a \rangle^s := \{ W\in Rep(Q,\beta): \dim \Hom(W,S_a)=s\}.   
\] 
Let $\underline{\gamma}=(\gamma^{0}=0 \leq \gamma^{1}\leq \gamma^{2} 
\leq \ldots \leq 
 \gamma^{N} = \beta)$
be a tuple of dimension vectors, where $\gamma^{i}\in \mathbb{Z}_{\geq 0}^{Q_0}$ for all $0\leq i\leq N$. 
Let $\vec{r}=(r^{0},r^{1},\ldots, r^{N})$ be a tuple of integers. Then for each $W\in Rep(Q,\beta)$, 
\[ 
Fl_{\underline{\gamma}}(W)\langle a \rangle^{\vec{r}} 
:=  
\{ U^{\bullet}\in Fl_{\gamma^{\bullet}}(W):\dim \Hom(U^{i},S_a)=r^{i}\}. 
\] 
We let $Rep(Q,\beta)\langle a\rangle := Rep(Q,\beta)\langle a \rangle^0$
and 
$Fl_{\underline{\gamma}}(W)\langle a \rangle := Fl_{\underline{\gamma}}(W)\langle a \rangle^{\vec{0}}$. 
\end{definition}

Note that $Fl_{\gamma^{\bullet}}(W) =
\displaystyle{\coprod_{\vec{r}\geq 0} Fl_{\underline{\gamma}}(W)\langle a \rangle^{\vec{r}}}$. 
Now suppose $b\in Q_0$ is a source in $Q$. 
Since $b^{op}\in Q^{op}$ is a sink in $Q^{op}$, we will dualize via the functor $D := \Hom_{\mathbb{C}}(-,\mathbb{C})$. 

\begin{definition}
Let $U^{\bullet}\in Fl_{\gamma^{\bullet}}(W)$ be a quiver flag of $W\in Rep(Q,\beta)$. Let 
$e^{N-i} =\gamma^N-\gamma^i$. 
Let $\underline{e}=(e^{0}\leq e^{1}\leq \ldots \leq e^{N})$. 
Define 
\[ 
\widehat{D}: Fl_{\gamma^{\bullet}}(W) \longrightarrow Fl_{\underline{e}}^{op}(DW), \hspace{4mm}
U^{\bullet} \mapsto (\widehat{D}(U^{\bullet}))^i := \ker(DW\longrightarrow DU^{i}) = D(W/U^{i}). 
\] 
\end{definition}

Note that $\widehat{D}^2 \cong \mathbf{1}$ and $\widehat{D}$ is an isomorphism of varieties. 

\begin{definition}
Let $b$ be a source, $\beta\in \mathbb{Z}_{\geq 0}^{Q_0}$, and $s\in \mathbb{Z}$. 
Define 
\[ 
Rep(Q,\beta)\langle b \rangle^s := \{W\in Rep(Q,\beta):\dim \Hom(S_b,W)=s \}.  
\] 
Let $\underline{\gamma}=(\gamma^{0} \leq \gamma^{1}\leq \gamma^{2}\leq \ldots \leq  \gamma^{N} = \beta)$ be a tuple of dimension vectors, where 
$\gamma^{i}\in \mathbb{Z}_{\geq 0}^{Q_0}$ for all $0\leq i\leq N$. 
Let $\vec{r}=(r^{0},r^{1},\ldots, r^{N})$ be a tuple of integers. Then for each 
$W\in Rep(Q,\beta)$, 
\[
Fl_{\underline{\gamma}}(W)\langle b \rangle^{\vec{r}} := \{U^{\bullet}\in Fl_{\gamma^{\bullet}}(W):\dim \Hom(S_b,W/U^{i})=r^{i} \}.  
\] 
Moreover, we let 
$Rep(Q,\beta)\langle b \rangle := Rep(Q,\beta)\langle b \rangle^0 $ and 
$Fl_{\underline{\gamma}}(W)\langle b \rangle := Fl_{\underline{\gamma}}(W)\langle b \rangle^{\vec{0}}$. 
\end{definition} 

\begin{remark} 
We have 
$U^{\bullet}\in Fl_{\underline{\gamma}}(W)\langle b \rangle^{\vec{r}}$ if and only if 
$\widehat{D}U^{\bullet}\in Fl_{\underline{e}}^{op}(DW)\langle b \rangle^{\vec{r}}$. 
\end{remark}

\begin{theorem}[Wolf]\label{theorem:wolf-reflection-functors}
Let $a$ be a sink of $Q$, $\underline{\gamma}$ be a tuple of dimension vectors, and $W\in Rep(Q,\beta)\langle a \rangle$. 
Then 
\[ 
S_a^+: Fl_{\underline{\gamma}}(W)\langle a \rangle \longrightarrow Fl_{\sigma_a\underline{\gamma}}^{\sigma_a Q}(S_a^+W)\langle a \rangle, \hspace{4mm}
U^{\bullet} \mapsto S_a^+U^{\bullet}  
\] 
is an isomorphism of varieties with inverse 
$\widehat{D}\circ S_a^+\circ \widehat{D}=S_a^-$. 
\end{theorem}

\begin{proof} 
For each $i$, we have the commutative diagram 
\begin{equation}\label{eq:quiver-flag-varieties-commutative-diagram} 
\xymatrix@-1pc{ 
& & 0\ar[dd]  & & 0\ar[dd]  & &  & 0\ar[dd] & &   \\  
& &   & &   & & &    & &   \\ 
0 \ar[rr] & & (S_a^+U^{i})_a \ar[rr] \ar[dd] & & \bigoplus_{j\rightarrow a} U_j^{i} \ar[rrr]^{\phi_a^{U^{i}}} \ar[dd] & & &  U_a^{i} \ar[rr] \ar[dd]  & & 0 \\ 
& &   & &   & &  &   & &   \\  
0 \ar[rr] & & (S_a^+W)_a \ar[rr] \ar[dd] & & \bigoplus_{j\rightarrow a} W_j \ar[rrr]^{\phi_a^W} \ar[dd] & & &  W_a  \ar[rr] \ar[dd]  & & 0 \\ 
& &   & &   &  & &   & &   \\ 
0 \ar[rr] & & (S_a^+W/S_a^+U^{i})_a \ar[rr] \ar[dd] & & (\bigoplus_{j\rightarrow a} W/U^{i})_j \ar[rrr]^{\phi_a^{W/U^{i}} } \ar[dd] & & & (W/U^{i})_a  \ar[rr] \ar[dd]  & & 0 \\  
& &   & &   & & &    & &   \\ 
& & 0   & & 0   & & &   0 & &   \\  
}
\end{equation} 
where the columns are exact. The first two rows are exact, which imply the third row is also exact by the snake lemma. 
Therefore, we have an inclusion 
\[ 
\xymatrix@-1pc{ 
(S_a^+W/S_a^+ U^{i})_a \ar@{^{(}->}[rrr]  & & & 
(\displaystyle{\bigoplus_{j\rightarrow a} W/U^{i}} )_j 
} 
\] 
of modules and 
$S_a^+W/S_a^+U^{i}\in Rep(\sigma_aQ,\sigma_a(
\beta-\gamma^i  
   )) \langle a \rangle$. 
The diagram \eqref{eq:quiver-flag-varieties-commutative-diagram} also gives 
$\widehat{D}\circ S_a^+ \circ \widehat{D} \circ S_a^+ =\Id$. 
Since $S_a^+$ is a functor where all choices were natural, we have $S_a^+$ is a natural transformation between the functors of points of the two varieties 
$Fl_{\underline{\gamma}}(W)\langle a \rangle$
and 
$Fl_{\sigma_a\underline{\gamma}}^{\sigma_a Q}(S_a^+W)\langle a \rangle$. 
Thus, $S_a^+$ is a morphism of varieties and dually, 
$S_a^+ \circ \widehat{D}\circ S_a^+ \circ \widehat{D}  =\Id$. 
\end{proof}

\chapter{Semi-invariants of filtered quiver  varieties}\label{chapter-semi-invariants-filtered-quiver-vars} 

In the following sections, we will give an explicit description of (semi)-invariants that generate the subring of invariant polynomials.  

\section{Filtered quiver varieties associated to finite type ADE Dynkin graphs}

We introduce multi-index notation. 
 
\begin{notation}\label{notation:multiindex}
Let $K=(k_1,\ldots, k_{r-1})\in \mathbb{Z}^{r-1}$ be an $r-1$-tuple of nonnegative integers and let 
\[ a_{ij}^K := \prod_{\alpha=1}^{r-1} {}_{(\alpha)}a_{ij}^{k_{\alpha}}.  
\] 
\end{notation} 
Recall our basic assumption: 
$F^{\bullet}$ is the complete standard filtration of vector spaces at each nonframed vertex and 
$F^{\bullet}Rep(Q,\beta)$ is a subspace of  
$Rep(Q,\beta)$ whose representations preserve $F^{\bullet}$; 
the product $\mathbb{U}_{\beta}$ of maximal unipotent subgroups 
acts on 
$F^{\bullet}Rep(Q,\beta)$ as a change-of-basis.  
Recall that a quiver is a Dynkin quiver if the underlying graph has the structure of a Dynkin graph (cf. Section~\ref{section:reps-of-quivers-intro}). 

Let $\mathfrak{t}_n$ be the set of complex diagonal matrices in the set $\mathfrak{gl}_n$ of $n\times n$ complex matrices. 

\begin{theorem}\label{theorem:filtered-ADE-Dynkin-quiver} 
If $Q$ is an $A_r$, $D_r$, or $E_r$-Dynkin quiver and  
$\beta=(n,\ldots, n)\in \mathbb{Z}_{\geq 0}^{Q_0}$, 
then  
$\mathbb{C}[F^{\bullet}Rep(Q,\beta)]^{\mathbb{U}_{\beta}}$ $=$ $\mathbb{C}[\mathfrak{t}_n^{\oplus r-1}]$. 
\end{theorem} 

We list the $ADE$-Dynkin graphs with a preferred orientation.   
\[
  \begin{tabular}{c|c|c}
  \multicolumn{1}{c}{\bfseries Type} & \multicolumn{1}{c}{\bfseries Graph} & \multicolumn{1}{c}{\bfseries Preferred Orientation}  \\ 
  \hline 
  $A_r$ &
  		$\xymatrix{
  		 \stackrel{1}{\bullet}\ar@{-}[r]^{ }    
  & \stackrel{2}{\bullet} \ar@{-}[r]^{ }   
  & \cdots \ar@{-}[r]^{ }  
  & \stackrel{r-1}{\bullet} \ar@{-}[r]^{ }  
  & \stackrel{r}{\bullet}
  		}$ & 
  		$\xymatrix{
  		 \stackrel{1}{\bullet}\ar[r]^{ }    
  & \stackrel{2}{\bullet} \ar[r]^{ }   
  & \cdots \ar[r]^{ }  
  & \stackrel{r-1}{\bullet} \ar[r]^{ }  
  & \stackrel{r}{\bullet}
  		}$
  \\ 
  $D_r$ &
    $\xymatrix{  
   & & & & \stackrel{r-1}{\bullet} \\  
  \stackrel{1}{\bullet}	\ar@{-}[r]^{ } &  
	\stackrel{2}{\bullet}	\ar@{-}[r]^{ } &  
	\ldots 								\ar@{-}[r]^{ } &  
\stackrel{r-2}{\bullet} \ar@{-}[ur]^{ } \ar@{-}[dr]^{ } & \\ 
  & & & & \stackrel{r}{\bullet} \\   
}$ & 
   $\xymatrix{  
   & & & & \stackrel{r-1}{\bullet} \\  
  \stackrel{1}{\bullet}	\ar[r]^{ } &  
	\stackrel{2}{\bullet}	\ar[r]^{ } &  
	\ldots 								\ar[r]^{ } &  
\stackrel{r-2}{\bullet} \ar[ur]^{ } \ar[dr]^{ } & \\ 
  & & & & \stackrel{r}{\bullet} \\   
}$  \\ 
  $E_6$ & 
$ 
\xymatrix@-1pc{
& & \stackrel{4}{\bullet} & &  \\ 
\stackrel{1}{\bullet}\ar@{-}[r]^{ }  & 
\stackrel{2}{\bullet}\ar@{-}[r]^{ }  & 
\stackrel{3}{\bullet}\ar@{-}[r]^{ } \ar@{-}[u]^{ } & 
\stackrel{5}{\bullet}\ar@{-}[r]^{ }  & 
\stackrel{6}{\bullet} \\ 
} $ & 
$ 
\xymatrix@-1pc{
& & \stackrel{4}{\bullet} & &  \\ 
\stackrel{1}{\bullet}\ar[r]^{ }  & 
\stackrel{2}{\bullet}\ar[r]^{ }  & 
\stackrel{3}{\bullet}\ar[r]^{ } \ar[u]^{ } & 
\stackrel{5}{\bullet}\ar[r]^{ }  & 
\stackrel{6}{\bullet} \\ 
 }$    \\ 
  $E_7$ & 
  $  \xymatrix@-1pc{
 & & \stackrel{4}{\bullet} & & & \\ 
\stackrel{1}{\bullet}\ar@{-}[r]^{ }  & 
\stackrel{2}{\bullet}\ar@{-}[r]^{ }  & 
\stackrel{3}{\bullet}\ar@{-}[r]^{ } \ar@{-}[u]^{ } & 
\stackrel{5}{\bullet}\ar@{-}[r]^{ }  & 
\stackrel{6}{\bullet}\ar@{-}[r]^{ }  & 
\stackrel{7}{\bullet} \\ 
}  $ & 
  $  \xymatrix@-1pc{
 & & \stackrel{4}{\bullet} & & & \\ 
\stackrel{1}{\bullet}\ar[r]^{ }  & 
\stackrel{2}{\bullet}\ar[r]^{ }  & 
\stackrel{3}{\bullet}\ar[r]^{ } \ar[u]^{ } & 
\stackrel{5}{\bullet}\ar[r]^{ }  & 
\stackrel{6}{\bullet}\ar[r]^{ }  & 
\stackrel{7}{\bullet} \\ 
}  $  \\ 
  $E_8$ & $\xymatrix@-1pc{
& & \stackrel{4}{\bullet} & & & &  \\ 
\stackrel{1}{\bullet}\ar@{-}[r]^{ }  & 
\stackrel{2}{\bullet}\ar@{-}[r]^{ }  & 
\stackrel{3}{\bullet}\ar@{-}[r]^{ } \ar@{-}[u]^{ } & 
\stackrel{5}{\bullet}\ar@{-}[r]^{ }  & 
\stackrel{6}{\bullet}\ar@{-}[r]^{ }  & 
\stackrel{7}{\bullet}\ar@{-}[r]^{ }  & 
\stackrel{8}{\bullet} 
 \\ 
} $   & 
 $\xymatrix@-1pc{
& & \stackrel{4}{\bullet} & & & &  \\ 
\stackrel{1}{\bullet}\ar[r]^{ }  & 
\stackrel{2}{\bullet}\ar[r]^{ }  & 
\stackrel{3}{\bullet}\ar[r]^{ } \ar[u]^{ } & 
\stackrel{5}{\bullet}\ar[r]^{ }  & 
\stackrel{6}{\bullet}\ar[r]^{ }  & 
\stackrel{7}{\bullet}\ar[r]^{ }  & 
\stackrel{8}{\bullet} 
 \\ 
} $   \\ 
  \end{tabular}
\]

We will now prove Theorem~\ref{theorem:filtered-ADE-Dynkin-quiver}.   
\begin{proof}  
If $Q$ is a quiver on $A_r$-Dynkin graph, then writing $a_{\alpha}$  
to be the arrow connecting vertices $\alpha$ and $\alpha+1$, define 
\[  
\epsilon(a_{\alpha})= 
\begin{cases} 
1 &\mbox{ if } a_{\alpha} \mbox{ is in the preferred orientation,} \\ 
0 &\mbox{ otherwise.} \\ 
\end{cases} 
\] 
If $Q$ is a quiver on $D_r$-Dynkin graph, then writing $a_{\alpha}$ 
to be the arrow connecting vertices $\alpha$ and $\alpha+1$ when $\alpha < r-1$ and 
$a_{r-1}$ is the arrow connecting vertices $r-2$ and $r$, define 
\[ 
\epsilon(a_{\alpha}) = 
\begin{cases} 
1 &\mbox{ if } a_{\alpha} \mbox{ is in the preferred orientation and }\alpha <r-1, \\ 
0 &\mbox{ otherwise and }\alpha <r-1  \\   
\end{cases}  
\]   
   and define 
\[ 
 \epsilon(a_{r-1}) = 
\begin{cases} 
 1 &\mbox{ if } a_{r-1} \mbox{ is in the preferred orientation,} \\ 
-1 &\mbox{ otherwise.} \\ 
\end{cases} 
 \]  
 If $Q$ is a quiver on $E_r$-Dynkin graph, where $r=6,7,8$, then letting $a_{\alpha}$  
 to be the arrow connecting vertices $\alpha$ and $\alpha+1$ if $\alpha\not=4$, or else, 
 $a_4$ is the arrow connecting vertices 
  $3$ and $5$, we have 
 \[ 
\epsilon(a_{\alpha}) = 
\begin{cases} 
1 &\mbox{ if } a_{\alpha} \mbox{ is in the preferred orientation and }\alpha\not=4, \\ 
0 &\mbox{ otherwise and } \alpha\not=4 \\ 
\end{cases} 
\]  
and 
 \[ 
\epsilon(a_{4}) = 
\begin{cases} 
 1 &\mbox{ if } a_{4} \mbox{ is in the preferred orientation,}  \\  
-1 &\mbox{ otherwise.}  \\  
\end{cases}   
\]   
Let $Q$ be a quiver on a Dynkin graph of finite type. 
  Then the filtered quiver variety  
  $F^{\bullet}Rep(Q,\beta)$   
 is $\mathfrak{b}_n^{\oplus r-1}$, where the product $\mathbb{U}_{\beta}:=U^r$ of unipotent groups acts as a change-of-basis.  
 We will prove the theorem by cases. 

 \underline{\textbf{Type A}}.
 In the $A_r$ type,    
 letting  
 $u_{\alpha}$ be an unipotent matrix at vertex $\alpha$ and writing $A_{\alpha}$ 
 to be a general representation of arrow 
 $a_{\alpha}$, 
 \[ 
 \begin{aligned} 
 (u_1,&\ldots, u_r).(A_1,\ldots, A_{r-1}) \\ 
 &= (u_{1+\epsilon(a_1)} A_1 u_{1-\epsilon(a_1)+1}^{-1}, \ldots, 
    u_{\alpha+\epsilon(a_{\alpha})} A_{\alpha} u_{\alpha-\epsilon(a_{\alpha})+1}^{-1}, \ldots,
    u_{r-1+\epsilon(a_{r-1})} A_{r-1} u_{r-1-\epsilon(a_{r-1})+1}^{-1}). \\ 
  \end{aligned}
 \] 
  On the level of functions,  
 \[ 
 \begin{aligned}
 (u_1,&\ldots, u_r).f(A_1,\ldots, A_{r-1}) \\ 
 &= f(u_{1+\epsilon(a_1)}^{-1} A_1   u_{1-\epsilon(a_1)+1}, \ldots, 
    u_{\alpha+\epsilon(a_{\alpha})}^{-1} A_{\alpha} u_{\alpha-\epsilon(a_{\alpha})+1}, 
    \ldots,
    u_{r-1+\epsilon(a_{r-1})}^{-1} A_{r-1} u_{r-1-\epsilon(a_{r-1})+1}). \\ 
 \end{aligned}
 \]  
 
 Since it is clear that $\mathbb{C}[\mathfrak{t}_n^{\oplus r-1}]\subseteq \mathbb{C}[\mathfrak{b}_n^{\oplus r-1}]^{\mathbb{U}_{\beta}}$, we will prove the converse containment. 
 Let the coordinate variables of  
 $A_{\alpha}$ be ${}_{(\alpha)}a_{st}$. 
 Fix a total ordering $\leq$ on pairs $(i,j)$, where $1\leq i\leq j\leq n$, by defining 
 $(i,j)\leq (i',j')$ if either  
 \begin{itemize}  
 \item $i < i'$ or  
 \item $i=i'$ and $j> j'$. 
 \end{itemize} 
 Let $f\in \mathbb{C}[F^{\bullet}Rep(Q,\beta)]^{\mathbb{U}_{\beta}}$. 
 Then for each $(i,j)$, we can write  
 \begin{equation}\label{eq:aij-invariant-function-borel}
 f = \sum_K a_{ij}^K f_{ij,K}, \mbox{ where } f_{ij,K}\in \mathbb{C}[\{ {}_{(\alpha)}a_{st}: (s,t)\not=(i,j)\} ],  
 \end{equation}   
 where $a_{ij}^K$ is the multi-index notation in Notation~\ref{notation:multiindex}.  
 Fix the least pair (under $\leq$) $(i,j)$ with $i< j$ for which there exists   $K\not=(0,\ldots, 0)$ with $f_{ij,K}\not=0$;  
 we will continue to denote it by $(i,j)$.  
 If no such $(i,j)$ exists, then $f\in \mathbb{C}[{}_{(\alpha)}a_{ii}]$ and we are done.

Let $K=(k_1,\ldots, k_{r-1})$. Let $m\geq 1$ be the least integer satisfying the following: 
for all $p<m$, if some component $k_p$ in $K$ is strictly greater than $0$, then $f_{ij,K}=0$.   
 Let $U_{ij}$ be the subgroup of matrices of the form 
 $u_{ij}:= (\I,\ldots, \I, \widehat{u}_m,\I, \ldots, \I)$, where $\I$ is the $n\times n$  
 identity matrix and  
 $\widehat{u}_m$ is the matrix with $1$ along the diagonal, the variable $u$ in the $(i,j)$-entry, and $0$ elsewhere. 
 Let $u_{ij}^{-1}:=(\I,\ldots, \I, \widehat{u}_m^{-1},\I, \ldots, \I)$. 
 Then since $u_{ij}$ acts on $f$ via 
 $u_{ij}.f(A_1,\ldots, A_{r-1})$ $=$ 
 $f(u_{ij}^{-1}.(A_1,\ldots, A_{r-1}))$ 
  $=$ 
  $f(A_1,$ 
  $\ldots,$  
  $\widehat{u}_{m-1+\epsilon(a_{m-1})}^{-1} A_{m-1}\widehat{u}_{m-1-\epsilon(a_{m-1})+1},$  
  $\widehat{u}_{m +\epsilon(a_{m})}^{-1}    A_m    \widehat{u}_{m-\epsilon(a_m)+1},$ 
  $\ldots,$  
  $A_{r-1})$, 
  where $\widehat{u}_{m-1}:=\I$, $\widehat{u}_{m+1}:=\I$,  
  and $\widehat{u}_m$ is the unipotent matrix defined as before,  
 \begin{equation}\label{eq:ar-dynkin-Uij-action-at-most-two-pathways} 
  u_{ij}.{}_{(\alpha)}a_{st}= 
 \begin{cases} 
 {}_{(m)}a_{ij} +{}_{(m)}a_{ii}u &\mbox{ if } \alpha=m, (s,t)=(i,j), \mbox{ and } \epsilon(a_m)=1, \\
 {}_{(m)}a_{ij} -{}_{(m)}a_{jj}u &\mbox{ if } \alpha=m, (s,t)=(i,j), \mbox{ and } \epsilon(a_m)=0, \\
 {}_{(\alpha)}a_{st} &\mbox{ if }s > i\mbox{ or }s=i\mbox{ and }t <j. \\ 
 \end{cases} 
 \end{equation}
  Now write 
 \[ f = \sum_{k\geq 0} {}_{(m)}a_{ij}^k F_k, \mbox{ where } 
      F_k\in \mathbb{C}[\{ {}_{(\alpha)}a_{st}: (s,t)\geq (i,j)\mbox{ and if }(s,t)=(i,j), \mbox{ then }\alpha >m
 %
       \} ].   
 \]   
 Let 
 $R_0 := \mathbb{C}[\{ {}_{(\alpha)}a_{st}: (s,t)\geq (i,j)\mbox{ and if }(s,t)=(i,j), \mbox{ then }\alpha > m \}]$. 
 If $\epsilon(a_m)=1$, then  
 \[  
 0 = u_{ij}.f-f=  \sum_{k\geq 1} \sum_{1\leq l\leq k}{}_{(m)}a_{ij}^{k-l} {}_{(m)}a_{ii}^l u^l \binom{k}{l}F_k.
 \] 
 Now $\{ {}_{(m)}a_{ij}^{k-l}u^l : 1\leq l\leq k, k\geq 0 \}$ is linearly independent over 
 $R_0$. 
  If $\epsilon(a_m)=0$, then 
 \[ 
 0 = u_{ij}.f-f= \sum_{k\geq 1} \sum_{1\leq l\leq k} (-1)^{l}  {}_{(m)}a_{ij}^{k-l} {}_{(m)}a_{jj}^l u^l \binom{k}{l}F_k, 
 \] 
 and $\{ {}_{(m)}a_{ij}^{k-l}u^l : 1\leq l\leq k, k\geq 0 \}$ is linearly independent over 
 $R_0$. 
 Hence each $F_k=0$ for $k\geq 1$, contradicting the choices of $(i,j)$ and $m$. 
 It follows that $f\in \mathbb{C}[{}_{(\alpha)}a_{ii}]$ as claimed.

\underline{\textbf{Type D}}.
Next, consider the $D_r$-Dynkin quiver. 
Let $A := (A_1,\ldots, A_{r-2}, A_{r-1})\in \mathfrak{b}_n^{\oplus r-1}$ 
be a tuple of matrices, where $A_{\alpha}$ is a general representation at arrow $a_{\alpha}$, 
and let $(u_1,\ldots, u_r)\in \mathbb{U}_{\beta}$ act on $A$ by 
\[ 
\begin{aligned} 
(u_1,&\ldots, u_r).(A_1,\ldots, A_{r-1})  \\ 
=  (&u_{1+\epsilon(a_1)} A_1 u_{1-\epsilon(a_1)+1}^{-1}, \ldots, 
    u_{\alpha+\epsilon(a_{\alpha})} A_{\alpha} u_{\alpha-\epsilon(a_{\alpha})+1}^{-1}, \ldots, \\
    &u_{r-2+\epsilon(a_{r-2})} A_{r-2} u_{r-2-\epsilon(a_{r-2})+1}^{-1}, 
    u_{r-1+\epsilon(a_{r-1})} A_{r-1} u_{r-1-\epsilon(a_{r-1})  }^{-1}). \\ 
\end{aligned}
 \] 
 On the level of functions, we have 
 \[
 \begin{aligned}
 (u_1,&\ldots, u_r).f(A_1,\ldots, A_{r-1}) \\ 
 = f(&u_{1+\epsilon(a_1)}^{-1} A_1 u_{1-\epsilon(a_1)+1}, \ldots, 
    u_{\alpha+\epsilon(a_{\alpha})}^{-1} A_{\alpha} u_{\alpha-\epsilon(a_{\alpha})+1}, 
    \ldots, \\ 
     &u_{r-2+\epsilon(a_{r-2})}^{-1} A_{r-2} u_{r-2-\epsilon(a_{r-2})+1}, 
    u_{r-1+\epsilon(a_{r-1})}^{-1} A_{r-1} u_{r-1-\epsilon(a_{r-1})  }).  \\ 
 \end{aligned}
 \]  
The argument for this quiver is the same as before if $m \leq r-3$. 
If $m=r-2$, then the variable ${}_{(r-2)}a_{ij}$ is of interest; 
this variable corresponds to a general matrix $A_{r-2}$  
which is associated to the arrow $a_{r-2}$ pointing northeast or southwest. 
So let $U_{ij}$ be the subgroup of matrices of the form $(\I,\ldots, \I,\widehat{u}_{r-1},\I)$, 
where the entries of $\widehat{u}_{r-1}$  are the following: 
$u$ in the $(i,j)$-entry, $1$ along the diagonal, and  $0$ otherwise. 
Thus, $u_{ij}\in U_{ij}$ acts on ${}_{(r-2)} a_{st}$ in the following way: 
\[ 
u_{ij}.{}_{(r-2)}a_{st} =  
\begin{cases}  
{}_{(r-2)}a_{ij}- {}_{(r-2)}a_{jj}u & \mbox{ if } (s,t)=(i,j) \mbox{ and }\epsilon(a_{r-2})=1,\\ 
{}_{(r-2)}a_{it}- {}_{(r-2)}a_{jt}u & \mbox{ if }  s=i, t >j  \mbox{ and } \epsilon(a_{r-2})=1,\\ 
{}_{(r-2)}a_{ij}+ {}_{(r-2)}a_{ii}u & \mbox{ if } (s,t)=(i,j) \mbox{ and }\epsilon(a_{r-2})=0,\\ 
{}_{(r-2)}a_{sj}+ {}_{(r-2)}a_{si}u & \mbox{ if }  s<i, t=j  \mbox{ and } \epsilon(a_{r-2})=0,\\ 
{}_{(r-2)}a_{st} &\mbox{ otherwise. } \\ 
\end{cases} 
\] 

If $m=r-1$, then the variable ${}_{(r-1)}a_{ij}$ is of interest; 
this variable corresponds to a general matrix $A_{r-1}$ which is associated to the arrow 
$a_{r-1}$ pointing southeast or northwest.   
In this case, let $U_{ij}$ be the subgroup of matrices of the form 
$(\I,\ldots, \I,\widehat{u}_{r})$, 
where the entries of the matrix $\widehat{u}_r$ are $u$ in the $(i,j)$-entry, $1$ along the main diagonal, and $0$ otherwise. 
Thus, 
\[
u_{ij}.{}_{(r-1)}a_{st} = 
\begin{cases}
{}_{(r-1)}a_{ij}- {}_{(r-1)}a_{jj}u & \mbox{ if } (s,t)=(i,j) \mbox{ and } \epsilon(a_{r-1})=1, \\ 
{}_{(r-1)}a_{it}- {}_{(r-1)}a_{jt}u & \mbox{ if }  s=i, t >j, \mbox{ and } \epsilon(a_{r-1})=1,\\ 
{}_{(r-1)}a_{ij}+ {}_{(r-1)}a_{ii}u & \mbox{ if } (s,t)=(i,j) \mbox{ and } \epsilon(a_{r-1})=-1, \\ 
{}_{(r-1)}a_{sj}+ {}_{(r-1)}a_{si}u & \mbox{ if }  s<i, t =j, \mbox{ and } \epsilon(a_{r-1})=-1, \\ 
{}_{(r-1)}a_{st} &\mbox{ otherwise. } \\ 
\end{cases} 
\] 
So if $m=r-2$ and $\epsilon(a_{r-2})=1$ or $m=r-1$ and $\epsilon(a_{r-1})=1$, we have 
\[ 
0= u_{ij}.f-f = \sum_{k\geq 1} \sum_{1\leq l\leq k}(-1)^l {}_{(m)}a_{ij}^{k-l} {}_{(m)}a_{jj}^l u^l \binom{k}{l}F_k. 
\] 
 If $m=r-2$ and $\epsilon(a_{r-2})=0$ or $m=r-1$ and $\epsilon(a_{r-1})=-1$, we have 
\[ 
0= u_{ij}.f-f = \sum_{k\geq 1} \sum_{1\leq l\leq k} {}_{(m)}a_{ij}^{k-l} {}_{(m)}a_{ii}^l u^l \binom{k}{l}F_k. 
\]
Since $\{ {}_{(m)}a_{ij}^{k-l}u^l : 1\leq l\leq k, k\geq 0 \}$ is linearly independent over 
 $R_0$,  
we conclude that  
$f\in \mathbb{C}[{}_{(\alpha)}a_{ii}]$.

\underline{\textbf{Type E}}.
Finally, consider the $E_{r}$-Dynkin quiver, where $r = 6,7,8$. 
Since there exists a natural embedding of 
$E_6 \hookrightarrow E_7 \hookrightarrow E_8$, 
the filtered representation space for $E_{\gamma}$ embeds into 
the filtered representation space for $E_{\gamma+1}$.  
So for  
$r=6,7,8$,  
let $u_{\alpha}$ be an unipotent matrix at vertex $\alpha$ and let $A_{\alpha}$  
be a general representation of arrow $a_{\alpha}$.  
Then $(u_1,\ldots, u_r)\in \mathbb{U}_{\beta}$ acts on $(A_1,\ldots, A_{r-1})$ via 
\[ 
\begin{aligned}
(u_1,&\ldots, u_r).(A_1,\ldots, A_{r-1}) \\ 
  = (&u_{1+\epsilon(a_1)} A_1 u_{1-\epsilon(a_1)+1}^{-1}, \ldots, 
   u_{4+\epsilon(a_4)} A_{4} u_{4-\epsilon(a_4)}^{-1}, 
    u_{\alpha+\epsilon(a_{\alpha})} A_{\alpha} u_{\alpha-\epsilon(a_{\alpha})+1}^{-1},  
    \ldots, \\ 
    &u_{r-1+\epsilon(a_{r-1})} A_{r-1} u_{r-1-\epsilon(a_{r-1})+1}^{-1}). \\ 
\end{aligned}
\]  
We use similar strategy as in the $A_r$ type setting when $1\leq m\leq 8$ 
but not when 
$m=3$ or $4$. 
If $m=3$, then let 
$U_{ij}$ be the subgroup of matrices of the form $(\I,\I,\I,\widehat{u}_4,\ldots, \I)$, where 
$\widehat{u}_4$ has $u$ in $(i,j)$-entry, $1$ along the diagonal, and $0$ elsewhere. 
Under the action by $u_{ij}\in U_{ij}$,  
the coordinates of $F^{\bullet}Rep(Q,\beta)$ change as follows: 
\[
u_{ij}. {}_{(\alpha)}a_{st} = 
\begin{cases} 
{}_{(3)}a_{ij}- {}_{(3)}a_{jj}u &\mbox{ if } \alpha =3, (s,t)=(i,j), \mbox{ and } \epsilon(a_3)=1,\\ 
{}_{(3)}a_{it}- {}_{(3)}a_{jt}u &\mbox{ if } \alpha =3, s=i, t > j,  \mbox{ and } \epsilon(a_3)=1,\\ 
{}_{(3)}a_{ij}+ {}_{(3)}a_{ii}u &\mbox{ if } \alpha =3, (s,t)=(i,j), \mbox{ and } \epsilon(a_3)=0,\\ 
{}_{(3)}a_{sj}+ {}_{(3)}a_{si}u &\mbox{ if } \alpha =3, s<i, t = j,  \mbox{ and } \epsilon(a_3)=0,\\ 
{}_{(\alpha)}a_{st} &\mbox{ otherwise,} \\ 
\end{cases}  
\]  
which will give us 
\[ 
0= u_{ij}.f-f = (-1)^{\epsilon(a_{3})-1} \sum_{k\geq 1} \sum_{1\leq l\leq k} (-1)^l {}_{(3)}a_{ij}^{k-l} {}_{(3)}a_{jj}^l u^l \binom{k}{l}F_k \mbox{ if }\epsilon(a_{3})=1 
\] 
or 
\[ 
0= u_{ij}.f-f = (-1)^{\epsilon(a_{3})-1} \sum_{k\geq 1} \sum_{1\leq l\leq k} {}_{(3)}a_{ij}^{k-l} {}_{(3)}a_{ii}^l u^l \binom{k}{l}F_k \mbox{ if }\epsilon(a_{3})=0.  
\]  
Since $\{ {}_{(3)}a_{ij}^{k-l}u^l : 1\leq l\leq k, k\geq 0 \}$ is linearly independent over 
 $R_0$,   
we conclude     
$f\in \mathbb{C}[{}_{(\alpha)}a_{ii}]$.

Finally, if $m=4$, then we use 
$U_{ij}$ to be the subgroup of matrices of the form $(\I,\I,\widehat{u}_3,\I,\ldots, \I)$, where 
$\widehat{u}_3$ has $u$ in $(i,j)$-entry,  
$1$ along the diagonal entries, and $0$ elsewhere. This implies 
\[ 
u_{ij}. {}_{(\alpha)}a_{st} = 
\begin{cases} 
{}_{(4)}a_{ij}+{}_{(4)}a_{ii}u &\mbox{ if } \alpha =4, (s,t)=(i,j),\mbox{ and }\epsilon(a_4)=1, \\ 
{}_{(4)}a_{sj}+{}_{(4)}a_{si}u &\mbox{ if } \alpha =4, s<i, t = j,\mbox{ and }\epsilon(a_4)=1,  \\ 
{}_{(4)}a_{ij}-{}_{(4)}a_{jj}u &\mbox{ if } \alpha =4, (s,t)=(i,j),\mbox{ and }\epsilon(a_4)=0, \\ 
{}_{(4)}a_{it}-{}_{(4)}a_{jt}u &\mbox{ if } \alpha =4, s=i, t > j,\mbox{ and }\epsilon(a_4)=0,  \\ 
{}_{(\alpha)}a_{st} &\mbox{ otherwise, } \\ 
\end{cases}   
\] 
which will give us 
\[ 
0= u_{ij}.f-f = (-1)^{\epsilon(a_{4})-1} \sum_{k\geq 1} \sum_{1\leq l\leq k}{}_{(4)}a_{ij}^{k-l} {}_{(4)}a_{ii}^l u^l \binom{k}{l}F_k \mbox{ if }\epsilon(a_4)=1 
\] 
or 
\[ 
0= u_{ij}.f-f = (-1)^{\epsilon(a_{4})-1} \sum_{k\geq 1} \sum_{1\leq l\leq k} (-1)^l {}_{(4)}a_{ij}^{k-l} {}_{(4)}a_{jj}^l u^l \binom{k}{l}F_k \mbox{ if }\epsilon(a_4)=0. 
\] 
Since $\{ {}_{(4)}a_{ij}^{k-l}u^l : 1\leq l\leq k, k\geq 0 \}$ is linearly independent over 
 $R_0$, 
  $f\in \mathbb{C}[{}_{(\alpha)}a_{ii}]$.  
\end{proof}

Recall the affine Dynkin graphs: 
$\widetilde{D}_r$, $\widetilde{E}_6$, $\widetilde{E}_7$, and  $\widetilde{E}_8$: 
\[  
 \begin{tabular}{c|c}
  \multicolumn{1}{c}{\bfseries Type} & \multicolumn{1}{c}{\bfseries Graph}  \\ 
  \hline 
$\widetilde{D}_r$ &  $\xymatrix{  
  \stackrel{r+1}{\bullet}	\ar@{-}[dr]^{ } & & & & \stackrel{r-1}{\bullet} \\  
   &  
	\stackrel{2}{\bullet}	\ar@{-}[r]^{ } &  
	\ldots 								\ar@{-}[r]^{ } &  
\stackrel{r-2}{\bullet} \ar@{-}[ur]^{ } \ar@{-}[dr]^{ } & \\ 
  \stackrel{1}{\bullet} \ar@{-}[ur]^{ } & & & & \stackrel{r}{\bullet} \\   
}$  \\   
$\widetilde{E}_6$ &   
$\xymatrix@-1pc{  
& & \stackrel{7}{\bullet} \ar@{-}[d] & & \\ 
& & \stackrel{4}{\bullet} & &  \\  
\stackrel{1}{\bullet}\ar@{-}[r]^{ }  &   
\stackrel{2}{\bullet}\ar@{-}[r]^{ }  &  
\stackrel{3}{\bullet}\ar@{-}[r]^{ } \ar@{-}[u]^{ } & 
\stackrel{5}{\bullet}\ar@{-}[r]^{ }  & 
\stackrel{6}{\bullet} \\ 
}$ \\ 
&  \\ 
$\widetilde{E}_7$ & 
 $  \xymatrix@-1pc{
& & & \stackrel{4}{\bullet} & & & \\ 
\stackrel{8}{\bullet} \ar@{-}[r]^{}  & 
\stackrel{1}{\bullet}\ar@{-}[r]^{ }  & 
\stackrel{2}{\bullet}\ar@{-}[r]^{ }  & 
\stackrel{3}{\bullet}\ar@{-}[r]^{ } \ar@{-}[u]^{ } & 
\stackrel{5}{\bullet}\ar@{-}[r]^{ }  & 
\stackrel{6}{\bullet}\ar@{-}[r]^{ }  & 
\stackrel{7}{\bullet} \\ 
}$    \\
&  \\ 
$\widetilde{E}_8$ &  
$\xymatrix@-1pc{
& & \stackrel{4}{\bullet} & & & &  & \\ 
\stackrel{1}{\bullet}\ar@{-}[r]^{ }  & 
\stackrel{2}{\bullet}\ar@{-}[r]^{ }  & 
\stackrel{3}{\bullet}\ar@{-}[r]^{ } \ar@{-}[u]^{ } & 
\stackrel{5}{\bullet}\ar@{-}[r]^{ }  &  
\stackrel{6}{\bullet}\ar@{-}[r]^{ }  &  
\stackrel{7}{\bullet}\ar@{-}[r]^{ }  &  
\stackrel{8}{\bullet}\ar@{-}[r]^{ }  &  
\stackrel{9}{\bullet}.  
 \\ 
} $ \\ 
  \end{tabular}  
\]

\begin{corollary}\label{corollary:affine-Dr-quiver-with-no-framing}  
For an affine $\widetilde{D}_r$, $\widetilde{E}_6$, $\widetilde{E}_7$, or   
$\widetilde{E}_8$-quiver with no framing and $\beta=(n,\ldots, n)\in \mathbb{Z}_{\geq 0}^{Q_0}$ 
with the complete standard filtration of vector spaces at each vertex of the quiver, we have 
$\mathbb{C}[\mathfrak{b}^{\oplus Q_1}]^{\mathbb{U}_{\beta}}\cong \mathbb{C}[\mathfrak{t}^{\oplus Q_1}]$. 
\end{corollary}

Since the proof of Corollary~\ref{corollary:affine-Dr-quiver-with-no-framing} is essentially the same as the proof of Theorem~\ref{theorem:filtered-ADE-Dynkin-quiver}, we omit the details.

\section{(Framed) filtered quiver varieties associated to affine Dynkin type \texorpdfstring{$\widetilde{A}_r$}{Ar}}\label{section:framed-filtered-quiver-var-affine-Dynkin-type}   
 
Let $\lambda=(\lambda_1\geq \lambda_2\geq \ldots \geq \lambda_l\geq 0)$ be a {\em partition} of size 
$|\lambda|$ $=$ $\sum_{i=1}^{l} \lambda_i$.  
One identifies to $\lambda$ a left-justified shape of $l$ rows of boxes of length $\lambda_1$, $\lambda_2$, $\ldots$, $\lambda_l$,  
which is called the {\em Young diagram} associated to $\lambda$. 
A {\em (Young) filling} or a {\em Young tableau} 
of $\lambda$ assigns a positive integer to each box of Young diagram. 

\begin{example}\label{example:young-diagram}  
Associated to $\lambda=(5,3,3,1)$ is the Young diagram  
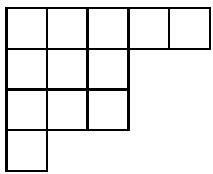
\[ 
\ydiagram{5,3,3,1} 
\] 
and a Young tableau    
\ytableausetup{mathmode, boxsize=2em} 
 \begin{ytableau} 
 4& 1& 2 & 3&3 \\ 
 3& 4& 5\\
 2& 5& 6\\    
 1 
 \end{ytableau}.    
\end{example}

\begin{definition}\label{definition:standard-form-bitableau}  
A Young tableau is   
{\em normal}   
if the entries in each row are strictly increasing from left to right. 
It is called {\em standard} if it is normal and the entries in each column are nondecreasing from top to bottom. 
A {\em bitableau} $J|I$ is a pair of Young tableaux $J$ and $I$ having the same shape and the bitableau is called {\em standard} if both $J$ and $I$ are standard Young tableaux. 
\end{definition}  
The {\em bideterminant} $(J|I)$  
of a bitableau is defined in the following way: 
the positive integer entries in $J$ in a fixed row correspond to the rows of a matrix while the positive integer entries in   
$I$ in the same row correspond to the columns of a matrix. Take the determinant of these minors of a matrix and repeat for each row in $J|I$ to obtain the bideterminant 
$(J|I)$.   

Note that the bideterminant $(J|I)$ associated to a bitableau $J|I$ is a product of minors of a matrix, where $J$ are the row indices and $I$ are the column indices. 
Furthermore, a matrix (or a product of matrices) must be specified when calculating the bideterminant associated to a bitableau; 
such specification will be denoted on the lower-right corner of each row of the bitableau (and the bideterminant), i.e., see 
\eqref{eq:bitableau-general-form-for-framed-quivers} 
and  
\eqref{eq:bideterminant-general-form-for-framed-quivers}.  

\begin{example} 
Consider the two Young tableaux:  
\[ J = 
\ytableausetup{mathmode, boxsize=2em} 
 \begin{ytableau} 
 1& 2& 4 \\ 
 2& 4    \\
 1& 4    \\    
 \end{ytableau}
\hspace{2mm} 
\mbox{ and } 
\hspace{2mm}
I = 
\ytableausetup{mathmode, boxsize=2em} 
 \begin{ytableau} 
 1& 2 & 4 \\ 
 2& 4 \\
 2& 5 \\    
 \end{ytableau}. 
\] 
$J$ is normal but not standard (since the entries in the first column are not nondecreasing when reading from top to bottom), while $I$ is standard. 
\end{example}

From this point forward, we will not draw a box around each entry of a Young (bi)tableau or the bideterminant of a bitableau.

Let $Q$ be a quiver with one framed vertex labelled as $1'$  
and all other (nonframed) vertices labelled as $1$, $2$, $\ldots$, $|Q_0|-1$,  
and the arrows labelled as $a_0$, $a_1$, $\ldots$, $a_{|Q_1|-1}$,  
where $ta_0=1'$ and $ha_0=1$.  
Let $A_{\phi_u}$ be a general representation of the arrow $a_{\phi_u}$. 
The product $A_{\phi_u}A_{\phi_u-1}A_{\phi_u-2}\cdots A_0$  
of general matrices is associated to the quiver path 
$a_{\phi_u}a_{\phi_u-1}a_{\phi_u-2}\cdots a_0$.  
Moreover, 
$A_{\phi_u}A_{\phi_u-1}A_{\phi_u-2}\cdots A_0$  
is uniquely associated to the sequence 
\begin{equation}\label{equation:unique-sequence-ass-to-prod-of-matrices}
\mathbf{\Phi_u} := [\phi_u, \phi_u-1, \phi_u-2, \ldots, 0]
\end{equation}  
of integers obtained by reading the indices of 
$A_{\phi_u}A_{\phi_u-1}A_{\phi_u-2}\cdots A_0$  
from right to left.  
Consider all $\mathbf{\Phi_u}$ whose quiver paths begin at the framed vertex;  
we will fix a partial ordering $\leq$ on these finite sequences of nonnegative integers. We say 
\[ 
\mathbf{\Phi_u} := [\phi_u, \phi_u-1, \phi_u-2, \ldots, 0] \leq 
[\phi_v, \phi_v-1, \phi_v-2, \ldots, 0] =: \mathbf{\Phi_v} 
\]  
if the number of integers in 
$\mathbf{\Phi_u}$ 
is less than the number of integers in 
$\mathbf{\Phi_v}$, 
or 
if the number of integers in  
$\mathbf{\Phi_u}$ 
equals the number of integers in 
$\mathbf{\Phi_v}$ 
and the right-most nonzero entry in 
$\mathbf{\Phi_v}-\mathbf{\Phi_u}$ is positive when subtracting component-wise.   
 
\begin{definition}\label{definition:bideterminant-product-of-general-matrices-framed-quiver}  
The $s^{th}$ row of the bitableau  
\begin{equation}\label{eq:bitableau-general-form-for-framed-quivers}
 \begin{matrix}
   j_1^{(1)} & j_2^{(1)} & \ldots & j_{v_1}^{(1)}  &| 
  i_1^{(1)}& i_2^{(1)} &\ldots & {i_{v_1}^{(1)}}_{A_{\phi_1}\cdots  A_0} \\
   j_1^{(2)} & j_2^{(2)} & \ldots & j_{v_2}^{(2)}  &| 
  i_1^{(2)}& i_2^{(2)} &\ldots & { i_{v_2}^{(2)}}_{A_{\phi_2}\cdots  A_0} \\
       & \vdots&        &    &           &  \vdots   &       &                   \\
   j_1^{(l)} & j_2^{(l)} & \ldots & j_{v_l}^{(l)}  &| 
  i_1^{(l)}& i_2^{(l)} &\ldots & {i_{v_l}^{(l)}}_{A_{\phi_l}\cdots  A_0} \\
  \end{matrix} 
\end{equation} 
is defined to be the bitableau associated to rows 
$j_1^{(s)}$, $j_2^{(s)}$, $\ldots$, $j_{v_s}^{(s)}$ and columns 
$i_1^{(s)}$, $i_2^{(s)}$, $\ldots$, $i_{v_s}^{(s)}$ in the product $A_{\phi_s}A_{\phi_s-1}\cdots A_0$ of general matrices. 
The bideterminant 
\begin{equation}\label{eq:bideterminant-general-form-for-framed-quivers} 
\begin{matrix}
   (j_1^{(1)} & j_2^{(1)} & \ldots & j_{v_1}^{(1)}  &| 
  i_1^{(1)}& i_2^{(1)} &\ldots & i_{v_1}^{(1)})_{A_{\phi_1}\cdots  A_0} \\
   (j_1^{(2)} & j_2^{(2)} & \ldots & j_{v_2}^{(2)}  &| 
  i_1^{(2)}& i_2^{(2)} &\ldots & i_{v_2}^{(2)})_{A_{\phi_2}\cdots  A_0} \\
       & \vdots&        &    &           &  \vdots   &       &                   \\
   (j_1^{(l)} & j_2^{(l)} & \ldots & j_{v_l}^{(l)}  &| 
  i_1^{(l)}& i_2^{(l)} &\ldots & i_{v_l}^{(l)})_{A_{\phi_l}\cdots  A_0}  
\end{matrix}
\end{equation}
is the product of bideterminants 
\[ 
\begin{matrix}
(j_1^{(s)} & j_2^{(s)} & \ldots & j_{v_s}^{(s)}  &| 
  i_1^{(s)}& i_2^{(s)} &\ldots & i_{v_s}^{(s)})_{A_{\phi_s}A_{\phi_s-1}\cdots  A_0}  
  \end{matrix} 
\]  
obtained by taking the determinant of minors of rows  
$j_1^{(s)}, j_2^{(s)},\ldots, j_{v_s}^{(s)}$ and columns  
$i_1^{(s)}, i_2^{(s)},\ldots, i_{v_s}^{(s)}$ in the product $A_{\phi_s}A_{\phi_s-1}\cdots A_0$ of general matrices.   
\end{definition}  
Note that the $s^{th}$ row of \eqref{eq:bitableau-general-form-for-framed-quivers} is associated to the sequence $\mathbf{\Phi_s}$ of integers (cf. \eqref{equation:unique-sequence-ass-to-prod-of-matrices}).  
We say the bitableau \eqref{eq:bitableau-general-form-for-framed-quivers}  
is in {\em block standard form}  
if the sequence $\mathbf{\Phi_s}$ of integers associated to each row of the bitableau is in nondecreasing order (with respect to the partial ordering $\leq$ defined earlier in this section) when ascending down the rows and in the case the sequence $\mathbf{\Phi_s}$ of integers for multiple rows are identical, then these rows are in standard form as defined in Definition~\ref{definition:standard-form-bitableau}.  
In the case that the bitableau \eqref{eq:bitableau-general-form-for-framed-quivers} is in block standard form, we will interchangeably say the bideterminant (associated to the bitableau) is in block standard form.

\begin{example}\label{example:standard-form-bitableau} 
Let 
\[  
A_0 = \begin{pmatrix}   
 x_{11}& x_{12} \\  
 x_{21} &x_{22}   
\end{pmatrix}
 \mbox{ and } 
 A_1= \begin{pmatrix} 
  a_{11}&a_{12} \\   
   0    &a_{22}  
 \end{pmatrix}. 
\] 
Then 
\[ 
\begin{matrix} 
 ( 2& 		  |& 1 & 		&)_{A_0} \\  
 (2&  		  |& 2 & 		&)_{A_0} \\
 (1& 		2 |& 1 & 2  &)_{A_1A_0} \\ 
 ( 2& 		  |& 1 & 		&)_{A_1A_0}  
\end{matrix}
\] 
is a bideterminant in block standard form. 
\end{example} 


\begin{example}\label{example:determinant-of-standard-bitableau} 
The polynomial associated to the bideterminant in Example~\ref{example:standard-form-bitableau} is 
\[
\begin{aligned}
x_{21} x_{22} \det(A_1A_0)\cdot (A_1A_0)_{2,1}&= 
x_{21}x_{22} a_{11}a_{22}(x_{11}x_{22}-x_{12}x_{21})a_{22} x_{21} \\ 
&= x_{21}^2 x_{22}  a_{11}a_{22}^2 (x_{11}x_{22} - x_{12} x_{21}).\\ 
\end{aligned}
\] 
\end{example}

The following is Theorem 13.1 in \cite{MR1489234}. 

\begin{theorem}\label{theorem:standard-basis-theorem-Grosshans}
Let $R= \mathbb{C}[\{ x_{ij}: 1\leq i\leq n, 1\leq j\leq m\}]$. 
Then bideterminants of standard bitableaux form a basis over $\mathbb{C}$ of $R$. 
\end{theorem}


\begin{lemma}\label{lemma:borel-invariants-on-bideterminants}  
Let invertible upper triangular matrices $B\subseteq GL_n(\mathbb{C})$ act on $M_{n\times m}$ via left translation. 
Consider the character  
$\chi_p(b)$ $=$ $\displaystyle{\prod_{i=p}^n b_{ii}}$ of $B$ and let 
$f=(p\:\: p+1\:\cdots n | i_1 \cdots i_{n-p+1})\in \mathbb{C}[M_{n\times m}]$ 
where $1\leq i_1 < i_2 < \ldots < i_{n-p+1}\leq m$. 
Then $b.f=\chi_p(b)f$. 
\end{lemma}

\begin{proof} 
 Write $b=tu$, where $t=(t_{ii})\in T$ and $u\in U$, $T$ is the maximal torus in $B$ and $U$ is the maximal unipotent subgroup of $B$. 
Then 
$t.f = \displaystyle{\prod_{i=p}^n t_{ii} f} = \chi_p(t)f$ 
and for the subgroup $U_{i,j}$ which has $1$ along the main diagonal, the variable $u$ in the $(i,j)$-entry and $0$ elsewhere,
we will show that $U_{i,j}$ fixes $f$. 
So let $\widehat{u}\in U_{i,j}$. 
Then since $\widehat{u}.f(x) = f(\widehat{u}^{-1}.x) = f(\widehat{u}^{-1}x)$, 
\[
 \widehat{u}.x_{st} = 
\begin{cases} 
x_{it}-x_{jt}u &\mbox{ if } s =i, \\ 
x_{st} &\mbox{ otherwise}. \\ 
\end{cases}
\]
So $\widehat{u}.f=f-u(p\:\: p+1\cdots j \cdots j \cdots n| i_1 \cdots i_{n-p+1})=f$ if $p\leq i$
and $\widehat{u}.f=f$ if $p>i$. 
This concludes the proof. 
\end{proof}

We will call the graph 
\[ 
\xymatrix@-1pc{ 
& & & &\stackrel{r+1}{\bullet} \ar@{-}[lld]_{} &\ldots  &\stackrel{5}{\bullet}  & &  \\   
\stackrel{1'}{\circ} \ar@{-}[rr]^{} & & \stackrel{1}{\bullet} \ar@{-}[rdr]^{} & & & & & & \stackrel{4}{\bullet} \ar@{-}[llu]_{} \\  
& & & & \stackrel{2}{\bullet} \ar@{-}[rr]^{} & &  \stackrel{3}{\bullet}  \ar@{-}[rru] & & \\   
} 
\]
 the {\em framed affine $\widetilde{A}_r$-Dynkin diagram}.

\begin{setup}\label{setup:framed-affine-Ar-dynkin-diagram} 
Let $Q$ be a quiver whose underlying graph is the framed affine $\widetilde{A}_r$-Dynkin diagram, 
where the orientation of the arrow between vertices $1'$ and $1$ starts at $1'$ and ends at $1$.  
 Let $\beta=(n,\ldots, n,m)\in \mathbb{Z}^{r+2}$ be the dimension vector of $Q$, where $m$ corresponds to the dimension at the framed vertex $1'$. 
 We impose the complete standard filtration of vector spaces only at the nonframed vertices. 
 Let $F^{\bullet}Rep(Q,\beta)$ be the representation subspace of $Rep(Q,\beta)$ whose representations preserve the filtration of vector spaces and suppose the product $\mathbb{U}_{\beta}:= U^{r+1}$ 
 of maximal unipotent subgroups of 
$\mathbb{G}_{\beta} = (GL_{n}(\mathbb{C}))^{r+1}$ 
acts on $F^{\bullet}Rep(Q,\beta)$ as a change-of-basis. 
Let $A_k$ be a general representation of the arrow between vertices $k$ and $k+1$ where $1\leq k\leq r$,  
$A_{r+1}$ be a general matrix of the arrow between vertices $1$ and $r+1$,  
and $A_0$ be a general representation of the arrow $a_0$ from the framed vertex  $1'$ to vertex $1$ (we note that $A_0$ 
is a general representation corresponding to a map from $\mathbb{C}^m$ to $\mathbb{C}^n$). 
\end{setup}

The following lemma generalizes Lemma~\ref{lemma:borel-invariants-on-bideterminants}. 

\begin{lemma}\label{lemma:borel-invariants-on-equioriented-acylic-quiver}
Consider the {\em equioriented} (all arrows are pointing in the same direction) quiver 
\[ 
\xymatrix@-1pc{
\stackrel{1'}{\circ} \ar[rr]^{a_0} & & \stackrel{1}{\bullet} \ar[rr]^{a_1} & & \stackrel{2}{\bullet} 
&\ldots & \stackrel{r}{\bullet} \ar[rr]^{a_r} & & \stackrel{r+1}{\bullet},   
}
\]  
where the dimension of the vector space at the framed vertex is $m$ and the dimension of the vector space at nonframed vertices is $n$. 
Impose the complete standard filtration $F^{\bullet}$ of vector spaces at the nonframed vertices. 
Let $\mathbb{U}_{\beta}\subseteq B^{r+1}$ be the product of maximal unipotent subgroups and let  
$A_m$ be a general matrix associated to arrow $a_m$ for each $0\leq m\leq r$.  
Then $F^{\bullet}Rep(Q,\beta)=M_{n\times m}\oplus \mathfrak{b}^{\oplus r}$ 
and standard bideterminants of the form 
\begin{equation}\label{eq:standard-bideterminant-one-row-only-lemma}
(p\:\: p+1 \: \cdots \: n \: | \: i_1 \: i_2 \: \cdots \: i_{n-p+1} )_{A_{m}\cdots A_0}, 
\hspace{4mm} \mbox{ where }  1\leq p\leq n \mbox{ and } 0 \leq m\leq r, 
\end{equation}  
are $\mathbb{U}_{\beta}$-invariant polynomials. 
\end{lemma}  

Note that the polynomial \eqref{eq:standard-bideterminant-one-row-only-lemma} is a row of \eqref{eq:bideterminant-general-form-for-framed-quivers}.   
  
\begin{proof}   
For $u=(u_1,\ldots, u_{r+1})\in \mathbb{U}_{\beta}$ and  
$(A_0,A_1,\ldots, A_r)\in M_{n\times m}\oplus \mathfrak{b}^{\oplus r}$, 
\[ 
u.(A_0,A_1,\ldots, A_r) = 
(u_1 A_0, u_2 A_1 u_1^{-1},\ldots, u_{\alpha+1}A_{\alpha}u_{\alpha}^{-1},\ldots, u_{r+1}A_r u_r^{-1}). 
\] 
We will write the entries of the product $A_{\alpha}\cdots A_0$ 
of matrices as $({}_{(\alpha)}y_{st})$, where $({}_{(0)}y_{st})=(x_{st})\in M_{n\times m}$. 
Then for the subgroup $U_{ij}$ of $\mathbb{U}_{\beta}$ which is    
\[ 
\begin{aligned} 
U_{ij} =   \{ u_{ij} = (\I,\ldots, \widehat{u}_m,\ldots, \I):   
\widehat{u}_m &\mbox{ is the matrix with the variable $u$ in the }(i,j)\mbox{-}entry, \mbox{ where }i<j, \\ 
&1 \mbox{ along the diagonal entries, and }
0 \mbox{ elsewhere}     \},    \\ 
\end{aligned}
\] 
$u_{ij}\in U_{ij}$ acts on $({}_{(\alpha)} y_{st}) \in A_{\alpha}\cdots A_1 A_0$ as follows: $U_{ij}$ changes the coordinate polynomial ${}_{(\alpha)}y_{st}$ via  
\[   
u_{ij}.{}_{(\alpha)}y_{st} = 
\begin{cases} 
{}_{(m-1)}y_{it}-{}_{(m-1)}y_{jt}u &\mbox{ if }\alpha = m-1 \mbox{ and }s=i, \\ 
{}_{(\alpha)}y_{st} &\mbox{ otherwise. } \\ 
\end{cases} 
\] 
So for $f=(p\:\: p+1\:\cdots \: n| i_1\: i_2 \: \cdots i_{n-p+1})_{A_{\alpha}\cdots A_0}$, 
\[ 
u_{ij}.f=
\begin{cases} 
f-u(p\:\: p+1\:\cdots j\:\cdots \:j \cdots \: n| i_1\: i_2 \: \cdots i_{n-p+1})_{A_{m-1}\cdots A_0} = f  &\mbox{ if } \alpha=m-1 \mbox{ and } p\leq i, \\ 
f &\mbox{ otherwise. }    \\ 
\end{cases}
\] 
Thus standard bideterminants of the form 
$(p\:\: p+1 \: \cdots \: n \: | \: i_1 \: i_2 \: \cdots \: i_{n-p+1} )_{A_{m}\cdots A_0}$
are $\mathbb{U}_{\beta}$-invariant polynomials, 
where $1\leq p\leq n$ 
 and 
$0 \leq m\leq r$. 
\end{proof}

\begin{definition}\label{definition:bideterminant-poly-for-theorem}
Assume Basic Set-up~\ref{setup:framed-affine-Ar-dynkin-diagram}. 
Define 
\begin{equation}\label{eq:bideterminant-block-standard-form-generic-notation}
(J|I)_{\mathbf{\Phi_{\phi_1}}\mathbf{\Phi_{\phi_2}} \cdots \mathbf{\Phi_{\phi_l}}} := 
  \begin{matrix}
   (j_1 & j_1+1 & \ldots & n  &| i_1^{(1)}& i_2^{(1)} &\ldots & i_{n-j_1+1}^{(1)})_{A_{\phi_1}\cdots  A_0} \\
   (j_2 & j_2+1 & \ldots & n  &| i_1^{(2)}& i_2^{(2)} &\ldots & i_{n-j_2+1}^{(2)})_{A_{\phi_2}\cdots  A_0} \\ 
       & \vdots&        &    &           &  \vdots   &       &                   \\
   (j_l & j_l+1 & \ldots & n  &| i_1^{(l)}& i_2^{(l)} &\ldots & i_{n-j_l+1}^{(l)})_{A_{\phi_l}\cdots  A_0}  
  \end{matrix}
\end{equation} 
as the bideterminant in block standard form, 
where  
$A_{\phi_k}\cdots A_0$ 
is a general representation of the quiver path $a_{\phi_k}\cdots a_0$ 
which begins at the framed vertex. 
\end{definition}

Note that each sequence $\mathbf{\Phi_{\phi_s}}$ of integers associated to each row of  \eqref{eq:bideterminant-block-standard-form-generic-notation}  
corresponds to a general representation of a quiver path that begins at the framed vertex (this is important as this will imply the uniqueness of \eqref{eq:bideterminant-block-standard-form-generic-notation}:   
 if 
 $(J|I)_{\mathbf{\Phi_{\phi_1}}\mathbf{\Phi_{\phi_2}} \cdots \mathbf{\Phi_{\phi_l}}}
 = 
 (J'|I')_{\mathbf{\Phi_{\phi_1'}}\mathbf{\Phi_{\phi_2'}} \cdots \mathbf{\Phi_{\phi_l'}}}$
 where $(J|I)_{\mathbf{\Phi_{\phi_1}}\mathbf{\Phi_{\phi_2}} \cdots \mathbf{\Phi_{\phi_l}}}$ and 
 $(J'|I')_{\mathbf{\Phi_{\phi_1'}}\mathbf{\Phi_{\phi_2'}} \cdots \mathbf{\Phi_{\phi_l'}}}$ are in block standard form, 
 then $J=J'$, $I=I'$, and $\mathbf{\Phi_{\phi_s}}=\mathbf{\Phi_{\phi_s'}}$ for all $1\leq s\leq l$).

\begin{theorem}\label{theorem:invariants-of-framed-affine-Ar-quiver}  
Assume Basic Set-Up~\ref{setup:framed-affine-Ar-dynkin-diagram}. 
 Then 
$F^{\bullet}Rep(Q,\beta)\cong \mathfrak{b}^{\oplus r+1}\oplus M_{n\times m}$ and 
$\mathbb{C}[F^{\bullet}Rep(Q,\beta)]^{\mathbb{U}_{\beta}}$ $\cong$ $\mathbb{C}[\mathfrak{t}^{\oplus r+1}]\otimes_{\mathbb{C}} \mathbb{C}[\{f \}]$, 
where  
\begin{equation}\label{eq:affine-dk-quiver-general-n-inv-poly}
  f = \sum_{\nu} g_{\nu}({}_{(\alpha)}a_{st})  
(J_{\nu}|I_{\nu})_{\mathbf{\Phi_{\phi_1}}\mathbf{\Phi_{\phi_2}} \cdots \mathbf{\Phi_{\phi_l}}}, 
\end{equation} 
where  
$g_{\nu}({}_{(\alpha)}a_{st} )\in \mathbb{C}[\mathfrak{t}^{\oplus r+1}]$ 
 and $(J_{\nu}|I_{\nu})_{\mathbf{\Phi_{\phi_1}}\mathbf{\Phi_{\phi_2}} \cdots \mathbf{\Phi_{\phi_l}}}$  
 is the block standard bideterminant given in Definition~\ref{definition:bideterminant-poly-for-theorem}. 
\end{theorem}

\begin{proof}
The identification $F^{\bullet}Rep(Q,\beta)\cong \mathfrak{b}^{\oplus r+1}\oplus M_{n\times m}$ is clear. 
Let 
${}_{(\alpha)}a_{st}$ be the entries of the general matrix $A_{\alpha}$  
 and let $x_{st}$ 
be the entries of the general matrix $A_0$.  

First, we prove that invariant polynomials independent of $x_{st}$ must be coming from the diagonal entries of $A_{\alpha}$. 
We then prove that invariants involving $x_{st}$ must be of the form \eqref{eq:affine-dk-quiver-general-n-inv-poly}. 
So suppose $f\in \mathbb{C}[F^{\bullet}Rep(Q,\beta)]^{\mathbb{U}_{\beta}}$ such that 
$f$ is a polynomial in only ${}_{(\alpha)}a_{st}$, i.e., $f$ does not depend on $x_{st}$. 
Writing $a_{\alpha}$ to be the arrow connecting vertices $\alpha$ and $\alpha+1$ for $1\leq \alpha\leq r$ and $a_{r+1}$ to be the arrow connecting vertices $r+1$ and $1$, 
define 
\[ 
\epsilon(a_{\alpha}) = 
\begin{cases}  
1 &  \mbox{ if } a_{\alpha} \mbox{ points in the counterclockwise direction, } \\ 
0 & \mbox{ if }a_{\alpha} \mbox{ points in the clockwise direction. } \\ 
\end{cases} 
\] 
Writing $U_{\alpha}$ to be the maximal unipotent group acting as a change-of-basis of the filtration of vector spaces at vertex $\alpha$, for $u_{\alpha}\in U_{\alpha}$, we have 
\[ 
\begin{aligned} 
(u_1, &\ldots, u_{r+1}).(A_1,\ldots, A_{r+1}) \\ 
&= (u_{1+\epsilon(a_1)} A_1 u_{1-\epsilon(a_1)+1}^{-1}, \ldots, 
    u_{\alpha+\epsilon(a_{\alpha})} A_{\alpha} u_{\alpha-\epsilon(a_{\alpha})+1}^{-1}, \ldots,
    u_{\floor{r+1+\epsilon(a_{r+1})}} A_{r+1} u_{\floor{r+1-\epsilon(a_{r+1})+1}}^{-1}), \\ 
\end{aligned}
\] 
where 
\[ 
\floor{r+1+\epsilon(a_{r+1})} := 
\begin{cases} 
1 & \mbox{ if } \epsilon(a_{r+1})=1, \\ 
r+1 &\mbox{ if }\epsilon(a_{r+1})=0, \\ 
\end{cases}
\] 
and 
\[ 
\floor{r+1-\epsilon(a_{r+1})+1} := 
\begin{cases} 
r+1 &\mbox{ if } \epsilon(a_{r+1})=1, \\ 
1 &\mbox{ if } \epsilon(a_{r+1})=0. \\  
\end{cases}  
\] 
On the level of functions, we have 
 \[ 
 \begin{aligned}
 (u_1,&\ldots, u_{r+1}).f(A_1,\ldots, A_{r+1}) \\ 
 &= f(u_{1+\epsilon(a_1)}^{-1} A_1   u_{1-\epsilon(a_1)+1}, \ldots, 
    u_{\alpha+\epsilon(a_{\alpha})}^{-1} A_{\alpha} u_{\alpha-\epsilon(a_{\alpha})+1}, 
    \ldots,
    u_{\floor{r+1+\epsilon(a_{r+1})}}^{-1} A_{r+1} u_{\floor{r+1-\epsilon(a_{r+1})+1}}). \\ 
 \end{aligned}
 \]

If $f$ is a polynomial only in ${}_{(\alpha)}a_{ii}$, 
then $f\in \mathbb{C}[\mathfrak{t}^{\oplus r+1}]$ and we are done. 
So suppose not. 
Fix a total ordering $\leq$ on pairs $(i,j)$, where $1\leq i\leq j\leq n$, by defining 
 $(i,j)\leq (i',j')$ if either  
 \begin{itemize}  
 \item $i < i'$ or  
 \item $i=i'$ and $j> j'$. 
 \end{itemize} 
 Let $f\in \mathbb{C}[F^{\bullet}Rep(Q,\beta)]^{\mathbb{U}_{\beta}}$. 
 For each $(i,j)$, we can write  
 \begin{equation}\label{eq:aij-invariant-function-borel-affine}
 f = \sum_K a_{ij}^K f_{ij,K}, \mbox{ where } f_{ij,K}\in \mathbb{C}[\{ {}_{(\alpha)}a_{st}: (s,t)\not=(i,j)\} ] 
 \mbox{ and }
 a_{ij}^K := \prod_{\alpha=1}^{r+1} {}_{(\alpha)}a_{ij}^{k_{\alpha}}. 
 \end{equation}  
 Fix the least pair (under $\leq$) $(i,j)$ with $i< j$ for which there exists $K\not=(0,\ldots, 0)$ with $f_{ij,K}\not=0$; 
 we will continue to denote it by $(i,j)$. 
 If no such $(i,j)$ exists, then $f\in \mathbb{C}[{}_{(\alpha)}a_{ii}]$ and we are done.

Let $K=(k_1,\ldots, k_{r+1})$. Let $m\geq 1$ be the least integer satisfying the following: 
for all $p<m$, if some component $k_p$ in $K$ is strictly greater than $0$, then $f_{ij,K}=0$.  
 Let $U_{ij}$ be the subgroup of matrices of the form 
 $u_{ij} = (\I,\ldots, \I, \widehat{u}_m,\I, \ldots, \I)$, where $\I$ is the $n\times n$  
 identity matrix and  
 $\widehat{u}_m$ is the matrix with $1$ along the diagonal, the variable $u$ in the $(i,j)$-entry, and $0$ elsewhere. 
 Let $u_{ij}^{-1}:=(\I,\ldots, \I, \widehat{u}_m^{-1},\I, \ldots, \I)$. 
 Then since $u_{ij}$ acts on $f$ via  
 $u_{ij}.f(A_1,\ldots, A_{r+1})$ $=$ $f(u_{ij}^{-1}.(A_1,\ldots, A_{r+1}))$ 
  $=$ $f(A_1, 
  \ldots,  
  \widehat{u}_{m-1+\epsilon(a_{m-1})}^{-1} A_{m-1}\widehat{u}_{m-1-\epsilon(a_{m-1})+1},  
  \widehat{u}_{m +\epsilon(a_{m})}^{-1}    A_m    \widehat{u}_{m-\epsilon(a_m)+1}, 
  \ldots,  
  A_{r+1})$, 
  where $\widehat{u}_{m-1}$ $:=$ $\I$, 
  $\widehat{u}_{m+1}$ $:=$ $\I$,  
  and $\widehat{u}_m$ is the unipotent matrix defined as before,  
  $u_{ij}$ changes coordinate variables of $\mathfrak{b}^{\oplus r+1}$ as follows:   
 \begin{equation}\label{eq:ar-dynkin-Uij-action-affine}
  u_{ij}.{}_{(\alpha)}a_{st}= 
 \begin{cases} 
 {}_{(m)}a_{ij} +{}_{(m)}a_{ii}u &\mbox{ if } \alpha=m, (s,t)=(i,j), \mbox{ and } \epsilon(a_m)=1, \\ 
 {}_{(m-1)}a_{ij} -{}_{(m-1)}a_{jj}u &\mbox{ if } \alpha=m-1, (s,t)=(i,j), \mbox{ and } \epsilon(a_{m-1})=1, \\
  {}_{(m)}a_{ij} -{}_{(m)}a_{jj}u &\mbox{ if } \alpha=m, (s,t)=(i,j), \mbox{ and } \epsilon(a_m)=0, \\
 {}_{(m-1)}a_{ij} +{}_{(m-1)}a_{ii}u &\mbox{ if } \alpha=m-1, (s,t)=(i,j), \mbox{ and } \epsilon(a_{m-1})=0, \\
 {}_{(\alpha)}a_{st} &\mbox{ if }s>i \mbox{ or }s=i \mbox{ and }t<j, \\ 
 \end{cases} 
 \end{equation}
where $m-1 := r+1$ if $m=1$.

First, suppose $m>1$.  
 Write 
 \[ f = \sum_{k\geq 0} {}_{(m)}a_{ij}^k F_k, \mbox{ where } 
      F_k\in 
      \mathbb{C}[\{ {}_{(\alpha)}a_{st}: 
      (s,t)\geq (i,j)\mbox{ and if } (s,t)=(i,j), \mbox{ then } \alpha > m 
      \} ]. 
 \]  
 Let $R_0:= \mathbb{C}[\{ {}_{(\alpha)}a_{st}: 
      (s,t)\geq (i,j)\mbox{ and if } (s,t)=(i,j), \mbox{ then } \alpha > m 
      \} ]$. 
If $\epsilon(a_m)=1$, then $u_{ij}.{}_{(m)}a_{ij}={}_{(m)}a_{ij}+{}_{(m)}a_{ii}u$.
So  
 \[  
 0 = u_{ij}.f-f=   \sum_{k\geq 1} \sum_{1\leq l\leq k}{}_{(m)}a_{ij}^{k-l} {}_{(m)}a_{ii}^l u^l \binom{k}{l}F_k, 
 \] 
 and 
 $\{ {}_{(m)}a_{ij}^{k-l}u^l : 1\leq l\leq k, k\geq 1 \}$ is linearly independent over 
 $R_0$.
 If $\epsilon(a_m)=0$, then $u_{ij}.{}_{(m)}a_{ij}$ $=$ ${}_{(m)}a_{ij}-{}_{(m)}a_{jj}u$. 
 So 
  \[ 
 0 = u_{ij}.f-f=  \sum_{k\geq 1} \sum_{1\leq l\leq k} (-1)^l {}_{(m)}a_{ij}^{k-l} {}_{(m)}a_{jj}^l u^l \binom{k}{l}F_k, 
 \] 
 and 
 $\{ {}_{(m)}a_{ij}^{k-l}u^l : 1\leq l\leq k, k\geq 1 \}$ is again linearly independent over 
 $R_0$. 
 Hence each $F_k=0$ for $k\geq 1$, contradicting the choices of $(i,j)$ and $m$. 
 
 Now suppose $m=1$. 
 If the least pair for $A_{m-1}:=A_{r+1}$ is strictly greater than the least pair for $A_m=A_1$, 
 then this case reduces to the previous case by choosing $U_{ij}$ as above. 
 If the least pair for $A_{m-1}=A_{r+1}$ equals the least pair for $A_m=A_1$, then writing $m-1$ to mean $r+1$
 (since $m=1$), write  
 \[ f = \sum_{k\leq d,k'\leq d'} {}_{(m-1)}a_{ij}^{k} {}_{(m)}a_{ij}^{k'} F_{k,k'}, 
 \]   
 where  $F_{k,k'} \in 
 \mathbb{C}[\{ {}_{(\alpha)}a_{st}: 
 (s,t)\geq (i,j) \mbox{ and if } (s,t)=(i,j), \mbox{ then }\alpha\not=m-1,m
 %
 \} ]$ and $d,d'\geq 1$. 
Let $R_1:=  \mathbb{C}[\{ {}_{(\alpha)}a_{st}: 
 (s,t)\geq (i,j) \mbox{ and if } (s,t)=(i,j), \mbox{ then }\alpha\not=m-1,m  \} ]$. 
If $\epsilon(a_m)=1$ and $\epsilon(a_{m-1})=1$, then 
\[ 
\begin{aligned} 
0 &= u_{ij}.f-f  \\ 
&= \sum_{k\leq d,k'\leq d'}  
			({}_{(m-1)}a_{ij}-{}_{(m-1)}a_{jj}u)^k  ({}_{(m)}a_{ij}+{}_{(m)}a_{ii} u)^{k'} F_{k,k'}-\sum_{k\leq d,k'\leq d'}{}_{(m-1)}a_{ij}^{k} {}_{(m)}a_{ij}^{k'} F_{k,k'} \\ 
			&= \sum_{k\leq d, k'\leq d'} \sum_{\stackrel{1\leq l+l'\leq  k+k' }{l\leq k,l'\leq k'}} \binom{k}{l}\binom{k'}{l'}{}_{(m-1)}a_{ij}^{k-l}{}_{(m)}a_{ij}^{k'-l'}u^{l+l'} (-{}_{(m-1)}a_{jj})^l {}_{(m)}a_{ii}^{l'}F_{k,k'}. \\ 
\end{aligned}
\]  
Since $\{ {}_{(m-1)}a_{ij}^{k-l} {}_{(m)}a_{ij}^{k'-l'}u^{l+l'}:1\leq l+l'\leq k+k', l\leq k, l'\leq k'\}$ is linearly independent over 
$R_1$,  
we have $F_{k,k'}=0$ for all $k+k'\geq 1$, which contradicts the choices of $(i,j)$ and $m=1$. 
Thus $f\in\mathbb{C}[\mathfrak{t}^{\oplus r+1}]$ as claimed.   
Moreover, we have similar arguments when $\epsilon(a_m)=0$ or $\epsilon(a_{m-1})=0$.

Finally, suppose $f$ is a $\mathbb{U}_{\beta}$-invariant polynomial in  $x_{st}$ (and possibly of ${}_{(\alpha)}a_{st}$).  
Without loss of generality, if 
$f(x_{st},{}_{(\alpha)}a_{st})=g(x_{st},{}_{(\alpha)}a_{st}) +h({}_{(\alpha)}a_{ii})$, 
where all the monomials of $g$ are divisible by some $x_{st}$ for some $s$ and $t$ and $h\in \mathbb{C}[\mathfrak{t}^{\oplus r+1}]$, 
then we only consider $g$ by subtracting off $h$ since we have already proved that $h=h({}_{(\alpha)}a_{ii})$ is an invariant polynomial. 
One may check directly that the polynomial in \eqref{eq:affine-dk-quiver-general-n-inv-poly} 
is a $\mathbb{U}_{\beta}$-invariant polynomial.  
Now suppose there exists a polynomial not in $\mathbb{C}[\mathfrak{t}^{\oplus r+1}]$ or not of the form \eqref{eq:affine-dk-quiver-general-n-inv-poly}  
 which is in $\mathbb{C}[F^{\bullet}Rep(Q,\beta)]^{\mathbb{U}_{\beta}}$.
 That is, 
 suppose there exists $f\in \mathbb{C}[F^{\bullet}Rep(Q,\beta)]$ fixed by $\mathbb{U}_{\beta}$, 
 with at least one of its monomials divisible by $x_{st}$ for some $s$ and $t$,  
 which cannot be written as 
 \eqref{eq:affine-dk-quiver-general-n-inv-poly}. 
Let $j\leq n-1$ be the biggest integer which satisfies the following: 
\begin{center}\label{center:unique-choice-of-j-affine-Ar-dynkin}
there exists a $\mathbb{U}_{\beta}$-invariant 
$F\in \mathbb{C}[F^{\bullet}Rep(Q,\beta)]$ such that when $F$ is written in terms of the block standard basis, i.e., 
$F=\sum_{\nu}g_{\nu}({}_{(\alpha)}a_{st})(J_{\nu}|I_{\nu})_{\mathbf{\Phi_{\phi_1}}\mathbf{\Phi_{\phi_2}} \cdots \mathbf{\Phi_{\phi_l}}}$
with each $(J_{\nu}|I_{\nu})_{\mathbf{\Phi_{\phi_1}}\mathbf{\Phi_{\phi_2}} \cdots \mathbf{\Phi_{\phi_l}}}$ 
a standard Young bideterminant in block standard form and each $g_{\nu}\not=0$, 
then there exists a $v$ and a row in $J_v$ where $j$ is not followed by $j+1$. 
Let us label this choice of $j$ as 
$(\dagger)$. 
\end{center}

Let $a_{\phi}$ be the arrow associated to a general representation $A_{\phi}$. 
Writing $ha_{\phi}$ to be the head of the arrow $a_{\phi}$, let 
$U_{j,j+1}$ be the subgroup consisting of matrices of the form 
$u_{j,j+1} = (\I,\ldots, u_{ha_{\phi}},\ldots, \I)$ where $ha_{\phi}$ has diagonal entries $1$, the variable $-u$ in $(j,j+1)$-entry, and $0$ elsewhere.  
Let's write the entries of the product $A_{\alpha}\cdots A_0$ of matrices as $y_{st}$. 
Then for $u_{j,j+1}\in U_{j,j+1}$, 
\[ 
u_{j,j+1}.y_{st} = 
\begin{cases} 
y_{jt}+u y_{j+1,t} &\mbox{ if }\alpha=\phi \mbox{ and } s=j, \\ 
y_{st} 						&\mbox{ otherwise} \\ 
\end{cases} 
\] 
since $A_{\alpha}\cdots A_0$ is a general representation of the quiver path $a_{\alpha}\cdots a_0$. 
To explain further, if the path $a_{\alpha}\cdots a_0$ includes $a_{\phi}$ 
somewhere strictly in the middle of the path, i.e., 
$a_{\alpha}\cdots a_0$ $=$ $a_{\alpha}\cdots a_{\phi}\cdots a_0$, then although $u_{ha_{\phi}}$ 
acts by left multiplication on $A_{\phi}$, 
$u_{ha_{\phi}}$ acts by right (inverse) multiplication on the general representation of the arrow in the path immediately following $a_{\phi}$ (this is the arrow which is immediately to the left of $a_{\phi}$ 
in the concatenation of the arrows $a_{\alpha}\cdots a_0$). 
Thus, the action by $u_{ha_{\phi}}$ is canceled.  
So for any $\alpha$,  
$u$ fixes every minor of the form   
\[ 
(\cdots j\:\: j+1 \cdots | \cdots )_{\mathbf{\Phi_{\phi_1}}\mathbf{\Phi_{\phi_2}} \cdots \mathbf{\Phi_{\phi_l}}}. 
\]

Now let us write 
\[ 
F = \sum_{\nu} g_{\nu}({}_{(\alpha)} a_{st})(J_{\nu}|I_{\nu})_{\mathbf{\Phi_{\phi_1}}\mathbf{\Phi_{\phi_2}} \cdots \mathbf{\Phi_{\phi_l}}}
+ \sum_{\gamma} g_{\gamma}({}_{(\alpha)} a_{st})(J_{\gamma}|I_{\gamma})_{\mathbf{\Phi_{\phi_1}}\mathbf{\Phi_{\phi_2}} \cdots \mathbf{\Phi_{\phi_l}}}, 
\] 
with the following properties: 
\begin{itemize} 
\item the $g_{\nu}$ are nonzero, 
\item there exists at least one row in each $J_{\nu}$ which contains $j$ but not $j+1$,
\item if $j$ appears in any row of $J_{\gamma}$, then so does $j+1$, 
\item the $(J_{\nu}|I_{\nu})_{\mathbf{\Phi_{\phi_1}}\mathbf{\Phi_{\phi_2}} \cdots \mathbf{\Phi_{\phi_l}}}$ and $(J_{\gamma}|I_{\gamma})_{\mathbf{\Phi_{\phi_1}}\mathbf{\Phi_{\phi_2}} \cdots \mathbf{\Phi_{\phi_l}}}$ are unique. 
\end{itemize} 
By Lemma~\ref{lemma:borel-invariants-on-equioriented-acylic-quiver}, $\mathbb{U}_{\beta}$ fixes each row of $(J_{\gamma}|I_{\gamma})_{\mathbf{\Phi_{\phi_1}}\mathbf{\Phi_{\phi_2}} \cdots \mathbf{\Phi_{\phi_l}}}$. Since 
$(J_{\gamma}|I_{\gamma})_{\mathbf{\Phi_{\phi_1}}\mathbf{\Phi_{\phi_2}} \cdots \mathbf{\Phi_{\phi_l}}}$  
is the product of the rows in the bideterminant, 
$\mathbb{U}_{\beta}$ 
fixes 
$(J_{\gamma}|I_{\gamma})_{\mathbf{\Phi_{\phi_1}}\mathbf{\Phi_{\phi_2}} \cdots \mathbf{\Phi_{\phi_l}}}$,   
which in turn fixes   
$\displaystyle{\sum_{\gamma} g_{\gamma}({}_{(\alpha)} a_{st})}  (J_{\gamma}|I_{\gamma})_{\mathbf{\Phi_{\phi_1}}\mathbf{\Phi_{\phi_2}} \cdots \mathbf{\Phi_{\phi_l}}}$.  
Among those rows with identical sequence $\mathbf{\Phi_u}$
in each $J_{\nu}$ in block standard form,   
the only possible occurrences of $j$ and $j+1$ in its rows are as follows:  
\begin{enumerate}[ {[[}1{]]} ]  
\item $j$ is followed by $j+1$,  
\item $j$ is followed by an integer larger than $j+1$,  
\item $j$ ends in a row,   
\item $j+1$ is preceded by an integer smaller than $j$, 
\item $j+1$ starts a row. 
\end{enumerate}
Since $J_{\nu}$ is in block standard form, all rows of type $[[i]]$ 
must occur above all rows of type $[[i+1]]$ within each block.

Since each $(J_{\nu}|I_{\nu})_{\mathbf{\Phi_{\phi_1}}\mathbf{\Phi_{\phi_2}} \cdots \mathbf{\Phi_{\phi_l}}}$ is unique, 
after re-numbering the indices $\nu$, 
let $J_1$ have the greatest number of rows, say $M$, of types $[[2]]$ and $[[3]]$. 
There may be other Young tableaux, say $J_2,\ldots, J_W$ in 
$(J_{\nu}|I_{\nu})_{\mathbf{\Phi_{\phi_1}}\mathbf{\Phi_{\phi_2}} \cdots \mathbf{\Phi_{\phi_l}}}$, 
with $M$ rows of types $[[2]]$ and $[[3]]$, but we ignore them for now. 
We label the sequence 
$\mathbf{\Phi_{\alpha}}$ of integers associated to the rows of 
$(J_{\nu}|I_{\nu})_{\mathbf{\Phi_{\phi_1}}\mathbf{\Phi_{\phi_2}} \cdots \mathbf{\Phi_{\phi_l}}}$  
as 
$\mathbf{\Phi_{\phi_1(\nu)}}$ $\leq$ $\mathbf{\Phi_{\phi_2(\nu)}}$ $\leq$ $\ldots$ $\leq$ $\mathbf{\Phi_{\phi_l(\nu)}}$. 
Let 
\[ 
\begin{aligned} 
U_{j,j+1} := \{ u_{j,j+1} &= (\I,\ldots, u_{ha_{\phi_1(1)}},\ldots, u_{ha_{\phi_l(1)}},\ldots, \I ): 
u_{ha_{\phi_1(1)}} = \ldots= u_{ha_{\phi_l(1)}} \mbox{ is the matrix with } \\ 
&\mbox{ 1 along the diagonal, the same variable $u$ in $(j,j+1)$-entry, and 0 elsewhere}  \}. \\ 
\end{aligned}
\] 
Applying $u_{j,j+1}\in U_{j,j+1}$ to $F$, we see that 
$g_1({}_{(\alpha)}a_{st})(J_1|I_1)_{\mathbf{\Phi_{\phi_1}}\mathbf{\Phi_{\phi_2}} \cdots \mathbf{\Phi_{\phi_l}}}$ gives a term 
\[ 
u^M g_1({}_{(\alpha)}a_{st}) (J_1'|I_1)_{\mathbf{\Phi_{\phi_1}}\mathbf{\Phi_{\phi_2}} \cdots \mathbf{\Phi_{\phi_l}}}, 
\]  
where $J_1'$ is obtained from $J_1$ by replacing each $j$ in rows of type 
$[[2]]$ and $[[3]]$ by $j+1$; 
furthermore, $(J_1'|I_1)_{\mathbf{\Phi_{\phi_1}}\mathbf{\Phi_{\phi_2}} \cdots \mathbf{\Phi_{\phi_l}}}$ is block standard. 
Thus, the tableau $J_1'$ uniquely determines $J_1$ for the following reasons: 
first, all rows of type 
$[[3]]$ have been changed to rows ending with $j+1$; 
such rows must end with $j+1$ in $J_1$ by our choice of $j$. 
Otherwise, to obtain $J_1$, 
we change $j+1$ to $j$ in $M$ rows of $J_1'$ reading from top to bottom 
(while ignoring those rows which contain both $j$ and $j+1$).

Now, in $u_{j,j+1}.F$, any other occurrence of $(J_1'|I_1)_{\mathbf{\Phi_{\phi_1}}\mathbf{\Phi_{\phi_2}} \cdots \mathbf{\Phi_{\phi_l}}}$ is with a coefficient 
$u^k g'({}_{(\alpha)}a_{st})$ 
where $k<M$ since $j$ was carefully chosen such that $j$ is the biggest integer satisfying $(\dagger)$. 
Since $F$ is a polynomial over a field of characteristic $0$, the coefficient of 
$(J_1'|I_1)_{\mathbf{\Phi_{\phi_1}}\mathbf{\Phi_{\phi_2}} \cdots \mathbf{\Phi_{\phi_l}}}$ depends on $u$. 
Thus, $F$ is not fixed by $\mathbb{U}_{\beta}$. 
This shows that $j$ cannot appear in a row of $J_{\nu}$ without $j+1$, which shows that $\mathbb{U}_{\beta}$-invariant polynomials must be of the form \eqref{eq:affine-dk-quiver-general-n-inv-poly}.  
Thus, if some of the terms of a 
$\mathbb{U}_{\beta}$-invariant polynomial $f$ are divisible by $x_{st}$ for some $s$ and $t$, then we write $f$ as 
\[ 
f(x_{st}, {}_{(\alpha)}a_{st}) = g(x_{st}, {}_{(\alpha)}a_{st}) + h({}_{(\alpha)}a_{st}),  
\] where all the terms in $g$ are divisible by $x_{st}$ for some $s$ and $t$ and $h$ is a polynomial in ${}_{(\alpha)}a_{st}$. 
By the first part of this proof, $h\in \mathbb{C}[\mathfrak{t}^{\oplus r+1}]$ while $g$ must be of the form \eqref{eq:affine-dk-quiver-general-n-inv-poly}. 
It is immediate by the first part of the proof that if none of the terms in a 
$\mathbb{U}_{\beta}$-invariant polynomial are divisible by $x_{st}$ for all $1\leq s\leq n$ and $1\leq t\leq m$, 
then the polynomial is in $\mathbb{C}[\mathfrak{t}^{\oplus r+1}]$. 
This concludes our proof. 
\end{proof}

\begin{corollary}\label{corollary:affine-quiver-no-framing} 
For an affine $\widetilde{A}_r$-quiver with no framing and 
$\beta=(n,\ldots, n)$ with the complete standard filtration of vector spaces at each vertex of the quiver, we have 
$\mathbb{C}[\mathfrak{b}^{\oplus r+1}]^{\mathbb{U}_{\beta}}\cong \mathbb{C}[\mathfrak{t}^{\oplus r+1}]$. 
\end{corollary}

\begin{example} 
Consider the framed $1$-Jordan quiver:  
$\xymatrix@-1pc{ \stackrel{1'}{\circ} \ar[rr]^{a_0} & & \stackrel{1}{\bullet} \ar@(ru,rd)^{a_1}}$.   
Let $\beta=(2,2)$ and let $F^{\bullet}$ be the complete standard filtration of vector spaces. Then 
$F^{\bullet}Rep(Q,\beta)=\mathfrak{b}_2\oplus M_{2\times 2}$. 
Let $(A_1,A_0)\in \mathfrak{b}_2\oplus M_{2\times 2}$, where 
\[ 
A_1 = \bordermatrix{
& & \cr  
& a_{11}& a_{12}\cr  
& 0     & a_{22}\cr  
}, \hspace{4mm} 
A_0 = \bordermatrix{ 
& & \cr  
&x_{11} &x_{12} \cr  
&x_{21} &x_{22} \cr  
},  
\]    
let $u\in U\subseteq B$ act on $(A_1,A_0)$ via $(u A_1 u^{-1},u A_0)$.
By Theorem~\ref{theorem:invariants-of-framed-affine-Ar-quiver},  
\[ 
\begin{aligned}
\mathbb{C}[\mathfrak{b}_2\oplus M_{2\times 2}]^{U} &= \mathbb{C}[a_{11},a_{22},x_{21},x_{22},\det(A_0), 
(A_1A_0 )_{2,1}, (A_1 A_0)_{2,2}, \det(A_1 A_0)] \\ 
&= \mathbb{C}[a_{11},a_{22},x_{21},x_{22},\det(A_0)].  \\ 
\end{aligned}
\]
Note that the polynomial $f(x_{st}, a_{st}) = a_{22}x_{22}$ can be written as $f=(2|2)_{A_1A_0}$ or as 
\begin{equation}\label{equation:block-standard-form-with-Cartan-coefficients-alternative-bitableau}
f = a_{22}(2|2)_{A_0} = (2|2)_{A_1} (2|2)_{A_0} = 
\begin{matrix} 
(2|2)_{A_0} \\ 
(2|2)_{A_1}
\end{matrix}, 
\end{equation} 
but \eqref{equation:block-standard-form-with-Cartan-coefficients-alternative-bitableau} is not in the form \eqref{eq:bideterminant-block-standard-form-generic-notation} since a general representation of the quiver path in the second row does not begin at the framed vertex. 
\end{example}

\section{Quivers with at most two pathways between any two vertices}\label{section:semi-invariants-quivers-at-most-two-pathways}

%
We refer to Definition~\ref{definition:pathway-of-quiver} for the definition of pathways.   
Recall our basic assumption: given any quiver $Q$, let $F^{\bullet}$ be the complete standard filtration of vector spaces at the nonframed vertices of $Q$ and let $\mathbb{U}_{\beta}$ be a product of the maximal unipotent subgroup of the standard Borel acting on $F^{\bullet}Rep(Q,\beta)$ as a change-of-basis. 
   
\begin{theorem}\label{theorem:two-paths-max-quiver-semi-invariants}   
Let $Q$ be a quiver and let $\beta=(n,\ldots, n)$, a dimension vector.   
Then  
$Q$ is a nonframed quiver with at most two distinct pathways between any two vertices if and only if  $\mathbb{C}[F^{\bullet}Rep(Q,\beta)]^{\mathbb{U}_{\beta}}=\mathbb{C}[\mathfrak{t}^{\oplus Q_1}]$, where $Q_1$ is the set of arrows whose tail and head are nonframed vertices. 
\end{theorem}  

Theorem~\ref{theorem:two-paths-max-quiver-semi-invariants} has a number of important consequences,  
with one being that if $\mathbb{U}_{\beta}$-invariants for filtered quiver varieties  
only come from diagonal blocks (i.e., the semisimple part), then $Q$   
has at most two pathways between any two vertices.  

\begin{theorem}\label{theorem:two-paths-max-quiver-semi-invariants-framed} 
Let $Q=(Q_0,Q_1)$ be a framed quiver with at most two distinct pathways between any two vertices, where $Q$ has one arrow from the framed vertex to any of the nonframed vertices. 
Let 
$\beta=(n,\ldots,n,m)\in \mathbb{Z}^{Q_0}$ be a dimension vector, where the framed vertex has dimension $m$.   
Then 
$\mathbb{C}[F^{\bullet}Rep(Q,\beta)]^{\mathbb{U}_{\beta}}$ 
$=$ 
$\mathbb{C}[\mathfrak{t}^{\oplus Q_1-1}]\otimes_{\mathbb{C}} \mathbb{C}[\{ f\}]$ 
where $f$ is of the form \eqref{eq:affine-dk-quiver-general-n-inv-poly}. 
\end{theorem} 

In Section~\ref{section:comment-regarding-at-most-two-paths}, 
we will give a reason for the condition for the quiver to have at most two distinct pathways between any two vertices.  
Furthermore, 
this implies that Domokos-Zubkov's technique in Section~\ref{subsection:Domokos-Zubkov-technique} is applicable in the (nontrivially) filtered setting only when a quiver has at most two pathways between any two vertices.

We will first prove Theorem~\ref{theorem:two-paths-max-quiver-semi-invariants}.  
 
\begin{proof} 
First, we will prove from right to left. Assume $\mathbb{C}[F^{\bullet}Rep(Q,\beta)]^{\mathbb{U}_{\beta}}=\mathbb{C}[\mathfrak{t}^{\oplus Q_1}]$, where $Q_1$ is the set of arrows whose tail and head are nonframed vertices. 
If $Q$ has a framing at some vertex, say at $1$: $\xymatrix@-1pc{ \stackrel{1'}{\circ}\ar[rr]^{a_0} & & \stackrel{1}{\bullet}}$, 
then consider a representation of vertices $1'$ and $1$ and the arrow $a_0$:   
$\xymatrix@-1pc{\stackrel{\mathbb{C}^m}{\circ}\ar[rr]^{A_0} & & \stackrel{\mathbb{C}^n}{\bullet} }$   
where $A_0\in M_{n\times m}$ is a general matrix. 
Consider the maximal unipotent subgroup $U\subseteq GL_n(\mathbb{C})$ acting on $A_0$ 
via left multiplication: 
$u.A_0 =uA_0$. 
By direct calculation, it is clear that for each $1\leq k\leq m$,  
$(n,k)$-entry of $A_0$ is an invariant polynomial. 
By assumption, invariant polynomials only come from general representation of arrows whose tail and head are nonframed vertices. This implies that $Q$ cannot be a framed quiver. 

Next, we will prove that $Q$ has at most two distinct pathways at each vertex. 
For a contradiction, suppose $Q$ has $3$ or more pathways at some vertices. 
There are two cases to consider: 
\begin{enumerate}[ {[}1{]} ]  
\item\label{item:3pathways-in-one-vertex} $Q$ has $3$ or more pathways from a vertex to itself, e.g., $\hspace{6mm}\xymatrix@-1pc{ \ar@(lu,ld) \bullet \ar@(ru,rd) }$, 
\item\label{item:3pathways-from-one-vertex-to-another} $Q$ has $3$ or more pathways from vertex $i$ to vertex $j$, where $i\not=j$, e.g., $\xymatrix@-1pc{ \stackrel{i}{\bullet} \ar@/^/[rr] \ar[rr] \ar@/_/[rr] & & \stackrel{j}{\bullet} }$
or $\xymatrix@-1pc{ \stackrel{i}{\bullet} \ar@/^/[rr] \ar@/_/[rr] & & \stackrel{j}{\bullet} \ar@(ru,rd)}\hspace{6mm}$
or $\xymatrix@-1pc{ \stackrel{i}{\bullet} \ar@/^/[rr] \ar@/_/[rr] & & \bullet \ar@/^/[rr] \ar@/_/[rr] & & \stackrel{j}{\bullet}  }$ 
or $\xymatrix@-1pc{ \stackrel{i}{\bullet} \ar[rr] & & \bullet \ar@(rd,ld) \ar@/^/[rr] \ar@/_/[rr] & & \stackrel{j}{\bullet}  }$.  
\end{enumerate}    
Consider Case~[\ref{item:3pathways-in-one-vertex}].  
Label the vertex as $1$ and label the pathways from vertex $1$ to itself as   
\[  
a_1 = p_{i_1} \cdots p_{i_{a_1}}, \hspace{4mm}
a_2 = p_{j_1} \cdots p_{j_{a_2}}, \hspace{2mm} \ldots,   
a_m = p_{k_1} \cdots p_{k_{a_m}},  
\]   
where $m\geq 2$. 
 Write a general representation of $a_1,\ldots, a_m$ as 
 $A_1,\ldots, A_m$, each of which is in $\mathfrak{b}$ with polynomial entries (as long as $a_l$ is not the trivial path).   
 Without loss of generality,   
 suppose $A_1=({}_{(1)}a_{ij})$ and $A_2 = ({}_{(2)}a_{ij})$   
 are general representations corresponding to a nontrivial pathway and consider the polynomial   
 \[ f(A_1,\ldots, A_m)=({}_{(1)}a_{11} - {}_{(1)}a_{22}){}_{(2)}a_{12}  
                    -({}_{(2)}a_{11} - {}_{(2)}a_{22}){}_{(1)}a_{12}. 
                    \]     
  For $u\in U\cong U\times \I^{Q_0-1}\subseteq \mathbb{U}_{\beta}$ where $u.(A_1,\ldots, A_m)=(u A_1 u^{-1},\ldots, u A_m u^{-1} )$,  
  the coordinate variable 
  ${}_{(l)}a_{12}\mapsto {}_{(l)}a_{12}+u_{12}( {}_{(l)}a_{22}-{}_{(l)}a_{11}  )$  
  under the group action.   
  So   
  \[   
  \begin{aligned}  
  u.&f(A_1,\ldots, A_m)   \\  
   &= ({}_{(1)}a_{11} - {}_{(1)}a_{22})({}_{(2)}a_{12}+u_{12}({}_{(2)} a_{22}-{}_{(2)}a_{11} ))
                    -({}_{(2)}a_{11} - {}_{(2)}a_{22})({}_{(1)}a_{12}+u_{12}({}_{(1)} a_{22}-{}_{(1)}a_{11} )) 
                    																										\\ 
 &= ({}_{(1)}a_{11} - {}_{(1)}a_{22}){}_{(2)}a_{12}   -({}_{(2)}a_{11} - {}_{(2)}a_{22}){}_{(1)}a_{12}  \\ 
                    &\hspace{4mm}+ u_{12}({}_{(1)}a_{11} - {}_{(1)}a_{22})({}_{(2)} a_{22}-{}_{(2)}a_{11})
                    -u_{12}({}_{(2)}a_{11} - {}_{(2)}a_{22})({}_{(1)} a_{22}-{}_{(1)}a_{11})  \\
                    &= f(A_1,\ldots, A_m). \\ 
  \end{aligned} 
  \] 
This implies $\mathbb{C}[F^{\bullet}Rep(Q,\beta)]^{\mathbb{U}_{\beta}} \not= \mathbb{C}[\mathfrak{t}^{\oplus Q_1}]$, which is a contradiction.

Now consider Case~[\ref{item:3pathways-from-one-vertex-to-another}]. 
Suppose 
\[ 
a_1 = p_{i_1} \cdots p_{i_{a_1}}, \hspace{4mm}
a_2 = p_{j_1} \cdots p_{j_{a_2}}, \hspace{2mm} \ldots,  
a_k = p_{l_1} \cdots p_{l_{a_k}}  
\] 
are pathways from vertex $i$ to vertex $j$, where $i\not=j$ and $k\geq 3$. 
Write a general representation of the pathways $a_1,\ldots, a_k$ as 
$A_1=({}_{(1)}a_{ij}),\ldots, A_k=({}_{(k)}a_{ij})\in \mathfrak{b}$. 
Consider the $\mathbb{U}_{\beta}$-action on $F^{\bullet}Rep(Q,\beta)$. 
In particular, consider $\I^{Q_0-2}\times U^2$ acting locally on 
$\mathfrak{b}^{\oplus k}$ via 
\[ 
(\I,\ldots, \I, u,v).(A_1,\ldots, A_k) = (u A_1v^{-1},\ldots, u A_k v^{-1} ). 
\] 
Consider the polynomial 
\[ 
\begin{aligned} 
f(A_1,\ldots, A_k) = ({}_{(1)} a_{11} {}_{(2)} a_{22} &- {}_{(1)} a_{22} {}_{(2)}a_{11} ) {}_{(3)} a_{12}   
+ ({}_{(3)} a_{11} {}_{(1)} a_{22} - {}_{(3)} a_{22} {}_{(1)}a_{11} ) {}_{(2)} a_{12}    \\  
&+ ({}_{(2)} a_{11} {}_{(3)} a_{22} - {}_{(2)} a_{22} {}_{(3)}a_{11} ) {}_{(1)} a_{12}.   \\ 
\end{aligned} 
\] 
Then 
\[ 
\begin{aligned} 
(\I,\ldots, \I,u,v).f(A_1,\ldots, A_k) 
&= ({}_{(1)} a_{11} {}_{(2)} a_{22} - {}_{(1)} a_{22} {}_{(2)}a_{11} ) ({}_{(3)} a_{12} -{}_{(3)}a_{22}u_{12}+{}_{(3)}a_{11}v_{12} )  \\ 
&\:\:+ ({}_{(3)} a_{11} {}_{(1)} a_{22} - {}_{(3)} a_{22} {}_{(1)}a_{11} ) ({}_{(2)} a_{12} -{}_{(2)}a_{22}u_{12}+{}_{(2)}a_{11}v_{12} )   \\  
&\:\:+ ({}_{(2)} a_{11} {}_{(3)} a_{22} - {}_{(2)} a_{22} {}_{(3)}a_{11} ) ({}_{(1)} a_{12} -{}_{(1)}a_{22}u_{12}+{}_{(1)}a_{11}v_{12} )   \\ 
&= f(A_1,\ldots, A_k) 
+  ({}_{(1)} a_{11} {}_{(2)} a_{22} - {}_{(1)} a_{22} {}_{(2)}a_{11} ) ( 
 -{}_{(3)}a_{22}u_{12}+{}_{(3)}a_{11}v_{12} )  \\ 
&\:\:+ ({}_{(3)} a_{11} {}_{(1)} a_{22} - {}_{(3)} a_{22} {}_{(1)}a_{11} ) (
 -{}_{(2)}a_{22}u_{12}+{}_{(2)}a_{11}v_{12} )   \\  
&\:\:+ ({}_{(2)} a_{11} {}_{(3)} a_{22} - {}_{(2)} a_{22} {}_{(3)}a_{11} ) ( 
 -{}_{(1)}a_{22}u_{12}+{}_{(1)}a_{11}v_{12} )  
 = f(A_1,\ldots, A_k). \\  
\end{aligned}
\] 
This also contradicts that $\mathbb{C}[F^{\bullet}Rep(Q,\beta)]^{\mathbb{U}_{\beta}}=\mathbb{C}[\mathfrak{t}^{\oplus Q_1}]$. 
Thus, $Q$ is a quiver with at most two pathways between any two vertices.

Now suppose $Q$ is a nonframed quiver with at most two distinct pathways between any two vertices.   
The heart of this proof is analogous to the proof of Theorem~\ref{theorem:filtered-ADE-Dynkin-quiver}:  
we first define a notion of total ordering on pairs and 
then choose the least pair.    
Since $Q$ is nonframed with at most two pathways between any two vertices, we list all possible local models of arrows of $Q$ at a vertex.   
Writing a polynomial similar as in the proof of Theorem~\ref{theorem:filtered-ADE-Dynkin-quiver},  
we carefully choose a subgroup of $\mathbb{U}_{\beta}$ and show that the invariant polynomial must only come from diagonal coordinates of each general matrix in the filtered representation space. 
We now prove the statement.

We label the arrows of $Q$ as $a_1$, $a_2$, $\ldots$, $a_{Q_1}$. 
It is clear that $\mathbb{C}[\mathfrak{t}^{\oplus Q_1}]\subseteq   \mathbb{C}[F^{\bullet}Rep(Q,\beta)]^{\mathbb{U}_{\beta}}$ so we will prove the other inclusion.  
Consider a general representation of $F^{\bullet}Rep(Q,\beta)$, which consists of a product of matrices. We define total ordering $\leq$ on pairs $(i,j)$, where $1\leq i\leq j\leq n$, by defining 
 $(i,j)\leq (i',j')$ if either  
 \begin{itemize}  
 \item $i < i'$ or  
 \item $i=i'$ and $j> j'$. 
 \end{itemize} 
 Consider $f\in \mathbb{C}[F^{\bullet}Rep(Q,\beta)]^{\mathbb{U}_{\beta}}$.   
  For each $(i,j)$, we can write  
 \begin{equation}\label{eq:aij-invariant-function-borel-at-most-two-pathways}
 f = \sum_{|K|\leq d} a_{ij}^K f_{ij,K}, \mbox{ where } f_{ij,K}\in \mathbb{C}[\{ {}_{(\alpha)}a_{st}: (s,t)\not=(i,j)\} ], 
 a_{ij}^K := \prod_{\alpha\in Q_1} {}_{(\alpha)}a_{ij}^{k_{\alpha}}, 
 \mbox{ and } 
 |K|=\sum_{\alpha=1}^{Q_1} k_{\alpha}.
 \end{equation}  
 Fix the least pair (under $\leq$) $(i,j)$ with $i< j$ for which there exists $K\not=(0,\ldots, 0)$ with $f_{ij,K}\not=0$; 
 we will continue to denote it by $(i,j)$. 
 If no such $(i,j)$ exists, then $f\in \mathbb{C}[{}_{(\alpha)}a_{ii}]$ and we are done. 
 Let $K=(k_1,\ldots, k_{Q_1})$. 
 Let $m\geq 1$ be the least integer satisfying the following: 
for all $p<m$, if some component $k_p$ in $K$ is strictly greater than $0$, then $f_{ij,K}=0$. 
%
%
 Relabel the head of the arrow corresponding to a general representation $A_m$ of arrow $a_m$
 as vertex $m$.  

 Let $U_{ij}$ be the subgroup of matrices of the form 
 $u_{ij} := (\I,\ldots, \I, \widehat{u}_m,\I, \ldots, \I)$, where $\I$ is the $n\times n$  
 identity matrix and  
 $\widehat{u}_m$ is the matrix with $1$ along the diagonal, the variable $u$ in the $(i,j)$-entry, and $0$ elsewhere. 
 Let $u_{ij}^{-1}:=(\I,\ldots, \I, \widehat{u}_m^{-1},\I, \ldots, \I)$. 
 Then since $u_{ij}^{-1}$ acts on $f$ via 
 \[ 
 \begin{aligned}
 u_{ij}^{-1}.f(A_1,\ldots, A_{Q_1}) &= 
 f(u_{ij}.(A_1,\ldots, A_{Q_1}))  \\ 
 &= 
 \begin{cases}  
 f(A_1, 
  \ldots,  
  \widehat{u}_{m}  A_{m'}     , 
  \ldots,  
  A_{Q_1}) &\mbox{ whenever vertex $m$ is a sink of arrow $a_{m'}$}, \\ 
   f(A_1, 
  \ldots,  
    A_{m'} \widehat{u}_{m}^{-1}    , 
  \ldots,  
  A_{Q_1}) &\mbox{ whenever vertex $m$ is a source of arrow $a_{m'}$}, \\ 
  f(A_1,  
  \ldots,  
  \widehat{u}_{m}  A_{m'}  \widehat{u}_{m}^{-1}, 
  \ldots,   
  A_{Q_1}) &\mbox{ whenever arrow $a_{m'}$ is a loop at vertex $m$}, \\ 
 \end{cases}  \\ 
 \end{aligned}
 \] 
 \begin{equation}\label{eq:ar-dynkin-Uij-action} 
  u_{ij}^{-1}.{}_{(\alpha)}a_{st}= 
 \begin{cases} 
 {}_{(m')}a_{ij} +{}_{(m')}a_{jj}u &\mbox{ if } \alpha=m', (s,t)=(i,j), \mbox{ and } m \mbox{ is a sink to } a_{m'}, \\ 
 {}_{(m')}a_{ij} -{}_{(m')}a_{ii}u &\mbox{ if } \alpha=m', (s,t)=(i,j), \mbox{ and } m \mbox{ is a source to } a_{m'}, \\
 {}_{(m')}a_{ij} +({}_{(m')}a_{jj}- {}_{(m')}a_{ii})u &\mbox{ if } \alpha=m', (s,t)=(i,j), \mbox{ and } a_{m'} \mbox{ is a loop at } m, \\ 
 {}_{(\alpha)}a_{st} &\mbox{ if }s>i \mbox{ or }s=i \mbox{ and }t<j.  \\ 
 \end{cases}  
 \end{equation}

Now, locally at vertex $m$, $Q$ looks like one of the following local models:     
\begin{multicols}{2}
 \begin{enumerate} 
 \item\label{item:twopathways-case1} $\xymatrix@-1pc{ \stackrel{m}{\bullet} \ar@(ru,rd)^{a_m}}$
 \item\label{item:twopathways-case2} $\xymatrix@-1pc{ \stackrel{m}{\bullet}\ar@/^/[rr]^{a_{m'}} & & \ar@/^/[ll]^{a_m} \stackrel{v_1}{\bullet} }$
 \item\label{item:twopathways-case3} $\xymatrix@-1pc{ 
 \stackrel{v_1}{\bullet} \ar[rr]^{a_m} & & \stackrel{m}{\bullet} & & \stackrel{v_l}{\bullet}\ar[ll]_{a_l} \\ 
 & & & & \\
 & \stackrel{v_2}{\bullet} \ar[uur]^{a_2} & \cdots & \stackrel{v_{l-1}}{\bullet} \ar[luu]_{a_{l-1}} & \\  
 }$
%
 \item\label{item:twopathways-case4}   $\xymatrix@-1pc{ 
 \stackrel{v_1}{\bullet} \ar@/^/[rr]^{a_m} \ar@/_/[rr]_{c_1}& & \stackrel{m}{\bullet} & & \stackrel{v_l}{\bullet}\ar@/^/[ll]^{c_l} \ar@/_/[ll]_{a_{l}} \\ 
 & & & & \\
 & \stackrel{v_2}{\bullet} \ar@/^/[uur]^{a_2} \ar@/_/[uur]_{c_2} & \cdots & \stackrel{v_{l-1}}{\bullet} \ar@/^/[luu]^{ } \ar@/_/[luu]_{ }& \\  
 }$ 
%
 \item\label{item:twopathways-case5}   $\xymatrix@-1pc{ 
 \stackrel{v_1}{\bullet} \ar[ddrr]^{a_m} & & & & \stackrel{w_{l'}}{\bullet} \\ 
 \stackrel{v_2}{\bullet} \ar[drr]_{a_2}  & & & &  \\ 
 \vdots & & \stackrel{m}{\bullet} \ar[drr]^{c_{2}} \ar[ddrr]_{c_{1}} \ar[uurr]^{c_{l'}} & & \vdots \\ 
 & & &  & \stackrel{w_{2}}{\bullet} \\  
 \stackrel{v_{l}}{\bullet} \ar[uurr]_{a_{l}} & & & &\stackrel{w_{1}}{\bullet} \\ 
 }$
 \item\label{item:twopathways-case6}  $\xymatrix@-1pc{
 \stackrel{v_1}{\bullet} & & \stackrel{m}{\bullet} \ar[lld]^{a_2} \ar[ll]_{a_1} \ar[dd]^{a_l} \ar@/^/[rr]^{c_1} & & \ar@/^/[ll]^{a_{m}}  \stackrel{w}{\bullet} \\
  \stackrel{v_2}{\bullet}  & \ddots & & & \\ 
 &  &\stackrel{v_l}{\bullet} & & \\
 }$ 
 \item\label{item:twopathways-case7}  $\xymatrix@-1pc{
  \stackrel{v_1}{\bullet} \ar[rr]^{a_1}  & & \stackrel{m}{\bullet} \ar@/^/[rr]^{c_1} & & \ar@/^/[ll]^{a_{m}}  \stackrel{w}{\bullet} \\
  \stackrel{v_2}{\bullet}  \ar[rru]^{a_2}  & \ddots & & & \\ 
 &  &\stackrel{v_l}{\bullet} \ar[uu]_{a_l} & & \\
 }$
 \item\label{item:twopathways-case8}     $\xymatrix@-1pc{ 
 \stackrel{v_1}{\bullet} \ar[rr]^{a_1} & & \stackrel{m}{\bullet} \ar@(lu,ru)^{a_m} & & \stackrel{v_l}{\bullet}\ar[ll]_{a_l} \\ 
 & & & & \\
 & \stackrel{v_2}{\bullet} \ar[uur]^{a_2} & \cdots & \stackrel{v_{l-1}}{\bullet} \ar[luu]_{a_{l-1}} & \\  
 }$  
 \item\label{item:twopathways-case9}  $\xymatrix@-1pc{ 
 \stackrel{v_1}{\bullet} & & \ar[ll]_{a_1} \ar[ddl]_{a_2} \stackrel{m}{\bullet} \ar@(lu,ru)^{a_m}  \ar[ddr]^{a_{l-1}}  \ar[rr]^{a_l} & & \stackrel{v_l}{\bullet} \\ 
 & & & & \\
 & \stackrel{v_2}{\bullet} & \cdots & \stackrel{v_{l-1}}{\bullet} & \\  
 }$  
 \item\label{item:twopathways-case10}   $\xymatrix@-1pc{ 
 \stackrel{v_1}{\bullet} \ar[drr]^{a_1}   & & & & \stackrel{w_{l'}}{\bullet} \\ 
 \stackrel{v_2}{\bullet} \ar[rr]_{a_2}  & & \stackrel{m}{\bullet} \ar@(lu,ru)^{a_m} \ar[drr]^{c_{2}} \ar[ddrr]_{c_{1}} \ar[urr]^{c_{l'}} & & \vdots \\ 
 \vdots & & &  & \stackrel{w_{2}}{\bullet} \\  
 \stackrel{v_{l}}{\bullet} \ar[uurr]_{a_{l}} & & & &\stackrel{w_{1}}{\bullet} \\ 
 }$   
%
%
 \item\label{item:twopathways-case11}   $\xymatrix@-1pc{ 
 \stackrel{v_1}{\bullet} \ar[rr]^{a_m}   & &   \stackrel{m}{\bullet} \ar@/^/[ddr]^{c_l} \ar@/_/[ddr]_{s_l} \ar@/^/[rr]^{c_l} \ar@/_/[rr]_{s_{1}}& & \stackrel{w_1}{\bullet}\\ 
 \vdots  &  & & & .^{.^.}   \\
 & \stackrel{v_{l'}}{\bullet}\ar[uur]^{a_{l'}}  &   & \stackrel{w_{l}}{\bullet} & \\  
 }$ 
 \item\label{item:twopathways-case12}  $\xymatrix@-1pc{ 
 \stackrel{v_1}{\bullet} \ar[rr]^{a_m}  & & \stackrel{m}{\bullet} & & \stackrel{w_1}{\bullet} \ar@/^/[ll]^{s_1} \ar@/_/[ll]_{c_{1}} \\ 
 \vdots & & & & .^{.^.}\\
 & \stackrel{v_{l'}}{\bullet} \ar[uur]^{a_{l'}}  &  & \stackrel{w_{l}}{\bullet} \ar@/^/[luu]^{s_l} \ar@/_/[luu]_{c_l}& \\  
 }$  
 \item\label{item:twopathways-case13}  $\xymatrix@-1pc{ 
 \stackrel{v_1}{\bullet}   & &  \ar[ll]_{a_1} \ar[ddl]_{a_{l'}}  \stackrel{m}{\bullet} & & \stackrel{w_1}{\bullet} \ar@/^/[ll]^{s_1} \ar@/_/[ll]_{a_{m}} \\ 
 \vdots & & & & .^{.^.}\\
 & \stackrel{v_{l'}}{\bullet}  &  & \stackrel{w_{l}}{\bullet} \ar@/^/[luu]^{s_l} \ar@/_/[luu]_{c_l}& \\  
 }$   
 \item\label{item:twopathways-case14}    $\xymatrix@-1pc{ 
 \stackrel{v_1}{\bullet} \ar[ddrr]^{a_m} & &  & & \stackrel{w_1}{\bullet} \\ 
  \vdots    & &  & &  \vdots \\  
 \stackrel{v_{l'}}{\bullet}\ar[rr]^{a_{l'}} & & \stackrel{m}{\bullet} \ar[ddl]_{r_1} \ar[ddr]^{r_{l''}} \ar@/^/[rruu]^{c_1} \ar@/_/[rruu]_{s_1}  \ar@/^/[rr]^{c_l} \ar@/_/[rr]_{s_l}& & \stackrel{w_l}{\bullet} \\ 
 & &  & &  \\ 
  & \stackrel{x_1}{\bullet} &\cdots  & \stackrel{x_{l''}}{\bullet} &   \\ 
 }$       
 \item\label{item:twopathways-case15}    $\xymatrix@-1pc{ 
 \stackrel{v_1}{\bullet} \ar[ddrr]^{a_m} & &  & & \stackrel{w_1}{\bullet} \ar@/^/[lldd]^{s_1} \ar@/_/[lldd]_{c_1}  \\ 
  \vdots    & &  & &  \vdots \\  
 \stackrel{v_{l'}}{\bullet}\ar[rr]^{a_{l'}} & & \stackrel{m}{\bullet} \ar[ddl]_{r_1} \ar[ddr]^{r_{l''}} & & \stackrel{w_l}{\bullet} \ar@/^/[ll]^{s_l} \ar@/_/[ll]_{c_l} \\ 
 & &  & &  \\ 
  & \stackrel{x_1}{\bullet} &\cdots  & \stackrel{x_{l''}}{\bullet} &   \\ 
 }$ 
 \end{enumerate}    
\end{multicols}   
Cases $\ref{item:twopathways-case1}$ and $\ref{item:twopathways-case2}$  
have been covered by Corollary~\ref{corollary:affine-quiver-no-framing}. 
So consider when $m$ is a sink, i.e., consider Cases $\ref{item:twopathways-case3}$, $\ref{item:twopathways-case4}$, and $\ref{item:twopathways-case12}$. 
Relabel the arrows in Cases $\ref{item:twopathways-case3}$, $\ref{item:twopathways-case4}$, and $\ref{item:twopathways-case12}$  
as $a_1,\ldots, a_l$ with 
$A_1,\ldots, A_l$ as general representations of  
$a_1,\ldots, a_l$, respectively, where $A_{\alpha}=({}_{(\alpha)}a_{st})$ is a general upper triangular matrix. 
Write 
\[ 
f=\sum_{|K'|\leq d} \hspace{2mm}\prod_{\alpha=1}^{l} {}_{(\alpha)}a_{ij}^{k_{\alpha}} f_{ij,K'}, 
\] 
where $K'=(k_1,\ldots, k_l)$, 
$|K'|= \displaystyle{\sum_{\alpha=1}^{l} k_{\alpha}}$, 
 and 
$f_{ij,K'}\in \mathbb{C}[\{ {}_{(\alpha)}a_{st}:(s,t)\not=(i,j) \mbox{ and }\alpha\not\in \{1,\ldots,l \} \}]$. 
Define 
$R_2:= \mathbb{C}[\{ {}_{(\alpha)}a_{st}:(s,t)\not=(i,j) \mbox{ and }\alpha\not\in \{1,\ldots,l \} \}]$. Then 
\[ 
\begin{aligned} 
0 &= u_{ij}^{-1}.f-f 
= \sum_{|K'|\leq d} \hspace{2mm}\prod_{\alpha=1}^{l} ({}_{(\alpha)}a_{ij}+{}_{(\alpha)}a_{jj}u )^{k_{\alpha}}f_{ij,K'} 
- \sum_{|K'|\leq d} \hspace{2mm}\prod_{\alpha=1}^{l} ({}_{(\alpha)}a_{ij} )^{k_{\alpha}}f_{ij,K'} \\ 
&= \sum_{1\leq |K'|\leq d} \hspace{2mm} \sum_{r_{\alpha}\leq k_{\alpha}} 
  \binom{k_1}{r_1} \binom{k_2}{r_2} \cdots \binom{k_{l}}{r_{l}} 
  \prod_{\alpha=1}^{l} {}_{(\alpha)}a_{ij}^{k_{\alpha}-r_{\alpha}} 
  \hspace{1mm} 
  u^{|R|} 
  \prod_{\alpha=1}^{l} ({}_{(\alpha)}a_{jj})^{r_{\alpha}} f_{ij,K'}, \\ 
\end{aligned} 
\] 
where $|R|=\displaystyle{\sum_{\alpha=1}^{l}r_{\alpha}}$,  
and we see that  
$\{\displaystyle{ \prod_{\alpha=1}^{l}{}_{(\alpha)}a_{ij}^{k_{\alpha}-r_{\alpha}}u^{|R|}}:
r_{\alpha}\leq k_{\alpha} \mbox{ for all }1\leq \alpha\leq l \}$ 
is linearly independent over $R_2$. 

Next, consider Cases 
$\ref{item:twopathways-case5}$, 
$\ref{item:twopathways-case6}$, 
$\ref{item:twopathways-case7}$, 
$\ref{item:twopathways-case11}$, 
$\ref{item:twopathways-case13}$, 
$\ref{item:twopathways-case14}$, 
and 
$\ref{item:twopathways-case15}$. 
Relabel the arrows in the following way: 
write $a_1,\ldots, a_l$ if $ta_i=m$, where $1\leq i\leq l$, and 
write $c_1,\ldots, c_k$ if $hc_j=m$, where $1\leq j\leq k$. 
Write $A_{\alpha}=({}_{(\alpha)}a_{st})$ as a general representation of $a_{\alpha}$ and 
$C_{\gamma}=({}_{(\gamma)}c_{st})$ as a general representation of $c_{\gamma}$. 
Let us write 
\[ 
f = \sum_{|K'|+|\Gamma| \leq d} \hspace{2mm}
\prod_{\alpha=1}^{l} {}_{(\alpha)}a_{ij}^{k_{\alpha}} 
\prod_{\gamma=1}^{k} {}_{(\gamma)}c_{ij}^{\mu_{\gamma}}
f_{ij,K', \Gamma}, 
\] 
where 
$K'=(k_1,\ldots, k_l)$,  
$\Gamma = (\mu_1,\ldots, \mu_k )$,  
$|K'| = \displaystyle{\sum_{\alpha=1}^{l}k_{\alpha}}$, 
$|\Gamma| = \displaystyle{\sum_{\gamma=1}^{k} \mu_{\gamma}}$, 
and 
$f_{ij,K',\Gamma}\in \mathbb{C}[\{ {}_{(\alpha)}a_{st}, {}_{(\gamma)}c_{st}:$ 
$(s,t)\not=(i,j)$ 
$\mbox{ and }$ 
$\alpha\not\in \{ 1,\ldots, l\}$
$\mbox{ and }$ 
$\gamma\not\in \{1,\ldots, k \}
\}]$. 
Define 
$R_3:= \mathbb{C}[\{ {}_{(\alpha)}a_{st}, {}_{(\gamma)}c_{st}: (s,t)\not=(i,j)$ 
$\mbox{ and }$ 
$\alpha\not\in \{ 1,\ldots, l\}$
$\mbox{ and }$ 
$\gamma\not\in \{1,\ldots, k \}
\}]$.  
Then  
\[  
\begin{aligned}  
0 &= u_{ij}^{-1}.f-f  
= \sum_{|K'|+|\Gamma| \leq d} \hspace{2mm} 
\prod_{\alpha=1}^{l} ({}_{(\alpha)}a_{ij}+{}_{(\alpha)}a_{jj}u )^{k_{\alpha}} 
\prod_{\gamma=1}^{k} ({}_{(\gamma)}c_{ij}-{}_{(\gamma)}c_{ii}u )^{\mu_{\gamma}}
f_{ij,K', \Gamma} \\
   &\hspace{4mm}-   
\sum_{|K'|+|\Gamma| \leq d}  \hspace{2mm} 
\prod_{\alpha=1}^{l} {}_{(\alpha)}a_{ij}^{k_{\alpha}} 
\prod_{\gamma=1}^{k} {}_{(\gamma)}c_{ij}^{\mu_{\gamma}}
f_{ij,K', \Gamma}  \\ 
&= \sum_{1\leq |K'|+|\Gamma| \leq d}   \hspace{2mm}   
   \sum_{\stackrel{r_{\alpha}\leq k_{\alpha} }{ 
s_{\gamma}\leq \mu_{\gamma} }} 
  \binom{k_1}{r_1} \binom{k_2}{r_2} \cdots \binom{k_{l}}{r_{l}} 
  \binom{\mu_1}{s_1} \binom{\mu_2}{s_2} \cdots \binom{\mu_{k}}{s_{k}}
  \prod_{\alpha=1}^{l} {}_{(\alpha)}a_{ij}^{k_{\alpha}-r_{\alpha}}
  \prod_{\gamma=1}^{k} {}_{(\gamma)}c_{ij}^{\mu_{\gamma}-s_{\gamma}}\cdot \\ 
  &\hspace{4mm}\cdot u^{|R|+|S|}  
  \prod_{\alpha=1}^{l} ({}_{(\alpha)}a_{jj})^{r_{\alpha}} 
  \prod_{\gamma=1}^{k} (-{}_{(\gamma)}c_{ii})^{s_{\gamma}}
  f_{ij,K',\Gamma}, \\ 
\end{aligned}
\] 
where 
$|R|=\displaystyle{\sum_{\alpha=1}^{l}r_{\alpha}}$ and 
$|S|=\displaystyle{\sum_{\gamma=1}^{k}s_{\gamma}}$. 
We see that 
$\{\displaystyle{ \prod_{\alpha=1}^{l}{}_{(\alpha)}a_{ij}^{k_{\alpha}-r_{\alpha}}
\prod_{\gamma=1}^{k} {}_{(\gamma)}c_{ij}^{\mu_{\gamma}-s_{\gamma}}
u^{|R|+|S|} }:$   
$r_{\alpha}\leq k_{\alpha}$, $s_{\gamma}\leq \mu_{\gamma}$ $\mbox{ for all } 1\leq \alpha\leq l 
\mbox{ and for all }1\leq \gamma\leq k \}$   
is linearly independent over $R_3$.

Finally, consider Cases \ref{item:twopathways-case8}, 
\ref{item:twopathways-case9}, 
and 
\ref{item:twopathways-case10}. 
Relabel the arrows in the following way: 
write 
$a_1,\ldots, a_l$ if $ta_i=m$, 
$c_1,\ldots, c_k$ if $hc_j=m$, and 
$\zeta$ if $\zeta$ is the loop at $m$. 
Write the general representations of $a_{\alpha}$ as $A_{\alpha} =({}_{(\alpha)}a_{st})$,
$c_{\gamma}$ as $C_{\gamma}=({}_{(\gamma)}c_{st})$, and 
$\zeta$ as $\Xi = (\zeta_{st})$. 
Write 
\[ 
f = \sum_{\rho+|K'|+|\Gamma| \leq d} 
\zeta_{ij}^{\rho} 
\prod_{\alpha=1}^{l} {}_{(\alpha)}a_{ij}^{k_{\alpha}} 
\prod_{\gamma=1}^{k} {}_{(\gamma)}c_{ij}^{\mu_{\gamma}} \hspace{1mm}
f_{ij,\rho, K', \Gamma}, 
\] 
where $K'=(k_1,\ldots, k_l)$, 
$\Gamma = (\mu_1,\ldots, \mu_k)$, 
$|K'|=\displaystyle{\sum_{\alpha=1}^{l} k_{\alpha}}$, 
$|\Gamma| = \displaystyle{ \sum_{\gamma=1}^{k} \mu_{\gamma}}$, 
and 
$f_{ij,K',\Gamma} \in  
\mathbb{C}[\{ {}_{(\alpha)}a_{st}$, ${}_{(\gamma)}c_{st}$, $\zeta_{st}:$ 
$(s,t)\not=(i,j)$ 
$\mbox{ and }$ 
$\alpha\not\in \{ 1,\ldots, l\}$
$\mbox{ and }$ 
$\gamma\not\in \{1,\ldots, k \}
\}]$. 
Define 
$R_4:= \mathbb{C}[\{ {}_{(\alpha)}a_{st}, {}_{(\gamma)}c_{st},\zeta_{st}$: $(s,t)\not=(i,j)$ 
$\mbox{ and }$ 
$\alpha\not\in \{ 1,\ldots, l\}$
$\mbox{ and }$ 
$\gamma\not\in \{1,\ldots, k \}  
\}]$.  
Then 
\[ 
\begin{aligned} 
0 &= u_{ij}^{-1}.f-f 
= \sum_{\rho+ |K'|+|\Gamma| \leq d}  
(\zeta_{ij}+(\zeta_{jj}-\zeta_{ii})u)^{\rho} 
\prod_{\alpha=1}^{l} ({}_{(\alpha)}a_{ij}+{}_{(\alpha)}a_{jj}u )^{k_{\alpha}} 
\prod_{\gamma=1}^{k} ({}_{(\gamma)}c_{ij}-{}_{(\gamma)}c_{ii}u )^{\mu_{\gamma}}
f_{ij,\rho, K', \Gamma} \\
   &\hspace{4mm}-
\sum_{\rho+ |K'|+|\Gamma| \leq d} 
\zeta_{ij}^{\rho}
\prod_{\alpha=1}^{l} {}_{(\alpha)}a_{ij}^{k_{\alpha}} 
\prod_{\gamma=1}^{k} {}_{(\gamma)}c_{ij}^{\mu_{\gamma}}\hspace{1mm}
f_{ij,\rho, K', \Gamma}  \\  
&= \sum_{1\leq \rho+ |K'|+|\Gamma| \leq d}  \sum_{ \stackrel{\tau\leq \rho}{
r_{\alpha}\leq k_{\alpha}, s_{\gamma}\leq \mu_{\gamma}
}}  
\binom{\rho}{\tau} 
  \binom{k_1}{r_1} \binom{k_2}{r_2} \cdots \binom{k_{l}}{r_{l}} 
  \binom{\mu_1}{s_1} \binom{\mu_2}{s_2} \cdots \binom{\mu_{k}}{s_{k}} \cdot \\  
  &\hspace{4mm}\cdot \zeta_{ij}^{\rho-\tau} 
  \prod_{\alpha=1}^{l} {}_{(\alpha)}a_{ij}^{k_{\alpha}-r_{\alpha}}
  \prod_{\gamma=1}^{k} {}_{(\gamma)}c_{ij}^{\mu_{\gamma}-s_{\gamma}} 
  \cdot u^{\tau+ |R|+|S|} 
  (\zeta_{jj}-\zeta_{ii})^{\tau}  
  \prod_{\alpha=1}^{l} ({}_{(\alpha)}a_{jj})^{r_{\alpha}} 
  \prod_{\gamma=1}^{k} (-{}_{(\gamma)}c_{ii})^{s_{\gamma}}
  f_{ij,\rho, K',\Gamma}, \\ 
\end{aligned}  
\] 
where $|R|=\displaystyle{\sum_{\alpha=1}^{l}r_{\alpha}}$ and 
$|S|=\displaystyle{\sum_{\gamma=1}^{k}s_{\gamma}}$,  
and we see that 
$\{\zeta_{ij}^{\rho-\tau} 
\displaystyle{ \prod_{\alpha=1}^{l}{}_{(\alpha)}a_{ij}^{k_{\alpha}-r_{\alpha}}
\prod_{\gamma=1}^{k} {}_{(\gamma)}c_{ij}^{\mu_{\gamma}-s_{\gamma}}
u^{\tau+ |R|+|S|} }:    
\tau\leq \rho,  
r_{\alpha}\leq k_{\alpha}, s_{\gamma}\leq \mu_{\gamma} \mbox{ for all }1\leq \alpha\leq l  
\mbox{ and for all }1\leq \gamma\leq k \}$   
is linearly independent over $R_4$.   
In all cases, 
we conclude that $f_{ij,K'}$, $f_{ij,K',\Gamma}$, and $f_{ij,\rho, K',\Gamma}$ $=$ $0$ 
for all $|K'|\geq 1$, $|\Gamma|\geq 1$, and $\rho\geq 1$, which contradict our choices of $(i,j)$ and $m$. We conclude that 
 $f\in \mathbb{C}[\mathfrak{t}^{\oplus Q_1}]$.   
\end{proof}    
Note that one needs to handle with care of the gluing of the arrows and vertices in 
the itemized list in the proof of Theorem~\ref{theorem:two-paths-max-quiver-semi-invariants}, or else, $Q$ will have more than two pathways between some of its vertices so Theorem~\ref{theorem:two-paths-max-quiver-semi-invariants} will not hold for such quivers.

Now, we will discuss Theorem~\ref{theorem:two-paths-max-quiver-semi-invariants-framed}.  
By Theorem~\ref{theorem:two-paths-max-quiver-semi-invariants}, it suffices to find all (semi-)invariants for paths starting at a framed vertex, 
but since the proof is essentially the same as the proof of  Theorem~\ref{theorem:invariants-of-framed-affine-Ar-quiver}, we will omit the details.

\section{Other interesting family of quivers}\label{section:other-interesting-quivers}

A comet-shaped quiver (of any orientation) has $k$ legs, each of length $s_k$, with $1$ loop on the central vertex, e.g., 
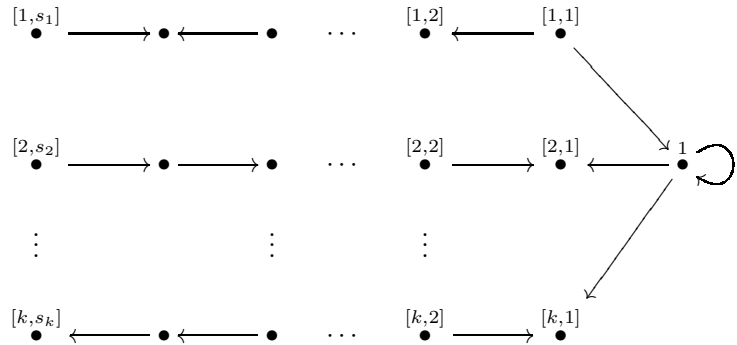
\[ 
\xymatrix@-1pc{
\stackrel{[1,s_1]}{\bullet} \ar[rr] & & \bullet & &\ar[ll] \bullet &\cdots &  \stackrel{[1,2]}{\bullet} & & \ar[ll] \stackrel{[1,1]}{\bullet} \ar[rrdd]& & \\ 
& &  & & & &   & &  & & \\ 
\stackrel{[2,s_2]}{\bullet} \ar[rr] & & \bullet \ar[rr]& & \bullet & \cdots &
 \stackrel{[2,2]}{\bullet} \ar[rr] & & \stackrel{[2,1]}{\bullet} & & \ar[ll] \stackrel{1}{\bullet}\ar@(ur,dr) \ar[ddll] \\ 
\vdots & &  & & \vdots & & \vdots   & &  & & \\ 
\stackrel{[k,s_k]}{\bullet} & & \bullet \ar[ll] & & \ar[ll]\bullet & \cdots & \stackrel{[k,2]}{\bullet} \ar[rr] & &  \stackrel{[k,1]}{\bullet} & & \\  
} 
\]  
 and 
 a star-shaped quiver is a comet-shaped quiver with $k$ legs, each of length $s_k$, but with no loops: 
 \[ 
\xymatrix@-1pc{
\stackrel{[1,s_1]}{\bullet} \ar[rr] & & \bullet & &\ar[ll] \bullet &\cdots &  \stackrel{[1,2]}{\bullet} & & \ar[ll] \stackrel{[1,1]}{\bullet} \ar[rrdd]& & \\ 
& &  & & & &   & &  & & \\ 
\stackrel{[2,s_2]}{\bullet} \ar[rr] & & \bullet \ar[rr]& & \bullet & \cdots &
 \stackrel{[2,2]}{\bullet} \ar[rr] & & \stackrel{[2,1]}{\bullet} & & \ar[ll] \stackrel{1}{\bullet}.  \ar[ddll] \\ 
\vdots & &  & & \vdots & & \vdots   & &  & & \\ 
\stackrel{[k,s_k]}{\bullet} & & \bullet \ar[ll] & & \ar[ll]\bullet & \cdots & \stackrel{[k,2]}{\bullet} \ar[rr] & &  \stackrel{[k,1]}{\bullet} & & \\  
} 
\] 
Theorem~\ref{theorem:two-paths-max-quiver-semi-invariants} 
holds for all families of comet-shaped and star-shaped quivers.

Since Theorem~\ref{theorem:invariants-of-framed-affine-Ar-quiver} holds for a quiver whose arrows are of any orientation with the underlying graph being the affine Dynkin $\widetilde{A}_r$-graph, consider the quiver with orientation 
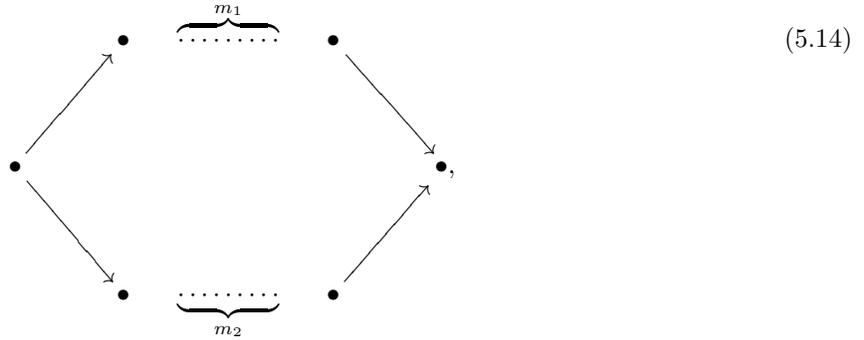
\begin{equation}\label{equation:generalized-2-kronecker-quiver}
\xymatrix@-1pc{ 
 & & \bullet & {\overbrace{\cdots\cdots\cdots}^{m_1}} & \bullet \ar[rrdd] & &  \\ 
 & &   & &   & & \\ 
\bullet \ar[rruu] \ar[rrdd] & &   & &    & & \bullet,      \\ 
& &  & &   & & \\ 
& & \bullet & {\underbrace{\cdots\cdots\cdots}_{m_2}} & \bullet \ar[rruu] & &   \\ 
}
\end{equation}
where the quiver has at least one source and one sink. Note that  Theorem~\ref{theorem:invariants-of-framed-affine-Ar-quiver} is being applied for a quiver without a framing.  
We will call the quiver \eqref{equation:generalized-2-kronecker-quiver} the 
{\em generalized $2$-Kronecker quiver} $K_{2,m_1,m_2}$.  
Consider when $r=1$, i.e., 
$\widetilde{A}_1$-quiver, with the above orientation: 
\[ 
\xymatrix@-1pc{
\bullet \ar@/^/[rrr] \ar@/_/[rrr] & & & \bullet. \\ 
}
\] 
Suppose the dimension at each vertex is $n$ and suppose we impose the complete standard filtration of vector spaces at each vertex. 
By Corollary~\ref{corollary:affine-quiver-no-framing}, $\mathbb{U}_{\beta}$-invariants are coming from the diagonal coordinates of $\mathfrak{b}^{\oplus 2}$, i.e.,  
$\mathbb{C}[F^{\bullet}Rep(\widetilde{A}_1, (n,n))]^{\mathbb{U}_{\beta}}=\mathbb{C}[\mathfrak{t}^{\oplus 2}]$. 
Thus, the following corollary follows: 
\begin{corollary}\label{corollary:generalized-2-Kronecker-quiver-semi-invariants}  
Let  $K_{2,m_1,m_2}=(Q_0,Q_1)$ be the generalized $2$-Kronecker quiver with at least one source and one sink as in \eqref{equation:generalized-2-kronecker-quiver}.   
Let $\beta=(n,\ldots, n)\in \mathbb{Z}_{\geq 0}^{Q_0}$ and let 
$F^{\bullet}$ be the complete standard filtration of vector spaces at each vertex. Let $\mathbb{U}_{\beta}:=U^{Q_0}$. 
Then $\mathbb{C}[F^{\bullet}Rep(K_{2,m_1,m_2}, \beta)]^{\mathbb{U}_{\beta}}=\mathbb{C}[\mathfrak{t}^{\oplus Q_1}]$. 
\end{corollary}  

Since the proof of Corollary~\ref{corollary:generalized-2-Kronecker-quiver-semi-invariants} is analogous to the proof of Theorem~\ref{theorem:invariants-of-framed-affine-Ar-quiver}, 
we will omit the details.

We will now generalize the generalized $2$-Kronecker quiver in \eqref{equation:generalized-2-kronecker-quiver} in the following way. 
Consider the graph $\Gamma_{K(M,N)}$, where $M=(m_1,\ldots, m_k)$ and $N=(n_1,\ldots, n_k)$: 
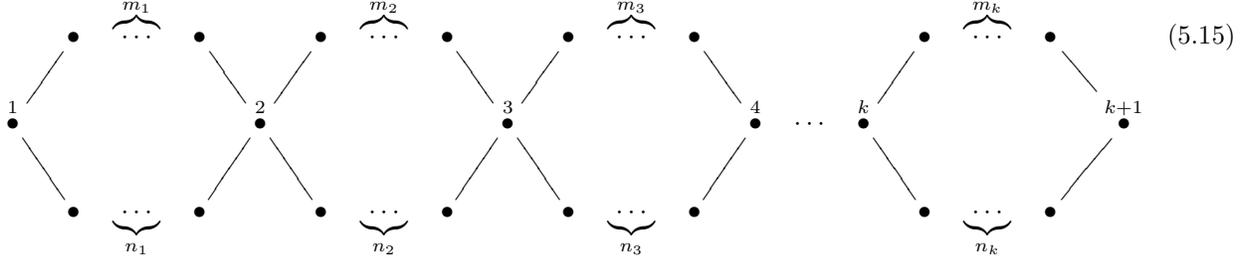
\begin{equation}\label{equation:generalized-2-Kronecker-quiver-repeated} 
\hspace{-0.5em}
\xymatrix@-1pc{
& \bullet & \hspace{-2mm}{\overbrace{\cdots}^{m_1}}\hspace{-2mm} & \bullet \ar@{-}[rd] &       & \bullet & \hspace{-2mm}{\overbrace{\cdots}^{m_2}}\hspace{-2mm} & \bullet \ar@{-}[rd] &   & \bullet & \hspace{-2mm}{\overbrace{\cdots}^{m_3}}\hspace{-2mm} & \bullet \ar@{-}[rd] &      & \hspace{-4mm}\hspace{-4mm}&     & \bullet & \hspace{-2mm} {\overbrace{\cdots}^{m_k}}\hspace{-2mm} & \bullet \ar@{-}[rd] &       \\  
\stackrel{1}{\bullet} \ar@{-}[ru] \ar@{-}[rd] & &  &  & \stackrel{2}{\bullet}    \ar@{-}[ru] \ar@{-}[rd] & &  &  & \stackrel{3}{\bullet}    \ar@{-}[ru] \ar@{-}[rd] & &  &  & \stackrel{4}{\bullet}   & \hspace{-4mm}\cdots \hspace{-4mm}&   \stackrel{k}{\bullet} \ar@{-}[ru] \ar@{-}[rd] & & \hspace{-4mm} &  & \stackrel{k+1}{\bullet}    \\ 
 & \bullet & \hspace{-2mm}{\underbrace{\cdots}_{n_1}}\hspace{-2mm} & \bullet \ar@{-}[ru]  &       & \bullet &\hspace{-2mm}{\underbrace{\cdots}_{n_2}}\hspace{-2mm} & \bullet \ar@{-}[ru]  &    & \bullet & \hspace{-2mm}{\underbrace{\cdots}_{n_3}} \hspace{-2mm}& \bullet \ar@{-}[ru]  &     & \hspace{-4mm}\hspace{-4mm}&   & \bullet &\hspace{-2mm} {\underbrace{\cdots}_{n_k}} \hspace{-2mm}& \bullet \ar@{-}[ru]  &       \\ 
}
\end{equation} 
where $m_i+2$ and $n_i+2$ denote the number of edges between vertices $i$ and $i+1$. 
As long as the vertices labeled as $2,3, \ldots, k$ are a source or a sink (in any order), 
then the quiver associated to $\Gamma_{K(M,N)}$ has at most two pathways. 
This implies Theorem~\ref{theorem:two-paths-max-quiver-semi-invariants} 
may also be applied quivers whose underlying graph is $\Gamma_{K(M,N)}$ with the above source/sink condition. 

\begin{lemma}   
Let $Q$ be a quiver whose graph is $\Gamma_{K(M,N)}$ as in \eqref{equation:generalized-2-Kronecker-quiver-repeated}.   
Assume the vertices labeled by $1,2,\ldots, k+1$ are a source or a sink.   
Let $\beta = (n,\ldots,n)\in \mathbb{Z}^{Q_0}$  and let $F^{\bullet}$ be the complete standard filtration of vector spaces. 
Then $\mathbb{C}[F^{\bullet}Rep(Q,\beta)]^{\mathbb{U}_{\beta}}=\mathbb{C}[\mathfrak{t}^{\oplus Q_1}]$. 
\end{lemma}

\begin{example}
For the $1$-Jordan quiver with the complete standard filtration of vector spaces imposed with dimension $\beta=2$, we have 
$\mathbb{C}[\mathfrak{b}_2]^U 
=\mathbb{C}[\mathfrak{b}_2]^{B_2} 
= \mathbb{C}[a_{11},a_{22}]
= \mathbb{C}[\mathfrak{t}_2]
\cong \mathbb{C}[\mathbb{C}^2]$.  
For the $2$-Jordan quiver with the complete standard filtration of vector spaces imposed with $\beta=2$, we have 
$\mathbb{C}[\mathfrak{b}_2^{\oplus 2}]^U \supseteq  
\mathbb{C}[a_{11},a_{22}, c_{11},c_{22}, (c_{11}-c_{22})a_{12}-(a_{11}-a_{22})c_{12}]$.   
For the $3$-Jordan quiver with the complete standard filtration of vector spaces imposed with $\beta=2$, we have 
\[ 
\begin{aligned}
\mathbb{C}[\mathfrak{b}_2^{\oplus 3}]^U &\supseteq  
\mathbb{C}[a_{11},a_{22},c_{11},c_{22},s_{11},s_{22}, 
(a_{11}-a_{22})c_{12}-(c_{11}-c_{22})a_{12},   \\ 
&(s_{11}-s_{22})a_{12}-(a_{11}-a_{22})s_{12},    (s_{11}-s_{22})c_{12} -(c_{11}-c_{22})s_{12}] \\ 
&=\dfrac{\mathbb{C}[f_1,f_2,g_1,g_2,h_1,h_2,k_1,k_2,k_3]}{ 
\langle (h_1-h_2)k_1-(g_1-g_2)k_2+(f_1-f_2)k_3\rangle}.  \\ 
\end{aligned}
\]   
Filtered representations of the $2$-Jordan and the $3$-Jordan quivers have semi-invariants coming from the nilradical of the Borel since they have more than two pathways at a vertex. 
Recall that $B=TU$, where $T$ is the (reductive) Levi factor while $U$ is the unipotent radical of $B$, with $\mathfrak{u}=\lie(U)$ being the nilradical of $B$.  
We believe that the two inclusions above are, in fact, equalities,  
and we will discuss this further in Section~\ref{section:comment-regarding-at-most-two-paths} and in Section~\ref{subsection:m-Jordan-quiver} as open problems.  
\end{example}

\begin{notation}\label{notation:undefined-sums} 
We define an empty sum or a summation whose indices are not defined to be zero, i.e., 
\[ \sum_{\alpha = \iota}^{\gamma} f(\alpha):=0 
		\mbox{ where }\iota > \gamma  
   	\mbox{ or }
   	\sum_{k} s_{k0}:=0 \mbox{ if }s_{k0} \mbox{ is not one of the coordinates of } s\in \mathfrak{b}^*.  
\]   
Furthermore, we will set all variables that are not defined to be zero, i.e., if $j\in (\mathbb{C}^n)^*$   
is a covector with entries  
$j_{\alpha}=y_{\alpha}$ where $1\leq \alpha\leq n$, then $y_0 :=0$. 
\end{notation}

\begin{lemma}\label{lemma:b-invariants-on-bdual}
For $B$-adjoint action on $\mathfrak{b}^*$, where  
$\mathfrak{b}^*= \mathfrak{gl}_n/\mathfrak{n}^+$, 
we have 
$\mathbb{C}[\mathfrak{b}^*]^B = \mathbb{C}[\tr(s)]$. 
\end{lemma} 

\begin{proof}
Let $F$ be a polynomial in $\mathbb{C}[\tr(s)]$. 
Since $F(\tr(bsb^{-1}))
= F(\tr(sb^{-1}b)) 
= F(\tr(s))$ for any $b\in B$, $F$ is in the $B$-invariant subring $\mathbb{C}[\mathfrak{b}^*]^B$.  

Now suppose $F\in \mathbb{C}[\mathfrak{b}^*]^B$ and let $s\in\mathfrak{b}^*$. 
Then for a $1$-parameter subgroup $\lambda_1(t)$ with coordinates 
\[ \lambda_1(t)_{\iota\gamma} = \left\{ 
\begin{aligned}
t^{\iota-1} 			& \mbox{ if } \iota =\gamma \\ 
0  \hspace{2mm} 	& \mbox{ if } \iota \not=\gamma, \\ 
\end{aligned}
\right.
\]
\[ 
(\lambda_1(t).s)_{\iota\gamma} 
    = (\lambda_1(t)s(\lambda_1(t))^{-1})_{\iota\gamma}  
		= \left\{ 
\begin{aligned} 
* 							\hspace{4mm} 		&	\mbox{ if } \iota < \gamma \\ 
s_{\iota\iota} 	\hspace{3mm} 		& \mbox{ if } \iota = \gamma \\ 
t^{\iota-\gamma}s_{\iota\gamma} & \mbox{ if }\iota > \gamma.  \\
\end{aligned} 
\right.
\] Taking the limit as $t\rightarrow 0$, we have 
\[ \lim_{t\rightarrow 0} (\lambda_1(t).s)_{\iota\gamma} = 
\left\{ 
\begin{aligned} 
* 				\hspace{1mm}			&	\mbox{ if } \iota < \gamma \\ 
s_{\iota\iota} 							& \mbox{ if } \iota = \gamma \\ 
0 				\hspace{1mm}			& \mbox{ if }	\iota > \gamma.  \\
\end{aligned}
\right. 
\]   
Since off-diagonal entries of $s$ are zero, our $B$-invariant polynomial $F$ 
is independent of the coordinates $\{ s_{\iota\gamma} \}_{\iota > \gamma}$. 
Now consider another $1$-parameter subgroup $\lambda_2(t)$, where 
\[\lambda_2(t)_{\iota\gamma} 
= \left\{ 
\begin{aligned} 
t^{\iota-1} 		& \mbox{ if } \iota \leq \gamma  \\ 
0 \hspace{2mm}	&	\mbox{ if } \iota > \gamma. \\
\end{aligned} 
\right.
\] 
Then $(\lambda_2(t).s)_{\iota\gamma} = $ 
\begin{equation}\label{eq:lambda2s} 
= \left\{ 
\begin{aligned} 
* \hspace{20mm}&\mbox{ if }\iota < \gamma  \\
t^{\iota-\gamma} \left( \sum_{k=\iota}^n s_{k \gamma} - \sum_{k=\iota}^n s_{k,\gamma-1}\right) 
 &\mbox{ if } \iota \geq \gamma \\ 
\end{aligned} 
\right. 
\hspace{3mm}
= \hspace{3mm}
\left\{ 
\begin{aligned} 
* \hspace{20mm}& \mbox{ if } \iota < \gamma \\
s_{\iota\iota}-\sum_{k=\iota}^n s_{k,\iota -1}+ \sum_{k=\iota+1}^n s_{k\iota}\hspace{4mm}				 & \mbox{ if } \iota = \gamma  \\ 
t^{\iota-\gamma} \left( \sum_{k=\iota}^n s_{k\gamma} - \sum_{k=\iota}^n s_{k,\gamma-1} \right)  & \mbox{ if }  \iota > \gamma.  \\ 
\end{aligned} 
\right. 
\end{equation} 
Since $F(s)=F(s')$ for any $s'\in \overline{B.s}$ (the polynomial $F$ must take the same value on any orbit closure), the equality 
$\displaystyle{\lim_{t\rightarrow 0}F(\lambda_1(t).s)}
= 
\displaystyle{\lim_{t\rightarrow 0}F(\lambda_2(t).s)}$ must hold for any values of  
 $\{ s_{\alpha\beta}\}_{\alpha>\beta}$. 
So for each $1\leq \iota < n$ (starting with $\iota=1$ in ascending order),  
choose $\{ s_{k\iota}\}_{k>\iota}$
in $\displaystyle{\sum_{\iota < k\leq n} s_{k\iota}}$ such that 
\begin{equation}\label{eq:choosing-free-vars}
-\sum_{k=\iota+1}^n s_{k\iota}=s_{\iota\iota}-\sum_{k=\iota}^n s_{k,\iota-1}.  
\end{equation} 

Move the sum in (\ref{eq:choosing-free-vars}) to the left-hand side so that for each $1\leq \iota <n$,  
$\displaystyle{\sum_{\iota-1<k\leq n}s_{k,\iota-1}} 
-
\displaystyle{\sum_{\iota < k\leq n} s_{k,\iota}}$ $=$ $s_{\iota\iota}$, which implies 
the sum of all such sum as $\iota$ varies over $1$ to $n-1$ is 
\[\begin{aligned}
\sum_{\iota=1}^{n-1} \left(\sum_{k=\iota}^n s_{k,\iota-1} -\sum_{k=\iota+1}^n s_{k \iota}\right) 
		& = \sum_{k=1}^n \not{s_{k0}} 
				+ \left(\sum_{\iota=2}^{n-1} \sum_{k=\iota}^n s_{k,\iota-1}
					  -\sum_{\iota=1}^{n-2} \sum_{ k=\iota+1 }^n s_{k\iota}  \right)
		   			-\sum_{k=n}^n s_{k,n-1} \\
		& = -\sum_{k=n}^n s_{k,n-1} = \sum_{\iota=1}^{n-1} s_{\iota\iota}. \\
\end{aligned}
\] 

By (\ref{eq:lambda2s}) and by choosing appropriate choices for $\{ s_{\alpha\beta}\}_{\alpha > \beta}$ in 
(\ref{eq:choosing-free-vars}), we have 
$(\lambda_2(t).s)_{\iota\iota}=0$ for each $1\leq \iota <n$ while  
$(\lambda_2(t).s)_{nn}=s_{nn}- 
\displaystyle{\sum_{n-1<k\leq n} s_{k,n-1}} = \tr(s)$.  
This means all coordinate entries are zero except the $(n,n)$-entry, which is $\tr(s)$. 
Thus for $F$ in $\mathbb{C}[\mathfrak{b}^*]^B$, $F(s)$ must be of the form $F(s')$, where all coordinates of $s'$ are zero   
except the entry $s_{nn}'$, which equals $\tr(s)$. So $F$ is a polynomial in $\tr(s)$. 
\end{proof}

\begin{lemma}\label{lemma:b-invariants-b-V} 
Let $B$ act on $\mathfrak{b}\times \mathbb{C}^n$ by $b.(r,i)=(brb^{-1},bi)$, where $b\in B$ and $(r,i)\in B\times \mathbb{C}^n$. 
Then 
$\mathbb{C}[\mathfrak{b}\times \mathbb{C}^n]^B=\mathbb{C}[\mathfrak{t}] \cong\mathbb{C}[\mathbb{C}^n]$ while 
 $\mathbb{C}[\mathfrak{b}\times \mathbb{C}^n]^U=\mathbb{C}[\mathfrak{t}]\otimes_{\mathbb{C}} \mathbb{C}[x_n]\cong \mathbb{C}[\mathbb{C}^{n+1}]$, where $x_n$ is the last component of $i\in \mathbb{C}^n$. 
\end{lemma} 

\begin{proof} 
The lemma holds by Theorem~\ref{theorem:two-paths-max-quiver-semi-invariants-framed}. 
\end{proof}

\begin{lemma}\label{lemma:b-invariants-bdual-V}
Let $\mathfrak{b}^*=\mathfrak{gl}_n/\mathfrak{n}^+$ and $(\mathbb{C}^n)^*$ be the dual vector space to $\mathbb{C}^n$. 
Let $B$ act on $\mathfrak{b}^*\times (\mathbb{C}^n)^*$ via $b.(s,j)=(bsb^{-1},jb^{-1})$, where $b\in B$ and $(s,j)\in \mathfrak{b}^*\times (\mathbb{C}^n)^*$.  
Then 
$\mathbb{C}[\mathfrak{b}^*\times (\mathbb{C}^n)^*]^B=\mathbb{C}[\tr(s)]$ while 
$\mathbb{C}[\mathfrak{b}^*\times (\mathbb{C}^n)^*]^U=\mathbb{C}[\tr(s)]\otimes_{\mathbb{C}}\mathbb{C}[y_1] \cong \mathbb{C}[\mathbb{C}^2]$, 
where $y_1$ is the first component of $j\in (\mathbb{C}^n)^*$.  
\end{lemma}

\begin{proof} 
 Let $F(s,j)\in\mathbb{C}[\mathfrak{b}^*\times (\mathbb{C}^n)^*]^B$. We will show that $F(s,j)\in\mathbb{C}[\tr(s)]$ 
since the other containment is clear. 
Choose a $1$-parameter subgroup $\lambda_1(t)$ in the Borel with coordinates 
\[ \lambda_1(t)_{\iota\gamma} = \left\{ 
\begin{aligned}
t^{\iota-n}    &\mbox{ if } \iota = \gamma \\ 
0 \hspace{2mm} &\mbox{ if } \iota \not=\gamma. \\
 \end{aligned}
\right. 
 \] 
Then 
\[
 (\lambda_1(t).s)_{\iota\gamma} = \left\{ 
\begin{aligned} 
*\hspace{4mm}                    &\mbox{ if }\iota < \gamma \\ 
s_{\iota\iota}\hspace{3mm}       &\mbox{ if } \iota =\gamma \\ 
t^{\iota-\gamma} s_{\iota\gamma} &\mbox{ if } \iota > \gamma \\
 \end{aligned} 
\right. 
\]
while 
      $(\lambda_1(t).j)_{\alpha} = t^{n-\alpha}y_{\alpha}$. 
Take the limit as $t\rightarrow 0$ to obtain 
\[ \lim_{t\rightarrow 0} (\lambda_1(t).s)_{\iota\gamma} = 
\left\{ \begin{aligned} 
         * \hspace{1mm} &\mbox{ if } \iota < \gamma \\ 
s_{\iota\iota}          &\mbox{ if } \iota = \gamma \\ 
0 \hspace{1mm}          &\mbox{ if } \iota > \gamma \\
        \end{aligned}
\right. 
\]
and 
\[
\lim_{t\rightarrow 0}(\lambda_1(t).j)_{\alpha} = \left\{ 
\begin{aligned}
0  \hspace{1mm} &\mbox{ if } \alpha < n \\ 
y_n             &\mbox{ if } \alpha =n. \\
\end{aligned}
\right. 
\] 
So the $B$-invariant function $F(s,j)$ is independent of the variables $\{ s_{\iota\gamma} \}_{\iota > \gamma}$
and $\{ y_{\alpha} \}_{\alpha < n}$. 
Now consider another $1$-parameter subgroup 
\[
 \lambda_2(t)_{\iota\gamma} = \left\{ 
\begin{aligned}
 t^{\iota-n}    &\mbox{ if }\iota \leq \gamma \\ 
 0 \hspace{2mm} &\mbox{ if } \iota > \gamma. \\
\end{aligned}
\right. 
\]
Then $(\lambda_2(t).s)_{\iota\gamma} = $ 
\begin{equation}\label{eq:lambda2s-b-times-V} 
= \left\{ 
\begin{aligned} 
* \hspace{20mm}&\mbox{ if }\iota < \gamma  \\
t^{\iota-\gamma} \left( \sum_{k=\iota}^n s_{k \gamma} - \sum_{k=\iota}^n s_{k,\gamma-1}\right) 
 &\mbox{ if } \iota \geq \gamma \\ 
\end{aligned} 
\right. 
= \left\{ 
\begin{aligned} 
* \hspace{20mm}& \mbox{ if } \iota < \gamma \\
s_{\iota\iota}-\sum_{k=\iota}^n s_{k,\iota -1}+ \sum_{k=\iota+1}^n s_{k\iota}\hspace{4mm}				 & \mbox{ if } \iota = \gamma  \\ 
t^{\iota-\gamma} \left( \sum_{k=\iota}^n s_{k\gamma} - \sum_{k=\iota}^n s_{k,\gamma-1} \right)  & \mbox{ if }  \iota > \gamma  \\ 
\end{aligned} 
\right. 
\end{equation} 
while 
$(\lambda_2(t).j)_{\alpha} =t^{n-\alpha}(y_{\alpha}-y_{\alpha-1} )$. 
Since $F(s,j)=F(s',j')$ for any $(s',j')\in \overline{B.(s,j)}$, 
the equality 
$\displaystyle{\lim_{t\rightarrow 0}F(\lambda_1(t).(s,j) )}
=
\displaystyle{\lim_{t\rightarrow 0}F(\lambda_2(t).(s,j))}$ 
 must hold regardless of the values of $\{ s_{\mu\nu}\}_{\mu>\nu}$ and $\{ y_{\alpha}\}_{\alpha<n}$. 
So for each  
$1\leq \iota < n$ (starting with $\iota=1$ in ascending order), choose $\{ s_{k\iota}\}_{k>\iota}$
in  $\displaystyle{\sum_{\iota < k\leq n} s_{k\iota}}$ so that 
\begin{equation}\label{eq:choosing-free-vars-b-times-V}
-\sum_{k=\iota+1}^n s_{k\iota}=s_{\iota\iota}-\sum_{k=\iota}^n s_{k,\iota-1}  
\end{equation} 
and choose $y_{n-1}$ so that $y_{n-1}=y_n$.

Rewrite (\ref{eq:choosing-free-vars-b-times-V}) so that for each $1\leq \iota <n$,  
$\displaystyle{\sum_{\iota-1<k\leq n}s_{k,\iota-1}}- 
\displaystyle{\sum_{\iota < k\leq n} s_{k,\iota}}$ $=$ $s_{\iota\iota}$, which implies 
the sum of all such sum as $\iota$ varies over $1$ to $n-1$ is 
\[\begin{aligned}
\sum_{\iota=1}^{n-1} \left(\sum_{k=\iota}^n s_{k,\iota-1} -\sum_{k=\iota+1}^n s_{k \iota}\right) 
		& = \sum_{k=1}^n \not{s_{k0}} 
				+ \left(\sum_{\iota=2}^{n-1} \sum_{k=\iota}^n s_{k,\iota-1}
					  -\sum_{\iota=1}^{n-2} \sum_{ k=\iota+1 }^n s_{k\iota}  \right)
		   			-\sum_{k=n}^n s_{k,n-1} \\
		& = -\sum_{k=n}^n s_{k,n-1} = \sum_{\iota=1}^{n-1} s_{\iota\iota}. \\
\end{aligned}
\] 

By (\ref{eq:lambda2s-b-times-V}) and by making appropriate choices for $\{ s_{\mu\nu}\}_{\mu>\nu}$ and $y_{n-1}$, 
we have 
$(\lambda_2(t).s)_{\iota\iota}=0$ for each $1\leq \iota <n$ while 
$(\lambda_2(t).s)_{nn}
=
s_{nn}- \displaystyle{\sum_{n-1<k\leq n} s_{k,n-1}} 
= 
\tr(s)$. Furthermore, all entries of $j$ are zero, so the entries of the covector $j$ do not contribute as $B$-invariant polynomials.

Thus for $F$ in $\mathbb{C}[\mathfrak{b}^*\times (\mathbb{C}^n)^* ]^B$, $F(s,j)=F(s',j')$,  
where all entries of $s'$ are zero   
except the entry $s_{nn}'$ (which equals $\tr(s)$), and all entries of $j$ are zero. 
Thus in order for $F$ to be a $B$-invariant polynomial, $F$ must a polynomial in $\tr(s)$. 

The second statement of the lemma holds due to this proof and Theorem~\ref{theorem:two-paths-max-quiver-semi-invariants-framed}.  
\end{proof}

\section{Comment regarding at most two pathways}\label{section:comment-regarding-at-most-two-paths}

Revisiting Theorems~\ref{theorem:two-paths-max-quiver-semi-invariants} and \ref{theorem:two-paths-max-quiver-semi-invariants-framed},  
we will now discuss the restriction of having at most two pathways at each vertex of a quiver.

Consider the $2$-Jordan quiver: 
$\xymatrix@-1pc{ \stackrel{1}{\bullet} \ar@(dl,ul)^{a_1} \ar@(ur,dr)^{a_2}  
}$ where $\beta=2$ and $F^{\bullet}$ is the complete standard filtration of vector spaces. 
Then this quiver has seven pathways at vertex $1$: 
$e_1$, $a_1$, $a_2$, $a_1 a_2$, $a_2 a_1$, $a_1 a_2 a_1$, and $a_2 a_1 a_2$. 
The filtered representation space is  
$F^{\bullet}Rep(Q,2) = \mathfrak{b}_2^{\oplus 2}$,  
and let  
\[  
u = \begin{pmatrix} 
1 & u_{12} \\ 
0 & 1 
\end{pmatrix} \in U,  
\hspace{4mm} 
A = \begin{pmatrix} 
a_{11} & a_{12} \\ 
0 & a_{22} 
\end{pmatrix} \in \mathfrak{b}_2, \mbox{ and } 
C = \begin{pmatrix} 
c_{11} & c_{12} \\ 
0 & c_{22} 
\end{pmatrix} \in \mathfrak{b}_2. 
\] 
Then 
\[ 
u.(A,C) = (uAu^{-1}, uCu^{-1}) 
= \left( 
\begin{pmatrix}
a_{11} & a_{12}-(a_{11}-a_{22})u_{12} \\ 
0 & a_{22} 
\end{pmatrix}, 
\begin{pmatrix}
c_{11} & c_{12}-(c_{11}-c_{22})u_{12} \\ 
0 & c_{22} 
\end{pmatrix}
\right). 
\] 
We can easily check that 
\begin{equation}\label{eq:2-Jordan-quiver-n2-5-generators-open-problem}
\mathbb{C}[\mathfrak{b}_2^{\oplus 2}]^U \supseteq 
\mathbb{C}[a_{11},a_{22},c_{11},c_{22}, (a_{11}-a_{22})c_{12}-(c_{11}-c_{22})a_{12}]. 
\end{equation} 
Thus $\mathbb{C}[\mathfrak{b}_2^{\oplus 2}]^U\supsetneq \mathbb{C}[\mathfrak{t}_2^{\oplus 2}]$. 
Furthermore, it remains an open problem that the containment in 
\eqref{eq:2-Jordan-quiver-n2-5-generators-open-problem} 
in the other direction holds. 
We refer the reader to Section~\ref{subsection:m-Jordan-quiver} for statements of open problems related to the $m$-Jordan quiver.

Now consider the $3$-Jordan quiver:    
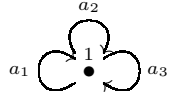
\[   
\xymatrix@-1pc{  
\stackrel{1}{\bullet} \ar@(dl,ul)^{a_1} \ar@(ul,ur)^{a_2} \ar@(ur,dr)^{a_3} \\
} 
\]   
with $\beta=3$ and $F^{\bullet}$ the complete standard filtration of vector spaces. 
Then $F^{\bullet}Rep(Q,\beta)=\mathfrak{b}_3^{\oplus 3}$ and under the maximal unipotent subgroup $U\subseteq B$ action, 
$\mathbb{C}[\mathfrak{t}_3^{\oplus 3}] \subsetneq \mathbb{C}[F^{\bullet}Rep(Q,\beta)]^{U}$ 
due to the following reason. 
Let $(A,C,S)\in \mathfrak{b}_3^{\oplus 3}$, where 
\[ 
A =\begin{pmatrix}  
 a_{11}& a_{12}& a_{13}\\ 
 0& a_{22}& a_{23}\\ 
 0& 0& a_{33}   
\end{pmatrix}, \hspace{2mm}
C =\begin{pmatrix}  
 c_{11}& c_{12}& c_{13}\\ 
 0& c_{22}& c_{23}\\ 
 0& 0& c_{33}  
\end{pmatrix}, \hspace{2mm} 
S =\begin{pmatrix} 
 s_{11}& s_{12}& s_{13}\\ 
 0& s_{22}& s_{23}\\ 
 0& 0& s_{33} 
\end{pmatrix}. 
\] 
 Consider the polynomial 
 $f(A,C,S)=(a_{11}-a_{22})c_{12}-(c_{11}-c_{22})a_{12}$. 
 Then for $u\in U$, where 
 \[ 
 u = \begin{pmatrix} 
 1 &u_{12} &u_{13} \\ 
 0 &1 & u_{23}\\  
 0 &0 & 1 
  \end{pmatrix}, 
 \]  
 \[ 
 \begin{aligned} 
 u.f(A,C,S) &= f(uAu^{-1},uCu^{-1}, uSu^{-1}) \\ 
 &= (a_{11}-a_{22})(c_{12}+(c_{22}-c_{11})u_{12} )-(c_{11}-c_{22})(a_{12}+(a_{22}-a_{11})u_{12}) \\ 
 &=(a_{11}-a_{22})c_{12}+(a_{11}-a_{22})(c_{22}-c_{11})u_{12} 
 -(c_{11}-c_{22})a_{12}-(c_{11}-c_{22})(a_{22}-a_{11})u_{12} \\ 
 &=f(A,C,S).\\
 \end{aligned} 
 \]  
 Thus, $f$ is a polynomial in $\mathbb{C}[F^{\bullet}Rep(Q,\beta)]^{U}$ but not in $\mathbb{C}[\mathfrak{t}_3^{\oplus 3}]$. 
 Similarly, we obtain the following $U$-invariant polynomials: 
 \[ 
 \begin{aligned} 
 (c_{11}-c_{22})s_{12}-&(s_{11}-s_{22})c_{12}, \hspace{4mm} (a_{11}-a_{22})s_{12}-(s_{11}-s_{22})a_{12}, \\ 
 (a_{22}-a_{33})c_{23}-&(c_{22}-c_{33})a_{23}, \hspace{4mm} (a_{22}-a_{33})s_{23}-(s_{22}-s_{33})a_{23},\\ 
                       &(c_{22}-c_{33})s_{23}-(s_{22}-s_{33})c_{23}.     \\ 
 \end{aligned}
 \] 
 We believe that there are other invariant polynomials, possibly of degree $3$ (for this particular example $F^{\bullet}Rep(3$-$\Jordan,3)$), 
 such that at least one of their monomials is divisible by the variable $a_{13}$, $c_{13}$, or $s_{13}$.
 In fact, for the $m$-Jordan quiver,  we conjecture that for $m\geq 1$, 
 \[ \dim_{\mathbb{C}} \mathfrak{b}^{\oplus m}/\!\!/U = \dim_{\mathbb{C}} \mathfrak{b}^{\oplus m}-\dim_{\mathbb{C}} U.
 \]  
 Note that by setting the coordinates $s_{ij}=0$, we obtain the $2$-Jordan quiver 
 $\xymatrix@-1pc{\stackrel{1}{\bullet} \ar@(ru,rd)^{a_2} \ar@(lu,ld)_{a_1}}$, which is still a quiver with more than $2$ distinct pathways, and it is clear that for the $2$-Jordan quiver, 
 $\mathbb{C}[\mathfrak{t}^{\oplus 2}]\subsetneq \mathbb{C}[\mathfrak{b}^{\oplus 2}/U]$ as we have seen earlier this section.  
 Returning to the $3$-Jordan quiver,  
 there are relations among the degree $1$ and the degree $2$ polynomials,  
 i.e.,  
 let 
 \[
 \begin{aligned} 
 f_1 = a_{11}, \hspace{2mm} 
 f_2 = a_{22}, \hspace{2mm}  
 &g_1 = c_{11}, \hspace{2mm} 
 g_2 = c_{22},                \\ 
 h_1 = s_{11}, \hspace{2mm}   
 h_2 = s_{22}, \hspace{2mm}   
 &k_1 = (a_{11}-a_{22})c_{12}-(c_{11}-c_{22})a_{12}, \\  
 k_2 = (c_{11}-c_{22})s_{12}-(s_{11}-s_{22})c_{12}, \hspace{2mm} 
 &k_3 = (a_{11}-a_{22})s_{12}-(s_{11}-s_{22})a_{12}.  \\   
 \end{aligned}   
 \]   
Then $(h_1-h_2)k_1 - (g_1-g_2)k_2 + (f_1-f_2)k_3 =0$.   
 We conjecture that there are relations (of higher degree) among all the invariant generators.  
 The $m$-Jordan quiver is revisited in Section~\ref{subsection:m-Jordan-quiver}.

Now, consider the $3$-Kronecker quiver: 
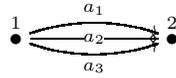
\[ 
\xymatrix@-1pc{
\stackrel{1}{\bullet} \ar@/^/[rrr]^{a_1} \ar[rrr]|{a_2} \ar@/_/[rrr]_{a_3} & & & \stackrel{2}{\bullet} \\ 
}
\]  
which has $3$ pathways from vertex $1$ to $2$. 
Let $(A,C,S)\in\mathfrak{b}_3^{\oplus 3}$, where 
\[
 A =\begin{pmatrix}  
 a_{11}& a_{12}& a_{13}\\ 
 0& a_{22}& a_{23}\\ 
 0& 0& a_{33}  
\end{pmatrix}, \hspace{2mm}
C =\begin{pmatrix}  
 c_{11}& c_{12}& c_{13}\\ 
 0& c_{22}& c_{23}\\ 
 0& 0& c_{33}   
\end{pmatrix}, \hspace{2mm} 
S =\begin{pmatrix}  
 s_{11}& s_{12}& s_{13}\\ 
 0& s_{22}& s_{23} \\ 
 0& 0& s_{33}  
\end{pmatrix}. 
\]
Suppose $(u,v)\in \mathbb{U}_{\beta}:=U_3^{2}$ acts on $(A,C,S)$ by $(uAv^{-1}, uCv^{-1}, uSv^{-1})$,
where \[
u = \begin{pmatrix}  
 1& u_{12}& u_{13}\\
 0& 1& u_{23}\\
 0& 0& 1  
\end{pmatrix},  \hspace{2mm} 
v= \begin{pmatrix}   
 1& v_{12}& v_{13}\\ 
 0& 1& v_{23}\\ 
 0& 0& 1   
\end{pmatrix}.  
\]
Consider 
$f(A,C,S) = 
  (a_{11} c_{22} - a_{22} c_{11}) s_{12}  
+ (s_{11} a_{22} - s_{22} a_{11}) c_{12}  
+ (c_{11} s_{22} - c_{22} s_{11}) a_{12}$.  
Then  
\begin{equation}\label{equation:3-kronecker-quiver-semi-invariant}
 \begin{aligned}  
(u,v).f(A,C,S)&=f(u^{-1}Av, u^{-1}Cv, u^{-1}Sv) \\  
&= (a_{11} c_{22} - a_{22} c_{11})(s_{12} - s_{22} u_{12} + s_{11} v_{12}) 
+ (s_{11} a_{22} - s_{22} a_{11})(c_{12} - c_{22} u_{12} + c_{11} v_{12}) \\ 
&+ (c_{11} s_{22} - c_{22} s_{11})(a_{12} - a_{22} u_{12} + a_{11} v_{12}) =f(A,C,S).\\  
\end{aligned} 
\end{equation}
Thus $f$ is an invariant polynomial which is not in $\mathbb{C}[\mathfrak{t}_3^{\oplus 3}]$.  
Analogously, we obtain one other degree $3$ polynomial  
\begin{equation}\label{eq:3-kronecker-quiver-2-3-entries}
(a_{22} c_{33} - a_{33} c_{22}) s_{23} - c_{23} (a_{22} s_{33} - a_{33} s_{22}) + 
 a_{23} (c_{22} s_{33} - c_{33} s_{22}) 
\end{equation} 
and we believe that there is at least one polynomial of degree $6$ with at least one of its terms divisible by 
$a_{13}$, $c_{13}$, or $s_{13}$. 
Thus, we conjecture that 
$\mathbb{C}[\mathfrak{b}_3^{\oplus 3}]^{U_3\times U_3}$ is generated polynomials of degrees $1$, $3$ and $6$. 
In general, we conjecture that for the $k$-Kronecker quiver,   
$\mathbb{C}[\mathfrak{b}^{\oplus k}]^{U\times U}$ is generated by polynomials of degrees $\dfrac{d(d+1)}{2}$ for each $1\leq d\leq n$,   
where $\mathfrak{b}\subseteq \mathfrak{gl}_n$.  
Furthermore, we conjecture that for $k\gneq 1$, 
$\dim_{\mathbb{C}} \mathfrak{b}^{\oplus k}/\!\!/U\times U=\dim_{\mathbb{C}} \mathfrak{b}^{\oplus k}-\dim_{\mathbb{C}} U\times U$.

One final remark is the following observation: 
note that the polynomial $f$ in \eqref{equation:3-kronecker-quiver-semi-invariant} is the determinant of the matrix 
\[
 \begin{pmatrix}  
 a_{11}&c_{11} &s_{11} \\ 
 a_{22}&c_{22} &s_{22} \\ 
 a_{12}&c_{12} &s_{12}  
\end{pmatrix} 
\]
and note that the polynomial in \eqref{eq:3-kronecker-quiver-2-3-entries} is the determinant of the matrix 
\[ 
 \begin{pmatrix}  
 a_{22}&c_{22} &s_{22} \\ 
 a_{33}&c_{33} &s_{33} \\ 
 a_{23}&c_{23} &s_{23}  
\end{pmatrix}.  
\] 
We believe that this observation may help us in finding other invariant polynomials of higher degrees. 
The $k$-Kronecker quiver is revisited in Section~\ref{subsection:k-Kronecker-quiver}.

\chapter{The \texorpdfstring{$B$}{B}-moment map restricted to the regular semisimple locus}\label{chapter:rss-locus}

\section{Introduction}

\begin{notation}
Let $B$ be the set of all upper triangular matrices in the general linear group $GL(n,\mathbb{C})$. We will call $B$ the Borel subgroup, or the Borel. 
\end{notation} 

\begin{notation}
We will write $\mathfrak{b}$ to denote the Lie algebra of $B$. It is the tangent space of $B$ at the identity element. 
\end{notation}

\begin{motivation}
Let $B$ act on the vector space $\mathfrak{b}\times \mathbb{C}^n$ with group action $b.(r,i)=(brb^{-1},bi)$. This action is induced onto the cotangent bundle $T^*(\mathfrak{b}\times \mathbb{C}^n) = \mathfrak{b}\times \mathfrak{b}^*\times \mathbb{C}^n\times (\mathbb{C}^n)^*$ of $\mathfrak{b}\times \mathbb{C}^n$ as: 
\[ 
b.(r,s,i,j)= (\Ad_b(r), \Ad_b^*(s),bi, jb^{-1})= (brb^{-1},bsb^{-1},bi,jb^{-1}).
\]  
 
 The infinitesimal action of $G$ induces the map $a:\mathfrak{b}\rightarrow Vect(\mathfrak{b}\times\mathbb{C}^n)$ which is given as $a(v)(r,i)=\frac{d}{dt}(g_t.(r,i))|_{t=0}=([v,r],vi)$ where $g_t=\exp(tv)$.
  Dualizing this map gives the moment map $\mu:T^*\left( \mathfrak{b}\times\mathbb{C}^n \right)\rightarrow \mathfrak{b}^*$ where $(r,s,i,j)$ is mapped to $\ad^*_r(s)+\overline{ij}$, which is  $[r,s]+\overline{ij}$ in our case, with $\overline{v}:\mathfrak{g}^* \rightarrow \mathfrak{b}^*$ being the projection map. 
For each character $\chi:B\rightarrow\mathbb{C}^*$, we are interested in understanding 
$\mu^{-1}(0)/\!\!/_{\chi}B$.  

Another main motivation for an interest in understanding the affine and GIT quotients of the above moment map is because $\mu^{-1}(0)/B$ is isomorphic to the cotangent bundle $T^*( G\times_B \mathfrak{b}\times\mathbb{C}^n/G )$ of a well-known stack where $G=GL(n,\mathbb{C})$. We refer the reader to \cite{Nevins-GSresolutions}
for the construction of the cotangent bundle of the Grothendieck-Springer resolution and to Chapter 3 of \cite{MR2838836} and Section 6 of \cite{MR1649626} for some background on the Grothendieck-Springer resolution.  
\end{motivation}  

We begin by restricting $\mathfrak{b}$ to its regular semisimple locus. 

\begin{notation}\label{notation:distinct-r-eigenvalues}
Let $\mu^{-1}(0)^{rss}$ be the set of quadruples 
\[ \{ (r,s,i,j)\in \mu^{-1}(0): r \mbox{ has distinct eigenvalues}\}.  
  \]  
\end{notation}

\begin{notation}
Let $r=(r_{\iota\gamma})$ be an $n\times n$ matrix. We write $\diag(r)$ to mean the diagonal matrix 
with entries $(r_{11},r_{22},\ldots, r_{nn})$ along its main diagonal. 
\end{notation}

\begin{trace}\label{trace:map-from-B-rss-to-complex-space}
For $(r,s,i,j)$ in $\mu^{-1}(0)^{rss}$,  choose $b\in B$ as in Proposition~\ref{proposition:distinct-ev-diagonalizable} so that 
\[(brb^{-1},bsb^{-1},bi,jb^{-1})=(\diag(r),s',i',j'). \] 
We prove in Proposition~\ref{proposition:deriving-diag-coordinates-of-bsb-inverse} that 
each diagonal coordinate function of $s'=(s_{\iota\gamma}')$ is 
\[ 
\begin{aligned}
&s_{\iota\iota}' 
=  s_{\iota\iota}
+ \sum_{\mu=\iota+1}^n \left( 
\dfrac{r_{\iota\mu} s_{\mu\iota} }{r_{\iota\iota}-r_{\mu\mu}} 
+ \sum_{v=1}^{\mu-\iota-1} \sum_{\iota < k_1 < \ldots < k_v<\mu} 
\dfrac{r_{\iota k_1}r_{k_v \mu} s_{\mu\iota} }{(r_{\iota\iota}-r_{k_1 k_1 })( r_{\iota\iota}-r_{\mu\mu})} \prod_{u=1}^{v-1} \dfrac{r_{k_u k_{u+1}}}{r_{\iota\iota}-r_{k_{u+1}k_{u+1}}}
\right) 
\\
&+ 
\sum_{\gamma=1}^{\iota-1}\left[ 
   \dfrac{ r_{\gamma\iota} s_{\iota\gamma}  }{r_{\iota\iota}-r_{\gamma\gamma}}
+\sum_{\widetilde{v}=1}^{\iota-\gamma-1} 
\sum_{\gamma< l_1 < \ldots < l_{\widetilde{v}}<\iota}
\dfrac{r_{\gamma l_1} r_{l_{\widetilde{v}}\iota} s_{\iota \gamma} }{
(r_{\iota\iota}-r_{l_{\widetilde{v}}l_{\widetilde{v}}})(r_{\iota\iota}-r_{\gamma\gamma})} \prod_{\widetilde{u}=1}^{\widetilde{v}-1} 
\dfrac{r_{l_{\widetilde{u}}l_{\widetilde{u}+1}}}{r_{\iota\iota}-r_{l_{\widetilde{u}}l_{\widetilde{u}}} } 
 \right. \\
&+ \sum_{\mu=\iota+1}^n \dfrac{r_{\gamma\iota} r_{\iota\mu}s_{\mu\gamma}}{ (r_{\iota\iota}-r_{\gamma\gamma})( r_{\iota\iota}-r_{\mu\mu})}
+ \sum_{\mu=\iota+1}^n   \sum_{\widetilde{v}=1}^{\iota-\gamma-1} \sum_{\gamma < l_1 < \ldots < l_{\widetilde{v}}<\iota} 
 \dfrac{r_{\gamma l_1}r_{l_{\widetilde{v}} \iota} r_{\iota\mu} s_{\mu\gamma} }{(r_{\iota\iota}-r_{l_{\widetilde{v}}l_{\widetilde{v}} })(r_{\iota\iota}-r_{\gamma\gamma})(r_{\iota\iota}-r_{\mu\mu})} \prod_{\widetilde{u}=1}^{\widetilde{v}-1} \dfrac{r_{l_{\widetilde{u}} l_{\widetilde{u}+1} }}{r_{\iota\iota}-r_{l_{\widetilde{u}}l_{\widetilde{u}} }}  
\\ 
&+ \sum_{\mu=\iota+1}^n \sum_{v=1}^{\mu-\iota-1}  
\sum_{\iota < k_1 < \ldots < k_v < \mu} 
 \dfrac{r_{\gamma\iota} r_{\iota k_1}r_{k_v \mu}s_{\mu\gamma} }{
(r_{\iota\iota}-r_{\gamma\gamma}) (r_{\iota\iota}-r_{k_1 k_1})(r_{\iota\iota}-r_{\mu\mu})} 
\prod_{u=1}^{v-1} \dfrac{r_{k_u k_{u+1}}}{r_{\iota\iota}-r_{k_{u+1}k_{u+1}}} 
 \\
&+ 
 \sum_{\mu=\iota+1}^n 
 \sum_{v=1}^{\mu-\iota-1} 
\sum_{\iota < k_1 < \ldots < k_v < \mu} 
 \sum_{\widetilde{v}=1}^{\iota-\gamma-1} 
\sum_{\gamma < l_1<\ldots < l_{\widetilde{v}}<\iota } 
\dfrac{r_{\gamma l_1}r_{l_{\widetilde{v}}\iota} r_{\iota k_1} r_{k_v \mu}s_{\mu\gamma}
}{(r_{\iota\iota}-
r_{l_{\widetilde{v}}l_{\widetilde{v}}})(r_{\iota\iota}-r_{\gamma\gamma})
(r_{\iota\iota}-r_{k_1 k_1})(r_{\iota\iota}-r_{\mu\mu})}\cdot \\
&\left. \;\;\;\;\;\;\;\;\;\;\;\;\cdot\prod_{u=1}^{v-1} \dfrac{r_{k_u k_{u+1}} }{r_{\iota\iota}-r_{k_{u+1}k_{u+1}}} 
 \prod_{\widetilde{u}=1}^{\widetilde{v}-1} 
\dfrac{r_{l_{\widetilde{u}}l_{\widetilde{u}+1}} }{r_{\iota\iota}-r_{l_{\widetilde{u}}l_{\widetilde{u}}} }  
 \right],\\
\end{aligned}
\] 
which is compactly written as 
\[ s_{\iota\iota}'=
 \left[ \tr\left( \prod_{
1\leq k\leq n, k\not=\iota 
} l_k(r) \right)\right]^{-1} 
  \tr \left(\prod_{
1\leq k\leq n, k\not=\iota
} l_k(r) \:s\right) 
\] where $l_k(r) = r -r_{kk}I$. 
\end{trace}

\begin{notation}\label{notation:definition-of-Delta-n-set}
We denote $\Delta_n\subseteq \mathbb{C}^{2n}$ to be the set $\{ (x_{11},\ldots, x_{nn},0,\ldots, 0): x_{\iota\iota}=x_{\gamma\gamma} 
$\mbox{ for some $\iota\not=\gamma$}$ \}$.
\end{notation}

Thus $\mathbb{C}^{2n}\setminus \Delta_n$ is the locus  
\[\{ (x_{11},\ldots, x_{nn},y_{11},\ldots,y_{nn}): x_{\iota\iota}\not= x_{\gamma\gamma} \mbox{ whenever } \iota\not= \gamma
\}.
\] 

We now state the main Propositions proved in this paper. 

\begin{proposition}\label{proposition:rss-extending-to-all-of-the-reg-semi-stable-locus}
The map 
\[ P:\mu^{-1}(0)^{rss}\rightarrow\hspace{-8pt}\rightarrow \mathbb{C}^{2n}\setminus \Delta_n 
\] 
given by sending 
\[(r,s,i,j)\mapsto (r_{11},\ldots, r_{nn},s_{11}',\ldots, s_{nn}'),
\] where $s_{\iota\iota}'$ is defined as in Trace~\ref{trace:map-from-B-rss-to-complex-space}, 
is a regular, well-defined, and $B$-invariant surjective map separating orbit closures. 
\end{proposition}

\begin{proposition}\label{proposition:rss-locus-to-complex-space-B-invariant}
The map given in Proposition~\ref{proposition:rss-extending-to-all-of-the-reg-semi-stable-locus}, which descends 
to a set-theoretic bijective homeomorphism 
$p: \mu^{-1}(0)^{rss}/\!\!/B\rightarrow \mathbb{C}^{2n}\setminus \Delta_n$, where 
\[ \overline{B.(r,s,i,j)} \mapsto (r_{11},\ldots, r_{nn},s_{11}',\ldots, s_{nn}') 
\]  with $s_{\iota\iota}'$ given as above, induces an isomorphism of varieties.  
\end{proposition}

\section{Moduli spaces, geometric invariant theory, and Hilbert schemes}\label{section:GIT-moduli-spaces-Hilbert-schemes}

Applications of quivers are prominent in many areas of mathematics, i.e., 
Hilbert schemes and almost-commuting varieties \cite{MR2210660}, 
symplectic geometry \cite{MR2838836},
representation theory  
\cite{MR1302318}, 
\cite{MR1604167},  
\cite{MR1649626},  
coding theory 
\cite{MR0424398}, 
\cite{MR0398664}, 
\cite{MR1643864}, 
etc. 
We begin by discussing moduli spaces.

\subsection{Moduli spaces}\label{subsection:moduli-spaces}

Important problems in mathematics include not only on understanding the geometry of various spaces, but on classifying a family of geometric objects. 
Moduli theory is the study of the geometry of families of algebraic objects by limiting the objects of study, deciding how two objects are considered equivalent, and how the objects will be parameterized. 
Some references include \cite{MR1631825}, \cite{MR0262240}, \cite{MR2258040}, and \cite{MR1315461}.

Invariants and semi-invariants are important for the description of the moduli spaces of representations of a quiver for fixed dimension vector.  
Thus, the notion of (semi-)stability from geometric invariant theory is important in moduli problems. Furthermore, constructions from the previous section apply to quiver varieties (\cite{MR1315461} and \cite{MR1908144}): 
%
%
%
a representation $W$ of a quiver variety $Rep(Q,\beta)$ is thought of as a point in an affine variety, and a point in the representation space is {\em semistable} if some nonconstant semi-invariant polynomial does not vanish at the point. 
If the orbit $G.W$ is closed of dimension $\dim G$, then $W$ is called a 
{\em stable point}. 
It is also a classical notion that  
the character $\chi$ is called  
{\em generic}  
if all $\chi$-semi-stable representations are $\chi$-stable.  

\subsection{Geometric invariant theory}\label{subsection:geometric-invariant-theory-introduction}   

An extensive treatment on geometric invariant theory is given in \cite{MR1304906} and \cite{MR2537067}. 
Another important reference connecting quiver representations and geometric invariant theory is A. King's \cite{MR1315461}. 
In this section, we will give some of the key ideas in this area of mathematics that are used in this thesis. 
Geometric invariant theory focuses on studying geometric interpretations of orbit spaces since group actions on a variety   
naturally arise in many areas of mathematics.   

Let $X$ be an affine variety with an algebraic $G$-action. Then $G$ acts on $\mathbb{C}[X]$ by 
\[ 
(g.f)(x)=f(g^{-1}.x).
\]

Throughout this section, let $\chi:G\rightarrow \mathbb{C}^*$ be a group homomorphism from $G$ to the units $\mathbb{C}^*$ in $\mathbb{C}$.  
We refer to Chapter~\ref{chapter:introduction} for the definition of invariant and semi-invariant polynomials, and the definition of invariant subring.  

Using Definition~\ref{definition:invariant-semi-invariant-polynomials-coordinate-rings-introduction}, we have the following constructions: 
\[ 
X/\!\!/G :=\spec(\mathbb{C}[X]^G), \hspace{4mm} 
X/\!\!/_{\chi}G :=\Proj(\bigoplus_{i\geq 0}\mathbb{C}[X]^{G,\chi^i}). 
\]


\begin{definition} 
A {\em $1$-parameter subgroup}   
$\lambda:\mathbb{C}^*\rightarrow G$ is a group homomorphism from $\mathbb{C}^*$ to $G$. 
\end{definition} 

Since the $G$-action on $X$ is linear, we have a composition of maps 
\[ 
\mathbb{C}^* \stackrel{\lambda}{\longrightarrow} G\stackrel{\sigma}{\longrightarrow} GL_{n+1}(\mathbb{C})
\] 
with the property that there exists a basis $e_0,\ldots, e_n\in \mathbb{C}^{n+1}$ such that 
\[ 
\sigma \circ \lambda(t) = 
\begin{pmatrix}
 t^{c_0} &        &0 \\ 
  			 & \ddots & \\
    0    & 		   &t^{c_n}  
\end{pmatrix}. 
\]

Next we closely follow A. King in \cite{MR1315461} 
which specializes to the case of a reductive group $G$ acting on a finite dimensional vector space $X$. 
For the rest of this section, we will write $\chi:G\rightarrow \mathbb{C}^*$ as a character of the group $G$ and $\lambda:\mathbb{C}^*\rightarrow G$ as a $1$-parameter subgroup. 
Given a character $\chi$, $G$ acts on the trivial line bundle $\mathcal{L}$, i.e., the total space $X\times \mathbb{C}$ of $\mathcal{L}^{-1}$,  
via 
$g.(x,z)=(g.x,\chi^{-1}(g)z)$. 
Then an invariant section of $\mathcal{L}^n$ is $f(x)z^n\in \mathbb{C}[X\times \mathbb{C}]$, 
where $f(x)\in \mathbb{C}[X]^{G,\chi^n}$. 
When $\mathbb{C}[X]^G=\mathbb{C}$, then $X/\!\!/_{\chi}G$ is a projective variety. 
Furthermore, $X/\!\!/_{\chi}G$ has a geometric description as the quotient of the open set $X^{\chi-ss}$ of $\chi$-semistable points by the equivalence relation that $x\sim y$ if and only if 
$\overline{G.x}\cap \overline{G.y}\not=\varnothing$ in $X^{\chi-ss}$. 
Two semistable points (or orbits) under this identification is called ``GIT equivalent." Each orbit closure contains a unique closed orbit, so the points of the quotient are in $1$-$1$ correspondence with the closed orbits in $X^{\chi-ss}$. 
This implies that there is an open subset of the quotient corresponding to $\chi$-stable orbits, all of which are closed. 
A. King allows a representation to have a kernel $\Delta$. 

\begin{lemma} 
Let $(x,z)\in X\times \mathbb{C}$ be a lift of $x\in X$, with $z\not=0$. Then 
\begin{enumerate} 
\item $x$ is $\chi$-semistable if and only if $\overline{G.(x,z)}\subseteq X\times \mathbb{C}$ is disjoint from the zero section $X\times \{ 0\}$. In particular, $\chi(\Delta)$ must be $\{ 1\}$. 
\item $x$ is $\chi$-stable if and only if $G.(x,z)$ is closed and the stabilizer group $G_{(x,z)}$ of $(x,z)$ contains $\Delta$ with finite index. 
\end{enumerate} 
\end{lemma}

The GIT equivalence relation on $X^{\chi-ss}$ can also be interpreted by 
$x\sim y$ if and only if there exist lifts $(x,z)$ and $(y,w)$ for which $\overline{G.(x,z)}\cap \overline{G.(y,w)}\not=\varnothing$ in $X\times \mathbb{C}$. 
In particular, the orbits $G.x$ which are closed in $X^{\chi-ss}$ are precisely those whose lifted orbits $G.(x,z)$ are closed in $X\times \mathbb{C}$. 

\begin{lemma}[Hilbert's Lemma] 
Let $(x,z)\in X\times \mathbb{C}$ be a lift of $x\in X$. Then 
\begin{enumerate} 
\item $x$ is $\chi$-semistable if and only if for all $1$-parameter subgroups $\lambda$, $\displaystyle{\lim_{t\rightarrow 0}} \lambda(t).(x,z)\not\in X\times \{ 0\}$. 
\item $x$ is $\chi$-stable if and only if $\lambda$ is a $1$-parameter subgroup such that $\displaystyle{\lim_{t\rightarrow 0}}\lambda(t).(x,z)$ exists, then $\lambda$ is in $\Delta$. 
\end{enumerate}
\end{lemma} 

King also introduces the integral pairing between $\lambda$ and $\chi$, defined by $\langle \chi,\lambda \rangle=m$ if $\chi(\lambda(t))=t^m$. 

\begin{proposition}[Mumford's Numerical Criterion]
A point $x\in X$ is $\chi$-semistable if and only if $\chi(\Delta)=\{ 1\}$ and every $\lambda$  
for which $\displaystyle{\lim_{t\rightarrow 0}} \lambda(t).x$ exists, then 
$\langle \chi, \lambda \rangle\geq 0$. 
Such a point is $\chi$-stable if and only if $\lambda$ is a $1$-parameter subgroup of $G$ for which $\displaystyle{\lim_{t\rightarrow 0}}\lambda(t).x$ exists and $\langle \chi,\lambda \rangle =0$, 
then $\lambda$ is in $\Delta$. 
\end{proposition} 

\begin{proposition} 
\mbox{} 
\begin{enumerate} 
\item An orbit $G.x$ is closed in $X^{\chi-ss}$ if and only if for each $\lambda$ with $\langle \chi,\lambda \rangle=0$ and $\displaystyle{\lim_{t\rightarrow 0}}\lambda(t).x$ exists, then 
$\displaystyle{\lim_{t\rightarrow 0}}\lambda(t).x\in G.x$. 
\item For $x,y\in X^{\chi-ss}$, $x\sim y$ if and only if there exist $\lambda_1,\lambda_2:\mathbb{C}^*\rightarrow G$ such that 
$\langle \chi,\lambda_1 \rangle = \langle \chi,\lambda_2 \rangle =0$ and 
$\displaystyle{\lim_{t\rightarrow 0}} \lambda_1(t).x$, 
$\displaystyle{\lim_{t\rightarrow 0}} \lambda_2(t).y\in G.x$.  
\end{enumerate} 
\end{proposition}

Since it is possible to approximate $X/G$ in more than one way, another interesting aspect is to study polarization and wall-crossing, which are interesting in their own right. Such problems are known as variational GIT, or VGIT (cf. \cite{MR1659282} and \cite{MR1333296}). 

\subsection{Hilbert schemes}\label{Section:hilbert-schemes}

Some references for this section are  
\cite{MR2223407}, 
\cite{MR2073194}, 
\cite{MR2221028}, 
\cite{MR2592504}, and 
\cite{MR0463157}.
{\em Hilbert scheme} is a scheme which is a parameter (moduli) space for the closed subschemes of some projective scheme. More precisely, 
let $M$ be a finitely generated graded module over a polynomial ring $R= \mathbb{C}[x_0,\ldots, x_n]$. Then the  
{\em Hilbert function}  
is defined as 
$H_M(t)= \dim_{\mathbb{C}}M_t$, 
where $\dim_{\mathbb{C}}$ means the dimension of $M_t$ as a vector space over $\mathbb{C}$, and it is a well-known fact that if $M$ is a finitely generated 
$R$-module, then a unique polynomial $P_M$ exists such that for all $t>\!\!>0$,
$H_M(t)=P_M(t)$. $P_M$ is called the 
{\em Hilbert polynomial} of $M$.

Let $p(t)\in \mathbb{Q}[t]$ be a polynomial such that $p(\mathbb{Z})\subseteq \mathbb{Z}$. 
Let $n>0$ and let $X\subseteq \mathbb{P}^n$ be a closed subscheme with Hilbert polynomial $p(t)=P_X(t)$. 

\begin{theorem} 
A projective scheme $\Hilb_p$ exists which represents flat families of subschemes of $\mathbb{P}^n$ with Hilbert polynomial $p$. 
\end{theorem}

Specifying a map $\phi:R\longrightarrow \Hilb_p$ is equivalent to giving a flat family $\pi:\mathcal{X}\rightarrow R$ of closed subschemes in $\mathbb{P}^n$ with Hilbert polynomial $p$. Thus there exists a universal flat family 
$\mathcal{U}\longrightarrow \Hilb_p$  
such that $\mathcal{X}$ $=$ $\mathcal{U}\times_{\Hilb_p} R$. 
It may be worth mentioning that the Hilbert scheme of closed schemes in a fixed subscheme of $\mathbb{P}^n$ with Hilbert polynomial $p$ exists, which is the following proposition. 

\begin{proposition} 
For a fixed closed subscheme $W\subseteq \mathbb{P}^n$, 
there exists a projective scheme $\Hilb_{p,W}$ which represents flat families of subschemes of $W$ with Hilbert polynomial $p$. 
\end{proposition}
 
 

We will now restrict to the Hilbert scheme parametrizing $n$ points in $X$; 
the Hilbert polynomial of such subscheme $Z$ is a constant polynomial since the degree of the Hilbert polynomial of a variety is equal to the dimension of the variety, i.e., $P_Z(t)=n$. 
We will write the parameter scheme as $X^{[n]}:=\Hilb_{\mathbb{C}}^n(X)$.    
The points of $X^{[n]}$ are $0$-dimensional subvarieties supported at at most $n$ points in $X$ satisfying 
$\dim H^0(Z,\mathcal{O}_Z)=n$. 

\begin{example} 
If $\{p_1,\ldots, p_n\}$ is the set of $n$ distinct closed points of $X$, then the sheaf $\mathcal{O}_Z$ is the direct sum of the skyscraper sheaves over each point $p_i$, where $1\leq i\leq n$. 
\end{example}

Next, we restrict to an open subset $X$ 
of a projective space and state a result of the Hilbert scheme of $n$ points on $X$. 

\begin{theorem} 
Let $X$ be irreducible, quasi-projective and smooth of dimension $d=1,2$. Then 
$X^{[n]}$ is smooth and irreducible of dimension $dn$. 
If $X$ is connected, then $X^{[n]}$ is connected. 
\end{theorem}

Let $M_n:=M_n(\mathbb{C})$ be the space of all $n\times n$ complex matrices. 
The following is a result by Nakajima (\cite{MR1711344}, Theorem 1.9). 

\begin{theorem} 
Let $GL_n(\mathbb{C})$ act on $M_n\times M_n\times \mathbb{C}^n$ by 
$g.(r,s,i)=(grg^{-1},gsg^{-1},gi)$. 
Then 
\[ 
(\mathbb{C}^2)^{[n]} \cong \{(r,s,i)\in M_n\times M_n\times \mathbb{C}^n :
[r,s]=0 \mbox{ and }\forall S\subsetneq \mathbb{C}^n \mbox{ s.t. }i\in S, rS\not\subseteq S 
\mbox{ and } sS\not\subseteq S \}/GL_n(\mathbb{C}). 
\] 
\end{theorem}

\begin{proof} 
 We will first construct a map from right to left.  
 Let 
 $\Phi_{r,s,i}:\mathbb{C}[z_1,z_2]\longrightarrow \mathbb{C}^n$ be the map $f\mapsto f(r,s)i$. 
 Consider $S=\im \Phi_{r,s,i}$. 
 Since  
 $rS\subseteq S$ and $sS\subseteq S$, $S=\mathbb{C}^n$. 
 So $\ker \Phi_{r,s,i}$ is an ideal of codimension $n$. 
 We send the triple $(r,s,i)$ to $\ker \Phi_{r,s,i}$, and we see that 
 $\ker\Phi_{r,s,i}$ is invariant under the $GL_n(\mathbb{C})$-action.  

 We will now construct a map from left to right. 
Since points of $(\mathbb{C}^2)^{[n]}$ are ideals $I\subseteq R=\mathbb{C}[z_1,z_2]$ such that 
$\dim_{\mathbb{C}}R/I=n$, we choose an isomorphism $R/I\cong \mathbb{C}^n$. 
Let $r\in M_n$ be the action of $z_1$, 
    $s\in M_n$ be the action of $z_2$, and 
    $i\in \mathbb{C}^n$ equal $1\in R/I$. 
    Since $R/I\cong \mathbb{C}^n$, any nonzero vector $i\in \mathbb{C}^n$ is cyclic and since 
    $z_1$ and $z_2$ commute, $[r,s]=0$. 
    Since our choices of $r$, $s$, and $i$ are unique up to the $GL_n(\mathbb{C})$-action, this gives us a map from left to right, 
    which, indeed, induces a bijection. 
%
%
%
\end{proof} 

Next, we will give an explicit correspondence between $(r,s,i)$ and its represented ideal in $\mathbb{C}[z_1,z_2]$. 
Suppose the matrices $r$ and $s$ have distinct eigenvalues. Then they are simultaneously diagonalizable by some $g\in GL_n(\mathbb{C})$. 
Let $i=(1,\ldots,1)\in \mathbb{C}^n$. 
Suppose the eigenvalues of $r$ are denoted as $\lambda_1,\ldots, \lambda_n$ 
and the eigenvalues of $s$ are denoted as $\mu_1,\ldots, \mu_n$. 
Then the corresponding ideal consists of all polynomials $f\in \mathbb{C}[z_1, z_2]$
such that $f(\lambda_i,\mu_i)=0$ for all $1\leq i\leq n$. 
Thus this closed subscheme represents $n$ distinct points on $\mathbb{C}^2$. 
In general, since $rs=sr$, we can simultaneously put $r$ and $s$ into upper triangular matrices under the $GL_n(\mathbb{C})$-adjoint action. 
Let $\lambda_1,\ldots, \lambda_n$ be the diagonal entries of $r$ and 
$\mu_1,\ldots, \mu_n$ be the diagonal entries of $s$. 
Let $S_n$ be the symmetric group of $n$ letters. 

\begin{theorem}[Hilbert-Chow morphism, \cite{MR0237496}]
The morphism 
\[ 
(\mathbb{C}^2)^{[n]}\stackrel{\pi}{\longrightarrow}S^n\mathbb{C}^2=\mathbb{C}^2\times \cdots \times \mathbb{C}^2/S_n \mbox{ where } 
(r,s,i)\mapsto \{(\lambda_1,\mu_1),\ldots, (\lambda_n,\mu_n) \}
\] 
is proper, surjective, and a resolution of singularities of the symmetric product. 
\end{theorem}

\section{Hamiltonian reduction construction in symplectic geometry}\label{section:ham-reduction-symplectic-geometry}

Let $G$ be a complex Lie group and suppose $G$ acts on a smooth, complex symplectic variety 
$M=(M,\omega)$ preserving the symplectic form, i.e., 
$\omega(gx,gy)=\omega(x,y)$ for all $g\in G$ and for all $x,y\in T_mM$, $m\in M$. 
Taking the derivative of the $G$-action, we get 
\[ 
\mathfrak{g}=\lie(G)\longrightarrow \vect(M). 
\] 

\begin{definition} 
A symplectic $G$-action is  
{\em Hamiltonian}  
if given a $G$-action on $M$ and a map $\partial$
\[ 
\xymatrix@-1pc{ 
 & & \mathfrak{g} \ar[dd] \\ 
 & & \\ 
\mathbb{C}[M] \ar[rr]^{\partial} & & \vect(M), \\ 
}
\] 
where $\partial:f\mapsto \xi_f$, the vector field at $f$, 
there exists a Lie algebra homomorphism 
$H:\mathfrak{g}\rightarrow \mathbb{C}[M]$, i.e., 
$H([X,Y])=\{HX,HY \}$,  
where $\{\bullet,\bullet \}$ is the Poisson bracket on $\mathbb{C}[M]$,  
 such that the above diagram commutes. 
\end{definition}

In the case that the Hamiltonian $G$-action exists, we call $H$ the 
{\em comoment map} 
and $H(X)$ the Hamiltonian for the vector field $X^{\#}$ on $M$ generated by the $1$-parameter subgroup $\{ \exp(tX):t\in \mathbb{R}\}\subseteq G$.

\begin{definition}
The map 
$\mu:M\rightarrow \mathfrak{g}^*$ such that 
$\mu(m):\mathfrak{g}\rightarrow \mathbb{C}$ with 
$\mu(m)(X)=H(X)(m)$ is called the {\em moment map}. 
\end{definition}

We explain another (more explicit) way to construct moment maps for symplectic varieties with a $G$-action and we include some details from symplectic geometry in \cite{MR1853077} for completeness. 
Let $(M,\omega)$ be a symplectic variety with a Lie group $G$-action. 
Let $\psi:G\rightarrow \Sympl(M,\omega)$ be a group homomorphism such that the evaluation map $\ev_{\psi}(g,p)=\psi_g(p)$ is smooth. 
For each  $X\in \mathfrak{g}$, let $\mu:M\rightarrow \mathfrak{g}^*$ and let 
$\mu^X:M\rightarrow \mathbb{R}$ such that $\mu^X(p):=\langle \mu(p),X\rangle$, which is the component of $\mu$ along $X$. 
Let $X^{\#}$ be the vector field on $M$ generated by the $1$-parameter subgroup given as above. 

\begin{definition} 
The action $\psi$ is  
{\em Hamiltonian} 
if there exists a map $\mu:M\rightarrow \mathfrak{g}^*$ such that 
\begin{enumerate}
\item for each $X\in \mathfrak{g}$, $d\mu^X=d\langle \mu,X \rangle =\iota_{X^{\#}}\omega$, 
\item $\mu$ is equivariant with respect to the given action $\psi$ of $G$ on $M$ and the coadjoint action $\Ad^*$ of $G$ on $\mathfrak{g}^*$, i.e., we have 
$\mu\circ\psi_g=\Ad_g^*\circ \mu$ for all $g\in G$. 
\end{enumerate} 
\end{definition}

We say $\mu^X$ is the {\em Hamiltonian function} for the vector field $X^{\#}$. 
The function $H(X)$ from earlier this section is the function $\mu^X$. 
The quadruple $(M,\omega, G,\mu)$ is called a {\em Hamiltonian $G$-space} 
and $\mu$ is called a {\em moment map}.  
If $G$ is connected, then $\mu$ is $G$-equivariant.  
Note that for all $g\in G$,  $\langle \mu(g.p),X\rangle =\langle \mu(p),g^{-1}.X \rangle$. 
More precisely, any $X$ in the Lie algebra $\mathfrak{g}=\lie(G)$ gives rise to a vector field $X^{\#}$ on $M$ by 
\[ 
(X.f)(p)= \dfrac{d}{dt}f(\exp(-tX).p)|_{t=0}. 
\] 
That is, the flow of the vector field $X^{\#}$ is the $1$-parameter group of global diffeomorphisms 
$p\mapsto \exp(-tX).p$, 
where we make the map $X\mapsto X^{\#}$ a Lie algebra homomorphism from $\mathfrak{g}$  to the Lie algebra of vector fields on $M$: $[X,Y]^{\#}=[X^{\#},Y^{\#}]$. 
%
%
We will give two examples connecting our moment map definition to linear and angular momentum in classical mechanics. 
\begin{example}\label{example:linear-momentum}
Let $a=(a_1,a_2,a_3),x=(x_1,x_2,x_3)\in \mathbb{R}^3$. 
Let $\mathbb{R}^3$ act on $\mathbb{R}^3$ via the linear translation: $a.x=x+a$.
The $\mathbb{R}^3$-action is induced to $T^*(\mathbb{R}^3)\simeq \mathbb{R}^6$ by 
$a.(x,y)=(x+a,y)$, where $y=(y_1,y_2,y_3)\in \mathbb{R}^3$. 
Let $\displaystyle{\omega=\sum_{i=1}^3 dx_i\wedge dy_i}$ and 
for notational purposes, let $X = X^{\#} = a_1\dfrac{\partial}{\partial x_1}+a_2\dfrac{\partial}{\partial x_2}
+a_3\dfrac{\partial}{\partial x_3}$. 
Then 
\[
d\mu^X =\iota_X\omega = \omega(X)
=\sum_{i=1}^3 dx_i\wedge dy_i(\sum_{j=1}^3a_j\dfrac{\partial}{\partial x_j}) 
=-\sum_{i=1}^3 dy_i\wedge dx_i (\sum_{j=1}^3a_j\dfrac{\partial}{\partial x_j}) 
=-\sum_{i=1}^3 a_idy_i. 
\]  
So 
$\mu^X:\mathbb{R}^6\rightarrow \mathbb{R}^1$ is 
$\mu^X(x,y)$ $=$ 
$\displaystyle{-\sum_{i=1}^{3}a_iy_i=- a \circ y}$, which is the usual dot product between two vectors,   
and 
$\mu^X(x,y)=\langle \mu(x,y),a\rangle$ 
imply the moment map $\mu:T^*\mathbb{R}^3\rightarrow \mathbb{R}^3$ is $\mu(x,y)=-y$. 
\end{example}

\begin{example}\label{example:angular-momentum} 
Let $SO(3)=\{ A\in GL_3(\mathbb{R}):\det(A)=1, \:\: AA^t=\Inoindex_3 \}$ act on $\mathbb{R}^3$ by rotation. Then $\mathfrak{so}(3)=\{A\in \mathfrak{gl}_3(\mathbb{R}):A+A^t=0 \}$ is isomorphic to $\mathbb{R}^3$ via the identification
\[
A =\begin{pmatrix} 
 0   &-a_3 &a_2 \\ 
 a_3 & 0   & -a_1\\  
 -a_2& a_1 & 0  
\end{pmatrix} \mapsto a=(a_1,a_2,a_3).
\] 
Let $x=(x_1,x_2,x_3)\in \mathbb{R}^3$. Consider the left multiplication 
\[ 
Ax = \begin{pmatrix}  
 -a_3 x_2 + a_2 x_3 \\ 
 a_3 x_1 - a_1 x_3 \\ 
 -a_2 x_1 + a_1 x_2  
\end{pmatrix} = a\times x. 
\] 
This implies 
$(AB-BA)x = A(Bx)-B(Ax) 
= A(B\times x)-B(A\times x) 
=(A\times B)x$; 
thus, 
$[A,B]=A\times B$. 
Furthermore, the infinitesimal $SO(3)$-action on $\mathbb{R}^3$ is $g.x=g\times x$. 
So the Hamiltonian is 
\[ 
\mu^X(x,y) = \langle A\times x, y\rangle 
= \det(A,x,y) = \det(x,y,A) 
= \langle x\times y,A \rangle =\langle \mu(x,y),A\rangle, 
\] and 
$\mu:T^*\mathbb{R}^3\rightarrow \mathbb{R}^3\simeq \mathfrak{so}(3)^*$ is precisely 
$\mu(x,y)=x\times y$. 
\end{example}

We refer the reader to \cite{MR0213476} for more background on symplectic geometry and representation theory. 


\begin{proposition}  
Let $G$ be a Lie group and let $X$ be a $G$-variety.  
The $G$-action on $X$ induces a Hamiltonian action on $T^*X$.  
\end{proposition} 

Now, let us discuss the Hamiltonian reduction of a moment map. 
The Hamiltonian reduction is a procedure to construct quotient symplectic and Poisson varieties in the Hamiltonian category. It is a systematic method to get intuition on symmetric systems as well as to reduce the number of coordinates.  
Let $G$ be a Lie group and suppose $N$ is a $G$-variety. Differentiate the group action to get 
\[ 
\mathfrak{g}=\lie(G)\stackrel{a}{\longrightarrow}\vect(N), \hspace{4mm}
a(v)(p)=\dfrac{d}{dt}(\exp(tv).p)|_{t=0}, 
\] 
and then dualize $a$ to get the moment map 
\[ T^*N\stackrel{\mu}{\longrightarrow }\mathfrak{g}^*. 
\] 
The {\em Hamiltonian reduction} of $N$ by $G$ is defined to be $\mu^{-1}(0)/G$.

\begin{example}\label{example:ham-reduction-gan-ginzburg}
Let $G=GL_n(\mathbb{C})$ and let 
$N=\mathfrak{g}\times \mathbb{C}^n$ be a matrix variety, where $\mathfrak{g}=\lie(G)$.
Let $G$ act on $N$ via $g.(r,i)=(grg^{-1}, gi)$. 
The infinitesimal action is given by 
\[ \mathfrak{g}\stackrel{a}{\longrightarrow} \vect(\mathfrak{g}\times \mathbb{C}^n), \hspace{4mm} 
a(v)(r,i)=\dfrac{d}{dt}(\exp(tv).(r,i))|_{t=0} = ([v,r],vi) 
\] 
since 
\[ 
\begin{aligned} 
\dfrac{d}{dt}(\exp(tv).r)|_{t=0} &= \dfrac{d}{dt}(\exp(tv)r\exp(-tv))|_{t=0} 
=\dfrac{d}{dt}((\I+tv+\mathcal{O}(t^2))r(\I-tv+\mathcal{O}(t^2)))|_{t=0} \\ 
&=\dfrac{d}{dt}((r+tvr+\mathcal{O}(t^2)) (\I-tv+\mathcal{O}(t^2)))|_{t=0}  
=\dfrac{d}{dt}(r+tvr-trv+\mathcal{O}(t^2))|_{t=0} \\ 
&=vr-rv+\mathcal{O}(t)|_{t=0} =[v,r] \\ 
\end{aligned} 
\] 
and 
\[ 
\begin{aligned} 
\dfrac{d}{dt}(\exp(tv).i)|_{t=0} &=\dfrac{d}{dt}((\I+tv+\mathcal{O}(t^2)).i)|_{t=0} 
=\dfrac{d}{dt}(i+tvi+\mathcal{O}(t^2))|_{t=0} 
=vi+\mathcal{O}(t)|_{t=0} = vi. 
\end{aligned}
\] 
Then identifying  
$\mathfrak{g}^*\cong \mathfrak{g}$ via $\mathfrak{g}\times \mathfrak{g}\rightarrow \mathbb{C}$, $(A,B)\mapsto \tr(AB)$,  
the moment map $\mu:= a^*$ is 
\[ 
T^*(\mathfrak{g}\times \mathbb{C}^n)\cong \mathfrak{g}\times \mathfrak{g}^* \times \mathbb{C}^n\times (\mathbb{C}^n)^*\stackrel{\mu}{\longrightarrow}\mathfrak{g}^*\cong \mathfrak{g},  \hspace{4mm} (r,s,i,j)\mapsto [r,s]+ij.
\]    
\end{example}

\begin{theorem}[Gan-Ginzburg, \cite{MR2210660}, Theorem 1.1.2]\label{theorem:almost-commuting-vars}
The subscheme $\mu^{-1}(0)$ in Example~\ref{example:ham-reduction-gan-ginzburg} is a complete intersection with $n+1$ irreducible components; $\mu^{-1}(0)$ is reduced of dimension $n^2+2n$.  
 \end{theorem}
 
The subscheme given in Theorem~\ref{theorem:almost-commuting-vars} is called an almost-commuting variety.   
An important result connecting Sections \ref{subsection:geometric-invariant-theory-introduction}, \ref{Section:hilbert-schemes}, and \ref{section:ham-reduction-symplectic-geometry} is the following theorem in \cite{MR1711344}. 

\begin{theorem}[Nakajima, Theorem 3.24 and Exercise 3.26]\label{theorem:nakajima-git-and-affine-quotient-identification}
Let $\mu$ be the moment map given in Example~\ref{example:ham-reduction-gan-ginzburg}. 
Then we have 
\[ 
(\mathbb{C}^2)^{[n]}\cong \mu^{-1}(0)/\!\!/_{\det}GL_n(\mathbb{C})\cong 
\mu^{-1}(0)/\!\!/_{\det^{-1}}GL_n(\mathbb{C})
\mbox{ and } 
S^n\mathbb{C}^2\cong \mu^{-1}(0)/\!\!/GL_n(\mathbb{C}). 
\]   
\end{theorem}

We also refer the reader to Remark~\ref{remark:hilbert-scheme-of-n-points-symmetric-product} for computational details when $n=2$. 
In \cite{MR1834739}, Crawley-Boevey extended Theorem~\ref{theorem:almost-commuting-vars} to all of Nakajima's quiver varieties.   
Definition~\ref{definition:double-quiver-classical-setting}  
and lecture notes by Ginzburg  \cite{Ginzburg-Nakajima-quivers} thoroughly discuss the notion of doubling quiver varieties, which is crucial in the construction of Nakajima quiver varieties.

\section{Grothendieck-Springer resolution}\label{section:GS-resolution-intro} 

The Grothendieck-Springer resolution plays one of the fundamental roles in geometric representation theory (cf. \cite{MR2838836}).

Let $G$ be a complex reductive group, where $\lie(G)=\mathfrak{g}$, and $W=N(T)/Z(T)$ be its corresponding Weyl group, where $T$ is any torus in $G$. 
Let $G/B$ be the variety of Borel subalgebras $\mathfrak{b}\subseteq \mathfrak{g}$. 
For each $\mathfrak{b}$, let $\mathfrak{u}\subseteq \mathfrak{b}$ be $\lie(U)$, where $U\subseteq B$ is the maximal unipotent subgroup of $B$. 
Let $\mathfrak{h}=\mathfrak{b}/\mathfrak{u}$ be the Cartan subalgebra. 
Let $\widetilde{\mathfrak{g}}=\{ (x,\mathfrak{b})\in \mathfrak{g}\times G/B:x\in \mathfrak{b}\}$ 
be the Grothendieck-Springer space of pairs and let 
$\widetilde{\mathcal{N}}=\{(x,\mathfrak{b})\in \mathcal{N}\times G/B:x\in \mathfrak{b}\}$ 
be the Springer space of pairs, where $\mathcal{N}\subseteq \mathfrak{g}$ is the nilpotent cone, i.e., $\mathcal{N}$ is the variety of nilpotent elements in $\mathfrak{g}$. 
For example, if $A\in \mathfrak{g}=\mathfrak{gl}_n$ is a nilpotent matrix, then the characteristic polynomial of $A$ is $p_A(t) = \det(t\I-A)=t^n$. 
It follows that we have the natural maps 
\[ 
\xymatrix@-1pc{ 
& & & T^*(G/B)\cong \widetilde{\mathcal{N}}\ar[ddlll]_p \ar[dd]^{\mu_{\widetilde{\mathcal{N}}}}  \ar@{^{(}->}[rrr]^{\widetilde{i}}& & & \widetilde{\mathfrak{g}}\ar[dd]^{\mu_{\widetilde{\mathfrak{g}}}} \ar[rrr]^{\widetilde{q}} & & & \mathfrak{h} \ar[dd]^h \\ 
& & &   & & &   & & & \\ 
G/B& & & \mathcal{N} \ar@{->>}[dd] \ar@{^{(}->}[rrr]^i  & & & \mathfrak{g}\ar@{->>}[dd] \ar[rrr]^q  & & &\mathfrak{h}/\!\!/W, \\ 
& & &   & & &   & & &   \\ 
& & & \{ 0\}=\mathcal{N}/\!\!/G \ar@{^{(}->}[rrr]^0 & & & \mathfrak{g}/\!\!/G \ar[uurrr]_{\cong} & & & \\
} 
\] 
where $p,\mu_{\widetilde{\mathcal{N}}},\mu_{\widetilde{\mathfrak{g}}}$ are projections, 
$h$ is a geometric quotient, 
$\mu_{\widetilde{\mathcal{N}}},\mu_{\widetilde{\mathfrak{g}}}$ are moment maps, where 
$\mu_{\widetilde{\mathcal{N}}}, \mu_{\widetilde{\mathfrak{g}}}:(x,\mathfrak{b})\mapsto x$, 
the maps $i,\widetilde{i}$ are inclusions, $q$ is the affine adjoint quotient map, and 
$\widetilde{q}$ maps the pair $(x,\mathfrak{b})$ to the class of $x$ in $\mathfrak{b}/\mathfrak{u}$. 
Note that $\widetilde{\mathcal{N}}$ and $\widetilde{\mathfrak{g}}$ are smooth and projective and $\mu_{\widetilde{\mathcal{N}}},\mu_{\widetilde{\mathfrak{g}}}$ are proper, with $\mu_{\widetilde{\mathfrak{g}}}$ being generically finite of degree $|W|$.   
Note that in the case $G=GL_n(\mathbb{C})$, 
then $\mathfrak{g}/\!\!/G = \spec (\mathbb{C}[\Tr(A),\ldots, \det(A)])\cong \mathbb{C}^n$, where $A$ is a general matrix in $\mathfrak{g}$.

Now let $\mathfrak{g}^{rss}\subseteq \mathfrak{g}$ be the regular semisimple locus and let $\mathfrak{h}^r\subseteq \mathfrak{h}$ denote the $W$-regular locus. If we let $\mathfrak{g}$ to be the space of all $n\times n$ matrices, then $\mathfrak{g}^{rss}$ is the subset of these matrices with pairwise distinct eigenvalues. 
Then the above diagram restricts as the following: 
\[ 
\xymatrix@-1pc{
\widetilde{\mathfrak{g}}^{rss} \ar[dd]_{\mu_{\widetilde{\mathfrak{g}}}} \ar[rrr]^{\widetilde{q}} & & & \mathfrak{h}^r \ar[dd]^h \\ 
& & & \\ 
\mathfrak{g}^{rss} \ar@{->>}[dd] \ar[rrr]^q & & & \mathfrak{h}^r/\!\!/W. \\ 
& & & \\ 
\mathfrak{g}^{rss}/\!\!/G \ar[uurrr]^{\cong} & & & \\ 
}
\] 

Note that $\widetilde{\mathfrak{g}}\stackrel{\mu_{\widetilde{\mathfrak{g}}}}{\longrightarrow}\mathfrak{g}$ factors as follows: 
\[ 
\xymatrix@-1pc{
\widetilde{\mathfrak{g}}\ar[rr]^{\phi} & & \mathfrak{g} \times_{\mathfrak{h}/W}\mathfrak{h}\ar[rr]^{\psi}&& \mathfrak{g}, \\ 
}
\] 
where $\phi$ is a resolution of singularities and $\psi$ is a finite projection. 
Now let $\mathfrak{g}^{reg}\subseteq \mathfrak{g}$ be the subspace of regular elements.
Centralizers of $\mathfrak{g}^{reg}$ have minimum possible dimension and they are not necessarily semisimple   
(if $\mathfrak{g}=\mathfrak{gl}_n$, then a matrix is regular if and only if its Jordan canonical form contains a single Jordan block for each distinct eigenvalue).  
Let  
$\widetilde{\mathfrak{g}}^{reg}:=\mu_{\widetilde{\mathfrak{g}}}^{-1}(\mathfrak{g}^{reg})$.
Then $\phi$ induces an isomorphism 
$\widetilde{\mathfrak{g}}^{reg}\cong \mathfrak{g}^{reg}\times_{\mathfrak{h}/W}\mathfrak{h}$. 
Note that $W$ acts on $\widetilde{\mathfrak{g}}^{reg}$ via its action on the second factor. 

Lemma~\ref{lemma:CG-G-S-resolution-without-modding-by-G} 
is also given in Chriss and Ginzburg (\cite{MR2838836}, Corollary 3.1.33);   
we will discuss this identification further in Section~\ref{section:cotangent-bundle-of-GS-resolutions}.
\begin{lemma}\label{lemma:CG-G-S-resolution-without-modding-by-G}
We have $\widetilde{\mathfrak{g}} \cong (G\times \mathfrak{b})/ B$, where 
the $B$-action on the pair $G\times \mathfrak{b}$ is   
$(g,X).b=(gb^{-1}, \Ad_b(X))$.  
\end{lemma}

\begin{proof}  
There are natural projection maps 
\[ 
\xymatrix@-1pc{ 
& & & \ar[llldd]_{\pi_1}^{\mbox{flat}} \widetilde{\mathfrak{g}} \ar[rrrdd]^{\pi_2}_{\mbox{proj}}& & & \\ 
& & & & & & \\  
G/B & & &  & & & \mathfrak{g} \\  
}
\] 
where both $\pi_1$ and $\pi_2$ are $G$-equivariant. 
The fiber of $x$ via $\pi_2$ consist of the set of all flags $\mathfrak{b}'$ satisfying $x\in \mathfrak{b}'$. 
Since any two Borels are conjugate under the $G$-action, 
$x$ is in $\mathfrak{b}'$ for some $\mathfrak{b}' \subseteq \mathfrak{g}$. 
Fibers of $\pi_1$ consist of Borels preserving a flag in $G/B$, 
so $\pi_1$ is flat with fiber $\mathfrak{b}'$. The $B$-action on $G\times \mathfrak{b}$ is free since 
$gb^{-1}=g$ implies $b$ is the identity element in $B$. 
So the map $G\times \mathfrak{b}/B\rightarrow \widetilde{\mathfrak{g}}$, 
where $(g,x)\mapsto (gxg^{-1}, (g.B)/B)$, is $G$-equivariant and is an isomorphism. 
\end{proof}

Finally, we end this section by proving $\widetilde{\mathfrak{g}}/ G\cong \mathfrak{b}/ B$. 
\begin{lemma}\label{lemma:g-s-resolution-and-B-action-on-borel}   
We have 
$\widetilde{\mathfrak{g}}/G\cong \mathfrak{b}/B$ as orbit spaces. 
\end{lemma}

\begin{proof} 
Consider the map 
$(G\times \mathfrak{b})/B=\widetilde{\mathfrak{g}} \rightarrow \mathfrak{b}/B$ where 
$(g,x)\mapsto gxg^{-1}$.  
The map is well-defined up to the $B$-conjugation action. 
Since the map is $G$-equivariant, it descends to an isomorphism 
$\widetilde{\mathfrak{g}}/G\cong \mathfrak{b}/B$ as orbit spaces. 
\end{proof}

\section{The cotangent bundle of the enhanced Grothendieck-Springer resolution}\label{section:cotangent-bundle-of-GS-resolutions} 

The details of the construction of the cotangent bundle of the enhanced Grothendieck-Springer resolution are given in \cite{Nevins-GSresolutions}. 
In this section, we will go through the details as an additional source for the construction. 

If $V$ is a linear space and  $\mathcal{N}=\mathcal{N}(V)$ is the set of nilpotent linear endomorphisms of $V$, then the 
{\em enhanced nilpotent cone} of $V$ is the set $\mathcal{N}\times V$; here, $GL(V)$ acts on $\mathcal{N}\times V$.  
In the similar manner, 
if $V=\mathbb{C}^n$ is a linear space and $\mathfrak{b}$ is the set of linear endomorphisms $\phi$  
such that $\phi(\mathbb{C}^k)\subseteq \mathbb{C}^k$ for all $1\leq k\leq n$, then the 
{\em enhanced Grothendieck-Springer resolution} of $V$ is $\mathfrak{b}\times V$; here, $B$ acts on 
$\mathfrak{b}\times V$ as a change-of-basis. 
We refer to Syu Kato's work for an extensive treatment on enhanced nilpotent representations. 

We refer to Section~\ref{section:GS-resolution-intro} for the relationship between the Springer resolution and the Grothendieck-Springer resolution.  
A {\em symplectic resolution}  $\widetilde{X}\rightarrow X$ 
is a resolution of singularities such that the symplectic $2$-form on $X$ 
pulls back to a nondegenerate and closed symplectic $2$-form on all of $\widetilde{X}$. 
The map $\widetilde{\mathfrak{g}}\stackrel{\mu_{\widetilde{\mathfrak{g}}}}{\longrightarrow}\mathfrak{g}^* 
\stackrel{\Tr}{\cong} \mathfrak{g}$ in Section~\ref{section:GS-resolution-intro} is a symplectic resolution which encapsulates the cotangent bundle 
$T^*(G/B)= \widetilde{\mathcal{N}} \subseteq \widetilde{\mathfrak{g}}$ 
of the flag variety. 
The map $\mu_{\widetilde{\mathfrak{g}}}$ is also a (projective) moment map for the $G$-action on $\widetilde{\mathfrak{g}}$ with each fiber of $x\in \mathfrak{b}'$ being a flag variety satisfying $x$. Furthermore, the Springer resolution $\widetilde{\mathcal{N}} \rightarrow \mathcal{N}$ 
of the nilpotent cone in a semisimple Lie algebra 
is a semi-small morphism while the Grothendieck-Springer
simultaneous resolution is a small morphism. 
Here, we use definitions of small and semi-small in the sense according to Goresky-MacPherson;  
a proper morphism $f:X\rightarrow Y$ is 
{\em semi-small} if $\dim (X\times_Y X)\leq \dim X$ 
and  
{\em small}  if 
$\dim (X\times_Y X\setminus \Delta (X))<\dim X$, where $\Delta (X)$ is the diagonal of $X$ in $X\times_Y X$. 
There are other interesting properties of these examples which have had a significant impact 
in representation theory, but we will not exhaustively list all of them. 
We begin with some propositions in symplectic geometry.  
 
%
%

\begin{proposition}\label{proposition:group-action-on-product-space}
Let $G$ be a Lie group with Hamiltonian action on symplectic manifolds 
$(M_1,\omega_1)$ and $(M_2,\omega_2)$ with moment maps 
$\mu_1:M_1\rightarrow \mathfrak{g}^*$ and 
$\mu_2:M_2\rightarrow \mathfrak{g}^*$. 
Then the diagonal action of $G$ on $M_1\times M_2$ is Hamiltonian with $\mu:M_1\times M_2 \rightarrow \mathfrak{g}^*$, 
where $\mu(p_1,p_2)=\mu_1(p_1)+\mu_2(p_2)$ 
for $p_i\in M_i$.  
\end{proposition}

Consider $M=T^*(G\times \mathfrak{b}\times \mathbb{C}^n)$ under the $G\times B$  action. We want to relate $M$ with 
$T^*(G\times_B \mathfrak{b}\times \mathbb{C}^n)=T^*(\widetilde{\mathfrak{g}}\times \mathbb{C}^n)$ under the $B$-action. 
Let $B$ act on $G\times \mathfrak{b}$ via $b.(g,r)=(gb^{-1},brb^{-1})$, 
Chriss and Ginzburg prove the identification 
$G\times_B \mathfrak{b}=\{(x,\mathfrak{b})\in \mathfrak{g}\times G/B:x\in \mathfrak{g} \}$ (\cite{MR2838836}, Corollary 3.1.33). 
Let $(\mathbb{C}^n)^*$ be the dual of $\mathbb{C}^n$. 
That is, if we think of  $i\in \mathbb{C}^n$ as a column vector, then we should think of $j\in (\mathbb{C}^n)^*$ as a row vector. 
Let $G\times B$ act on $T^*(G\times \mathfrak{b}\times \mathbb{C}^n)$ in the following way:  
for $g\in G$ and  
$(g',\theta,r,s,i,j)\in G\times \mathfrak{g}\times\mathfrak{b}\times \mathfrak{b}^*\times \mathbb{C}^n \times (\mathbb{C}^n)^*$, 
\begin{equation}\label{equation:constructing-gs-resolution-g-action}
g.(g',\theta,r,s,i,j) = (g'g^{-1},g\theta g^{-1},r,s,gi,jg^{-1})
\end{equation}
and for $b\in B$ and $(g',\theta,r,s,i,j)$, 
\begin{equation}\label{equation:constructing-gs-resolution-b-action}  
b.(g',\theta,r,s,i,j) = (g'b^{-1},b\theta b^{-1},brb^{-1},bsb^{-1}, i,j). 
\end{equation}
From 
\eqref{equation:constructing-gs-resolution-g-action} and \eqref{equation:constructing-gs-resolution-b-action}, 
we have the following moment maps: 
\[ 
\begin{aligned}  
T^*(G)\cong G\times \mathfrak{g}^* \rightarrow \mathfrak{g}^*, &\hspace{4mm}(g',\theta)\mapsto \theta,  
		\hspace{6mm}
T^*(\mathbb{C}^n)\cong \mathbb{C}^n \times (\mathbb{C}^n)^*\rightarrow \mathfrak{g}^*, \hspace{4mm}(i,j)\mapsto -ij, \\ 
T^*(G)\cong G\times \mathfrak{g}^*\rightarrow \mathfrak{b}^*, &\hspace{4mm}(g',\theta)\mapsto \overline{\theta},  
		\hspace{6mm}
T^*(\mathfrak{b})\cong \mathfrak{b}\times \mathfrak{b}^*\rightarrow \mathfrak{b}^*, \hspace{4mm}(r,s)\mapsto [r,s], \\ 
\end{aligned} 
\] 
where $\overline{v}:\mathfrak{g}^*\cong \mathfrak{g} \rightarrow \mathfrak{b}^*$ is a projection. 
By Proposition~\ref{proposition:group-action-on-product-space}, 
$T^*(G\times \mathfrak{b}\times \mathbb{C}^n)\rightarrow \mathfrak{g}^*\times \mathfrak{b}^*$ is the map 
\[ 
(g',\theta,r,s,i,j)\mapsto (\theta -ij, [r,s] + \overline{\theta}).  
\]  
We will explain how to obtain some of the above moment maps. 
Let $G$ act on $T^*(G)=G\times \mathfrak{g}^*\stackrel{1\times \Tr}{\cong} G\times \mathfrak{g}$  
via  
$g.(g',\theta)=(g'g^{-1},g\theta g^{-1})$.  
Let $g_i$ be the coordinates of $g'$,   
$\theta_i$ be the coordinates of $\theta$, and  
let $\omega=-\displaystyle{\sum_{i} dg_i\wedge d\theta_i}=\displaystyle{\sum_i d\theta_i\wedge dg_i}$ be a symplectic form on $T^*(G)$. 
Let $X=\displaystyle{\sum_i \partial/\partial g_i}$. 
Then 
\[
d\mu^X =\iota_X\omega = \sum_i d\theta_i = \omega(X). 
\] 
So $\mu^X = \displaystyle{\sum_i\theta_i}.$ 
Viewing $\theta\in \mathfrak{g}$ as a vector in $\mathbb{C}^{n^2}$, 
$\displaystyle{\sum_i\theta_i} = \theta\circ (1,\ldots, 1)$. 
Since $\mu^X(g',\theta)=\langle \mu(g',\theta),1\rangle$
where $\mu:T^*G\rightarrow \mathfrak{g}^*\stackrel{\Tr}{\cong}\mathfrak{g}$, 
$\mu(g',\theta)=\theta$.

Suppose $G$ acts on $T^*\mathbb{C}^n$ via $g.(i,j)=(gi,jg^{-1})$.
Let $\omega = 
\displaystyle{\sum_{i=1}^n dx_i\wedge dy_i}$, 
where $(x_i)$ are the coordinates of $\mathbb{C}^n$ and $(y_i)$ are the coordinates of $(\mathbb{C}^n)^*$.
We will explicitly write down the moment map for $n=2$.
Let $X=E_{11}\in \mathfrak{gl}_2$, which is a $2\times 2$ matrix with $1$ in the $(1,1)$ entry and $0$ elsewhere. 
Since 
\[ 
\begin{aligned} 
\dfrac{d}{dt}(\exp(t 
\begin{pmatrix}   
 1& 0\\ 
 0& 0 
\end{pmatrix}
).(i,j))|_{t=0}
&= \left( 
\begin{pmatrix}  
 1& 0\\ 
 0& 0 
\end{pmatrix}
\begin{pmatrix} 
 x_1\\ 
 x_2 
\end{pmatrix}, 
- \begin{pmatrix}  
 y_1& y_2  
\end{pmatrix}
\begin{pmatrix}  
1 &0 \\ 
0 &0  
\end{pmatrix}
\right) \\ 
&= \left( 
\begin{pmatrix}   
 x_1\\ 
 0   
\end{pmatrix}, 
\begin{pmatrix} 
 -y_1& 0  
\end{pmatrix} 
\right) = x_1 \dfrac{\partial}{\partial x_1} - y_1\dfrac{\partial}{\partial y_1}, 
\end{aligned}
\] 
let $X=  x_1 \dfrac{\partial}{\partial x_1} - y_1\dfrac{\partial}{\partial y_1}$. 
Then 
\[ 
d\mu^X =\iota_X\omega 
= (dx_1 \wedge dy_1 + dx_2\wedge dy_2 )(x_1 \dfrac{\partial}{\partial x_1} - y_1\dfrac{\partial}{\partial y_1})
=- (x_1dy_1+y_1dx_1)
\] 
imples $\mu^X= - x_1y_1$. 
So 
$\mu^X
=\langle \mu(i,j),X\rangle
=\langle \mu(i,j),E_{11}\rangle= - x_1y_1$, 
and the $(1,1)$-entry of the moment map $\mu$ is $\mu(i,j)_{1,1}= - x_1y_1$.

Now let $X=E_{12}$, with $1$ in the $(1,2)$ entry and $0$ elsewhere. 
Then we have 
\[ 
\begin{aligned} 
\dfrac{d}{dt}(\exp(t 
\bordermatrix{
& &   \cr 
& 0& 1\cr 
& 0& 0\cr 
}
).(i,j))|_{t=0}
&= \left( 
\begin{pmatrix}  
 0& 1\\ 
 0& 0 
\end{pmatrix}
\begin{pmatrix}
 x_1 \\ 
 x_2 
\end{pmatrix}, 
- \begin{pmatrix} 
 y_1& y_2  
\end{pmatrix}
\begin{pmatrix} 
0 &1 \\ 
0 &0  
\end{pmatrix}
\right) \\ 
&= \left( 
\begin{pmatrix}  
 x_2\\ 
 0   
\end{pmatrix}, 
\begin{pmatrix}  
 0   & -y_1   
\end{pmatrix}  
\right) = x_2 \dfrac{\partial}{\partial x_1} -y_1\dfrac{\partial}{\partial y_2},   
\end{aligned}
\] 
which implies 
$d\mu^X=\iota_X\omega = (dx_1\wedge dy_1+dx_2\wedge dy_2)(x_2 \dfrac{\partial}{\partial x_1} -y_1\dfrac{\partial}{\partial y_2})=-(x_2dy_1+y_1dx_2)$. 
Thus we obtain the Hamiltonian $\mu^X=- x_2y_1$. 
Repeat the above calculation for $E_{22}$ to get 
$X= x_2 \dfrac{\partial}{\partial x_2} -y_2\dfrac{\partial}{\partial y_2}$. 
Since $d\mu^X =\iota_X\omega = -(x_2dy_2 +y_2dx_2)$, 
$\mu^X = - x_2y_2$. 
Finally, let $X=E_{21}\in \mathfrak{g}^*$ to get that the infinitesimal action is given in coordinates as 
$x_1 \dfrac{\partial}{\partial x_2} -y_2\dfrac{\partial}{\partial y_1}$. 
Then 
$d\mu^X =\iota_X\omega=-(x_1dy_2+y_2dx_1)$ implies $\mu^X = - x_1y_2$. 
Thus we get four Hamiltonian functions from $T^*(\mathbb{C}^2)\rightarrow \mathbb{C}$, which are  
$- x_1y_1, - x_1y_2,- x_2y_1$, and $- x_2y_2$. 
Since $\mu^X=\langle \mu(i,j),X\rangle$ and 
since the derivative 
$d\mu:TM\cong T^*M \rightarrow \mathfrak{g}^*$ 
of the moment map 
$\mu:M\rightarrow \mathfrak{g}^*$ 
is the transpose of the derivative 
$a:\mathfrak{g}\rightarrow \vect(M)$ 
of the $G$-action,   
the moment map is given as 
\[ 
\mu(i,j)=- \begin{pmatrix}
 x_1y_1& x_1y_2\\
 x_2y_1&x_2y_2 \\   
\end{pmatrix} = - ij. 
\] 

\subsection{Hamiltonian reduction in stages}\label{subsection:ham-reduction-in-stages-symplectic-geometry}

Next, we will discuss a concept known as reduction in stages in symplectic geometry. 
We will closely follow the manuscript \cite{Nevins-GSresolutions} so the reader may skip the rest of this section if desired to do so. 
Let $G\times \mathfrak{b}\times \mathbb{C}^n\stackrel{f}{\longrightarrow} G\times_B\mathfrak{b}\times \mathbb{C}^n$ 
be a principle $B$-bundle, i.e. $B$ acts freely and transitively on the fibers of $G\times \mathfrak{b}\times \mathbb{C}^n$. 
The map $f$ preserves the fibers of 
$G\times \mathfrak{b}\times \mathbb{C}^n$, i.e., 
if $y\in (G\times \mathfrak{b}\times \mathbb{C}^n)_x$ is a point in the fiber of $x$, then $b.y\in (G\times \mathfrak{b}\times \mathbb{C}^n)_x$ still lies in the fiber of $x$ for all $b\in B$. 
Thus each fiber is homeomorphic to $B$ and 
$(G\times \mathfrak{b}\times \mathbb{C}^n)/B \cong G\times_B \mathfrak{b}\times \mathbb{C}^n$.  
Next, note that $f$ has a $G$-action which commutes with the $B$-action, i.e., the map from above 
\[
T^*(G\times \mathfrak{b}\times \mathbb{C}^n) \stackrel{\mu_{G\times B}}{\longrightarrow} 
\mathfrak{g}^*\times \mathfrak{b}^*, \hspace{4mm} 
(g',\theta,r,s,i,j)\mapsto (\theta - ij,[r,s] + \overline{\theta}), 
\] 
is $G\times B$-equivariant. 
Let 
$\psi_G:T^*(G\times \mathfrak{b}\times \mathbb{C}^n)\rightarrow \mathfrak{g}^*$ and 
$\psi_B:T^*(G\times \mathfrak{b}\times \mathbb{C}^n)\rightarrow \mathfrak{b}^*$  
be the projection maps obtained from $\mu_{G\times B}$.  
Consider $\psi_B^{-1}(0)\subseteq T^*(G\times \mathfrak{b}\times \mathbb{C}^n)$.  
Since $B$ acts freely on $\psi_B^{-1}(0)$,  
\[  
\psi_B^{-1}(0)/B\cong T^*(G\times \mathfrak{b}\times \mathbb{C}^n/B)\simeq T^*(G\times_B\mathfrak{b}\times \mathbb{C}^n) 
\]  
is a reduced space.  

Now consider $G$ acting on $\psi_B^{-1}(0)$, 
which commutes with the $B$-action. Since $G$ preserves the symplectic form, say $\omega$, 
on $T^*(G\times \mathfrak{b}\times \mathbb{C}^n)$, 
$G$ acts symplectically on the reduced space $(\psi_B^{-1}(0)/B,\omega')$, where $\omega'$ is the symplectic form on $\psi_B^{-1}(0)/B$. 
Since $B$ preserves $\psi_G:T^*(G\times \mathfrak{b}\times \mathbb{C}^n)\longrightarrow \mathfrak{g}^*$, 
$B$ preserves 
$\psi_G\circ \iota_1:\psi_B^{-1}(0)\rightarrow \mathfrak{g}^*$. 
So $\psi_G\circ \iota_1$ is constant on fibers of 
$\psi_B^{-1}(0)\stackrel{p_1}{\longrightarrow}\psi_B^{-1}(0)/B$. 
Thus, there exists a smooth map 
\[ 
\mu_G:\psi_B^{-1}(0)/B\simeq T^*(G\times_B\mathfrak{b}\times \mathbb{C}^n)\longrightarrow \mathfrak{g}^* 
\] 
such that 
$\mu_G\circ p_1=\psi_G\circ \iota_1:\psi_B^{-1}(0)\longrightarrow \mathfrak{g}^*$. 

\begin{proposition}
\mbox{} 
\begin{enumerate}
\item $\mu_G:T^*(G\times_B\mathfrak{b}\times \mathbb{C}^n)\rightarrow \mathfrak{g}^*$ is a moment map for $G$ acting on $(T^*(G\times_B\mathfrak{b}\times \mathbb{C}^n), \omega')$. 
\item Since $G$ acts freely on $\mu_{G\times B}^{-1}(0)$, $G$ acts freely on $\mu_G^{-1}(0)$ and a symplectomorphism 
$\mu_G^{-1}(0)/G\cong \mu_{G\times B}^{-1}(0)/G\times B$ exists. 
\end{enumerate}
\end{proposition}

Now, consider the reduction of $T^*(G\times \mathfrak{b}\times \mathbb{C}^n)$ by $G$ first: 
\[ 
\psi_G^{-1}(0)
= \{(g',\theta,r,s,i,j)\in T^*(G\times \mathfrak{b}\times \mathbb{C}^n):\theta-ij=0\} 
= \{(g',ij,r,s,i,j)\in T^*(G\times \mathfrak{b}\times \mathbb{C}^n) \}.  
\] 
Since the $G$-action is transitive, we will identify 
\[ 
T^*(\mathfrak{b}\times \mathbb{C}^n) 
\cong \{(1,ij,r,s,i,j)\in T^*(G\times \mathfrak{b}\times \mathbb{C}^n)\}\simeq \psi_G^{-1}(0)/G.  
\] 
As above, 
$\mu_B:T^*(\mathfrak{b}\times \mathbb{C}^n)\longrightarrow\mathfrak{b}^*$ is a moment map for $B$ acting on 
$(T^*(\mathfrak{b}\times \mathbb{C}^n),\omega_B)$, where $\omega_B$ is the symplectic form on $T^*(\mathfrak{b}\times \mathbb{C}^n)$. 
Secondly, since $B$ acts freely on 
$\mu_{G\times B}^{-1}(0)$, $B$ acts freely on $\mu_B^{-1}(0)$ and we have 
\[ 
\mu_B^{-1}(0)/B \cong \mu_{G\times B}^{-1}(0)/G\times B. 
\] 

Thus studying the $G$-orbits on $T^*(G\times_B \mathfrak{b}\times \mathbb{C}^n)$ is equivalent to studying the $B$-orbits on $T^*(\mathfrak{b}\times \mathbb{C}^n)$.

Now, let $B$ be the set of invertible upper triangular matrices and let $\mathfrak{b}=\lie(B)$. The adjoint action of $B$ on $\mathfrak{b}\times \mathbb{C}^n$ is $b.(r,i)=(brb^{-1},bi)$, which is induced onto 
$T^*(\mathfrak{b\times \mathbb{C}^n})$ via $b.(r,s,i,j)=(brb^{-1},bsb^{-1},bi,jb^{-1})$. 
In terms of Section~\ref{subsection:filtered-quiver-varieties}, we will explain how $T^*(\mathfrak{b}\times \mathbb{C}^n)$ is identified as the cotangent bundle of filtered quiver variety. 

Let $Q$ be the $1$-Jordan quiver $\xymatrix@-1pc{\stackrel{1}{\bullet}\ar@(ur,dr)^a}$ 
and let $\beta=n$.  
Let $F^{\bullet}$ be the complete standard filtration of vector spaces of $\mathbb{C}^n$. 
Then  $F^{\bullet}Rep(Q,n)=\mathfrak{b}$. 
Now we adjoin a framing to the $1$-Jordan quiver to get 
$Q^{\dagger}: \xymatrix@-1pc{ \stackrel{1^{\natural}}{\circ}\ar[rr]^{\iota_1} & & \stackrel{1}{\bullet}\ar@(ur,dr)^a }$.  
Let $\beta^{\dagger}=(n,1)$ where $F^{\bullet}$ is imposed only on the representation at the nonframed vertex. 
This gives us a $B$-action on 
$F^{\bullet}Rep(Q^{\dagger},\beta^{\dagger})\cong \mathfrak{b}\times \mathbb{C}^n$.  
One doubles this quiver $Q^{\dagger}$ to obtain 
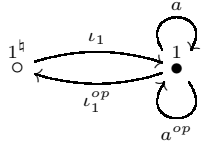
\[ 
\xymatrix@-1pc{
\stackrel{1^{\natural}}{\circ} \ar@/^/[rrr]^{\iota_1} & & & \stackrel{1}{\bullet} \ar@/^/[lll]^{\iota_1^{op}} \ar@(ul,ur)^{a} \ar@(dr,dl)^{a^{op}}  \\ 
}
\] 
and we have the identification  
$T^*(F^{\bullet}Rep(Q^{\dagger},\beta^{\dagger}))\cong 
T^*(\mathfrak{b}\times \mathbb{C}^n)$. 

One fact we would like to mention is that by defining $\widetilde{Rep(Q,\beta)} := \mathbb{G}\times_{\mathbb{P}}F^{\bullet}Rep(Q,\beta)$,  
it follows immediately from Lemma~\ref{lemma:g-s-resolution-and-B-action-on-borel}  
that we have an isomorphism  
$\widetilde{Rep(Q,\beta)}/\mathbb{G}\cong F^{\bullet}Rep(Q,\beta)/\mathbb{P}$ as orbit spaces,  
where 
$\mathbb{G}:= \displaystyle{\prod_{i\in Q_0}GL_{\beta_i}(\mathbb{C})}$ and 
$\mathbb{P}:= \displaystyle{\prod_{i\in Q_0}P_i}$. 
We refer the reader to Section~\ref{subsection:P-moment-map-filtered-quiver-varieties} for further discussion on 
$\widetilde{Rep(Q,\beta)}$.

%
%

\section{Preliminaries}

\subsection{Orthonormal projection operators in \texorpdfstring{$B$}{B}} 

We begin with a discussion on a certain set of orthonormal projection operators in $B$. 

\begin{notation}\label{notation:empty-sum-empty-product-convention}
We will use the convention that an empty sum is defined to be 0 while an empty product is defined to be 1; that is, 
\[ \sum_{k= \iota}^{\gamma} f(k) := 0 \mbox{ and } \prod_{k=\iota}^{\gamma} f(k) := 1 
\] if $\gamma <\iota$. 
We also like to remark that 
\[ \sum_{\iota < k_1 < \ldots < k_v < \mu} f(k_i) :=0 
\]  if $v\geq \mu-\iota$.   
\end{notation}

\begin{notation}\label{notation:coordinatefunction-of-product-of-rs}
For $l_k(r)=r-r_{kk}I$, 
we will write 
\[ L^{\iota} := \left[\tr\left(\prod_{1\leq k\leq n, k\not= \iota} l_k(r)\right)\right]^{-1} \prod_{1\leq k\leq n, k\not= \iota} 
l_k(r) 
\] 
and 
will write  
\[ 
L_{\gamma\mu}^{\iota} := 
\left( \left[\tr\left(\prod_{1\leq k\leq n, k\not= \iota} l_k(r)\right)\right]^{-1} \prod_{1\leq k\leq n, k\not= \iota} 
l_k(r) \right)_{\gamma\mu} 
\]   to denote coordinate functions of $L^{\iota}$. 
\end{notation} 
For Lemma~\ref{lemma:Baction-on-Liota-operator} and in Section~\ref{subsection:B-invariantfunctions} (Prop~\ref{proposition:diagonal-rss-natural-matrix-representation} and \ref{proposition:B-invariant-functions-involving-s}), we will write $L^{\iota}(r)$, instead of $L^{\iota}$ since $L^{\iota}$ depends on $r$.

\begin{lemma}\label{lemma:coordinate-fns-of-product-of-lkr}
The operator $L^{\iota}$ in Notation~\ref{notation:coordinatefunction-of-product-of-rs} has 
coordinate functions 
\[
\begin{aligned}
L_{\gamma\mu}^{\iota} 
&=\left\{ 
\begin{aligned}
0\qquad\qquad\qquad\qquad\qquad\qquad\qquad\qquad & \mbox{ if } \mu < \iota 
						\\
\dfrac{r_{\gamma\iota}}{r_{\iota\iota}-r_{\gamma\gamma}} 
+ \sum_{\widetilde{v}=1}^{\iota-\gamma-1}
\sum_{\gamma < l_1<\ldots < l_{\widetilde{v}}<\iota }
   \dfrac{r_{\gamma l_1} r_{l_{\widetilde{v}}\iota}}{(r_{\iota\iota}-r_{l_{\widetilde{v}}l_{\widetilde{v}}})(r_{\iota\iota}-r_{\gamma\gamma})} \prod_{\widetilde{u}=1}^{\widetilde{v}-1} \dfrac{r_{l_{\widetilde{u}}l_{\widetilde{u}+1}}}{r_{\iota\iota}-r_{l_{\widetilde{u}}l_{\widetilde{u}}}} 
   \qquad\qquad\qquad
 & \mbox{ if }\gamma < \iota = \mu 
 						\\
1\qquad\qquad\qquad\qquad\qquad\qquad\qquad\qquad & \mbox{ if }\gamma = \iota = \mu 
						\\
\dfrac{r_{\iota \mu}}{r_{\iota\iota}-r_{\mu\mu}} 
+ \sum_{v=1}^{\mu-\iota-1} \sum_{\iota < k_1 < \ldots < k_v<\mu}
\dfrac{r_{\iota k_1}r_{k_v \mu}}{(r_{\iota\iota}-r_{k_1k_1})(r_{\iota\iota}-r_{\mu\mu})} 
\prod_{u=1}^{v-1} \dfrac{r_{k_u k_{u+1}}}{r_{\iota\iota}-r_{k_{u+1}k_{u+1}}}
\qquad\qquad  
&
 \mbox{ if }\gamma = \iota < \mu 
 						\\
0\qquad\qquad\qquad\qquad\qquad\qquad\qquad\qquad & \mbox{ if }\gamma >\iota  
						\\ 
 		\left( \dfrac{r_{\gamma\iota}r_{\iota\mu}}{(r_{\iota\iota}-r_{\gamma\gamma})(r_{\iota\iota}-r_{\mu\mu})}
 		\qquad\qquad\qquad\qquad\qquad\qquad\qquad\qquad\qquad\qquad\qquad 
 		  \right. & \\	
 		 + \sum_{\widetilde{v}=1}^{\iota-\gamma-1} 
\sum_{\gamma <l_1<\ldots < l_{\widetilde{v}}<\iota} 
\dfrac{r_{\gamma  l_1}r_{l_{\widetilde{v}}\iota}r_{\iota\mu}}{(r_{\iota\iota}-r_{l_{\widetilde{v}}l_{\widetilde{v}} })(r_{\iota\iota}-r_{\gamma\gamma})(r_{\iota\iota}-r_{\mu\mu})} \prod_{\widetilde{u}=1}^{\widetilde{v}-1} \dfrac{r_{l_{\widetilde{u}} l_{\widetilde{u}+1}} }{r_{\iota\iota}-r_{l_{\widetilde{u}}l_{\widetilde{u}}}}\qquad & 
						\\ 
+\sum_{v=1}^{\mu-\iota-1} \sum_{\iota < k_1< \ldots < k_v < \mu} 
\dfrac{r_{\gamma\iota}r_{\iota k_1}r_{k_v \mu} }{(r_{\iota\iota}-r_{\gamma\gamma})(r_{\iota\iota}-r_{k_1 k_1})(r_{\iota\iota}-r_{\mu\mu})} \prod_{u=1}^{v-1}
\dfrac{r_{k_u k_{u+1}}}{r_{\iota\iota}-r_{k_{u+1}k_{u+1}}} 
\qquad  \;\;  & \mbox{ if }\gamma < \iota < \mu. 
						\\ 
+\sum_{v=1}^{\mu -\iota-1} \sum_{\iota < k_1 < \ldots < k_v < \mu} 
\sum_{\widetilde{v}=1}^{\iota-\gamma-1} 
\sum_{\gamma < l_1 < \ldots < l_{\widetilde{v}} < \iota} 
\dfrac{r_{\gamma l_1}r_{l_{\widetilde{v}}\iota }r_{\iota k_1}r_{k_v \mu}}{
(r_{\iota\iota}-r_{l_{\widetilde{v}} l_{\widetilde{v}} })(r_{\iota\iota}-r_{\gamma\gamma})(r_{\iota\iota}-r_{k_1k_1})(r_{\iota\iota}-r_{\mu\mu})}& 
						\\ 
\left. \prod_{u=1}^{v-1} \dfrac{r_{k_u k_{u+1}}}{r_{\iota\iota}-r_{k_{u+1}k_{u+1} }} 
\prod_{\widetilde{u}=1}^{\widetilde{v}-1} \dfrac{r_{l_{\widetilde{u}}l_{\widetilde{u}+1}}}{r_{\iota\iota}-r_{l_{\widetilde{u}}l_{\widetilde{u}}}} \right)
\qquad\qquad\qquad\qquad 
& 
						\\ 
\end{aligned} 
\right.    		\\ 
\end{aligned}
\]
\end{lemma}

\begin{lemma}\label{lemma:lkrmatrices-orthogonal-rss}  
The $L^{\iota}$'s form mutually orthonormal projection matrices. That is, 
we have 
\[ \det(L^{\iota})=1,\:\: \left( L^{\iota} \right)^2 = L^{\iota}\:\: \mbox{ and }
\:\: L^{\iota} L^{\gamma}=0 
\] for any $\iota\not=\gamma$. 
In particular, any row of $L^{\iota}$
is orthogonal to any column of $L^{\gamma}$ 
for $\iota\not=\gamma$. 
\end{lemma} 

We like to mention that $L^{\iota}$'s in Lemma~\ref{lemma:lkrmatrices-orthogonal-rss} 
are $n$ idempotents satisfying 
$\displaystyle{\sum_{\iota=1}^{n} L^{\iota}}=\I$.

\begin{proof} We refer to Lemma~\ref{lemma:coordinate-fns-of-product-of-lkr} for the coordinates of $L^{\iota}$.  
It is clear that $\det(L^{\iota})=1$ 
since $L_{\mu\alpha}^{\iota}=0$ if $\mu>\iota$ or $\alpha <\iota$, and the only nonzero diagonal entry is  $L_{\iota\iota}^{\iota}=1$. 

Suppose 
$\iota=\gamma$.  
Then 
\begin{gather}\label{gather:L-iota-operator-options}
\begin{aligned}
L_{\mu\alpha}^{\iota} =0 &\mbox{ if } \mu>\iota \mbox{ or } \alpha <\iota, \mbox{ and }  \\
L_{\alpha\nu}^{\iota} =0 &\mbox{ if } \alpha >\iota \mbox{ or } \nu<\iota.   \\
\end{aligned}
\end{gather}
So 
\[\begin{aligned}
\left(L^{\iota}L^{\iota} \right)_{\mu\nu} 
    &= 
\sum_{\alpha=1}^n L_{\mu\alpha}^{\iota} L_{\alpha\nu}^{\iota}  
  \quad
  \stackrel{\mbox{ by }(\ref{gather:L-iota-operator-options})}{=} 
  \quad
  \sum_{\alpha=\iota}^{\iota} L_{\mu\alpha}^{\iota} L_{\alpha\nu}^{\iota}  
   = 
L_{\mu\iota}^{\iota} L_{\iota\nu}^{\iota}  \\
&= \left\{ 
\begin{aligned} 
L_{\mu\iota}^{\iota} = L_{\mu\nu}^{\iota} 
		\quad
			&\mbox{ if } \mu < \iota, \nu = \iota \\
1\;\; = L_{\mu\nu}^{\iota} 
		\quad  
			&\mbox{ if } \mu=\iota, \nu = \iota \\
L_{\iota\nu}^{\iota} = L_{\mu\nu}^{\iota}  
		\quad 
			&\mbox{ if } \mu = \iota, \nu > \iota \\  
 L_{\mu\nu}^{\iota} \mbox{ (see $\dagger$)} 
		\quad 
 			&\mbox{ if } \mu <\iota, \nu > \iota \\  
0\;\; = L_{\mu\nu}^{\iota}   
		\quad  
 			&\mbox{ if } \mu > \iota \mbox{ or } \nu<\iota,   
 \\  
\end{aligned}  
\right. \\ 
\end{aligned}  
\] where $\dagger$ from above is 
\[ \begin{aligned}
L_{\mu\iota}^{\iota} L_{\iota\nu}^{\iota} &\stackrel{\dagger}{=}  \left( 
\dfrac{r_{\mu\iota}}{r_{\iota\iota}-r_{\mu\mu}} + 
\sum_{\widetilde{v}=1}^{\iota-\mu-1} 
\sum_{\mu<l_1<\ldots < l_{\widetilde{v}}<\iota}
\dfrac{r_{\mu l_1}r_{l_{\widetilde{v} }\iota} }{(r_{\iota\iota}-r_{l_{\widetilde{v}}l_{\widetilde{v}} })(r_{\iota\iota}-r_{\mu\mu}) } \prod_{\widetilde{u}=1}^{\widetilde{v}-1} \dfrac{r_{l_{\widetilde{u}} l_{\widetilde{u}+1} } }{r_{\iota\iota}-r_{l_{\widetilde{u}} l_{\widetilde{u}} }}   
\right)\\ 
&\quad 
   \cdot \left( 
\dfrac{r_{\iota\nu}}{r_{\iota\iota}-r_{\nu\nu}}+ 
\sum_{v=1}^{\nu-\iota-1} \sum_{\iota < k_1 < \ldots < k_v < \nu}  
\dfrac{r_{\iota k_1}r_{k_v \nu} }{(r_{\iota\iota}-r_{k_1 k_1})(r_{\iota\iota}-r_{\nu\nu}) } \prod_{u=1}^{v-1} \dfrac{r_{k_u k_{u+1}} }{r_{\iota\iota}-r_{k_{u+1} k_{u+1}}} 
\right) \\
&= \dfrac{r_{\mu\iota}r_{\iota\nu}}{(r_{\iota\iota}-r_{\mu\mu})(r_{\iota\iota} 
-r_{\nu\nu}) } + \sum_{\widetilde{v}=1}^{\iota-\mu-1} \sum_{\mu < l_1 < \ldots < l_{\widetilde{v}}<\iota } \dfrac{r_{\mu l_1}r_{l_{\widetilde{v}} \iota} r_{\iota \nu} }{(r_{\iota\iota}-r_{l_{\widetilde{v}}l_{\widetilde{v}} })(r_{\iota\iota}-r_{\mu\mu}) (r_{\iota\iota}-r_{\nu\nu} ) } \prod_{\widetilde{u}=1}^{\widetilde{v}-1}
 \dfrac{r_{l_{\widetilde{u}}l_{\widetilde{u}+1} } }{r_{\iota\iota}-r_{l_{\widetilde{u}} l_{\widetilde{u}} } } \\ 
&\quad 
   + \sum_{v=1}^{\nu-\iota-1} \sum_{\iota < k_1 < \ldots < k_v < \nu} 
\dfrac{r_{\mu\iota} r_{\iota k_1} r_{k_v \nu} }{ (r_{\iota\iota}-r_{\mu\mu} )
(r_{\iota\iota}-r_{k_1 k_1} )(r_{\iota\iota}-r_{\nu\nu} ) } 
\prod_{u=1}^{v-1} \dfrac{r_{k_u k_{u+1}} }{r_{\iota\iota}-r_{k_{u+1}k_{u+1} } } 
\\
&\quad  
   + \sum_{v=1}^{\nu-\iota-1} \sum_{\iota < k_1 < \ldots < k_v < \nu} 
\sum_{\widetilde{v}=1}^{\iota -\mu-1} 
\sum_{\mu < l_1 < \ldots < l_{\widetilde{v}}<\iota } \dfrac{r_{\mu l_1}r_{l_{\widetilde{v}}\iota }r_{\iota k_1}r_{k_v \nu} }{ 
(r_{\iota\iota}-r_{l_{\widetilde{v}} l_{\widetilde{v}} } ) (r_{\iota\iota}-r_{\mu\mu}) (r_{\iota\iota}-r_{k_1 k_1})(r_{\iota\iota}-r_{\nu\nu} ) } 
\\ 
&\quad  
    \cdot \prod_{u=1}^{v-1} \dfrac{r_{k_u k_{u+1} } }{r_{\iota\iota}-r_{k_{u+1} k_{u+1} } } \prod_{\widetilde{u}=1 }^{\widetilde{v}-1 } \dfrac{r_{l_{\widetilde{u}}l_{\widetilde{u}+1} } }{r_{\iota\iota}-r_{l_{\widetilde{u}} l_{\widetilde{u}}} }  \\ 
&= L_{\mu\nu}^{\iota}. \\
\end{aligned}
\] Thus $\left( L^{\iota}\right)^2 = L^{\iota}$.

For 
$\iota > \gamma$, 
\[
\left( L^{\iota}L^{\gamma}\right)_{\mu\nu} 
=  \sum_{\alpha=1}^n L_{\mu\alpha}^{\iota}L_{\alpha\nu}^{\gamma}=0
\] since $L_{\mu\alpha}^{\iota}=0$ for each  
$\alpha <\iota$ 
and $L_{\alpha\nu}^{\gamma}=0$ for each $\alpha>\gamma$. 
 Thus $L^{\iota}L^{\gamma}=0$ whenever $\iota >\gamma$.

Finally for 
$\iota < \gamma$,  
\[ L^{\iota}L^{\gamma} = L^{\gamma} L^{\iota} =0 
\] since $L^{\iota}$'s commute and by previous case.  
\end{proof}

\begin{corollary}\label{corollary:lkrmatrices-orthogonal}
The product of matrices 
\[ \prod_{1\leq k\leq n, k\not=\iota} l_k(r) \mbox{ and }
\prod_{1\leq \widetilde{k} \leq n, \widetilde{k} \not=\gamma} l_{\widetilde{k}}(r)
\] is zero for any $\iota\not=\gamma$. In particular, any row of 
\[ \prod_{1\leq k\leq n, k\not=\iota} l_k(r) 
\] is orthogonal to any column of 
\[ \prod_{1\leq \widetilde{k}\leq n, \widetilde{k}\not=\gamma} l_{\widetilde{k}}(r) 
\] for $\iota\not=\gamma$. 
\end{corollary}

\begin{proof}
This follows from Lemma~\ref{lemma:lkrmatrices-orthogonal-rss} 
by clearing the denominators of $L^{\iota}$ and $L^{\gamma}$. 
\end{proof}

\begin{lemma}\label{lemma:product-of-r-minus-rkk} 
For each $1\leq \iota\leq n$, 
\[ \prod_{1\leq k\leq n, k\not= \iota} l_k(r) 
\] is zero strictly to the left of $\iota^{th}$ column or 
strictly below $\iota^{th}$ row. 
\end{lemma} 

\begin{proof} 
This follows from Lemma~\ref{lemma:coordinate-fns-of-product-of-lkr}.
\end{proof}

%
%
 
\begin{lemma}\label{lemma:lkr-times-r-identity}
For each $\gamma$ and $r\in\mathfrak{b}$,  
\[ (L^{\iota} r)_{\gamma\gamma} = (rL^{\iota})_{\gamma\gamma} = r_{\iota\iota}L_{\gamma\gamma}^{\iota}.  
\]  
\end{lemma}

\begin{proof}
For a fixed $\iota$, it is clear that $L^{\iota}r=rL^{\iota}$ since $r$ and $L^{\iota}$ commute. We will thus show 
 for each $\gamma$, 
		$(L^{\iota}r)_{\gamma\gamma} = r_{\iota\iota}L_{\gamma\gamma}^{\iota}$.  
So 
\[
  \begin{aligned} 
\left( 
  L^{\iota}r\right)_{\gamma\gamma} &= \sum_{\mu=1}^n 
L_{\gamma\mu}^{\iota} r_{\mu\gamma} \\ 
&= \sum_{\mu=\iota}^n L_{\gamma\mu}^{\iota}r_{\mu\gamma} \mbox{ since }L_{\gamma\mu}^{\iota} =0 \;\;\;\mbox{ for }\mu<\iota \\  
&= \left\{  
\begin{aligned}  
\sum_{\mu=\iota}^n 0\cdot r_{\mu\gamma} = 0=r_{\iota\iota}L_{\gamma\gamma}^{\iota} 
			\qquad\qquad\qquad 
			&\mbox{ if }\gamma >\iota \\
 \sum_{\mu=\iota}^n L_{\gamma\mu}^{\iota} r_{\mu\gamma} 
 	\stackrel{\dagger}{=}		\sum_{\mu=\iota}^{\iota} L_{\iota\mu}^{\iota} r_{\mu\iota} 
	= 	r_{\iota\iota}\cdot 1= r_{\iota\iota}L_{\gamma\gamma}^{\iota}      
			\;\; 
			&\mbox{ if }\gamma = \iota \\
\sum_{\mu=\iota}^n L_{\gamma\mu}^{\iota} r_{\mu\gamma} 
	\stackrel{\ddagger}{=}   \sum_{\mu=\iota}^n L_{\gamma\mu}^{\iota}\cdot 0 = r_{\iota\iota}L_{\gamma\gamma}^{\iota}  
			\qquad 
			&\mbox{ if }\gamma < \iota,  \\  
\end{aligned} 
\right.  
\end{aligned}  
\] 
where 
	$\dagger$ holds since $r_{\mu\gamma}=0$ for all $\mu > \gamma$ and since $\gamma=\iota$, 
 		and 
 			$\ddagger$ holds since $r_{\mu\gamma}=0$ for all $\mu > \gamma$ 
 				and since $\mu$ ranges over $\iota \leq \mu\leq n$.  
\end{proof}

%
%
%

\begin{lemma}\label{lemma:productoflkr-r}  
For each $\iota$ and for any $r\in\mathfrak{b}$, 
\[ \prod_{k\not=\iota} l_k(r)\: r = r \prod_{k\not=\iota} l_k(r) 
\]  is zero strictly to the left of $\iota^{th}$ column or strictly below $\iota^{th}$ row. 
\end{lemma}

\begin{proof}  
We have  
\[
		\begin{aligned} 
\left( L^{\iota}r\right)_{\gamma\nu} &= \sum_{\mu=1}^n 
L_{\gamma\mu}^{\iota} r_{\mu\nu} \\ 
&= \sum_{\mu=\iota}^n L_{\gamma\mu}^{\iota}r_{\mu\nu} \mbox{ since }L_{\gamma\mu}^{\iota} =0 \;\;\;\mbox{ for }\mu<\iota \\
&= \left\{ 
\begin{aligned}
\sum_{\mu=\iota}^n 0\cdot r_{\mu\nu} = 0 
			\quad\qquad	&\mbox{ if }\gamma >\iota \\ 
\sum_{\mu=\iota}^n L_{\gamma\mu}^{\iota} r_{\mu\nu} =
		\sum_{\mu=\iota}^n L_{\gamma\mu}^{\iota}\cdot 0 = 0 
			\;\;	&\mbox{ if }\gamma \leq \iota,  \nu<\iota  \\ 
\end{aligned} 
%
%
%
\right.  
\end{aligned}  
\]  
		since $r_{\mu\nu}=0$ for each $\mu > \nu$.
Clear the denominator to obtain the result. 
\end{proof}

%

\begin{lemma}\label{lemma:productoflkr-s-RHS}
For each $\iota$ and for any $s\in\mathfrak{b}^*$,  
\[ 
   \prod_{1\leq k\leq n, k\not=\iota} l_k(r)\: s 
\] 
    is zero strictly below $\iota^{th}$ row. 
\end{lemma}

Unlike the case in Lemma~\ref{lemma:productoflkr-r}, 
we make a note that 
$\displaystyle{\prod_{1\leq k\leq n, k\not=\iota} l_k(r)\: s }$ is not zero strictly to the left of $\iota^{th}$ column. 

\begin{proof}
For $\gamma > \iota$, 
\[ 
\left( \prod_{1\leq k \leq n, k\not=\iota} l_k(r) \: s \right)_{\gamma\mu} 
	= \sum_{\kappa=1}^n \left(\prod_{1\leq k\leq n, k\not=\iota }l_k(r) \right)_{\gamma\kappa} s_{\kappa\mu}  
			\stackrel{\dagger}{=} 
	\sum_{\kappa=1}^n \: 0\cdot s_{\kappa\mu}  = 0  \\  
\] 
where $\dagger$ holds by Lemma~\ref{lemma:product-of-r-minus-rkk}. 
\end{proof} 

%

The following Corollary will be used in the proof of Proposition~\ref{proposition:deriving-diag-coordinates-of-bsb-inverse}.

\begin{corollary}\label{corollary:projectionmatrices-with-s} 
For each $\gamma > \iota$, $\left(L^{\iota}s\right)_{\gamma\mu}=0$.  
\end{corollary}

\begin{proof} 
This follows from Lemma~\ref{lemma:productoflkr-s-RHS}. 
\end{proof}

We will see an application of the following Lemma in Section~\ref{subsection:B-invariantfunctions}. 

\begin{lemma}\label{lemma:Baction-on-Liota-operator}
For any $d\in B$, $L^{\iota}(d.r)=d.L^{\iota}(r)$, where the $B$-action on the operator is by conjugation. 
\end{lemma}

\begin{proof}
For any $d\in B$, 
\[\begin{aligned} 
L^{\iota}(d.r) &=  \left[
\tr\left( \prod_{
1\leq k\leq n, k\not=\iota 
} 
      drd^{-1} -r_{kk}I   
\right)\right]^{-1}   \left(\prod_{
1\leq k\leq n, k\not=\iota
}    
   drd^{-1} - r_{kk}I
\right) \\  
&=  \left[\tr\left(  d\left( \prod_{1\leq k \leq n,k\not=\iota} 
     r-r_{kk}I \right) 
 d^{-1} \right) \right]^{-1}  d   
   \left( 
   \prod_{1\leq k\leq n, k\not=\iota} r-r_{kk}I 
  \right)
 d^{-1}  
\\ 
&= d \left( \left[\tr\left(   \prod_{1\leq k \leq n,k\not=\iota} 
     r-r_{kk}I   
  \right) \right]^{-1}    
   \prod_{1\leq k\leq n, k\not=\iota} 
   \left(  r-r_{kk}I 
 					 \right) 
  \right) d^{-1}   \\    
   &= d.L^{\iota}(r)   \\  
\end{aligned} 
\] 
\end{proof}

%
%
%

%
%
%
%

\subsection{On the fixed points by the Borel} 
  
We now prove some basic facts about the action of $B$. 
Proposition~\ref{proposition:diagonal-being-invariant} shows that the diagonal entries of an upper triangular matrix do not change under $B$-conjugation whilst Lemma~\ref{lemma:s-diagonal-conjugation-still-diagonal} shows that diagonal entries of a diagonal matrix are preserved under $B$-conjugation.

\begin{proposition}\label{proposition:diagonal-being-invariant}
For any $r$ in $\mathfrak{b}$ and for any $b$ in $B$, we have $\diag(brb^{-1})=\diag(r)$. 
\end{proposition}

\begin{proof}
Denote the entries of $b$, $r$, and $b^{-1}$ as 
$b_{\iota\gamma}$, 
$r_{\iota\gamma}$, and  
$b_{\iota\gamma}'$, respectively,  
with 
$b_{\iota\iota}'=b_{\iota\iota}^{-1}$. 
Then   
$(br)_{\iota\gamma} 
= \displaystyle{\sum_{k=1}^n b_{\iota k}r_{k\gamma}}$. 
Since $b_{\iota k}$ and $r_{\iota k}$ equal 0 whenever $\iota >k$, 
\[ (br)_{\iota\gamma} 
= \left\{ 
 \begin{aligned}
 \sum_{k=\iota}^{\gamma} b_{\iota k}r_{k\gamma}\qquad & \mbox{ if } \iota <\gamma \\ 
   b_{\iota\iota}r_{\iota\iota}\qquad       \quad    & \mbox{ if } \iota =\gamma \\ 
   0 \qquad\qquad                               &  \mbox{ if } \iota>\gamma. \\ 
 \end{aligned}
\right. 
\] 
Renaming 
$(br)_{\iota k}$ as $d_{\iota k}$, we obtain  
 $(brb^{-1})_{\iota \gamma}
=
\displaystyle{\sum_{k=1}^n d_{\iota k}b_{k\gamma}'}$. 
Since $d_{\iota k}$ and $b_{\iota k}'$ equal 0 whenever $\iota >k$, 
\[ (brb^{-1})_{\iota\gamma} = \left\{ 
\begin{aligned} 
\sum_{k=\iota }^{\gamma} d_{\iota k}b_{k\gamma}' \qquad 
       & \mbox{ if $\iota <\gamma$} \\ 
  d_{\iota\iota}b_{\iota\iota }'  \qquad\quad 
       & \mbox{ if $\iota =\gamma$} \\ 
  0\qquad\qquad 
       &  \mbox{ if $\iota>\gamma$}.  \\   
\end{aligned}
\right. 
\] 
Since $(brb^{-1})_{\iota\iota}=d_{\iota\iota}b_{\iota\iota}'=b_{\iota\iota}r_{\iota\iota}b_{\iota\iota}^{-1} = r_{\iota\iota}$, we are done.  
\end{proof}

\begin{lemma}\label{lemma:s-diagonal-conjugation-still-diagonal}
Let $s=\diag(s)$. For any $b$ in $B$, $\diag(bsb^{-1})=\diag(s)$.  
\end{lemma} 

\begin{proof} 
Restrict $r$ in Proposition~\ref{proposition:diagonal-being-invariant} to those in  $\mathfrak{b}\cap\mathfrak{b}^*=\mathfrak{h}$. 
\end{proof}

\begin{remark}\label{remark:diagonalize-borel-and-its-dual}
For a diagonal matrix $r$ in $\mathfrak{b}$ and for a general $b$ in $B$, 
$brb^{-1}$ need not equal $r$ (that is, $brb^{-1}$ need not be diagonal). 
However, for a diagonal matrix $s$ in $\mathfrak{b}^*$ and for any $b$ in $B$, $bsb^{-1}=s$ (i.e., $b\:\diag(s)\: b^{-1}=\diag(s)$) always holds in $\mathfrak{b}^* = \mathfrak{b}/\mathfrak{n}$ since the strictly upper triangular part is killed in $\mathfrak{b}^*$. 
\end{remark}

\begin{proposition}\label{proposition:diag-entries-r-is-trace}
For each $\iota$, 
\[  r_{\iota\iota} = \tr(L^{\iota}r). 
\] 
\end{proposition}

\begin{proof}
We have 
$\tr(L^{\iota}r)=\tr(r_{\iota\iota}L^{\iota})=r_{\iota\iota}$ where the first equality holds by Lemma~\ref{lemma:lkr-times-r-identity} and the second equality holds by Lemma~\ref{lemma:coordinate-fns-of-product-of-lkr} ($L_{\gamma\gamma}^{\iota}$ is 0 if $\gamma\not=\iota$  and equals 1 if $\gamma = \iota$).  
\end{proof}

\begin{proposition}\label{proposition:difference-of-r-is-trace}
For each $1\leq \gamma < \nu\leq n$, 
\[ (r_{\nu\nu} - r_{\gamma\gamma})^{-1} = [\tr((L^{\nu}-L^{\gamma})r)]^{-1}.  
\] 
\end{proposition}

\begin{proof}
Applying Proposition~\ref{proposition:diag-entries-r-is-trace}, we have  $\tr(L^{\nu}r-L^{\gamma}r)$ $=$ $\tr(L^{\nu}r)-\tr(L^{\gamma}r)$ $=$ $r_{\nu\nu}-r_{\gamma\gamma}$. 
\end{proof}

We now apply the operators from Notation~\ref{notation:coordinatefunction-of-product-of-rs} for the following 

\begin{proposition}\label{proposition:distinct-ev-diagonalizable}
For each $r$ in $\mathfrak{b}^{rss}$, there exists $b$ 
in the Borel with coordinates  
\[ b_{\iota\gamma} = \left\{ 
\begin{aligned} 
0 \;\;\;\;\;\;\;\;\;\;\;\;\;\;\;\;\;\;\;\;\;\;\;\;\;\;\;\;\;\;\; \;\;\;\;\;\;\;\;\;\;\;\;\;\;\;\;\;\;\;\;\;\;\;\;\;\;\;\;\;\;\;\;\;\;\; &\mbox{ if } \iota >\gamma \\
b_{\iota\iota}\;\;\;\;\;\;\;\;\;\;\;\;\;\;\;\;\;\;\;\;\;\;\;\;\;\;\;\;\;\; \;\;\;\;\;\;\;\;\;\;\;\;\;\;\;\;\;\;\;\;\;\;\;\;\;\;\;\;\;\;\;\;\;\;\; &\mbox{ if } \iota = \gamma \\ 
b_{\iota\iota}\left(
\dfrac{r_{\iota\gamma}}{r_{\iota\iota}-r_{\gamma\gamma}}  
+
\sum_{v=1}^{\gamma-\iota-1}\sum_{\iota<k_1<\ldots <k_{v}<\gamma} 
\dfrac{r_{\iota k_1}r_{k_v \gamma} }{(r_{\iota\iota}-r_{k_1k_1})(r_{\iota\iota}-r_{\gamma\gamma})} 
\prod_{u=1}^{v-1} \dfrac{r_{k_uk_{u+1}}}{r_{\iota\iota}-r_{k_{u+1}k_{u+1}}} 
\right) &\mbox{ if } \iota < \gamma \\ 
\end{aligned}
\right. 
\] 
which diagonalizes $r$. 
This $b$ can be written as 
\begin{equation}\label{eqn:diagonalizingmatrix-B}
b = \sum_{\iota =1}^n  
 E_{\iota\iota} \: \diag(b) L^{\iota},  
\end{equation} 
where $L^{\iota}$ is from 
Notation~\ref{notation:coordinatefunction-of-product-of-rs} 
and 
$E_{\iota\iota}$ has 1 in $(\iota,\iota)$-entry and $0$'s elsewhere. 
\end{proposition}

One may deduce from the proof of Proposition~\ref{proposition:distinct-ev-diagonalizable} that 
one may replace $\diag(b)$ with $b$ on the right-hand side of (\ref{eqn:diagonalizingmatrix-B})  
and still obtain a matrix that diagonalizes $r$. 

\begin{proof}
For $r=(r_{ij})$, 
we will show $brb^{-1}=\diag(r)$, or equivalently, $br=\diag(r)b$. So 
\[ 
\begin{aligned} 
br &=  \sum_{\iota =1}^n  
 \left[\tr\left( \prod_{
1\leq k\leq n, k\not=\iota 
} l_k(r) \right)\right]^{-1} E_{\iota\iota} \: \diag(b) \left(\prod_{
1\leq k\leq n, k\not=\iota
} l_k(r) \right) r\\
&=\sum_{\iota =1}^n  
 \left[ \tr\left( \prod_{
1\leq k\leq n, k\not=\iota 
} l_k(r) \right)\right]^{-1} b_{\iota\iota} E_{\iota\iota} \: r \left(\prod_{
1\leq k\leq n, k\not=\iota
} l_k(r) \right) \\ 
&=\sum_{\iota =1}^n  
 \left[ \tr\left( \prod_{
1\leq k\leq n, k\not=\iota 
} l_k(r) \right)\right]^{-1} \left[ 
\begin{array}{c} 
0 \\ 
\vdots \\ 
b_{\iota\iota} \vec{r}_{\iota} \\ 
\vdots \\ 
0 \\  
\end{array}
\right] \left(\prod_{
1\leq k\leq n, k\not=\iota
} l_k(r) \right) \mbox{ where }\vec{r}_{\iota}\mbox{ is the }
\iota^{th} \mbox{ row of } r \\ 
&= \sum_{\iota =1}^n  
 \left[ \tr\left( \prod_{
1\leq k\leq n, k\not=\iota 
} l_k(r) \right)\right]^{-1} b_{\iota\iota} r_{\iota\iota} E_{\iota\iota}  \left(\prod_{
1\leq k\leq n, k\not=\iota
} l_k(r) \right)\mbox{ by Lemma~\ref{lemma:product-of-r-minus-rkk}} \\ 
&= \sum_{\iota =1}^n  
\left[ \tr\left( \prod_{
1\leq k\leq n, k\not=\iota 
} l_k(r) \right)\right]^{-1} \diag(r) E_{\iota\iota} \diag(b) \left(\prod_{
1\leq k\leq n, k\not=\iota
} l_k(r) \right) \\ 
&= \diag(r)b. \\
\end{aligned} 
\]     
\end{proof}

\begin{corollary}\label{corollary:inverse-of-diagonalizingmatrix-rss-locus}
The inverse of $b$ in 
Proposition~\ref{proposition:distinct-ev-diagonalizable} has coordinates 
\[ (b^{-1})_{\iota\gamma} = \left\{ 
\begin{aligned} 
0 \;\;\;\;\;\;\;\;\;\;\;\;\;\;\;\;\;\;\;\;\;\;\;\;\;\;\;\;\;\;\; \;\;\;\;\;\;\;\;\;\;\;\;\;\;\;\;\;\;\;\;\;\;\;\;\;\;\;\;\;\;\;\;\;\;\; &\mbox{ if } \iota >\gamma \\
b_{\gamma\gamma}^{-1}\;\;\;\;\;\;\;\;\;\;\;\;\;\;\;\;\;\;\;\;\;\;\;\;\;\;\;\;\;\; \;\;\;\;\;\;\;\;\;\;\;\;\;\;\;\;\;\;\;\;\;\;\;\;\;\;\;\;\;\;\;\;\;\;\; &\mbox{ if } \iota = \gamma \\ 
b_{\gamma\gamma}^{-1}\left(\dfrac{r_{\iota\gamma}}{
r_{\gamma\gamma}-r_{\iota\iota}} 
+
 \sum_{v=1}^{\gamma-\iota-1}
 \sum_{\iota <k_1<\ldots <k_{v}<\gamma} 
\dfrac{r_{\iota k_1}r_{k_v \gamma} }{
(r_{\gamma\gamma}-r_{k_vk_v})
(r_{\gamma\gamma}-r_{\iota\iota})} 
\prod_{u=1}^{v-1} \dfrac{r_{k_uk_{u+1}}}{r_{\gamma\gamma}-r_{k_{u}k_{u}}} 
\right) &\mbox{ if } \iota < \gamma \\ 
\end{aligned}
\right.  
\]  
with   
 matrix representation 
\[ 
b^{-1} = \sum_{\iota =1}^n  L^{\iota} 
\diag(b)^{-1} \: E_{\iota\iota}. 
\]  
\end{corollary}

\begin{proof}
We will prove that $b b^{-1}=\I$. 
The $\iota^{th}$ row of $b$ is $e_{\iota}^* \diag(b)L^{\iota}$ and the 
$\gamma^{th}$ column of $b^{-1}$ is 
$L^{\gamma}\diag(b)^{-1}e_{\gamma}$ 
  where 
   $e_{\gamma}$ is an elementary column vector 
and $e_{\iota}^*$ is an elementary covector. 
By Lemma~\ref{lemma:lkrmatrices-orthogonal-rss},  
\[
e_{\iota}^* \diag(b)L^{\iota}  L^{\gamma}\diag(b)^{-1}e_{\gamma} =
\left\{  
\begin{aligned} 
1 \quad  &\mbox{ if }\iota=\gamma \\  
0 \quad  &\mbox{ if }\iota\not=\gamma.  \\ 
\end{aligned} 
\right. 
\] 
\end{proof}

\begin{proposition}\label{proposition:deriving-diag-coordinates-of-bsb-inverse}
For each $\iota$, 
\[ (bsb^{-1})_{\iota\iota} = \tr(L^{\iota}s), 
\] where 
$b$ is from  Proposition~\ref{proposition:distinct-ev-diagonalizable} and $L^{\iota}$ is from Notation~\ref{notation:coordinatefunction-of-product-of-rs}. 
\end{proposition}

\begin{strategy}\label{strategy:deriving-diag-coordinates-of-bsb-inverse}
We outline a proof for  Proposition~\ref{proposition:deriving-diag-coordinates-of-bsb-inverse}. 
For a fixed $\iota$ and for each $\gamma \leq \iota$, consider the entry 
\[ (bs)_{\iota\gamma}=\sum_{\mu=\iota}^{n} b_{\iota\mu}s_{\mu\gamma}. 
\] Then 
\[ (bsb^{-1})_{\iota\iota} = \sum_{\gamma=1}^{\iota} (bs)_{\iota\gamma} (b^{-1})_{\gamma\iota} 
\]   
since $(bs)_{\iota\gamma} (b^{-1})_{\gamma\iota}=0$ for each $\gamma >\iota$.  
Now considering the product of matrices 
$L^{\iota}s$, 
$(L^{\iota}s)_{\gamma\gamma}=0$
for all $\gamma > \iota$ by Corollary~\ref{corollary:projectionmatrices-with-s}. 
We will thus show 
\[ (bs)_{\iota\gamma} (b^{-1})_{\gamma\iota} =
 \left( 
 L^{\iota}s 
\right)_{\gamma\gamma}  
 \] for each $\gamma$, which implies 
 \[ (bsb^{-1})_{\iota\iota} = \sum_{\gamma=1}^{\iota} 
 \underbrace{(bs)_{\iota\gamma} (b^{-1})_{\gamma\iota}}_{
 (L^{\iota}s)_{\gamma\gamma}  
 } = s_{\iota\iota}'. 
\]  
\end{strategy} 

We now prove Proposition~\ref{proposition:deriving-diag-coordinates-of-bsb-inverse}, thus deriving Trace~\ref{trace:map-from-B-rss-to-complex-space}.  

\begin{proof}
Fix $\iota$ and let $\gamma \leq \iota$. We will adopt   Notation~\ref{notation:coordinatefunction-of-product-of-rs} near the end of the proof. 
Since 
\[  
\begin{aligned} 
(bs)_{\iota\gamma} &= \sum_{\mu=\iota}^n b_{\iota\mu}s_{\mu\gamma} \\ 
&= b_{\iota\iota}s_{\iota\gamma}+ \sum_{\mu=\iota+1}^n b_{\iota\iota}
\left( \dfrac{ r_{\iota\mu} }{r_{\iota\iota}-r_{\mu\mu} } 
+ \sum_{v=1}^{\mu -\iota-1} \sum_{\iota < k_1 < \ldots < k_v < \mu} 
\dfrac{r_{\iota k_1}r_{k_v \mu} }{(r_{\iota\iota}-r_{k_1 k_1})(r_{\iota\iota}-r_{\mu\mu}) } \prod_{u=1}^{v-1} 
\dfrac{r_{k_u k_{u+1}}}{r_{\iota\iota}-r_{k_{u+1}k_{u+1}} } 
\right) s_{\mu\gamma},  \\ 
\end{aligned} 
\]  we have 
\[
\begin{aligned} 
&(bsb^{-1})_{\iota\iota} = \sum_{\gamma=1}^{\iota} (bs)_{\iota\gamma} 
(b^{-1})_{\gamma\iota} \\ 
&= \sum_{\gamma=1}^{\iota-1} 
\left[ \left( \not{b}_{\iota\iota}s_{\iota\gamma} \right.\right. \\ 
&  \left.+ \sum_{\mu=\iota+1}^n \not{b}_{\iota\iota}  
\left(
\dfrac{r_{\iota\mu}}{r_{\iota\iota}-r_{\mu\mu}} 
+  \sum_{v=1}^{\mu-\iota-1} 
\sum_{\iota < k_1<\ldots < k_v < \mu} 
\dfrac{r_{\iota k_1}r_{k_v \mu}}{(r_{\iota\iota}-r_{k_1 k_1})(r_{\iota\iota}-r_{\mu\mu})} 
\prod_{u=1}^{v-1}\dfrac{r_{k_u k_{u+1}}}{r_{\iota\iota}-r_{k_{u+1}k_{u+1}}}  
 \right)s_{\mu\gamma}
\right)    \\ 
&\cdot\left.\left( \not{b}_{\iota\iota}^{-1} 
\left( \dfrac{r_{\gamma\iota}}{r_{\iota\iota}-r_{\gamma\gamma}}   
+\sum_{\widetilde{v}=1}^{\iota-\gamma-1} 
\sum_{\gamma < l_1 <\ldots < l_{\widetilde{v}}<\iota }
\dfrac{r_{\gamma l_1}r_{l_{\widetilde{v}}\iota} }{(r_{\iota\iota}-r_{l_{\widetilde{v}}l_{\widetilde{v}}})(r_{\iota\iota}-r_{\gamma\gamma}) } \prod_{\widetilde{u}=1}^{\widetilde{v}-1} 
\dfrac{r_{l_{\widetilde{u}}l_{\widetilde{u}+1} } }{r_{\iota\iota}-r_{l_{\widetilde{u}}l_{\widetilde{u}} }} 
\right)\right)\right] \\
&+ \underbrace{ \left( \not{b}_{\iota\iota}s_{\iota\iota}
+\hspace{-1mm} \sum_{\mu=\iota+1}^n \hspace{-1mm}\not{b}_{\iota\iota}\left( 
\dfrac{r_{\iota\mu}}{r_{\iota\iota}-r_{\mu\mu}} 
+ \sum_{v=1}^{\mu-\iota-1} \sum_{\iota < k_1 < \ldots < k_v<\mu} 
\dfrac{r_{\iota k_1}r_{k_v \mu}}{(r_{\iota\iota}-r_{k_1 k_1 })( r_{\iota\iota}-r_{\mu\mu})} \prod_{u=1}^{v-1} \dfrac{r_{k_u k_{u+1}}}{r_{\iota\iota}-r_{k_{u+1}k_{u+1}}}
\right)s_{\mu\iota}
\right)\hspace{-1mm}\not{b}_{\iota\iota}^{-1}}_{\mbox{ for }\gamma=\iota} \\
\end{aligned}
\] 
\[ 
\begin{aligned}  
&=  \sum_{\gamma=1}^{\iota-1}\left[ \dfrac{ r_{\gamma\iota} s_{\iota\gamma}  }{r_{\iota\iota}-r_{\gamma\gamma}}
+\sum_{\widetilde{v}=1}^{\iota-\gamma-1} 
\sum_{\gamma< l_1 < \ldots < l_{\widetilde{v}}<\iota}
\dfrac{r_{\gamma l_1} r_{l_{\widetilde{v}}\iota} s_{\iota \gamma} }{
(r_{\iota\iota}-r_{l_{\widetilde{v}}l_{\widetilde{v}}})(r_{\iota\iota}-r_{\gamma\gamma})} \prod_{\widetilde{u}=1}^{\widetilde{v}-1} 
\dfrac{r_{l_{\widetilde{u}}l_{\widetilde{u}+1}}}{r_{\iota\iota}-r_{l_{\widetilde{u}}l_{\widetilde{u}}} } 
\right. \\
&+ \sum_{\mu=\iota+1}^n \dfrac{r_{\iota\mu}s_{\mu\gamma}}{r_{\iota\iota}-r_{\mu\mu}} \left( \dfrac{r_{\gamma\iota}}{r_{\iota\iota}-r_{\gamma\gamma}} 
+ \sum_{\widetilde{v}=1}^{\iota-\gamma-1} \sum_{\gamma < l_1 < \ldots < l_{\widetilde{v}}<\iota} \dfrac{r_{\gamma l_1}r_{l_{\widetilde{v}} \iota} }{(r_{\iota\iota}-r_{l_{\widetilde{v}}l_{\widetilde{v}} })(r_{\iota\iota}-r_{\gamma\gamma})} \prod_{\widetilde{u}=1}^{\widetilde{v}-1} \dfrac{r_{l_{\widetilde{u}} l_{\widetilde{u}+1} }}{r_{\iota\iota}-r_{l_{\widetilde{u}}l_{\widetilde{u}} }}  
\right) 
\\ 
&+ \sum_{\mu=\iota+1}^n \sum_{v=1}^{\mu-\iota-1} 
\sum_{\iota < k_1 < \ldots < k_v < \mu} 
\left[ 
\dfrac{r_{\iota k_1}r_{k_v \mu}s_{\mu\gamma} }{
(r_{\iota\iota}-r_{k_1 k_1})(r_{\iota\iota}-r_{\mu\mu})} 
\prod_{u=1}^{v-1} \dfrac{r_{k_u k_{u+1}}}{r_{\iota\iota}-r_{k_{u+1}k_{u+1}}} 
\right. \\
&\cdot \left. \left.\left( \dfrac{r_{\gamma\iota}}{r_{\iota\iota}-r_{\gamma\gamma}}
+ \sum_{\widetilde{v}=1}^{\iota-\gamma-1} \hspace{-2mm}
\sum_{\gamma < l_1<\ldots < l_{\widetilde{v}}<\iota } \hspace{-2mm} 
\dfrac{r_{\gamma l_1}r_{l_{\widetilde{v}}\iota} }{(r_{\iota\iota}-
r_{l_{\widetilde{v}}l_{\widetilde{v}}})(r_{\iota\iota}-r_{\gamma\gamma}) } \prod_{\widetilde{u}=1}^{\widetilde{v}-1} 
\dfrac{r_{l_{\widetilde{u}}l_{\widetilde{u}+1}} }{r_{\iota\iota}-r_{l_{\widetilde{u}}l_{\widetilde{u}}} }  
\right) \right]\right]\\
&+ \underbrace{ \left( s_{\iota\iota}
+ \sum_{\mu=\iota+1}^n \left( 
\dfrac{r_{\iota\mu}}{r_{\iota\iota}-r_{\mu\mu}} 
+ \sum_{v=1}^{\mu-\iota-1} \sum_{\iota < k_1 < \ldots < k_v<\mu} 
\dfrac{r_{\iota k_1}r_{k_v \mu}}{(r_{\iota\iota}-r_{k_1 k_1 })( r_{\iota\iota}-r_{\mu\mu})} \prod_{u=1}^{v-1} \dfrac{r_{k_u k_{u+1}}}{r_{\iota\iota}-r_{k_{u+1}k_{u+1}}}
\right)s_{\mu\iota}
\right) }_{\mbox{ for }\gamma=\iota} \\
\end{aligned}
\] 
%
\[  
\begin{aligned} 
&=  \sum_{\gamma=1}^{\iota-1}\left[ 
  \left( 
   \dfrac{ r_{\gamma\iota} s_{\iota\gamma}  }{r_{\iota\iota}-r_{\gamma\gamma}}
+\sum_{\widetilde{v}=1}^{\iota-\gamma-1} 
\sum_{\gamma< l_1 < \ldots < l_{\widetilde{v}}<\iota}
\dfrac{r_{\gamma l_1} r_{l_{\widetilde{v}}\iota} s_{\iota \gamma} }{
(r_{\iota\iota}-r_{l_{\widetilde{v}}l_{\widetilde{v}}})(r_{\iota\iota}-r_{\gamma\gamma})} \prod_{\widetilde{u}=1}^{\widetilde{v}-1} 
\dfrac{r_{l_{\widetilde{u}}l_{\widetilde{u}+1}}}{r_{\iota\iota}-r_{l_{\widetilde{u}}l_{\widetilde{u}}} } 
\right) \right. \\
&+ \sum_{\mu=\iota+1}^n \dfrac{r_{\gamma\iota} r_{\iota\mu}s_{\mu\gamma}}{ (r_{\iota\iota}-r_{\gamma\gamma})( r_{\iota\iota}-r_{\mu\mu})}
+ \sum_{\mu=\iota+1}^n   \sum_{\widetilde{v}=1}^{\iota-\gamma-1} \sum_{\gamma < l_1 < \ldots < l_{\widetilde{v}}<\iota} 
 \dfrac{r_{\gamma l_1}r_{l_{\widetilde{v}} \iota} r_{\iota\mu} s_{\mu\gamma} }{(r_{\iota\iota}-r_{l_{\widetilde{v}}l_{\widetilde{v}} })(r_{\iota\iota}-r_{\gamma\gamma})(r_{\iota\iota}-r_{\mu\mu})} \prod_{\widetilde{u}=1}^{\widetilde{v}-1} \dfrac{r_{l_{\widetilde{u}} l_{\widetilde{u}+1} }}{r_{\iota\iota}-r_{l_{\widetilde{u}}l_{\widetilde{u}} }}  
\\ 
&+ \sum_{\mu=\iota+1}^n \sum_{v=1}^{\mu-\iota-1} 
\sum_{\iota < k_1 < \ldots < k_v < \mu} 
 \dfrac{r_{\gamma\iota} r_{\iota k_1}r_{k_v \mu}s_{\mu\gamma} }{
(r_{\iota\iota}-r_{\gamma\gamma}) (r_{\iota\iota}-r_{k_1 k_1})(r_{\iota\iota}-r_{\mu\mu})} 
\prod_{u=1}^{v-1} \dfrac{r_{k_u k_{u+1}}}{r_{\iota\iota}-r_{k_{u+1}k_{u+1}}} 
 +  \sum_{\mu=\iota+1}^n 
 \sum_{v=1}^{\mu-\iota-1} 
 \sum_{\iota < k_1 < \ldots < k_v < \mu} 
 \\
&\sum_{\widetilde{v}=1}^{\iota-\gamma-1} 
\sum_{\gamma < l_1<\ldots < l_{\widetilde{v}}<\iota } 
\dfrac{r_{\gamma l_1}r_{l_{\widetilde{v}}\iota} r_{\iota k_1} r_{k_v \mu}s_{\mu\gamma}
}{(r_{\iota\iota}-
r_{l_{\widetilde{v}}l_{\widetilde{v}}})(r_{\iota\iota}-r_{\gamma\gamma})
(r_{\iota\iota}-r_{k_1 k_1})(r_{\iota\iota}-r_{\mu\mu})}  
\left. 
 \prod_{u=1}^{v-1} \dfrac{r_{k_u k_{u+1}} }{r_{\iota\iota}-r_{k_{u+1}k_{u+1}}} 
 \prod_{\widetilde{u}=1}^{\widetilde{v}-1} 
\dfrac{r_{l_{\widetilde{u}}l_{\widetilde{u}+1}} }{r_{\iota\iota}-r_{l_{\widetilde{u}}l_{\widetilde{u}}} }  
 \right]\\
&+ \underbrace{ \left( s_{\iota\iota}
+ \sum_{\mu=\iota+1}^n \left( 
\dfrac{r_{\iota\mu}}{r_{\iota\iota}-r_{\mu\mu}} 
+ \sum_{v=1}^{\mu-\iota-1} \sum_{\iota < k_1 < \ldots < k_v<\mu} 
\dfrac{r_{\iota k_1}r_{k_v \mu}}{(r_{\iota\iota}-r_{k_1 k_1 })( r_{\iota\iota}-r_{\mu\mu})} \prod_{u=1}^{v-1} \dfrac{r_{k_u k_{u+1}}}{r_{\iota\iota}-r_{k_{u+1}k_{u+1}}}
\right)s_{\mu\iota}
\right) }_{\mbox{ for }\gamma=\iota} \\ 
\end{aligned}
\] 
\[ 
\begin{aligned}  
&=  \sum_{\gamma=1}^{\iota-1}\left[ 
\overbrace{  \left( 
 \dfrac{ r_{\gamma\iota} }{r_{\iota\iota}-r_{\gamma\gamma}}
+\sum_{\widetilde{v}=1}^{\iota-\gamma-1} 
\sum_{\gamma< l_1 < \ldots < l_{\widetilde{v}}<\iota}
\dfrac{r_{\gamma l_1} r_{l_{\widetilde{v}}\iota} }{
(r_{\iota\iota}-r_{l_{\widetilde{v}}l_{\widetilde{v}}})(r_{\iota\iota}-r_{\gamma\gamma})} \prod_{\widetilde{u}=1}^{\widetilde{v}-1} 
\dfrac{r_{l_{\widetilde{u}}l_{\widetilde{u}+1}}}{r_{\iota\iota}-r_{l_{\widetilde{u}}l_{\widetilde{u}}} } \right) }^{L_{\gamma\iota}^{\iota} }
 s_{\iota\gamma} \right. \\
&\;\;\;\;\;\;\;\;\;\;\;+ \sum_{\mu=\iota+1}^n \left( 
\dfrac{r_{\gamma\iota} r_{\iota\mu} }{ (r_{\iota\iota}-r_{\gamma\gamma})( r_{\iota\iota}-r_{\mu\mu})}
+    \sum_{\widetilde{v}=1}^{\iota-\gamma-1} \sum_{\gamma < l_1 < \ldots <  l_{\widetilde{v}}<\iota} 
 \dfrac{r_{\gamma l_1}r_{l_{\widetilde{v}} \iota} r_{\iota\mu}  }{(r_{\iota\iota}-r_{l_{\widetilde{v}}l_{\widetilde{v}} })(r_{\iota\iota}-r_{\gamma\gamma})(r_{\iota\iota}-r_{\mu\mu})} \prod_{\widetilde{u}=1}^{\widetilde{v}-1} \dfrac{r_{l_{\widetilde{u}} l_{\widetilde{u}+1} }}{r_{\iota\iota}-r_{l_{\widetilde{u}}l_{\widetilde{u}} }}  \right.
\\ 
&\;\;\;\;\;\;\;\;\;\;\;\;\;\;\;\;\;\;\;\;\;\;\;\;\;\;\;\;\;+
  \sum_{v=1}^{\mu-\iota-1} 
\sum_{\iota < k_1 < \ldots < k_v < \mu} 
 \dfrac{r_{\gamma\iota} r_{\iota k_1}r_{k_v \mu}  }{
(r_{\iota\iota}-r_{\gamma\gamma}) (r_{\iota\iota}-r_{k_1 k_1})(r_{\iota\iota}-r_{\mu\mu})} 
\prod_{u=1}^{v-1} \dfrac{r_{k_u k_{u+1}}}{r_{\iota\iota}-r_{k_{u+1}k_{u+1}}} 
 \\
&\;\;\;\;\;\;\;\;\;\;\;\;\;\;\;\;\;\;\;\;\;\;\;\;\;\;\;\;\;+ 
  \sum_{v=1}^{\mu-\iota-1} 
\sum_{\iota < k_1 < \ldots < k_v < \mu} 
 \sum_{\widetilde{v}=1}^{\iota-\gamma-1} 
\sum_{\gamma < l_1<\ldots < l_{\widetilde{v}}<\iota } 
\dfrac{r_{\gamma l_1}r_{l_{\widetilde{v}}\iota} r_{\iota k_1} r_{k_v \mu} 
}{(r_{\iota\iota}-
r_{l_{\widetilde{v}}l_{\widetilde{v}}})(r_{\iota\iota}-r_{\gamma\gamma})
(r_{\iota\iota}-r_{k_1 k_1})(r_{\iota\iota}-r_{\mu\mu})} \cdot  \\
\end{aligned}
\] 
\[ 
\begin{aligned}
&\;\;\;\;\;\;\;\;\;\;\;\;\;\;\;\;\;\;\;\;\;\;\;\;\;\;\;\;\;\left.
\underbrace{ \left.
\cdot 
\prod_{u=1}^{v-1} \dfrac{r_{k_u k_{u+1}} }{r_{\iota\iota}-r_{k_{u+1}k_{u+1} } } 
\prod_{\widetilde{u}=1}^{\widetilde{v}-1} 
\dfrac{r_{l_{\widetilde{u}}l_{\widetilde{u}+1}} }{r_{\iota\iota}-r_{l_{\widetilde{u}}l_{\widetilde{u}}} } \right) }_{ 
L_{\gamma\mu}^{\iota}, \mbox{ which includes all terms for }\iota+1 \leq \mu\leq n} 
   s_{\mu\gamma} 
 \right] 
+ \underbrace{ \left( L_{\iota\iota}^{\iota} s_{\iota\iota}
+ 
\sum_{\mu=\iota+1}^n L_{\iota\mu}^{\iota}  s_{\mu\iota}
\right) }_{\mbox{ for }\gamma=\iota} \\ 
&= \sum_{\gamma=1}^{\iota-1}\underbrace{\left(
 \sum_{\mu=\iota}^n L_{\gamma\mu}^{\iota} s_{\mu\gamma}  \right) }_{ (L^{\iota}s)_{\gamma\gamma} }  
+ \underbrace{\left(  \sum_{\mu=\iota}^{n} L_{\iota\mu}^{\iota}s_{\mu\iota} \right)   }_{  (L^{\iota}s)_{\iota\iota}  } 
= \tr(L^{\iota}s). 
\end{aligned} 
\] 
\end{proof}

\begin{remark}\label{remark:deriving-diag-coordinates-of-bsb-inverse-opposite-order} 
Suppose we multiply $bsb^{-1}$ from right to left than from left to right as in Strategy~\ref{strategy:deriving-diag-coordinates-of-bsb-inverse}. That is,  
for $\gamma\geq \iota$, 
\[ (sb^{-1})_{\gamma\iota} = \sum_{\mu=1}^{\iota} s_{\gamma\mu}(b^{-1})_{\mu\iota}
\] with 
\[ (bsb^{-1})_{\iota\iota} = \sum_{\gamma=\iota}^{n} b_{\iota\gamma}
(sb^{-1})_{\gamma\iota}, 
\] 
   then we have 
\[ b_{\iota\gamma}
(sb^{-1})_{\gamma\iota} 
= 
\left( s L^{\iota} 
\right)_{\gamma\gamma}, 
\] 
and similar as before, for $\gamma < \iota$, 
\[ b_{\iota\gamma}
(sb^{-1})_{\gamma\iota} =0 
= 
\left( sL^{\iota}
 \right)_{\gamma\gamma}.  
\] 
\end{remark}

\begin{remark}
Proposition~\ref{proposition:distinct-ev-diagonalizable} may also be thought of as obtaining 1-dimensional eigenspaces for each eigenvalue of $r$. First writing $r$ as $(r_{\iota\gamma})$, 
find a nonzero vector $v_k=(v_1^{(k)},\ldots, v_n^{(k)})$ in the null space $N(r-r_{kk}I)$ of $r-r_{kk}I$ starting with $k=1$ and iterate the process for each $k> 1$, i.e.,   for each fixed $k$, 
\[(r-r_{kk}I)v_k = \left[ 
\begin{array}{ccccc}
r_{11}-r_{kk} & & & \ldots & \\ 
 & r_{22}-r_{kk} & & & r_{\iota\gamma }\\
 & & \ddots & &  \vdots \\ 
 & & & & r_{nn}-r_{kk} \\  
\end{array}
\right] v_k = 0. 
\] We obtain $v_{k+1}^{(k)},\ldots, v_n^{(k)}=0$ from the above constraints. Apply backward substitution to the rest of the relations 
\[ (r_{\iota\iota}-r_{kk})v_{\iota}^{(k)} + r_{\iota,\iota +1}v_{\iota+1}^{(k)} + \ldots + r_{\iota k}v_k^{(k)}=0 
\] for $1 \leq \iota < k$ 
and obtain 
$v_k$ to be 
\[ \left( *,\ldots, *, -\frac{r_{k-1,k}}{r_{k-1,k-1}-r_{kk}}, 1,0,\ldots, 0\right),  
\] 
where $*$'s are uniquely determined by the relations. 
By putting the eigenvectors $v_1,\ldots, v_n$ in a matrix form 
$E= [v_1 \ldots v_n]$  
with respect to the standard basis for  $\mathbb{C}^n$, 
 we see that $E$ is in the Borel and $rE = E\diag(r)$, or 
equivalently, $E^{-1}rE$ $=$ $\diag(r)$. Thus the matrix $b$ in the proof for Proposition~\ref{proposition:distinct-ev-diagonalizable} is precisely $E^{-1}$ after choosing each $b_{\iota\iota}^{-1}$ to equal $v_{\iota}^{(\iota)}$.   
\end{remark}

\section{\texorpdfstring{$B$}{B}-orbits on the \texorpdfstring{$rss$}{rss}-locus}

We now analyze the affine quotient of $\mu^{-1}(0)^{rss}$ by $B$. 

\begin{proposition}\label{proposition:s-in-HR-is-diagonalizable}
For $(r,s,0,0)$ in $\mu^{-1}(0)^{rss}$, $s$ is diagonal.  
\end{proposition} 

We prove Proposition~\ref{proposition:s-in-HR-is-diagonalizable} using strong induction. 

\begin{proof} 
Let $r=(r_{\iota\gamma})$ and $s=(s_{\iota\gamma })$. 
Let $n=2$. Then 
\begin{align*} 
[r,s] &= \left[ 
\begin{array}{cc} 
r_{11} & r_{12} \\ 
 0 & r_{22} \\ 
\end{array}
\right] 
\left[ 
\begin{array}{cc} 
s_{11} & 0 \\ 
 s_{21} & s_{22} \\ 
\end{array}
\right] 
-
\left[ 
\begin{array}{cc} 
s_{11} & 0 \\ 
 s_{21} & s_{22} \\ 
\end{array}
\right] 
\left[ 
\begin{array}{cc} 
r_{11} & r_{12} \\ 
 0 & r_{22} \\ 
\end{array}
\right]  \\
&= \left[ 
\begin{array}{cc} 
r_{11}s_{11}+r_{12}s_{21} & * \\ 
 r_{22}s_{21} & r_{22}s_{22} \\ 
\end{array}
\right] 
-\left[ 
\begin{array}{cc} 
r_{11}s_{11} & * \\ 
 r_{11}s_{21} & r_{22}s_{22}+s_{21}r_{12} \\ 
\end{array}
\right]  \\
&= \left[ 
\begin{array}{cc} 
r_{12}s_{21} & * \\ 
 (r_{22}-r_{11})s_{21} & -s_{21}r_{12} \\ 
\end{array}
\right].  
\end{align*}
Since $r\in\mathfrak{b}^{rss}$, $s_{21}=0$ and since $s=\diag(s)$, we are done.  

Now assuming that the proposition holds when 
the rank of $r$ is $n-1$, 
consider $r'$ and $s'$ where 
\[ r' = 
\left[ \begin{array}{cc}
r & r_{\iota n} \\ 
0 & r_{nn} \\ 
\end{array}
\right] 
\mbox{ and }
s' = 
\left[ \begin{array}{cc}
s & 0    \\ 
s_{n\gamma} & s_{nn} \\ 
\end{array}
\right],  
\] which are $n\times n$ matrices whose upper left block are matrices $r$ and $s$, respectively, 
with $(r',s',0,0)$ in $\mu^{-1}(0)^{rss}$. 
Then $[r',s']=0$ in $\mathfrak{b}^*$, and $r_{nn}$ distinct from $r_{ll}$ for all $l<n$.   
So 
\[
([r',s'])_{\iota\gamma} 
=
 \left\{ 
\begin{aligned} 
  0   \;\;\;\;\;\;\;\;\;\;\;\;\;\;\;\;\;\;\;\;\;\;\;\;\;\;\;\;     & \mbox{ if $\iota< \gamma$} \\ 
(r_{\iota\iota}-r_{\gamma\gamma}) s_{\iota\gamma} +\sum_{k>\iota} r_{\iota k}s_{k\gamma }-\sum_{k<\gamma}s_{\iota k}r_{k\gamma}  \:\:\: &  \mbox{ if $\iota\geq \gamma$}.  
\end{aligned} 
\right. 
\] 
We rewrite $[r',s']$ as the sum 
\[ 
\begin{aligned}
&\left[ 
\begin{array}{cc}
[r,s] & 0 \\ 
0 & 0 \\  
\end{array}  
\right]           \\ 
&+ 
\left[ 
\begin{array}{ccccc}
r_{1n}s_{n1} &   & &  & \\ 
r_{2n}s_{n1} & r_{2n}s_{n2} & & &   \\ 
r_{3n}s_{n1} & r_{3n}s_{n2} & r_{3n}s_{n3} & &  \\ 
 & \vdots  & &   &   \\
 r_{\iota n}s_{n1} & r_{\iota n}s_{n2} & r_{\iota n}s_{n3} &  \ddots  & \\ 
 & \vdots & & &   \\  
(r_{nn}-r_{11})s_{n1} & (r_{nn}-r_{22})s_{n2} -s_{n1}r_{12} & (r_{nn}-r_{33})s_{n3} 
-\sum_{\iota =1}^{2} s_{n\iota}r_{\iota 3} 
& & -\sum_{k<n}s_{nk}r_{kn} \\ 
\end{array}
\right].     \\ 
\end{aligned}
\] 
Since $[r,s]=0$,  
the large matrix on the right hand side of the sum equals zero. 
Now consider the entry in the lower left corner of the large matrix. 
 Since the eigenvalues of $r'$ are pairwise distinct, $s_{n1}=0$. 
 Moving 
 over one column to the right, and remaining in the last row, we see that $s_{n2}=0$ as well. 
Continue recursively until we get to the last column in the large matrix:  
each term in the sum in the bottom right entry equals zero since 
each $s_{nk}=0$ for all $k<n$. 
By the strong induction hypothesis, $s=\diag(s)$ and by the above argument, $s'=\diag(s')$ and this completes the proof. 
\end{proof}

\begin{corollary}\label{corollary:i-or-j-equals-zero-then-s-diagonal}
If $(r,s,i,j)$ is in $\mu^{-1}(0)^{rss}$ with $i$ or $j$ equaling zero, then $s$ is diagonal. 
\end{corollary}

\begin{proof}
This is immediate from Proposition~\ref{proposition:s-in-HR-is-diagonalizable}.
\end{proof}

\begin{lemma}\label{lemma:r-s-0-0-s-is-diagonal}
Let $(r,s,0,0)$ be in $\mu^{-1}(0)^{rss}$. Then there exists $b\in B$ such that $(brb^{-1},bsb^{-1},0,0)$
 $=$ $(\diag(r),\diag(s),0,0)$. 
\end{lemma}

We note the subtlety in Lemma~\ref{lemma:r-s-0-0-s-is-diagonal} that the $B$-action on the points in $B.(r,s,0,0)$ does not change the diagonal coordinates of $r$ and $s$. 

\begin{proof}
By Proposition~\ref{proposition:s-in-HR-is-diagonalizable}, $s$ is diagonal. By 
Proposition~\ref{proposition:distinct-ev-diagonalizable}, there exists a matrix $b$ in the Borel so that $brb^{-1}$ is diagonal. By the second statement in  Remark~\ref{remark:diagonalize-borel-and-its-dual}, we see that $bsb^{-1}$ in $\mathfrak{b}^*$ is always  diagonal. By Proposition~\ref{proposition:diagonal-being-invariant} and Lemma~\ref{lemma:s-diagonal-conjugation-still-diagonal}, 
we see that the diagonal coordinates of $r$ and $s$ are not affected by the $B$-action.  
\end{proof}

\begin{proposition}\label{proposition:i-j-zero-closed-orbit}
Each $B$-orbit containing the quadruple $(r,s,0,0)$, where $r$ is regular semisimple and the commutator of $r$ and $s$ is zero, is closed.
\end{proposition}

\begin{proof}
By Proposition~\ref{proposition:s-in-HR-is-diagonalizable}, $s$ must be diagonal.  
Choose an appropriate 1-parameter subgroup $\lambda(t)$ so that 
\[ \lim_{t\rightarrow 0} \lambda(t).( r,\diag(s),0,0)=(\diag(r),\diag(s),0,0).  \]
Since $(\diag(r),\diag(s),0,0)$ is in the $B$-orbit by 
Lemma~\ref{lemma:r-s-0-0-s-is-diagonal}, we are done. 
\end{proof}

\begin{corollary}\label{corollary:Dreg-affine-quotient-closed-orbits}
The affine quotient $\{(r,s,0,0):r \mbox{ regular}, [r,s]=0 \}/\!\!/B$ consists of closed orbits. 
\end{corollary}

\begin{proof}
This follows from Proposition~\ref{proposition:i-j-zero-closed-orbit}. 
\end{proof}

\begin{corollary}\label{corollary:closed-orbits-are-disjoint}
If the points 
$(\diag(r),\diag(s),0,0)$ and $(\diag(r'),\diag(s'),0,0)$ are distinct, then the $B$-orbits 
$B.(\diag(r),\diag(s),0,0)$ and $B.(\diag(r'),\diag(s'),0,0)$ are disjoint. 
\end{corollary}

\begin{proof}
Suppose the $B$-orbits are not disjoint. 
We will show that the quadruples 
$(\diag(r),\diag(s),0,0)$ are $(\diag(r'),\diag(s'),0,0)$ are the same point in $T^*(\mathfrak{b}\times\mathbb{C}^n)$. 
By assumption, 
the intersection of the $B$-orbits is nomempty, so choose $b\in B$ so that $b.(\diag(r),\diag(s),0,0)=(\diag(r'),\diag(s'),0,0)$. 
This implies  
$(b\diag(r)b^{-1}, b\diag(s)b^{-1},0,0)$ $=$ $(\diag(r')$, $\diag(s')$, $0$, $0)$.  
Setting 
the first coordinates equal, 
$b\diag(r)b^{-1}=\diag(r')$. So $b\diag(r)b^{-1}$ is diagonal. By Lemma~\ref{lemma:s-diagonal-conjugation-still-diagonal}, 
$\diag(r)$ is fixed under the $B$-action: $b\diag(r)b^{-1}=\diag(r)=\diag(r')$. 
Next, 
we see that 
$b\diag(s)b^{-1}=\diag(s)=\diag(s')$,  
where the first equality holds by   
the second sentence in Remark~\ref{remark:diagonalize-borel-and-its-dual}. 
\end{proof} 

\begin{remark}\label{remark:exhausting-all-closed-orbits}
Proposition~\ref{proposition:openlocus-affine-quotient} together with Proposition~\ref{proposition:i-j-zero-closed-orbit}
show that all closed $B$-orbits in $\mu^{-1}(0)^{rss}$ contain points of the form, and thus must be of the form, $(r,s,0,0)$.  
\end{remark}

\begin{proposition}\label{proposition:openlocus-affine-quotient}
Each orbit closure in $\mu^{-1}(0)^{rss}/\!\!/B$ contains a point of the form $(\diag(r),\diag(s'),0,0)$.  
\end{proposition}

\begin{proof}
For $(r,s,i,j)$ be in $\mu^{-1}(0)^{rss}$,  
we will show that there exists a quadruple of the form 
$(\diag(r)$, $\diag(s')$, $0$, $0)$ in its $B$-orbit closure, where $s'$ is some other element in $\mathfrak{b}^*$. 

Firstly, consider points of the form $(r,s,0,0)$. By Lemma~\ref{lemma:r-s-0-0-s-is-diagonal}, we are done.  
Thus consider points of the form $(r,s,i,j)$ in $\mu^{-1}(0)^{rss}$
with $i$ or $j$ not necessarily 0. 
By 
Proposition~\ref{proposition:distinct-ev-diagonalizable}, there is $b\in B$ so that 
$(brb^{-1},bsb^{-1},bi,jb^{-1})=(\diag(r),s',i',j')$. 
Let us write $r= (r_{\iota\gamma})$, $s'=(s_{\iota\gamma}')$, $i'=(x_{\iota}')$, and $j'=(y_{\iota}')$.  
Since $\mu$ is $B$-equivariant, $[r,s]+ij=0$ implies $[\diag(r),s']+i'j'=0$ in $\mathfrak{b}^*$. 
Since 
\[
([\diag(r),s'] + i'j')_{\iota\gamma} = \left\{ 
\begin{aligned} 
0       \;\;\;\;\;\;\;\;\;\;\;\;\;\;\;\; & \mbox{ if $\iota< \gamma$} \\ 
x_{\iota}'y_{\iota}' \;\;\;\;\;\;\;\;\;\;\;\; \;\; &  \mbox{ if $\iota=\gamma$} \\ 
(r_{\iota\iota}-r_{\gamma\gamma})s_{\iota\gamma}' + x_{\iota}'y_{\gamma}' \;\; & \mbox{ if $\iota>\gamma$},   
\end{aligned} 
\right. 
\] 
  $x_{\iota}'=0$ or $y_{\iota}'=0$ for each $\iota$. 
At this point, we give a recipe for choosing the $a_{\iota}$'s in the 1-parameter subgroup $\lambda(t)=\diag(t^{a_1},\ldots,t^{a_n})$ where $a_{\iota}\in\mathbb{Z}$. Choose 
\[
a_{\iota}= \left\{ 
\begin{aligned} 
1   \;\;\;    & \mbox{ if $x_{\iota}' \not=0$} \\ 
-1  \;\;\; &  \mbox{ if $y_{\iota}'   \not=0$} \\ 
0   \;\;\; & \mbox{ if $x_{\iota}'=y_{\iota}'=0$}.  
\end{aligned} 
\right. 
\] 
Now consider $\lambda(t).(\diag(r),s',i',j')$, which equals 
\[\left( 
\diag(r), \left[ 
\begin{array}{cccc}
s_{11} &                    &         &  \\ 
       & \ddots             &         &  \\  
       &  t^{a_{\iota}-a_{\gamma}}s_{\iota\gamma}'& \ddots  & \\ 
       &                    &         & s_{nn}\\
\end{array}
\right], 
\left[ \begin{array}{c}
t^{a_1}x_1' \\ 
\vdots \\  
t^{a_n}x_n' \\  
\end{array}  
\right],  
\left[    
\begin{array}{ccc}
t^{-a_1} y_1' & \ldots & t^{-a_n} y_n' \\ 
\end{array}
\right] 
\right). 
 \]
If $x_{\iota}'\not=0$, then we have $t x_{\iota}'$. If $y_{\gamma}'\not=0$, then we also have $ty_{\gamma}'$. 
If $s_{\iota\gamma}'\not=0$, then since $s_{\iota\gamma}'=-x_{\iota}'y_{\gamma}'/(r_{\iota\iota}-r_{\gamma\gamma})$, we have $t^2 s_{\iota\gamma}'$. 
We conclude 
\[ \lim_{t\rightarrow 0} \lambda(t).(\diag(r),s',i',j')=(\diag(r),\diag(s'),0,0).
\] 
\end{proof}

\begin{remark}\label{remark:orbitclosure-at-most-one-limit-point} 
We would like to mention that the proof for  Proposition~\ref{proposition:rss-extending-to-all-of-the-reg-semi-stable-locus} 
includes showing that 
the closure of each $B$-orbit on $\mu^{-1}(0)^{rss}/\!\!/B$ contains at most one point of the form $(\diag(r),\diag(s'),0,0)$.   
\end{remark}

\begin{remark}\label{remark:orbitclosures-emptyintersection-implies-orbitsempty-intersection}
It is clear that if $\overline{B.P_1}\cap\overline{B.P_2}=\varnothing$ where $P_1, P_2\in T^*(\mathfrak{b}\times \mathbb{C}^n)$, then $B.P_1\cap B.P_2=\varnothing$. 
\end{remark}

\section{Proof of Proposition  \texorpdfstring{$\ref{proposition:rss-extending-to-all-of-the-reg-semi-stable-locus}$}{rss categorical quotient}}

We now prove  Proposition~\ref{proposition:rss-extending-to-all-of-the-reg-semi-stable-locus}.

\begin{proof} 
It is clear that $P$ is regular and by Proposition~\ref{proposition:diagonal-being-invariant} and \ref{proposition:deriving-diag-coordinates-of-bsb-inverse}, 
 $P$ is $B$-invariant. 
By 
  Remark~\ref{remark:exhausting-all-closed-orbits}, closed orbits are precisely those that contain a point of the form $(r,s,0,0)$ as orbits that contain a point of the form $(r,s,i,j)$, where $i$ or $j$ is nonzero, are not closed. 
Since 
 each 
   closed orbit contains a unique point of the form $(\diag(r),\diag(s),0,0)$ by Lemma~\ref{lemma:r-s-0-0-s-is-diagonal} and by Corollary~\ref{corollary:closed-orbits-are-disjoint}, it is clear that two such points  go to different points in $\mathbb{C}^{2n}\setminus \Delta_n$ via the map $P$. In particular, we see that two  $B$-orbits $B.(\diag(r),\diag(s),0,0)$ and $B.(\diag(r'),\diag(s'),0,0)$ cannot be in the same $B$-orbit closure  
 for 
 if it were nonempty, then two such closed orbits would not be separated by $P$. 
Thus each $B$-orbit closure contains a unique closed orbit, and 
exactly 1 point of the form $(\diag(r),\diag(s),0,0)$. 
Thus $P$ is injective on orbit closures.

We will now show that $P$ is surjective. Consider $(x_{11},\ldots,x_{nn},y_{11},\ldots,y_{nn})$ in $\mathbb{C}^{2n}\setminus \Delta_n$. 
A point $(r,s,i,j)$ in $\mu^{-1}(0)^{rss}$ is constrained by $[r,s]+ij=0$. 
Since 
\[ 
([r,s]+ij)_{\iota\gamma} =
\left\{ 
\begin{aligned}
0 \;\;\;\;\;\;\;\; \;\;\;\;\;\; \;\;\;\; \;\;\;\;\;\;\;\;    &\mbox{ if }\iota < \gamma  \\  
\sum_{\iota <k\leq n} r_{\iota k}s_{k \iota} 
- \sum_{1\leq k< \iota } s_{\iota k}r_{k \iota}+ x_{\iota}y_{\iota} \;\;\;\; &\mbox{ if } \iota = \gamma \\ 
\sum_{\iota \leq k\leq n} r_{\iota k}s_{k\gamma} 
- \sum_{1\leq k\leq \gamma}s_{\iota k}r_{k\gamma}+ x_{\iota} y_{\gamma}\;\;    &\mbox{ if }\iota >\gamma, \\ 
\end{aligned}
\right. 
\] 
none 
 of 
 the 
 coordinate functions of $[r,s]+ij$ involve $s_{\iota\iota}$'s. So $s_{\iota\iota}$ is a free parameter. 
Take $s_{\iota\iota}$ to equal 
\[ 
\begin{aligned}
&\:\: s_{\iota\iota} 
=  y_{\iota\iota}
-\left[  
\sum_{\mu=\iota+1}^n \left( 
\dfrac{r_{\iota\mu} s_{\mu\iota} }{r_{\iota\iota}-r_{\mu\mu}} 
+ \sum_{v=1}^{\mu-\iota-1} \sum_{\iota < k_1 < \ldots < k_v<\mu} 
\dfrac{r_{\iota k_1}r_{k_v \mu} s_{\mu\iota} }{(r_{\iota\iota}-r_{k_1 k_1 })( r_{\iota\iota}-r_{\mu\mu})} \prod_{u=1}^{v-1} \dfrac{r_{k_u k_{u+1}}}{r_{\iota\iota}-r_{k_{u+1}k_{u+1}}}
\right)\right.  
\\
&+ 
\sum_{\gamma=1}^{\iota-1}\left[ 
   \dfrac{ r_{\gamma\iota} s_{\iota\gamma}  }{r_{\iota\iota}-r_{\gamma\gamma}}
+\sum_{\widetilde{v}=1}^{\iota-\gamma-1} 
\sum_{\gamma< l_1 < \ldots < l_{\widetilde{v}}<\iota}
\dfrac{r_{\gamma l_1} r_{l_{\widetilde{v}}\iota} s_{\iota \gamma} }{
(r_{\iota\iota}-r_{l_{\widetilde{v}}l_{\widetilde{v}}})(r_{\iota\iota}-r_{\gamma\gamma})} \prod_{\widetilde{u}=1}^{\widetilde{v}-1} 
\dfrac{r_{l_{\widetilde{u}}l_{\widetilde{u}+1}}}{r_{\iota\iota}-r_{l_{\widetilde{u}}l_{\widetilde{u}}} } 
 \right. \\
&+ \sum_{\mu=\iota+1}^n \dfrac{r_{\gamma\iota} r_{\iota\mu}s_{\mu\gamma}}{ (r_{\iota\iota}-r_{\gamma\gamma})( r_{\iota\iota}-r_{\mu\mu})}
+ \sum_{\mu=\iota+1}^n   \sum_{\widetilde{v}=1}^{\iota-\gamma-1} 
\sum_{\gamma < l_1 < \ldots < l_{\widetilde{v}}<\iota} 
 \dfrac{r_{\gamma l_1}r_{l_{\widetilde{v}} \iota} r_{\iota\mu} s_{\mu\gamma} }{(r_{\iota\iota}-r_{l_{\widetilde{v}}l_{\widetilde{v}} })(r_{\iota\iota}-r_{\gamma\gamma})(r_{\iota\iota}-r_{\mu\mu})} \prod_{\widetilde{u}=1}^{\widetilde{v}-1} \dfrac{r_{l_{\widetilde{u}} l_{\widetilde{u}+1} }}{r_{\iota\iota}-r_{l_{\widetilde{u}}l_{\widetilde{u}} }}  
\\ 
&+ \sum_{\mu=\iota+1}^n \sum_{v=1}^{\mu-\iota-1} 
\sum_{\iota < k_1 < \ldots < k_v < \mu} 
 \dfrac{r_{\gamma\iota} r_{\iota k_1}r_{k_v \mu}s_{\mu\gamma} }{
(r_{\iota\iota}-r_{\gamma\gamma}) (r_{\iota\iota}-r_{k_1 k_1})(r_{\iota\iota}-r_{\mu\mu})} 
\prod_{u=1}^{v-1} \dfrac{r_{k_u k_{u+1}}}{r_{\iota\iota}-r_{k_{u+1}k_{u+1}}} 
 \\
&+ 
 \sum_{\mu=\iota+1}^n 
 \sum_{v=1}^{\mu-\iota-1} 
\sum_{\iota < k_1 < \ldots < k_v < \mu} 
 \sum_{\widetilde{v}=1}^{\iota-\gamma-1} 
\sum_{\gamma < l_1<\ldots < l_{\widetilde{v}}<\iota } 
\dfrac{r_{\gamma l_1}r_{l_{\widetilde{v}}\iota} r_{\iota k_1} r_{k_v \mu}s_{\mu\gamma}
}{(r_{\iota\iota}-
r_{l_{\widetilde{v}}l_{\widetilde{v}}})(r_{\iota\iota}-r_{\gamma\gamma})
(r_{\iota\iota}-r_{k_1 k_1})(r_{\iota\iota}-r_{\mu\mu})} \cdot\\
&\left. \left. \;\;\;\;\;\;\;\cdot\prod_{u=1}^{v-1} \dfrac{r_{k_u k_{u+1}} }{r_{\iota\iota}-r_{k_{u+1}k_{u+1}}} 
 \prod_{\widetilde{u}=1}^{\widetilde{v}-1} 
\dfrac{r_{l_{\widetilde{u}}l_{\widetilde{u}+1}} }{r_{\iota\iota}-r_{l_{\widetilde{u}}l_{\widetilde{u}}} }  
 \right]\right],\\
\end{aligned}
\] 
and 
 take $r_{\iota\iota}$ to equal $x_{\iota\iota}$. 
Then a quadruple $(r,s,i,j)$ satisfying such conditions, whose $B$-orbit closure contains the unique point $(\diag(x_{\iota\iota}),\diag(y_{\iota\iota}),0,0)$, will map to $(x_{11},\ldots,x_{nn},y_{11},\ldots,y_{nn})$. 
\end{proof}

\section{The coordinate ring of the \texorpdfstring{$rss$}{rss}-locus}

We will prove that 
the coordinate ring of the affine quotient $\mu^{-1}(0)^{rss}/\!\!/B$ is isomorphic to the coordinate ring of pairwise diagonal matrices in $\mu^{-1}(0)^{rss}$.

\subsection{Changing coordinates}

\begin{definition}\label{definition:subdiagonal-level-k}
Let $(a_{\iota\gamma})$ be a matrix. {\em Level $k$ subdiagonal entries} consist of those coordinates $a_{\iota\gamma}$ below the main diagonal that satisfy $\iota-\gamma=k$. 
\end{definition}

\begin{example}\label{example:subdiagonal-level-k-nxn-matrix}
For an $n\times n$ matrix, level $0$ subdiagonal entries are precisely those along the main diagonal. Level $1$ subdiagonal entries are those immediately below the main diagonal. Level $n-1$ subdiagonal entry is the $(n,1)$-entry, in the lower left corner. 
\end{example}
 
In the next Proposition, we prove that for $\iota > \gamma$,   $(\mu(r,s,i,j))_{\iota\gamma}$
 may be solved for the coordinate function $s_{\iota\gamma}$, which depends on those $s_{ij}$ satisfying $i-j>\iota-\gamma$. That is,  each $s_{\iota\gamma}$ is a regular function of $s_{ij}$'s in level $k$ subdiagonal, where $k>\iota-\gamma$. 

\begin{proposition}\label{proposition:off-diagonal-coordinate-function-defining-variety}
For each $\iota>\gamma$, the coordinate function $([r,s]+ij)_{\iota\gamma}$ 
may be solved for $s_{\iota\gamma}$, which is in 
\[\im \left( \mathbb{C}[ 
\{r_{\iota j}\}_{\iota<j}, \{r_{i\gamma} \}_{i<\gamma},\{s_{ij}\}_{i-j>\iota-\gamma},
x_{\iota},y_{\gamma}][(r_{\iota\iota}-r_{\gamma\gamma})^{-1}]
\longrightarrow\mathbb{C}[\mu^{-1}(0)^{rss}]\right). 
\] 
\end{proposition}

\begin{proof}
For $\iota>\gamma$, the sequence of equalities 
\[  
0 =([r,s]+ij)_{\iota\gamma}= \sum_{\iota\leq j}r_{\iota j}s_{j \gamma}-\sum_{i\leq \gamma}s_{\iota i}r_{i\gamma}+x_{\iota}y_{\gamma} \\ 
=(r_{\iota\iota}-r_{\gamma\gamma} )s_{\iota\gamma} +\sum_{\iota <j}r_{\iota j}s_{j\gamma}-\sum_{i<\gamma}s_{\iota i}r_{i\gamma} + x_{\iota}y_{\gamma} 
\] 
 implies 
\[ 
   s_{\iota\gamma} =\dfrac{1}{r_{\gamma\gamma}-r_{\iota\iota}}\left( \sum_{\iota <j}r_{\iota j}s_{j\gamma}-\sum_{i<\gamma}s_{\iota i}r_{i\gamma} + x_{\iota}y_{\gamma}\right).   
\] 
\end{proof}

\begin{coordinatering}\label{coordinatering:coordinate-ring-off-diagonal}
We apply Proposition~\ref{proposition:off-diagonal-coordinate-function-defining-variety}  starting 
  from level $n-1$ subdiagonal of 
  $[r,s]+ij$, 
    add $(r_{nn}-r_{11})^{-1}$ to, and thus will be able to remove the parameter $s_{n1}$ from, the coordinate ring $\mathbb{C}[\mu^{-1}(0)^{rss}]$. 
     We then move to level $n-2$ subdiagonal and repeat the procedure by adding 
$(r_{n-1,n-1}-r_{11})^{-1}$ to the ring and then removing $s_{n-1,1}$, and then 
adding  
 $(r_{nn}-r_{22})^{-1}$ to the coordinate ring and then removing 
 $s_{n2}$. 
  Continue by moving up to the next subdiagonal.  
 \end{coordinatering}

\begin{corollary}\label{corollary:off-diag-coord-fn-defining-variety-iterated-steps}
It follows from Proposition~\ref{proposition:off-diagonal-coordinate-function-defining-variety} that $s_{\iota\gamma}$ is in 
\[ \begin{aligned} 
\im (\: \mathbb{C}[\{r_{ij}\}_{j>i\geq \iota\mbox{ or }i<j\leq \gamma}, 
&\{ x_k\}_{k\geq \iota},\{ y_l\}_{l\leq \gamma}][\{(r_{ii}-r_{jj})^{-1} \}_{i-j\geq \iota-\gamma} ] \\ 
& \longrightarrow \mathbb{C}[\mu^{-1}(0)^{rss}] \: ). \\
\end{aligned}
\] 
\end{corollary}

Corollary~\ref{corollary:off-diag-coord-fn-defining-variety-iterated-steps} shows each $s_{\iota\gamma}$ (where $\iota>\gamma$) may be solved so that it does not depend on any of the entries of $s\in\mathfrak{b}^*$. 

\begin{proof}
 Exhaust Systematic Procedure~\ref{coordinatering:coordinate-ring-off-diagonal} 
 recursively decreasing to the next sublevel (and thus moving closer to the main diagonal). 
\end{proof}

\begin{corollary}\label{corollary:diagonalfunction-matrix-variety}
After replacing 
 each parameter $s_{\mu\nu}$ in the coordinate function $([r,s]+ij)_{\iota\iota}$ by recursively applying Systematic Procedure~\ref{coordinatering:coordinate-ring-off-diagonal}, 
 we obtain that $([r,s]+ij)_{\iota\iota}$ is in the image 
\[ 
   \im(\mathbb{C}[r_{ij},\{ x_k\}_{k\geq \iota} , \{ y_l\}_{l\leq \iota} ] [(r_{ii}-r_{jj})^{-1}]\rightarrow \mathbb{C}[\mu^{-1}(0)^{rss}]). 
\] 
\end{corollary}

\begin{proof}
This follows from Corollary~\ref{corollary:off-diag-coord-fn-defining-variety-iterated-steps} since for each $\iota > \gamma$,  $s_{\iota\gamma}$ may be solved 
so that it is independent of 
the coordinates of $s$, 
and also since 
$s_{\iota\iota}$ are not constrained under the moment map. 
\end{proof}

\begin{corollary}\label{corollary:diag-fn-defining-matrixvariety-explicit-eqn}
Writing 
  $F_{\iota}:=(\mu(r,s,i,j))_{\iota\iota}$, 
the image under the map given in Corollary~\ref{corollary:diagonalfunction-matrix-variety}
is  
\[ 
\begin{aligned}
&F_{\iota} 
=  x_{\iota}y_{\iota} 
+ \sum_{\mu=\iota+1}^n \left( 
\dfrac{r_{\iota\mu} x_{\mu}y_{\iota} }{r_{\iota\iota}-r_{\mu\mu}} 
+ \sum_{v=1}^{\mu-\iota-1} \sum_{\iota < k_1 < \ldots < k_v<\mu} 
\dfrac{r_{\iota k_1}r_{k_v \mu} x_{\mu}y_{\iota}  }{(r_{\iota\iota}-r_{k_1 k_1 })( r_{\iota\iota}-r_{\mu\mu})} \prod_{u=1}^{v-1} \dfrac{r_{k_u k_{u+1}}}{r_{\iota\iota}-r_{k_{u+1}k_{u+1}}}
\right) 
\\
&+ 
\sum_{\gamma=1}^{\iota-1}\left[ 
   \dfrac{ r_{\gamma\iota} x_{\iota} y_{\gamma} }{r_{\iota\iota}-r_{\gamma\gamma}}
+\sum_{\widetilde{v}=1}^{\iota-\gamma-1} 
\sum_{\gamma< l_1 < \ldots < l_{\widetilde{v}}<\iota}
\dfrac{r_{\gamma l_1} r_{l_{\widetilde{v}}\iota} 
x_{\iota} y_{\gamma} }{
(r_{\iota\iota}-r_{l_{\widetilde{v}}l_{\widetilde{v}}})(r_{\iota\iota}-r_{\gamma\gamma})} \prod_{\widetilde{u}=1}^{\widetilde{v}-1} 
\dfrac{r_{l_{\widetilde{u}}l_{\widetilde{u}+1}}}{r_{\iota\iota}-r_{l_{\widetilde{u}}l_{\widetilde{u}}} } 
 \right. \\
&+ \sum_{\mu=\iota+1}^n \dfrac{r_{\gamma\iota} r_{\iota\mu}
x_{\mu} y_{\gamma}
}{ (r_{\iota\iota}-r_{\gamma\gamma})( r_{\iota\iota}-r_{\mu\mu})}
+ \sum_{\mu=\iota+1}^n   \sum_{\widetilde{v}=1}^{\iota-\gamma-1} 
\sum_{\gamma < l_1 < \ldots < l_{\widetilde{v}}<\iota} 
 \dfrac{r_{\gamma l_1}r_{l_{\widetilde{v}} \iota} r_{\iota\mu} 
 x_{\mu} y_{\gamma}
 }{(r_{\iota\iota}-r_{l_{\widetilde{v}}l_{\widetilde{v}} })(r_{\iota\iota}-r_{\gamma\gamma})(r_{\iota\iota}-r_{\mu\mu})} \prod_{\widetilde{u}=1}^{\widetilde{v}-1} \dfrac{r_{l_{\widetilde{u}} l_{\widetilde{u}+1} }}{r_{\iota\iota}-r_{l_{\widetilde{u}}l_{\widetilde{u}} }}  
\\ 
&+ \sum_{\mu=\iota+1}^n \sum_{v=1}^{\mu-\iota-1} 
\sum_{\iota < k_1 < \ldots < k_v < \mu} 
 \dfrac{r_{\gamma\iota} r_{\iota k_1}r_{k_v \mu}
 x_{\mu} y_{\gamma}  
 }{
(r_{\iota\iota}-r_{\gamma\gamma}) (r_{\iota\iota}-r_{k_1 k_1})(r_{\iota\iota}-r_{\mu\mu})} 
\prod_{u=1}^{v-1} \dfrac{r_{k_u k_{u+1}}}{r_{\iota\iota}-r_{k_{u+1}k_{u+1}}} 
 \\
&+ 
 \sum_{\mu=\iota+1}^n 
 \sum_{v=1}^{\mu-\iota-1} 
\sum_{\iota < k_1 < \ldots < k_v < \mu} 
 \sum_{\widetilde{v}=1}^{\iota-\gamma-1} 
\sum_{\gamma < l_1<\ldots < l_{\widetilde{v}}<\iota } 
\dfrac{r_{\gamma l_1}r_{l_{\widetilde{v}}\iota} r_{\iota k_1} r_{k_v \mu} 
x_{\mu} y_{\gamma} 
}{(r_{\iota\iota}-
r_{l_{\widetilde{v}}l_{\widetilde{v}}})(r_{\iota\iota}-r_{\gamma\gamma})
(r_{\iota\iota}-r_{k_1 k_1})(r_{\iota\iota}-r_{\mu\mu})} \cdot\\
&\;\;\;\;\;\;\;\;\;\;\;\;\;\left. \prod_{u=1}^{v-1} \dfrac{r_{k_u k_{u+1}} }{r_{\iota\iota}-r_{k_{u+1}k_{u+1}}} 
 \prod_{\widetilde{u}=1}^{\widetilde{v}-1} 
\dfrac{r_{l_{\widetilde{u}}l_{\widetilde{u}+1}} }{r_{\iota\iota}-r_{l_{\widetilde{u}}l_{\widetilde{u}}} }  
 \right],\\
\end{aligned}
\] 
%
%
which is compactly written as $jL^{\iota}i$ with $L^{\iota}$ as defined in 
Notation~\ref{notation:coordinatefunction-of-product-of-rs}. 
\end{corollary}

\subsection{B-invariant functions}\label{subsection:B-invariantfunctions}

We will show that $B$-invariant functions on $\mu^{-1}(0)^{rss}$ include $F_{\iota}$ (involving $r$, $i$, and $j$) as  in Proposition~\ref{proposition:diagonal-rss-natural-matrix-representation}, $G_{\iota}$ (involving $r$ and $s$) as in Proposition~\ref{proposition:B-invariant-functions-involving-s}, $H_{\iota}$
(involving $r$) as in Proposition~\ref{proposition:B-invariant-functions-diagonal-entries-of-r}, and $K_{\gamma\nu}$ (involving the inverse of the difference of the diagonal coordinates of $r$) as in Proposition~\ref{proposition:B-invariant-invert-difference-diag-r}. They are summarized in the box
\[ 
\framebox{$
\begin{aligned}
F_{\iota}(r,s,i,j) &=  \tr\left(j L^{\iota} i  \right)
			\\
G_{\iota}(r,s,i,j)&=
  \tr \left( L^{\iota} s\right) 
			\\ 
H_{\iota}(r,s,i,j)&=\tr(L^{\iota}r) 
			\\
K_{\gamma\nu}(r,s,i,j) &= [\tr((L^{\nu}-L^{\gamma})r)]^{-1} 
			\\
\mbox{ where } &1\leq \iota\leq n\mbox{ and }\: 1\leq \gamma< \nu \leq n. \\
\end{aligned}
$}  
\]

\begin{proposition}\label{proposition:diagonal-rss-natural-matrix-representation}
Denoting 
$F_{\iota}(r,s,i,j):= ([r,s]+ij)_{\iota\iota}$ from Corollary~\ref{corollary:diag-fn-defining-matrixvariety-explicit-eqn},  
\[ 
F_{\iota}(r,s,i,j) =  \left[
\tr\left( \prod_{
1\leq k\leq n, k\not=\iota 
} l_k(r) \right)\right]^{-1}  \tr\left(  j \left(\prod_{
1\leq k\leq n, k\not=\iota
} l_k(r) \right)i \right) 
\]  is $B$-invariant. 
\end{proposition}

\begin{proof}
For any $d\in B$, 
\[\begin{aligned} 
F_{\iota}(d.(r,s,i,j)) &=  \tr(jd^{-1} L^{\iota}(d.r) di) \\
&= \tr(  j d^{-1} \: d L^{\iota} (r)  d^{-1} \:  d i)   \mbox{ by Lemma~\ref{lemma:Baction-on-Liota-operator}} \\
&=\tr(  jL^{\iota}(r) i) = F_{\iota}(r,s,i,j).  \\
\end{aligned} 
\] 
\end{proof}

\begin{proposition}\label{proposition:B-invariant-functions-involving-s}
Denoting $G_{\iota}(r,s,i,j):=s_{\iota\iota}'$ from  Trace~\ref{trace:map-from-B-rss-to-complex-space},  
\[ G_{\iota}(r,s,i,j)=
\left[ \tr\left( \prod_{
1\leq k\leq n, k\not=\iota 
} l_k(r) \right)\right]^{-1} 
  \tr \left(\prod_{ 
1\leq k\leq n, k\not=\iota 
} l_k(r) \; s\right)  
\] 
is $B$-invariant.  
\end{proposition} 

\begin{proof}
 For any $d\in B$, 
\[\begin{aligned} 
G_{\iota}(d.(r,s,i,j)) &= \tr(L^{\iota}(d.r) dsd^{-1}) \\
 &= \tr(d L^{\iota}(r)d^{-1} \:   dsd^{-1}) \mbox{ by Lemma~\ref{lemma:Baction-on-Liota-operator}}  \\ 
 &= \tr(d L^{\iota}(r)  sd^{-1}) \\ 
  &= \tr(L^{\iota}(r)  s) = G_{\iota}(r,s,i,j).\\  
\end{aligned} 
\] 
\end{proof}

\begin{proposition}\label{proposition:B-invariant-functions-diagonal-entries-of-r}
Denoting $H_{\iota}(r,s,i,j):=r_{\iota\iota}$ from Proposition~\ref{proposition:diag-entries-r-is-trace}, 
\[ H_{\iota}(r,s,i,j)=
\left[ \tr\left( \prod_{
1\leq k\leq n, k\not=\iota 
} l_k(r) \right)\right]^{-1} 
  \tr \left(\prod_{ 
1\leq k\leq n, k\not=\iota 
} l_k(r) \; r\right) 
\] is $B$-invariant. 
\end{proposition}

\begin{proof}
This proof is analogous to the proof of  Proposition~\ref{proposition:B-invariant-functions-involving-s}: to prove this Proposition, we would need to replace $s$ with $r$.  
\end{proof}   
 
\begin{corollary}\label{corollary:B-invar-fns-operator-and-r}
$H_{\iota}(r,s,i,j)$ in Proposition~\ref{proposition:B-invariant-functions-diagonal-entries-of-r} may be written as $e_{\iota}^* \: r\: e_{\iota}$, 
 where $e_{\iota}$ is the standard basis vector for $\mathbb{C}^n$ and $e_{\iota}^*$ is an elementary covector. 
\end{corollary}

\begin{proof}  
This is clear, and the product $e_{\iota}^* \: r\: e_{\iota}$ of matrices is $B$-invariant since for any $d\in B$, $(drd^{-1})_{\iota\iota}=d_{\iota\iota}r_{\iota\iota}d_{\iota\iota}^{-1}=r_{\iota\iota}$. 
\end{proof} 

\begin{proposition}\label{proposition:B-invariant-invert-difference-diag-r}
Denoting $K_{\gamma\nu}(r,s,i,j):= (r_{\nu\nu} - r_{\gamma\gamma})^{-1}$ 
from Proposition~\ref{proposition:difference-of-r-is-trace},  
where 
$1\leq \gamma < \nu \leq n$,  
\[ K_{\gamma\nu}(r,s,i,j) = [\tr((L^{\nu}-L^{\gamma})r)]^{-1}
\] is $B$-invariant.  
\end{proposition} 

\begin{proof} 
For any $1\leq \gamma< \nu\leq n$, 
\[ 
\left( K_{\gamma\nu}(r,s,i,j)\right)^{-1} =  \tr(L^{\nu}r - L^{\gamma}r)  
= 
  \underbrace{\tr(L^{\nu}r)}_{B-\mbox{inv}} - \underbrace{\tr(L^{\gamma}r)}_{B-\mbox{inv}}    
\]  implies $(K_{\gamma\nu})^{-1}$ is $B$-invariant. Since it is never vanishing, 
$K_{\gamma\nu}$ is $B$-invariant. 
\end{proof} 

\begin{corollary}\label{corollary:B-invar-fns-difference-of-operator-and-r} 
 $K_{\gamma\nu}(r,s,i,j)$ in Proposition~\ref{proposition:B-invariant-invert-difference-diag-r} 
 may be written as 
$(e_{\nu}^* (r-r_{\gamma\gamma}I)e_{\nu})^{-1}$, where $e_\nu$ 
is the standard basis vector for $\mathbb{C}^n$ and $e_\nu^*$ is its dual.  
\end{corollary}  

\begin{proof}
This is clear, and the product of matrices 
$(e_{\nu}^* (r-r_{\gamma\gamma}I)e_{\nu})^{-1}$ is $B$-invariant since for any $d\in B$ and for any $\gamma < \nu$, 
$(d(r-r_{\gamma\gamma}I)d^{-1})_{\nu\nu} 
 =(drd^{-1}-r_{\gamma\gamma}I)_{\nu\nu} 
 = (drd^{-1})_{\nu\nu} -(r_{\gamma\gamma}I)_{\nu\nu}
=r_{\nu\nu}-r_{\gamma\gamma}$, which is never zero; so its inverse is well-defined. 
\end{proof}

\subsection{The initial ideal and regular sequence}

\begin{notation}\label{notation:passingfrom-regular-to-polynomial-concept}
We will define $z_{ij}^{(kl)}:= \dfrac{r_{kl}}{r_{ii}-r_{jj}} $ for $i\not= j$. 
\end{notation}


\begin{weight}\label{weight:weight-on-coordinate-ring} 
Using Notation~\ref{notation:passingfrom-regular-to-polynomial-concept}, the coordinate ring $\mathbb{C}[x_k,y_l,z_{ij}^{(kl)}]$ has the following integral weight on each variable: $\wt(x_k)= 1, \wt(y_l)=1, \wt(z_{ij}^{(kl)})=0$. 
\end{weight} 

\begin{monomial}\label{monomial:grobner-basis-monomial-ordering}
We fix a term order $>$ on $\mathbb{C}[x_k, y_l, z_{ij}^{(kl)}]$ via the following refinement: 
write 
\[ m= x_1^{a_1}\cdots x_n^{a_n}
y_n^{b_n}\cdots y_1^{b_1}
(z_{ij}^{(kl)})^{c_{ij}^{(kl)}} >_{\lex,\rev}
 x_1^{a_1'}\cdots x_n^{a_n'} 
 y_n^{b_n'}\cdots y_1^{b_1'}
 (z_{ij}^{(kl)})^{{c_{ij}^{(kl)}}'}=m'
\] 
  if
\[ (a_1,\ldots, a_n,b_n,\ldots, b_1,c_{ij}^{(kl)}) - 
  (a_1',\ldots, a_n', b_n',\ldots, b_1', c_{ij}^{(kl)'})>0 
\] in the sense that the left-most nonzero coordinate of the difference of the exponent vector is positive. 

We impose any ordering on the $z_{ij}^{(kl)}$ as long as they succeed the ordering on the $x_{\iota}$'s and the $y_{\gamma}$'s; thus, we will view them as constants, which coincide with their weights as imposed in Remark~\ref{weight:weight-on-coordinate-ring}. 
\end{monomial}	

We will write $>$ rather than $>_{\lex,\rev}$ throughout this section. 

\begin{remark}\label{remark:monomial-ordering-no-mention-of-total-ordering} 
Total ordering by total degree in Monomial Ordering~\ref{monomial:grobner-basis-monomial-ordering} does not need to be mentioned since each $F_{\iota}$ is a homogeneous quadratic function. Furthermore, if we want to view $z_{ij}^{(kl)}$ as a rational function (rather than as a constant) and impose an ordering, one may define such ordering by $\dfrac{f_1}{f_2}>\dfrac{g_1}{g_2}$ if $f_1 g_2 > f_2 g_1$. 
%
%
\end{remark} 

\begin{remark}\label{remark:monomial-ordering}
We have imposed lexicographical order on the $x_i$'s, and 
reversed the indices on the $y_i$'s and applied lex on the $y_i$'s
(caution: this is not the same as reverse lex order since that has infinite descending sequences; thus it is not a monomial order), with the ordering on the $x_i$'s preceding the $y_i$'s. Monomial Ordering~\ref{monomial:grobner-basis-monomial-ordering} of  $\mathbb{C}[x_k,y_l,z_{ij}^{(kl)}]$ is multiplicative (i.e., if $m>m'$, then $m\widetilde{m}> m'\widetilde{m}$) and artinian ($m>1$ for all nonunit monomials $m$). 
\end{remark}

\begin{lemma}\label{lemma:initial-term-of-each-F}
With respect to Monomial Ordering~\ref{monomial:grobner-basis-monomial-ordering}, the initial term $\In(F_{\iota})$ of each $F_{\iota}$ equals $x_{\iota}y_{\iota}$.
\end{lemma}

\begin{proof} 
Since each monomial corresponds to a unique exponent vector, 
write the exponents of each monomial of $F_{\iota}$ as a pair of multi-indices $\mathbf{a}=(a_1,\ldots, a_n)$ and $\mathbf{b}=(b_1,\ldots, b_n)$, i.e., 
$\mathbf{a}$ and $\mathbf{b}$ 
may be thought of as column vectors living in $\mathbb{Z}_{\geq 0}^n$ (for the time being, we omit keeping track of $z_{ij}^{(kl)}$). 
It is clear by Corollary~\ref{corollary:diag-fn-defining-matrixvariety-explicit-eqn} that 
both $\mathbf{a}$ and $\mathbf{b}$ are in 
$\{e_1,\ldots, e_n \}$, where $e_i$ is the standard basis vector for $\mathbb{Z}^n$, since exactly one of the exponents for $x_{\gamma}$'s and one of the exponents for  $y_{\nu}$'s are nonzero for each monomial of $F_{\iota}$. Since the degree of each monomial of $F_{\iota}$ is 2, higher powers of $x_{\gamma}$ or $y_{\nu}$ cannot occur. 

Now for a fixed $F_{\iota}$, we inspect the monomials in Corollary~\ref{corollary:diag-fn-defining-matrixvariety-explicit-eqn} to conclude 
the inclusion of sets 
\[\{ \mathbf{a}: \mathbf{a} \mbox{ is the multi-index of some monomial of }F_{\iota} \}\supseteq \{e_{\iota},\ldots, e_{n} \}.  
\] 
 The vector $\mathbf{a}$ corresponding to the first term $x_{\iota}y_{\iota}$ is $e_{\iota}$ while all the other summations show that $\mathbf{a}$ is in $\{ e_{\iota+1},\ldots, e_n\}$.   
When $\mathbf{a} =e_{\iota}$, 
 the possibilities for its corresponding  $\mathbf{b}$-vector 
 take on all values $e_{1},\ldots, e_{\iota}$,  
 which one may check by looking at the monomials in 
 Corollary~\ref{corollary:diag-fn-defining-matrixvariety-explicit-eqn}.  
In order to determine the leading term, 
if $\mathbf{a}=e_{\alpha}$, we want $\alpha$ to be as small as possible 
 since we have imposed lex on the $x_{\gamma}$'s, and  
 if $\mathbf{b}=e_{\beta}$, we want $\beta$ to be as big as possible  
since we have imposed a reverse ordering on the $y_{\nu}$'s. 
Since $x_{\iota}y_{\iota}$ occurs once in $F_{\iota}$ with coefficient 1 with 
$c_{ij}^{(kl)}=0$, the initial term of $F_{\iota}$ is $x_{\iota}y_{\iota}$.  
\end{proof}


\begin{lemma}\label{lemma:initial-term-of-Fiota}
The initial terms of $\{ F_{\iota}\}_{1\leq \iota \leq n}$ form a regular sequence. 
\end{lemma} 

\begin{proof}
This follows from Lemma~\ref{lemma:initial-term-of-each-F}.  
\end{proof} 

\begin{lemma}\label{lemma:deriving-reg-seq-from-initialterms} 
The set $\{ F_{\iota}\}_{1\leq \iota\leq n}$ of functions is $\mathbb{C}[T^*(\mathfrak{b}\times \mathbb{C}^n)^{rss}]$-regular. 
\end{lemma}  

\begin{proof} 
The $F_{\iota}$'s form a regular sequence since their initial terms form a regular sequence by Proposition 15.15 in \cite{MR1322960}.  
\end{proof}

\section{Proof of Proposition \texorpdfstring{$\ref{proposition:rss-locus-to-complex-space-B-invariant}$}{rss isomorphism of varieties}}

The following proves Proposition~\ref{proposition:rss-locus-to-complex-space-B-invariant}. 

\begin{proof}
We have 
\[ 
\begin{aligned}
\mathbb{C}[\mu^{-1}(0)^{rss}] &= \mathbb{C}[T^*(\mathfrak{b}\times\mathbb{C}^n)^{rss}
]/\! \left<\: (\mu(r,s,i,j))_{\iota\gamma} \:\right>  \\ 
&\cong \dfrac{\mathbb{C}[r_{\alpha\beta}, \: x_k, \: y_l ][(r_{\nu\nu}-r_{\gamma\gamma})^{-1}]}{
\left<\: ([r,s]+ij)_{\iota\iota} \:\right> } \otimes \mathbb{C}[s_{11},\ldots, s_{nn}]  \\ 
&= \dfrac{\mathbb{C}[r_{\alpha\beta}, \: x_k, \: y_l ][(r_{\nu\nu}-r_{\gamma\gamma})^{-1}]}{
\left<\: F_{\iota}(r,s,i,j) \: \right> } \otimes \mathbb{C}[s_{11},\ldots, s_{nn}],  
\end{aligned}
\] 
where the second isomorphism holds by Corollary~\ref{corollary:diagonalfunction-matrix-variety} and the third 
equality holds by Corollary~\ref{corollary:diag-fn-defining-matrixvariety-explicit-eqn}. 
The locus 
		$\mu^{-1}(0)^{rss}$ is a complete intersection by Lemma~\ref{lemma:deriving-reg-seq-from-initialterms}. 
By Propositions~\ref{proposition:diagonal-rss-natural-matrix-representation}, \ref{proposition:B-invariant-functions-involving-s},  \ref{proposition:B-invariant-functions-diagonal-entries-of-r}, and \ref{proposition:B-invariant-invert-difference-diag-r}, 
\[ \begin{aligned} 
\mathbb{C}[\mu^{-1}(0)^{rss}]^B &=
   \dfrac{ \mathbb{C}\left[
    F_{\iota}(r,s,i,j),  G_{\iota}(r,s,i,j), H_{\iota}(r,s,i,j)\right]
    \left[ K_{\gamma\nu}(r,s,i,j) 
 \right]}{
\left< F_{\iota}(r,s,i,j)\right> } \\
&\cong  \mathbb{C}\left[r_{11},\ldots, r_{nn}, s_{11}',\ldots,s_{nn}'  
\right]\left[(r_{\nu\nu}-r_{\gamma\gamma})^{-1} \right] \\   
&\cong \mathbb{C}[\mathbb{C}^{2n}\setminus \Delta_n]. 
\end{aligned}
\] 
\end{proof}

\section{Extension to a general flag}

The results in this paper may be generalized in the following way. 
Let $F: V_0 = \{ 0\}\subsetneq V_{\iota_1}\subsetneq \ldots \subsetneq V_{\iota_k}=\mathbb{C}^n$ be a flag with $\dim(V_{\iota_l})=d_{\iota_l}$ for each $1\leq l\leq k$. Let $n_{\iota_l}=d_{\iota_l}-d_{\iota_{l-1}}$ with $d_{\iota_0}:=0$. 
 Let  $P$ be the corresponding parabolic group fixing $F$, and let $\mathfrak{p}=\lie(P)$.

Analogous to the construction for the complete flag, consider $T^*(\mathfrak{p}\times\mathbb{C}^n)\stackrel{\mu_P}{\longrightarrow} \mathfrak{p}^*$ where $(r,s,i,j)\mapsto \ad_r^*(s) + \overline{ij}$. 
Then the Hamiltonian reduction $\mu_P^{-1}(0)/P$ is isomorphic to $T^*(G\times_P \mathfrak{p}\times \mathbb{C}^n/G  )$ where $G$ is the general linear group over $\mathbb{C}$. 
Using Notations~\ref{notation:distinct-r-eigenvalues} and \ref{notation:definition-of-Delta-n-set}, 
we believe that 
$\mu_P^{-1}(0)^{rss}/\!\!/P$ is scheme-theoretically isomorphic to the dense open subset $\mathbb{C}^{2n}\setminus \Delta_n$. 
Laying out the key ideas for such a proof,  we first view the parabolic subgroup $P$ whose block diagonal matrices are in 
$GL(n_{\iota_l},\mathbb{C})$, where $1\leq l\leq k$. 
There exists an element in $GL(n_{\iota_l},\mathbb{C})$ putting the corresponding entries $\End(V_{\iota_l})$ in $\mathfrak{p}$ 
 into an upper triangular form, i.e., $g_{n_{\iota_l}}r_{n_{\iota_l}}g_{n_{\iota_l}}^{-1}$ is upper triangular (this amounts to imposing a finer refinement of the fixed flag by ordering and choosing a basis). Putting 
\[ p = \left[ \begin{array}{ccccc}
g_{n_{\iota_1}} & *      & * & *  & *   \\
                & \ddots & * & *  &  *  \\ 
                &        & g_{n_{\iota_l}}& * & * \\ 
                &        & & \ddots & * \\ 
                &        & & & g_{n_{\iota_k}}  \\
\end{array} 
\right] \in P,  
 \] 
   $prp^{-1}$ is now in $\mathfrak{b}^{rss}$, and we are now working with $(prp^{-1},psp^{-1},pi,jp^{-1}) \in T^*(\mathfrak{b}\times \mathbb{C}^n)$.  
 Modify $b\in B$ in Proposition~\ref{proposition:distinct-ev-diagonalizable} appropriately (replace all coordinates in $r$ to those in $prp^{-1}$) to obtain that as $b$ diagonalizes $prp^{-1}$, we simultaneously obtain $P$-invariant functions coming from $psp^{-1}$. 
We believe that an analogous form of traces will appear as $P$-invariants, and by mimicking the above proof, we may conclude that $\mu_P^{-1}(0)^{rss}/\!\!/P\cong \mathbb{C}^{2n}\setminus \Delta_n$.

\chapter{Future direction and statements of conjectures}\label{chapter:future-direction}

\section{Reflection functors for filtered quiver varieties}

Relating the category of filtered quiver representations and the category of associated graded quiver representations remains an open problem. 
Furthermore, 
let $X$ be a general representation in $F^{\bullet}Rep(Q,\beta)$ and let $S_i^+$ and $S_i^-$ be reflection functors at vertex $i\in Q_0$. 
Analogous to the classical setting, 
it would be interesting to find appropriate conditions such that 
 $S_i^+ \gr X \cong \gr S_i^+ X$ and $S_i^- \gr X \cong \gr S_i^- X$ for any $i\in Q_0$.   
Furthermore, it is interesting to construct a new filtration  
$\sigma_iF^{\bullet}$ of $F^{\bullet}$ such that 
$\widehat{D}:F^{\bullet}Rep(Q,\beta)\rightarrow (\sigma_iF^{\bullet})Rep(Q^{op},\beta')$ is an isomorphism of varieties, where $\beta'$ need not equal $\beta$,  
and if $i$ is a sink, construct a map 
$S_i^+:F^{\bullet}Rep(Q,\beta)\rightarrow (\sigma_iF^{\bullet})Rep(\sigma_iQ,\sigma_i\beta)$ such that 
$\widehat{D}\circ S_i^+=S_i^-\circ \widehat{D}$ (cf. Theorem~\ref{theorem:wolf-reflection-functors}).

\section{Semi-invariants for generalized filtered quiver varieties}\label{section:semi-invariants-generalized-filtered-quiver-vars} 

\subsection{\texorpdfstring{$m$}{m}-Jordan quiver}\label{subsection:m-Jordan-quiver}

Consider the $m$-Jordan quiver $Q$ where $m\geq 2$: 
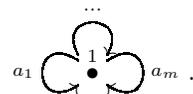
\[\xymatrix@-1pc{ \ar@(lu,ld)_{a_1} \stackrel{1}{\bullet}\ar@(lu,ru)^{\cdots} \ar@(ru,rd)^{a_m} \\ 
}.
\]  
The first question we would like to ask is describe   $\mathbb{C}[F^{\bullet}Rep(Q,\beta)]^{\mathbb{U}_{\beta}}$.  
In the case when   
$\beta=n$ and $F^{\bullet}$ is the complete standard filtration of vector spaces over vertex $1$, 
we believe that 
$\mathbb{C}[\mathfrak{b}^{\oplus m}]^{\mathbb{U}_{\beta}}$ is generated by polynomials of degrees $d$, for each $1\leq d\leq n$. 
We refer the reader to Section~\ref{section:comment-regarding-at-most-two-paths} for the example when $Q$ is the $2$-Jordan quiver $\xymatrix@-1pc{ \stackrel{1}{\bullet} \ar@(dl,ul)^{a_1} \ar@(ur,dr)^{a_2}  
}$  
and $\beta=2$. In that example, we produce $5$ algebraically independent polynomials invariant under the $U$-action,  
where $U\subseteq B\subseteq GL_2(\mathbb{C})$: four polynomials of degree $1$ and one polynomial of degree $2$.     
If $\mathbb{C}[\mathfrak{b}^{\oplus m}]^{\mathbb{U}_{\beta}}$ is generated by polynomials of degree $d$,  
where $1\leq d\leq n$,  
then the upper bound on the degrees of invariant polynomial generators is $n$. 
After writing down the relations among the generators, it is natural to identify the quotient space 
$\mathfrak{b}^{\oplus m}/\!\!/U$ with well-known algebraic spaces. 
We believe that for all $m\geq 1$, 
\[ 
\dim_{\mathbb{C}} \mathfrak{b}^{\oplus m}/\!\!/U = \dim_{\mathbb{C}}\mathfrak{b}^{\oplus m}-\dim_{\mathbb{C}} U.   
\] 
Finally, for any two characters $\chi$ and $\chi'$ of the standard Borel $B\subseteq GL_n(\mathbb{C})$, 
it would be interesting to construct birational maps 
\[ 
\xymatrix@-1pc{
\Proj(\bigoplus_{i\geq 0} \mathbb{C}[\mathfrak{b}^{\oplus m}]^{B,\chi^i})\ar@{-->}[rrr] & & &   
\Proj(\bigoplus_{i\geq 0} \mathbb{C}[\mathfrak{b}^{\oplus m}]^{B,{\chi'}^i})   \\ 
}
\]  and study the birational geometry of GIT quotients.

\subsection{\texorpdfstring{$k$}{k}-Kronecker quiver}\label{subsection:k-Kronecker-quiver} 

Consider the $k$-Kronecker quiver where $k\geq 3$: 
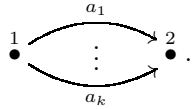
\[ 
\xymatrix@-1pc{
\stackrel{1}{\bullet} \ar@/^1pc/[rrr]^{a_1}_{\vdots} \ar@/_1pc/[rrr]_{a_k} & & & 
\stackrel{2}{\bullet}.  \\  
}  
\]  
We refer the reader to Section~\ref{section:comment-regarding-at-most-two-paths}  
for the example of the $3$-Kronecker quiver when $\beta=(3,3)$ and $F^{\bullet}$ is the complete standard filtration of vector spaces at each vertex.   
For the $k$-Kronecker quiver, 
a problem we would like to suggest is give a complete description of  $\mathbb{C}[F^{\bullet}Rep(Q,\beta)]^{\mathbb{U}_{\beta}}$.  
In the case when   
$\beta=(n,n)$ and $F^{\bullet}$ is the complete standard filtration of vector spaces, 
then we believe that $\mathbb{C}[\mathfrak{b}^{\oplus k}]^{U\times U}$ 
is generated by polynomials of degrees $\dfrac{d(d+1)}{2}$, for each $1\leq d\leq n$; 
if this is true, then the upper bound on the degrees of invariant polynomial generators is $n(n+1)/2$. 
Next, after finding all generators, 
write down the relations among the generators and identify the quotient space $\mathfrak{b}^{\oplus k}/\!\!/U\times U$ 
with well-known algebraic spaces. For $k>1$, we believe that 
\[ 
\dim_{\mathbb{C}} \mathfrak{b}^{\oplus k}/\!\!/U\times U =\dim_{\mathbb{C}} \mathfrak{b}^{\oplus k}-\dim_{\mathbb{C}}  U\times U,
\] 
and finally, for any two characters $\chi$ and $\chi'$ of $B^2$, it would be interesting to study birational maps 
\[ 
\xymatrix@-1pc{
\Proj(\bigoplus_{i\geq 0}\mathbb{C}[\mathfrak{b}^{\oplus k}]^{B\times B,\chi^i})\ar@{-->}[rrr] & & & 
\Proj(\bigoplus_{i\geq 0}\mathbb{C}[\mathfrak{b}^{\oplus k}]^{B\times B,{\chi'}^i}) \\
}  
\]   
to understand the birational geometry of GIT quotients.

\subsection{Generalized dimension vector, generalized flags, and quivers with more than \texorpdfstring{$2$}{2} pathways}\label{subsection:generalized-quiver-varieties}  

We refer the reader to Definition~\ref{definition:pathway-of-quiver} for the definition of a pathway at a vertex. 
It remains an open problem to vary the dimension vector $\beta\in \mathbb{Z}_{\geq 0}^{Q_0}$ and the filtration of vector spaces $F^{\bullet}$ 
to study the natural $\mathbb{P}_{\beta}$-action on $F^{\bullet}Rep(Q,\beta)$,  
where some vertices of $Q$ may have more than $2$ pathways.  

\begin{example} 
Consider an $A_3$-Dynkin quiver: $\xymatrix@-1pc{
\stackrel{1}{\bullet}\ar[rr]^{a_1} & & \stackrel{2}{\bullet} \ar[rr]^{a_2} & & \stackrel{3}{\bullet} }$
and consider $\beta=(3,3,3)$. Let $\gamma=(1,2,1)$, 
where $\mathbb{C}^{\gamma_i}$ is the space spanned by the first $\gamma_i$ standard basis vectors. 
Then a general representation $W$ in $F^{\bullet}Rep(Q,\beta)$ is of the form 
\[ 
W(a_1) = 
\begin{pmatrix} 
a_{11} & a_{12} & a_{13} \\ 
a_{21} & a_{22} & a_{23} \\ 
 0     & a_{32} & a_{33}  
\end{pmatrix}
\mbox{ and } 
W(a_2) = 
\begin{pmatrix}
c_{11} & c_{12} & c_{13} \\ 
  0    &   0    & c_{23} \\ 
  0    &   0    & c_{33}    
\end{pmatrix}.  
\] 
Under $\mathbb{P}_{\beta}=P_1\times P_2\times P_3$-action, where a general point in $\mathbb{P}_{\beta}$ has coordinates 
\[ p_{\beta_1} = 
\begin{pmatrix}
p_{11} & p_{12} & p_{13} \\ 
   0   & p_{22} & p_{23} \\ 
   0   & p_{32} & p_{33}  
\end{pmatrix}\in P_1,  \hspace{2mm}
p_{\beta_2} =  
\begin{pmatrix} 
q_{11} & q_{12} & q_{13} \\ 
q_{21} & q_{22} & q_{23} \\ 
   0   &   0    & q_{33}  
\end{pmatrix} \in P_2,  
\mbox{ and }  \hspace{2mm}
p_{\beta_3} =  
\begin{pmatrix}  
r_{11} & r_{12} & r_{13} \\ 
  0    & r_{22} & r_{23} \\ 
  0    & r_{32} & r_{33}  
\end{pmatrix}\in P_3,   
\] 
we believe that 
$\displaystyle{\bigoplus_{\chi}} \mathbb{C}[F^{\bullet}Rep(Q,\beta)]^{\mathbb{P}_{\beta},\chi} \cong  \mathbb{C}[\det(W(a_1))]$. 
\end{example}  

\section{Moment maps and GIT for filtered quiver varieties}\label{section:moment-map-filtered-quiver-varieties-open-problems} 

\subsection{\texorpdfstring{$B$}{B}-moment map}\label{subsection:B-moment-map-open-problems} 

Returning to the moment map $T^*(\mathfrak{b}\times \mathbb{C}^n)\stackrel{\mu_B}{\longrightarrow}\mathfrak{b}^*$, where $\mu_B:(r,s,i,j)\mapsto [r,s]+ij$, 
it would be interesting to identify the affine quotient $\mu_B^{-1}(0)/\!\!/B$ of the entire locus with other varieties.  
In this section, we will show the calculations to $\mu_B^{-1}(0)/\!\!/B$ for $n=2$ since $\mu_B^{-1}(0)$ is a complete intersection when $n=2$. So consider 
\[ 
\mu_B(r,s,i,j)=\left[ 
\begin{array}{cc}  
r_{12}s_{21}+x_1y_1 & * \\ 
(r_{22}-r_{11})s_{21} +x_2 y_1 & -r_{12}s_{21}+x_2 y_2 \\ 
\end{array}\right] \in \mathfrak{b}^* = \mathfrak{gl}_2/\mathfrak{u}^+. 
\] 
By Theorem~\ref{theorem:two-paths-max-quiver-semi-invariants-framed}, 
classical techniques are applicable for this $B$-action on $\mathfrak{b}\times \mathbb{C}^n$ setting (since the framed $1$-Jordan quiver has at most $2$ pathways at each vertex).  
So using Domokos-Zubkov's technique (Section~\ref{subsection:Domokos-Zubkov-technique}) and omiting those polynomials that vanish mod the ideal  
$\langle r_{12}s_{21}+x_1y_1, -r_{12}s_{21}+x_2 y_2, (r_{22}-r_{11})s_{21} +x_2 y_1 \rangle$, 
the following are $B$-invariant generators on $T^*(\mathfrak{b}\times \mathbb{C}^2)$: 
\[ 
\begin{aligned} 
&r_{11},  \hspace{2mm} r_{22},  \hspace{2mm} s_{11}+s_{22} = \Tr(s), \hspace{2mm} x_1y_1+x_2y_2 = \Tr(ij),  \\ 
f &= (r_{11}-r_{22})s_{11} +r_{12}s_{21}, \hspace{2mm}  
g = (r_{11}-r_{22})s_{22}-r_{12}s_{21},    \\ 
h &= r_{11}s_{11}+r_{22}s_{22}+r_{12}s_{21},\hspace{2mm}  
k = r_{11}s_{22}+r_{22}s_{11}-r_{12}s_{21}.    \\ 
\end{aligned} 
\] 
Since 
\[   
\begin{aligned}  
\mathbb{C}[\mu_B^{-1}(0)] = \dfrac{\mathbb{C}[T^*(\mathfrak{b}\times \mathbb{C}^2)]}{
I(\mu_B^{-1}(0))}  
&= \dfrac{\mathbb{C}[T^*(\mathfrak{b}\times \mathbb{C}^2)]}{
\langle r_{12}s_{21}+x_1 y_1, -r_{12}s_{21}+x_2 y_2, (r_{22}-r_{11})s_{21}+x_2 y_1\rangle} \\ 
&=  \dfrac{\mathbb{C}[T^*(\mathfrak{b}\times \mathbb{C}^2)]}{
\langle r_{12}s_{21}+x_1 y_1, x_1y_1 + x_2 y_2, (r_{22}-r_{11})s_{21}+x_2 y_1\rangle},  \\ 
\end{aligned} 
\] 
\[ 
\begin{aligned}
\mathbb{C}[\mu_B^{-1}(0)]^B 
&=  \mathbb{C}[\overline{r_{11}},\overline{r_{22}},\overline{s_{11}+s_{22}},\overline{x_1y_1+x_2y_2}, \overline{f},\overline{g},\overline{h},\overline{k}] \hspace{4mm}\mbox{ where }\overline{f}:=f\mod I(\mu_B^{-1}(0)) \\ 
&= \mathbb{C}[\overline{r_{11}},\overline{r_{22}},\overline{s_{11}+s_{22}}, 
\overline{f},\overline{g},\overline{h},\overline{k}]  \hspace{4mm} \mbox{ since } x_1y_1+x_2y_2\in  I(\mu_B^{-1}(0)) \\ 
&= \mathbb{C}[\overline{r_{11}},\overline{r_{22}},\overline{s_{11}+s_{22}}, \overline{f},\overline{g},\overline{h}]  \hspace{4mm} \mbox{ since } \overline{f}-\overline{g}-\overline{h}+\overline{k} = 0 \\ 
&= \mathbb{C}[\overline{r_{11}},\overline{r_{22}},\overline{s_{11}+s_{22}}, \overline{f},\overline{g}]  \hspace{4mm} 
\mbox{ since } \overline{g}-\overline{r_{11}}(\overline{s_{11}}+\overline{s_{22}})+\overline{h} = 0 \\ 
&= \dfrac{\mathbb{C}[R_1,R_2,T,S_1,S_2]}{
\langle 
R_1 S_1 + R_1 S_2 + R_1 R_2 T - R_1^2 T  
\rangle 
} 
= \dfrac{\mathbb{C}[R_1,R_2,T,S_1,S_2]}{
\langle R_1 (S_1 + S_2 + (R_2 - R_1) T) \rangle. 
} 
 \\ 
\end{aligned}
\]

In Remark~\ref{remark:hilbert-scheme-of-n-points-symmetric-product}, we write down the details to  
$\mu_G^{-1}(0)/\!\!/G$ for the classical setting (cf., \cite{MR2210660}, \cite{MR1711344})  
when $n=2$ since it is interesting to relate it to the affine quotient 
$\mu_B^{-1}(0)/\!\!/B$. 

\begin{remark}[$G$-moment map]\label{remark:hilbert-scheme-of-n-points-symmetric-product} 
Consider the $G:=GL_n(\mathbb{C})$-adjoint action on $\mathfrak{gl}_n\times \mathbb{C}^n$, i.e., $g.(r,i)=(grg^{-1},gi)$, which induces the moment map: 
\[ \mu_G:T^*(\mathfrak{gl}_n \times \mathbb{C}^n)\longrightarrow \mathfrak{g}^*, \hspace{4mm} \mu_G:(r,s,i,j)\mapsto [r,s]+ij. 
\] 
Using coordinates $\mathbb{C}[\mathbb{C}^2 \times \mathbb{C}^2]=\mathbb{C}[r_{11},r_{22},s_{11},s_{22}]$,  
we have 
$\mu_G^{-1}(0)/\!\!/G\cong S^2\mathbb{C}^2$ by Nakajima (see also Theorem~\ref{theorem:nakajima-git-and-affine-quotient-identification}), with 
\[ \begin{aligned} 
\mathbb{C}[S^2\mathbb{C}^2] &= \mathbb{C}[\mathbb{C}^2\times \mathbb{C}^2/S_2] =\mathbb{C}[\mathbb{C}^2\times \mathbb{C}^2]^{S_2} \\ 
&=\mathbb{C}[r_{11}+r_{22}, s_{11}+s_{22}, r_{11}^2+r_{22}^2, s_{11}^2+s_{22}^2, 
r_{11}s_{11}+r_{22}s_{22}] \\ 
&= \dfrac{\mathbb{C}[R_1,S_1,R_2,S_2,T]}{
\langle  
-2T^2+2R_1S_1T+2R_2S_2-S_2R_1^2-R_2S_1^2 
\rangle}.   
\end{aligned}  
\] 
So $S^2\mathbb{C}^2$ is a singular $4$ dimensional variety sitting in $\mathbb{C}^5$. 
\end{remark} 

Next, it is worth pointing out that it is interesting to study the singular locus of $\mu_B^{-1}(0)/\!\!/B$. 
Note that the singular locus of $\mu_B^{-1}(0)/\!\!/B$ is precisely the points where all the partials of 
$\phi$ vanish, where  
$\phi$ $=$ $R_1 S_1 + R_1 S_2 + R_1 R_2 T - R_1^2 T$. 
So 
\[ 
\begin{aligned} 
\phi_{R_1} &= S_1+S_2+R_2T - 2 R_1 T,  \hspace{4mm}
\phi_{R_2} = R_1 T,     \\ 
\phi_T 		&= R_1 (R_2 -R_1),  \hspace{4mm}
\phi_{S_1} = R_1, \hspace{4mm}
\phi_{S_2} = R_1,     \\ 
\end{aligned} 
\] 
imply the singular locus of the affine quotient consists of the points where 
$R_1=0$ and $S_1 + S_2 + R_2 T = 0$. 
In terms of the coordinates $r_{ii}$ and $s_{ii}$, $R_1=r_{11}=0$ while 
\[ 
S_1+S_2+R_2T = 
(r_{11}-r_{22})s_{11} + r_{12}s_{21} + (r_{11}-r_{22})s_{22}-r_{12}s_{21} +r_{22}(s_{11}+s_{22})=0, 
\] 
which simplifies as 
$(r_{11}-r_{22})(s_{11}+s_{22}) + r_{22}(s_{11}+s_{22})=0$. 
So the singular locus is when $r_{11}=0$.

\begin{remark}[Singular locus of the $G$-moment map]
Since 
\[ 
\mathbb{C}[\mu_G^{-1}(0)]^G 
= 
\dfrac{\mathbb{C}[R_1,S_1,R_2,S_2,T]}{
\langle  
-2T^2+2R_1S_1T+2R_2S_2-S_2R_1^2-R_2S_1^2 
\rangle}, 
\]  
we will calculate the partials of $\psi$ $=$ $-2T^2+2R_1S_1T+2R_2S_2-S_2R_1^2-R_2S_1^2$: 
\[   
\begin{aligned} 
\psi_{R_1} &= 2 S_1 T - 2S_2 R_1, \hspace{4mm}
\psi_{S_1} = 2 R_1 T -2 R_2 S_1,  \\ 
\psi_{R_2} &= 2 S_2 - S_1^2, \hspace{4mm}
\psi_{S_2} = 2 R_2 - R_1^2, \hspace{4mm}
\psi_T = 2 R_1 S_1 -4T. 
\end{aligned}
\]   
Singular locus is precisely when $S_2=S_1^2/2$, $R_2=R_1^2/2$, and $T=R_1S_1/2$. 
In terms of $r_{ii}$ and $s_{ii}$, 
we have 
$S_1^2/2=(s_{11}+s_{22})^2/2 = s_{11}^2+ s_{22}^2=S_2$. 
This gives 
$s_{11}^2+s_{22}^2+2s_{11}s_{22}=2s_{11}^2 + 2s_{22}^2$, which simplifies as 
$s_{11}^2-2s_{11}s_{22}+s_{22}^2=(s_{11}-s_{22})^2 = 0$. 
We also have 
$R_1^2/2=(r_{11}+r_{22})^2/2 = r_{11}^2+ r_{22}^2=R_2$. 
This gives 
$r_{11}^2+r_{22}^2+2r_{11}r_{22} = 2r_{11}^2 + 2r_{22}^2$, which simplifies as 
$r_{11}^2-2r_{11}r_{22}+r_{22}^2=(r_{11}-r_{22})^2 = 0$. 
Lastly, 
we have 
\[ \dfrac{R_1S_1}{2}=\dfrac{(r_{11}+r_{22})(s_{11}+s_{22})}{2} = r_{11}s_{11}+r_{22}s_{22} = T. 
\] 
Expanding the left-hand side and moving all terms to one side, we have 
$r_{11}s_{11}-r_{22}s_{11}-r_{11}s_{22}+r_{22}s_{22} = (r_{11}-r_{22})(s_{11}-s_{22})=0$. 
Thus the singular locus is precisely the points when $r_{11}=r_{22}$ and $s_{11}=s_{22}$. 
Geometrically, this is precisely the locus when the two points $(r_{11},s_{11})$ and $(r_{22},s_{22})$ on $\mathbb{C}^2$ are the same. 
\end{remark}

Next, we discuss the irreducible components of $\mu_B^{-1}(0)$. 
It is given in \cite{Nevins-GSresolutions} that $\mu_B^{-1}(0)$ has at least $2^n$ irreducible components. 
If the moment map $\mu_B$ is a complete intersection, then $\mu_B^{-1}(0)$ has at most $2^n$ irreducible components corresponding to the vector $i=e_k$ or the covector $j=e_k^*$, 
where $e_k$ ($e_k^*$) is an elementary vector (covector), $1\leq k\leq n$. 
After having shown that the generators of $I(\mu_B^{-1}(0))$ form a regular sequence,  
we may try to single out an irreducible component with a stability condition. 
That is, 
given characters $\chi$ and $\chi'$ of the standard Borel, 
construct the morphisms 
\[ 
\xymatrix@-1pc{
 & & \ar[dll]  \mu_B^{-1}(0)  \ar[drr] & & \\ 
\mu_B^{-1}(0)/\!\!/_{\chi}B \ar@{-->}[rrrr]^{\psi_{\chi,\chi'}} \ar[drr]& & & & \mu_B^{-1}(0)/\!\!/_{\chi'} B \ar[dll]  \\ 
   & & \mu_B^{-1}(0)/\!\!/B & & \\ 
}
\] 
and then fix $\chi(b)=\det(b)$ and relate the map 
$\mu_B^{-1}(0)/\!\!/_{\det}B\longrightarrow \mu_B^{-1}(0)/\!\!/B$  to the Hilbert-Chow morphism $(\mathbb{C}^2)^{[n]}\stackrel{HC}{\longrightarrow}S^n\mathbb{C}^2$; $\psi_{\chi,\chi'}$ is known as wall-crossing from one chamber to another. 
Since we have found at least one semi-invariant polynomial for each character $\chi$ of the Borel, 
this means  
$\mu_B^{-1}(0)/\!\!/_{\chi} B$ can never be the empty scheme.

\subsection{\texorpdfstring{$\mathbb{P}_{\beta}$}{P}-moment map}\label{subsection:P-moment-map-filtered-quiver-varieties}

Let $Q$ be any (framed) quiver and let $F^{\bullet}$ be any filtration of vector spaces. Let $\beta\in \mathbb{Z}_{\geq 0}^{Q_0}$  
be a dimension vector. 
Let $\mathbb{G}_{\beta} = \displaystyle{\prod_{i\in Q_0} GL_{\beta_i}(\mathbb{C})}$
and 
$\mathbb{P}_{\beta} = \displaystyle{\prod_{i\in Q_0} P_{\beta_i}}$, $P_{\beta_i}\subseteq GL_{\beta_i}(\mathbb{C})$, 
where $\mathbb{G}_{\beta}$ acts on $Rep(Q,\beta)$ and $\mathbb{P}_{\beta}$ acts on $F^{\bullet}Rep(Q,\beta)$ as a change-of-basis. 
Let $\widetilde{Rep(Q,\beta)}:=\{(W,F^{\bullet}Rep(Q,\beta))\in Rep(Q,\beta)\times \mathbb{G}_{\beta}/\mathbb{P}_{\beta}:W\in F^{\bullet}Rep(Q,\beta) \}$ 
be the {\em generalized Grothendieck-Springer space of pairs}  
and let $V$ be an $m$-dimensional complex vector space. 
Then 
$\mathbb{G}_{\beta} \times_{\mathbb{P}_{\beta}}F^{\bullet}Rep(Q,\beta)\cong \widetilde{Rep(Q,\beta)}$, 
$\widetilde{Rep(Q,\beta)}/\mathbb{G}_{\beta}\cong F^{\bullet}Rep(Q,\beta)/\mathbb{P}_{\beta}$ as orbit spaces, and 
$\mathbb{G}_{\beta}$-orbits on the cotangent bundle $T^*(\widetilde{Rep(Q,\beta)}\times V)$ of the extended generalized Grothendieck-Springer resolution correspond to 
$\mathbb{P}_{\beta}$-orbits on the cotangent bundle $T^*(F^{\bullet}Rep(Q,\beta)\times V)$ of the extended filtered quiver variety. 
This motivates us to generalize Section~\ref{subsection:B-moment-map-open-problems} and study the Hamiltonian reduction of the $\mathbb{P}_{\beta}$-equivariant moment map 
\[ 
\xymatrix@-1pc{
T^*(F^{\bullet}Rep(Q,\beta)\times V) \ar[rrr]^{\mu_{\mathbb{P}_{\beta}}} & & & \lie(\mathbb{P}_{\beta})^* 
}
\] 
using GIT and symplectic geometry techniques,  
where $\mu_{\mathbb{P}_{\beta}}$ is the moment map for $\mathbb{P}_{\beta}$-action on $F^{\bullet}Rep(Q,\beta)$. 
After determining the criteria for which $\mu_{\mathbb{P}_{\beta}}^{-1}(0)$ is a complete intersection,  
find the semi-stable locus $\mu_{\mathbb{P}_{\beta}}^{-1}(0)^{ss}$ 
 and analyze when 
 $\mu_{\mathbb{P}_{\beta}}^{-1}(0)^{ss}\rightarrow \mu_{\mathbb{P}_{\beta}}^{-1}(0)/\!\!/_{\chi}\mathbb{P}_{\beta} = 
 \displaystyle{\Proj(\bigoplus_{i\geq 0}\mathbb{C}[\mu_{\mathbb{P}_{\beta}}^{-1}(0)]^{\mathbb{P}_{\beta},\chi^i})}$ 
 is surjective for we know that projectivity and the existence of a proper map follows from classical GIT in the reductive setting.  
 Construct the maps  
 \[
 \xymatrix@-1pc{
 & & \mu_{\mathbb{P}_{\beta}}^{-1}(0)^{ss} \ar[dll] \ar[drr] & & \\ 
 \mu_{\mathbb{P}_{\beta}}^{-1}(0)/\!\!/_{\chi}\mathbb{P}_{\beta} 
 \ar[drr]\ar@{-->}[rrrr]^{\psi_{\chi,\chi'}} & & & & \mu_{\mathbb{P}_{\beta}}^{-1}(0)/\!\!/_{\chi'}\mathbb{P}_{\beta} \ar[dll] \\
 & & \mu_{\mathbb{P}_{\beta}}^{-1}(0)/\!\!/\mathbb{P}_{\beta} =\spec(\mathbb{C}[\mu_{\mathbb{P}_{\beta}}^{-1}(0)]^{\mathbb{P}_{\beta}})   & & \\ 
 } 
 \]   
 for any two characters $\chi$ and $\chi'$ of $\mathbb{P}_{\beta}$. 
Note that the moduli space parameterizes isomorphic classes of $\mathbb{P}_{\beta}$-orbits of $\chi$ (or $\chi'$)-semi-stable representations of $Q$ with dimension vector $\beta$. 

Finally, it may be interesting to determine when the space $\mathcal{M}_{\chi}^{F^{\bullet}}(Q,\beta):=F^{\bullet}Rep(Q,\beta)/\!\!/_{\chi}\mathbb{P}_{\beta}$ 
is a projective variety, and if $\beta$ is indivisible, i.e., it is not a nontrivial multiple of another integer vector,  
the stable locus 
$\mathcal{M}_{\chi}^{F^{\bullet}}(Q,\beta)^s$ 
is a fine moduli space for families of 
$\chi$-stable representations. 
Finally, stating a remark from \cite{MR1315461} (Remark 5.4), 
if $\beta$ is indivisible, then the set of generic characters  
(cf. Section~\ref{subsection:moduli-spaces})  
is dense, and for any such character $\chi$, we believe that  
$\mathcal{M}_{\chi}^{F^{\bullet}}(Q,\beta)^s  
= \mathcal{M}_{\chi}^{F^{\bullet}}(Q,\beta)$ is smooth.




\appendix*

\include{Appendix.tex}

\backmatter

\bibliographystyle{amsalpha} 

\setcounter{tocdepth}{2}

\bibliography{inv-and-semi-inv-of-all-filtered-quivers}

%

\end{document}